\documentclass[reqno,10pt]{amsart}
\usepackage{amsmath,mathrsfs}
\usepackage{amssymb, amsmath}
\usepackage{graphicx}
\usepackage{pdfsync}
\usepackage{color}
\usepackage[colorlinks,
            linkcolor=red,
            anchorcolor=blue,
            citecolor=green
            ]{hyperref}

\def\inte#1{
\displaystyle\mathop{#1\kern0pt}^\circ }

\let\pa=\partial

\let\f=\frac

\def\pa{\partial}

\def\virgp{\raise 2pt\hbox{,}}
\def\cdotpv{\raise 2pt\hbox{;}}

\def\C{\mathop{\mathbb C\kern 0pt}\nolimits}
\def\DD{\mathop{\mathbb D\kern 0pt}\nolimits}
\def\EE{\mathop{{\mathbb E \kern 0pt}}\nolimits}
\def\K{\mathop{\mathbb K\kern 0pt}\nolimits}
\def\N{\mathop{\mathbb N\kern 0pt}\nolimits}
\def\Q{\mathop{\mathbb Q\kern 0pt}\nolimits}
\def\R{\mathop{\mathbb R\kern 0pt}\nolimits}
\def\SS{\mathop{\mathbb S\kern 0pt}\nolimits}
\def\ZZ{\mathop{\mathbb Z\kern 0pt}\nolimits}
\def\TT{\mathop{\mathbb T\kern 0pt}\nolimits}
\def\P{\mathop{\mathbb P\kern 0pt}\nolimits}

\newcommand{\Z}{{\ZZ}}

\def\na{\nabla}

\def\th{\theta}

\newcommand{\beq}{\begin{equation}}
\newcommand{\eeq}{\end{equation}}
\newcommand{\ben}{\begin{eqnarray}}
\newcommand{\een}{\end{eqnarray}}
\newcommand{\beno}{\begin{eqnarray*}}
\newcommand{\eeno}{\end{eqnarray*}}

\newtheorem{defi}{Definition}[section]
\newtheorem{thm}{Theorem}[section]
\newtheorem{lem}{Lemma}[section]
\newtheorem{rmk}{Remark}[section]
\newtheorem{col}{Corollary}[section]
\newtheorem{prop}{Proposition}[section]
\renewcommand{\theequation}{\thesection.\arabic{equation}}

\begin{document}
\title[Boltzmann equation with Rutherford scattering cross section]
{Boltzmann equation with cutoff Rutherford scattering cross section near Maxwellian}

\author[L.-B. He and Y.-L. Zhou]{Ling-Bing He and Yu-long Zhou}
\address[L.-B. He]{Department of Mathematical Sciences, Tsinghua University\\
Beijing, 100084,  P. R.  China.} \email{hlb@tsinghua.edu.cn}
\address[Y.-L. Zhou]{School of Mathematics, Sun Yat-Sen University, Guangzhou, 510275, P. R.  China.} \email{zhouyulong@mail.sysu.edu.cn}

\begin{abstract}  The well-known Rutherford differential  cross section, denoted by $ d\Omega/d\sigma$, corresponds to a two body interaction with  Coulomb potential.   It leads to the logarithmically divergence of the momentum transfer (or the transport cross section) which is described by $$\int_{\SS^2} (1-\cos\theta) \frac{d\Omega}{d\sigma} d\sigma\sim \int_0^{\pi} \theta^{-1}d\theta. $$ Here $\theta$ is the deviation angle in the scattering event. Due to screening effect, physically one can assume that $\theta_{\min}$ is the order of magnitude of the smallest angles for which the scattering can still be regarded as Coulomb scattering. Under ad hoc cutoff $\theta \geq \theta_{\min}$ on the deviation angle, L. D. Landau derived a new equation in  \cite{landau1936transport} for the weakly interacting gas which is now referred to as the Fokker-Planck-Landau or Landau equation. In the present work, we establish a unified framework to  justify  Landau's formal derivation in  \cite{landau1936transport} and the so-called Landau approximation problem proposed in \cite{alexandre2004landau} in the close-to-equilibrium regime. Precisely,
(i). we prove global well-posedness of the Boltzmann equation with cutoff Rutherford  cross section which is perhaps the most singular kernel both in relative velocity and deviation angle.
(ii). we prove a global-in-time error estimate between solutions to Boltzmann and Landau equations with logarithm accuracy, which is consistent with the famous Coulomb logarithm.
Key ingredients into the proofs of these results include a complete coercivity estimate of the linearized Boltzmann collision operator, a uniform spectral gap estimate and a novel linear-quasilinear method.
\end{abstract}

\maketitle
\setcounter{tocdepth}{1}
\tableofcontents




\noindent {\sl AMS Subject Classification (2010):} {35Q20, 35R11, 75P05.}

\renewcommand{\theequation}{\thesection.\arabic{equation}}
\setcounter{equation}{0}


\section{Introduction}
The present work aims at the mathematical justification of Landau's   derivation of the Landau equation and the Landau's approximation problem from the Boltzmann equation with angular cutoff Rutherford scattering cross section. These problems have a long history and we first recall relevant physical backgrounds.


\subsection{Review of Landau's    derivation} In 1936, Landau published his paper \cite{landau1936transport} on the derivation of the effective equation for weakly interaction by Coulomb field in plasma physics. Loosely speaking, he derived Landau equation from the Boltzmann equation with cutoff Rutherford cross section.

\subsubsection{Boltzmann equation with cutoff Rutherford cross section} The typical Boltzmann equation can be written as follows:
\beno
\pa_tf+v\cdot \na_x f=Q(f,f).
\eeno
Here  $Q$ is the Boltzmann collision operator defined by
\beno Q(f,f)(v):=
\int_{\R^3}\int_{\mathbb{S}^{2}}|v-v_*|\frac{d\Omega}{d\sigma}\left(f(v'_*) f(v')-f(v_*)f(v)\right) dv_* d\sigma,
\eeno
where $d\Omega/d\sigma $ is the differential scattering cross section determined by the potential function $\phi$ for particles, and  $v'$, $v_*'$ are given by
\begin{eqnarray}\label{e3}
v'=\frac{v+v_{*}}{2}+\frac{|v-v_{*}|}{2}\sigma,\quad
v'_{*}=\frac{v+v_{*}}{2}-\frac{|v-v_{*}|}{2}\sigma, \quad \sigma\in\mathbb{S}^{2}.
\end{eqnarray}

In the scattering event between two electrons
governed by the Coulomb potential,
\begin{equation}\label{coulomb-potential-between-electrons}
\phi(r)=\frac{e^{2}}{4 \pi \varepsilon_{0} r},
\end{equation}
the deviation angle $\theta$ of relative velocity is determined through (see \cite{swanson2008plasma} for instance)
\begin{equation}\label{Relationthetab}
\tan \frac{\theta}{2}=\frac{e^{2}}{2 \pi \epsilon_{0} m \left|v-v_{*}\right|^{2} b},
\end{equation}
where $\varepsilon_{0}$ is the permittivity of vacuum, $m$ is the mass of an electron and $e$ its charge, $\left|v-v_{*}\right|$ is the relative velocity before the event and $b$ is the impact parameter which is defined as the distance of closest approach
if the trajectory were undeflected.  The well-known Rutherford differential cross section is computed as
\begin{equation}
\frac{d\Omega}{d\sigma}:= \frac{b}{\sin \theta} \frac{db}{d\theta}
=\frac{\left(e^{2} /\left(4 \pi \varepsilon_{0} m\right)\right)^{2}}{\big(\left|v-v_{*}\right| \sin(\theta / 2)\big)^4}.
\end{equation}
Then the corresponding Boltzmann collision kernel $B$ reads
\ben \label{boltzmann-kernel-Rutherford-cs}B(v-v_{*}, \sigma) = |v-v_*|\frac{d\Omega}{d\sigma}=\frac{\left(e^{2} /\left(4 \pi \varepsilon_{0} m\right)\right)^{2}}{\left|v-v_{*}\right|^{3} \sin ^{4}(\theta / 2)}
 = K|v-v_*|^{-3} b(\cos \theta),
\een
where  \ben \label{definition-K-b}
K:= \left(e^{2} /\left(4 \pi \varepsilon_{0} m\right)\right)^{2}, \quad b(\cos \theta):= \sin^{-4}(\theta / 2),
 \quad \cos \theta := \frac{v-v_{*}}{|v-v_{*}|} \cdot \sigma.\een

\subsubsection{Divergence of the momentum transfer} One may check that the momentum transfer defined by $\int_{\mathbb{S}^{2}} b(\cos \theta){\sin^{2}(\theta/2)}  d\sigma$ is divergent in a logarithmic manner due to the singularity at $\theta =0$. Indeed,
\beno \frac{1}{8\pi} \int_{\mathbb{S}^{2}} b(\cos \theta){\sin^{2}(\theta/2)}  d\sigma = \frac{1}{8\pi} \int_{0}^{2\pi}\int_{0}^{\pi} {\sin^{-2}(\theta/2)}\sin\theta  d\theta d\varphi = \frac{1}{2} \int_{0}^{\pi} \frac{\cos(\theta/2)}{\sin(\theta/2)} d\theta
=  \int_{0}^{1} \frac{1}{u} d u = \infty,
\eeno
where the change of variable $u=\sin(\theta/2)$ is used.

 The reason of the divergence is due to the long-range interaction of Coulomb potential. As indicated in \cite{lifshitz1981physical}, the divergence at the lower limit has a physical cause: the slowness of the decrease of the Coulomb forces, which leads to a high probability of small-angle scattering. However, the phenomenon of screening effect implies that the role of collisions with a high impact parameter is not as important as the Coulomb potential suggests. Thanks to \eqref{Relationthetab}, a rough and artificial approximation is just to ignore grazing collisions. Such an argument can be found in \cite{lifshitz1981physical}:
{\it ``In reality, however, in an electrically neutral plasma the Coulomb field of a particle at sufficiently large distances is screened by other charges; let $\theta_{\min}$ denote the order of magnitude of the smallest angles for which the scattering can still be regarded as Coulomb scattering".}
In this way, one has
\beno  \frac{1}{8\pi} \int_{\mathbb{S}^{2}} b(\cos \theta){\sin^{2}(\theta/2)}\mathrm{1}_{\theta\ge \theta_{\min}}  d\sigma  = -\ln(\sin(\theta_{\min}/2)) \sim \ln (1/\theta_{\min}),
\eeno
which is relevant to the so-called ``Coulomb logarithm''  denoted by $\ln \Lambda$.

\begin{rmk}In most physical books (for instance, see \cite{lifshitz1981physical,montgomery1964plasma,landau1936transport}), $\ln \Lambda$ is derived through the integration with respect to the impact parameter $b$, that is,
\beno \ln \Lambda:= \int_{\lambda_L}^{\lambda_D} b^{-1} db=\ln \f{\lambda_D}{\lambda_L},\eeno
where $\lambda_D$ is the Debye length which characterizes electrostatic screening and $\lambda_L$ is the Landau length which identities strong interactions. In other words,   the ``weak interaction'' is defined through the truncation of the impact parameter $b$ onto the interval $[\lambda_L,\lambda_D]$.  Thanks to \eqref{Relationthetab}, by approximation (see \cite{montgomery1964plasma}), it holds that
\ben\label{Lambdathetamin} \Lambda=\f{2}{\theta_{\min}}=24\pi n_0\lambda_D^3, \een
where $n_0$ is the   density of the particle. Since $n_0\lambda_D^3\gg 1$(for instance, for electron-proton gas), one has
\ben\label{asymthetalambda} \theta_{\min}\ll 1,\qquad  \Lambda\gg 1. \een
\end{rmk}

\subsubsection{Landau's strategy  in \cite{landau1936transport}} For simplicity, we restrict ourselves to one species of particle and neglect the self-consistent electrostatic  field.  Landau's strategy can be summarized as follows:

\begin{itemize}
\item[Step 1:] Based on \eqref{boltzmann-kernel-Rutherford-cs}, the Boltzmann kernel with cutoff Rutherford cross section is defined by
  \beno B^c(v-v_*,\sigma)=B(v-v_*,\sigma)\mathrm{1}_{\sin(\theta/2)\ge \sin(\theta_{\min}/2)}. \eeno Landau's first assumption is  that the Boltzmann equation   with cutoff Rutherford cross section:
\ben\label{BoltzmannCRcs} \left\{ \begin{aligned}
&\partial _t F +  v \cdot \nabla_{x} F=Q^{c}(F,F), ~~t > 0, x \in \mathbb{T}^{3}, v \in\R^3,\\
&F|_{t=0} = F_{0},
\end{aligned} \right.
\een
admits a smooth solution. Here $\mathbb{T}^{3}:=[-\pi, \pi]^{3}$ is the torus. Here
\ben \label{KernelCRcs}
Q^{c}(F,F):=
\int_{\R^3}\int_{\mathbb{S}^{2}}B^c(v-v_*,\sigma) \left(F(v'_*) F(v')-F(v_*)F(v)\right) dv_* d\sigma,\een

\item[Step 2:] To derive an effective equation for weakly coupling particles,   Landau further assumed $v\sim v'$ and $v_*\sim v_*'$. In his language, if $q:= v-v'$, then  $|q|\ll 1$. Thanks to Taylor expansion,
  \ben\label{Ftaylor} F(v')=F(v+q)=F(v)+\na_vF(v)\cdot q+\f12\na^2_vF(v): q\otimes q+ O(|q|^3). \een
Inserting \eqref{Ftaylor} into the Boltzmann collision operator  \eqref{KernelCRcs}
    and taking the truncation of the impact parameter $b$ onto the interval $[\lambda_L,\lambda_D]$, he derived the leading equation, which is now named as the Landau equation:
\ben\label{landauori} \pa_tF +v\cdot\na_xF=(\ln \Lambda )Q^L(F,F),\een
where the Landau collision operator $Q^{L}$ is defined as:
\ben \label{oroginal-definition-Laudau-oprator} Q^{L}(g,h)(v) :=
\nabla_{v}\cdot\bigg\{\int_{\R^3}a(v-v_{*})[g(v_{*})\nabla_{v}h(v)-\nabla_{v_{*}}g(v_{*})h(v)]dv_{*}\bigg\}.
\een
Here the symmetric matrix $a$ is given by
\begin{eqnarray}\label{matrix}
a(z)  = 2 \pi K |z|^{-1}(I_{3} - \frac{z \otimes z}{|z|^{2}}),
\end{eqnarray}
where $I_{3}$ is the $3 \times 3$ identity matrix.
\end{itemize}

\subsection{Mathematical problems on the derivation}
To set up   mathematical problems, thanks to \eqref{asymthetalambda},
 we first introduce a small parameter $\epsilon$, which is related to the physical cutoff for the angle, that is,
\ben\label{Defeps} \epsilon:=\sin(\theta_{\min}/2)\ll1.\een
For simplicity of presentation, we take $K=1$ for the constant $K$ in \eqref{definition-K-b}.

\subsubsection{Mathematical assumptions on the kernel} Throughout the paper, we will consider the kernel  $B^{\epsilon}$ verifying the following  assumptions.
  \smallskip
\begin{itemize}
\item[{\bf  (A1).} ]  The kernel $B^{\epsilon}(v-v_*,\sigma)$  takes the form:
\ben \label{DefScaledBker} B^{\epsilon}(v-v_*,\sigma)&:= &|\ln \epsilon|^{-1}B^c(v-v_*,\sigma)
= \left|v-v_{*}\right|^{-3} b^\epsilon(\cos \theta),
\\\label{DefScaled-b} b^\epsilon(\cos \theta) &:=&|\ln \epsilon|^{-1} \sin^{-4}(\theta/2)\mathrm{1}_{\sin(\theta/2) \geq \epsilon}.
 \een

\item[{\bf  (A2).} ]   The kernel $B^{\epsilon}(v-v_*,\sigma)$ is supported in the set $0\leq \th\leq \pi/2$, i.e. $\cos\theta \geq0$, for otherwise $B^\epsilon$ can be replaced by its symmetrized form:
\ben
\overline{B^\epsilon}(v-v_*,\sigma)=[B^\epsilon(v-v_*,\sigma)+B^\epsilon(v-v_*,-\sigma)]\mathrm{1}_{\cos\theta\>>0}.
\een
\end{itemize}
Associated to $B^{\epsilon}$, the Boltzmann collision operator $Q^{\epsilon}$ is defined by
\ben\label{DefScaledBOP} Q^{\epsilon}(g,h)(v):=
\int_{\R^3}\int_{\mathbb{S}^{2}}B^{\epsilon}(v-v_*,\sigma)\left(g(v'_*) h(v')-g(v_*)h(v)\right)dv_* d\sigma .
\een

\subsubsection{Reformulation of the equation}
It is compulsory to rewrite equations \eqref{BoltzmannCRcs} and \eqref{landauori} by taking into account the parameter $\epsilon$.
To do that, we introduce the scaling:
\ben\label{scalingsol}
\tilde{F}(t,x,v)=F(|\ln \epsilon|^{-1}t, x, |\ln \epsilon|v).\een
Thanks to the facts $Q^{c}(\tilde{F},\tilde{F})(t,x,v) = Q^{c}(F,F)(|\ln \epsilon|^{-1}t, x, |\ln \epsilon|v)$ and $Q^{\epsilon} = |\ln \epsilon|^{-1}Q^{c}$,   \eqref{BoltzmannCRcs} becomes
\begin{equation}\label{coboltzmann} \left\{ \begin{aligned}
&\partial _t \tilde{F} +  v \cdot \nabla_{x} \tilde{F}=Q^{\epsilon}(\tilde{F},\tilde{F}), ~~t > 0, x \in \mathbb{T}^{3}, v \in\R^3 ;\\
&\tilde{F}|_{t=0} = F^{\epsilon}_{0}:= F_{0}(x, |\ln \epsilon|v).
\end{aligned} \right. \end{equation}
Similarly, thanks to $Q^{L}(\tilde{F},\tilde{F})(t,x,v)  = Q^{L}(F,F)(|\ln \epsilon|^{-1}t, x, |\ln \epsilon|v)$, \eqref{landauori} reduces to
 {(thanks to \eqref{Lambdathetamin} and \eqref{Defeps}, we choose $\epsilon$ satisfying $|\ln \epsilon| \sim \ln \Lambda$ and absorb some unimportant constant into the Landau operator)}
\begin{equation}\label{landau-equation} \left\{ \begin{aligned}
&\partial _t \tilde{F}+  v \cdot \nabla_{x} \tilde{F}=Q^{L}(\tilde{F}, \tilde{F}), ~~t > 0, x \in \mathbb{T}^{3}, v \in\R^3 ;\\
&\tilde{F}|_{t=0} = F^{\epsilon}_{0} : = F_{0}(x, |\ln \epsilon|v).
\end{aligned} \right. \end{equation}

Let us give  comments  on the scaling  \eqref{scalingsol}. Roughly speaking, it enables us to consider Landau's derivation and the Landau approximation proposed in \cite{alexandre2004landau} in a unified framework.

\smallskip
\noindent $\bullet$ {\bf Validity of the scaling in physical sense.} In physical books and lectures, one may check that for typical weakly coupled plasmas, Coulomb logarithm $\ln \Lambda$ lies in the range: $[5, 20]$. Recalling \eqref{Lambdathetamin} and \eqref{Defeps}, we have $|\ln \epsilon| \sim \ln \Lambda \in [5, 20]$, which means $\epsilon$ is sufficiently  small (around $[e^{-20}, e^{-5}]$) but $|\ln \epsilon|$ is relatively ``normal'' (around $[5, 20]$). In this range,
     the scaling \eqref{scalingsol} is harmless and thus
       \eqref{BoltzmannCRcs} and \eqref{landauori} are equivalent to  \eqref{coboltzmann}  and \eqref{landau-equation} respectively. For the same reason, it is mathematically equivalent to ignore the dependence of the initial data on the parameter $|\ln \epsilon|$ in \eqref{coboltzmann} and \eqref{landau-equation}, as we will do in \eqref{linearizedBE} and \eqref{linearizedLE}. This reduces the justification of Landau's derivation to the consideration of \eqref{coboltzmann}.

\smallskip

\noindent $\bullet$ {\bf  Relation between the scaling and Landau approximation.}  Landau approximation(or the grazing collsions limit) is a mathematical framework to
derive   Landau equation from Boltzmann equation with general potentials. The main idea is as follows: when the deviation angle is truncated up to the order $O(\epsilon)$ with a proper scaling, the grazing collisions will dominate and then the Boltzmann equation  will formally converge to the Landau equation. The convergence has drawn extensive attention from mathematicians(see   \cite{degond1992fokker,desvillettes1992asymptotics,he2014well,villani1998new}).

In view of \eqref{DefScaled-b}, only grazing collisions can survive in the limit in which $\epsilon$ goes to zero.
The pioneering work \cite{alexandre2004landau} of Alexandre-Villani derived Landau equation \eqref{landau-equation} from Boltzmann equation \eqref{coboltzmann} in the inhomogeneous setting under physical assumptions of finite mass, energy, entropy and entropy production.
We emphasize that the  Coulomb logarithm $\ln \Lambda$ in front of the collision operator $Q^L$ in \eqref{landauori} has been  normalized in \eqref{landau-equation}. In other words,  the derivation through Landau approximation will lose some information from the potential function $\phi$ for particles.

\smallskip

  \noindent $\bullet$ {\bf Effect of the scaling on the linearized Boltzmann collision operator.}  The scaling factor $|\ln \epsilon|^{-1}$ in \eqref{DefScaled-b} and \eqref{scalingsol} plays an essential role  in getting the spectral gap estimates for the linearized operator of $Q^\epsilon$. We refer readers to Theorem \ref{micro-dissipation} for details.
     For Maxwellian molecules, the Boltzmann kernel $B$ takes the special form:
 \ben\label{KernelMaxweM}B(v-v_*,\sigma)=b(\cos\theta).\een
 Chang-Uhlenbeck in \cite{Chang-Uhlenbeck} proved that the first (smallest) positive eigenvalue $\lambda_1$ to the associated linearized Boltzmann collision operator can be computed explicitly:
\ben\label{firstegenvaluemaxwellion} \lambda_1\sim \int_0^\pi b(\cos\theta)(1-\cos\theta)\sin\theta d\theta. \een
 The scaling factor $|\ln \epsilon|^{-1}$ in \eqref{DefScaled-b} ensures that in the limit process($\epsilon\rightarrow0$), it holds that (see Lemma \ref{integral-angular-function})
\ben\label{firstegenvalueLA} \lambda_1^\epsilon:=\int_0^\pi b^\epsilon(\cos\theta)(1-\cos\theta)\sin\theta d\theta\sim 1. \een
This motivates us to link the spectral gap estimates with Chang-Uhlenbeck's work \cite{Chang-Uhlenbeck}.

\subsubsection{Mathematical problems} Our setup enables us to handle the Landau's derivation and   the Landau approximation proposed in \cite{alexandre2004landau} in a unified framework. We consider these problems near Maxwellian (a small perturbation around equilibrium: Maxwellian) since it is widely used in the study of kinetic equations (for instance, wave phenomena in plasma physics).  Our work can be summarized as follows.
\begin{itemize}
\item[(1).] We consider global well-posedness of \eqref{coboltzmann} near Maxwellian. As a result, it shows that all the computation in Landau's paper \cite{landau1936transport} is valid globally in time  and thus we justify his derivation rigorously in mathematics.
\item[(2).] We revisit Landau approximation  near Maxwellian from \eqref{coboltzmann} to \eqref{landau-equation}. Compared to \cite{alexandre2004landau}, we work with classical solution rather than weak solution.
    Mathematically we need to establish  a unified framework to solve Boltzmann and Landau equations simultaneously and obtain an explicit expansion formula for  the approximation.
\end{itemize}

\noindent$\bullet$ {\bf Global wellposedness of \eqref{coboltzmann}  near Maxwellian.} Recall the standard Maxwellian density function $\mu(v) :=(2\pi)^{-3/2}e^{-|v|^{2}/2}$. With the perturbation $\tilde{F}=\mu +\mu^{1/2}f$, \eqref{coboltzmann} yields:
\begin{equation}\label{linearizedBE} \left\{ \begin{aligned}
&\partial_{t}f + v\cdot \nabla_{x} f + \mathcal{L}^{\epsilon}f=  \Gamma^{\epsilon}(f,f), ~~t > 0;\\
&f|_{t=0} = f_{0}.
\end{aligned} \right.\end{equation}
Here the linearized Boltzmann operator $\mathcal{L}^{\epsilon}$ and the nonlinear term $\Gamma^{\epsilon}$ are defined by
\ben\label{DefLep}\Gamma^{\epsilon}(g,h):= \mu^{-1/2} Q^{\epsilon}(\mu^{1/2}g,\mu^{1/2}h), \mathcal{L}^{\epsilon}_{1}g:=-\Gamma^{\epsilon}(\mu^{1/2},g), \mathcal{L}^{\epsilon}_{2}g:=- \Gamma^{\epsilon}(g, \mu^{1/2}), \mathcal{L}^{\epsilon}g:= \mathcal{L}^{\epsilon}_{1}g + \mathcal{L}^{\epsilon}_{2}g.
\een
We aim at not only global well-posedness but also propagation of regularity which holds uniformly in $\epsilon$. They are crucial because of the following reasons:

\begin{itemize}
\item[(1).] 	
Global well-posedness of \eqref{linearizedBE} exactly corresponds to the first step of Landau's original strategy.  Propagation of regularity is necessary in order to apply the Taylor expansion in the second step of Landau's strategy.

\item[(2).]  In order to find asymptotic formula between solutions of the Boltzmann and  Landau equation, some uniform estimates for propagation of regularity are essential.
\end{itemize}

\noindent$\bullet$ {\bf Asymptotics of \eqref{coboltzmann}  near Maxwellian.}
It is relevant to the second step of Landau's strategy and establishes Landau approximation global-in-time in classical solution sense. As we reviewed before, Landau just used the Taylor expansion of the solution to get the desired equation.
It is not formulated by a direct limit from Boltzmann equation to Landau equation while Landau approximation
seeks to do so.

In \cite{he2014well}, it was shown that at least locally in time,
\beno \tilde{F}^\epsilon=\tilde{F}^L+O(|\ln \epsilon|^{-1}),\eeno
where $\tilde{F}^\epsilon$ and $\tilde{F}^L$ are solutions to \eqref{coboltzmann} and \eqref{landau-equation}
with the same initial data. The order $|\ln \epsilon|^{-1}$ reflects the logarithmic accuracy  as derived in \cite{lifshitz1981physical}.

In this paper, we will reconsider such approximation near Maxwellian. More precisely, if we
set $\tilde{F}=\mu +\mu^{1/2}f$, then   \eqref{landau-equation} gives the linearized equation:
\begin{equation}\label{linearizedLE} \left\{ \begin{aligned}
&\partial_{t}f + v\cdot \nabla_{x} f + \mathcal{L}^{L}f= \Gamma^{L}(f,f), ~~t > 0;\\
&f|_{t=0} = f_{0}.
\end{aligned} \right.
\end{equation}
Here the linearized Landau operator $\mathcal{L}^{L}$ and the nonlinear term $\Gamma^{L}$  are defined by:
\ben\label{DefL}\Gamma^{L}(g,h):= \mu^{-1/2} Q^{L}(\mu^{1/2}g,\mu^{1/2}h), \mathcal{L}^{L}_{1}g:=-\Gamma^{L}(\mu^{1/2},g), \mathcal{L}^{L}_{2}g:=- \Gamma^{L}(g, \mu^{1/2}), \mathcal{L}^{L}g:= \mathcal{L}^{L}_{1}g + \mathcal{L}^{L}_{2}g. \een
One may regard \eqref{linearizedLE} as the limit case ($\epsilon=0$) of \eqref{linearizedBE}. Our goal is to establish an asymptotic formula to describe the limit process when $\epsilon$ tends to zero. That is, if $f^\epsilon$ and $f^L$ are solutions to \eqref{linearizedBE} and \eqref{linearizedLE} with the same initial data, then it holds
\ben\label{ASF}  f^\epsilon=f^L+O(|\ln \epsilon|^{-1}),\een
globally in time in weighted Sobolev spaces.

\subsubsection{Basic properties} At the end of this subsection, we recall some special properties of the Boltzmann equations.  The solutions  to  \eqref{coboltzmann} and \eqref{landau-equation} have the fundamental physical properties of conserving total mass, momentum and kinetic energy, that is, for all $t\ge0$,
\ben\label{conserveq}  \int_{\mathbb{T}^{3} \times \R^3} F(t,x,v)\phi(v)dxdv=\int_{\mathbb{T}^{3} \times \R^3} F(0,x,v)\phi(v)dxdv,\quad \phi(v)=1,v_{j},|v|^2, \quad j=1,2,3.\een
Without loss of generality, we assume that the initial data $f_0$ in \eqref{linearizedBE} and \eqref{linearizedLE}
verifies
\ben\label{initial-condition} \int_{\mathbb{T}^{3} \times \R^3} \sqrt{\mu}f_0\phi dxdv=0, \quad \phi(v)=1,v_{j},|v|^2, \quad j=1,2,3.\een
As a result of \eqref{conserveq}, the solutions  to  \eqref{linearizedBE} and \eqref{linearizedLE} verify for all $t\ge0$,
\ben \label{Nuspace}  \int_{\mathbb{T}^{3} \times \R^3} \sqrt{\mu}f(t)\phi dxdv=0, \quad \phi(v)=1,v_{j},|v|^2, \quad j=1,2,3.\een
Recall that $\mathcal{N}(\mathcal{L}^\epsilon)$ and $\mathcal{N}(\mathcal{L}^{L})$, the kernel space of  $\mathcal{L}^\epsilon$ and $\mathcal{L}^{L}$ respectively, verify
 \beno  \mathcal{N}(\mathcal{L}^{L})= \mathcal{N}(\mathcal{L}^\epsilon)= \mathrm{span}\{\sqrt{\mu}, \sqrt{\mu}v_1, \sqrt{\mu}v_2,\sqrt{\mu}v_3, \sqrt{\mu}|v|^2 \} :=\mathcal{N}. \eeno

\subsection{Main results}
Our main results are global well-posedness and propagation of regularity for the Boltzmann equation \eqref{linearizedBE} with cutoff Rutherford cross section. Moreover, we derive the global-in-time asymptotic formula for the Landau approximation from the equation \eqref{linearizedBE} to the equation \eqref{linearizedLE}.  We refer readers to  subsection 1.5 to check details on the function spaces used throughout the paper.

Our results are based on the following energy functional
\ben \label{energy-functional}\mathcal{E}^{N,l}(f):=\sum_{j=0}^{N}
\|f\|^{2}_{H^{N-j}_{x}\dot{H}^{j}_{l+j\gamma}},\een
where $\gamma=-3, l\geq 3N+2$.

\begin{thm}\label{asymptotic-result}
Let $0 \leq \epsilon \leq \epsilon_0$ where $\epsilon_0>0$ is a small constant. Suppose $f_0$ verify \eqref{initial-condition}. There is a  universal constant $\delta_0>0$ such that, if
\beno  \mu+\mu^{\f12}f_0\ge0,\quad \mathcal{E}^{4,14}(f_{0}) \leq \delta_0, \eeno
then \eqref{linearizedBE} admits a unique global strong solution $f^\epsilon$ verifying $\mu+\mu^{\f12}f^\epsilon\ge0$ and
\ben \label{uniform-controlled-by-initial} \sup_{t\ge 0} \mathcal{E}^{4,14}(f^{\epsilon}({t}))\leq C \mathcal{E}^{4,14}(f_{0}),\een
 for some universal constant $C$. Moreover, the family of solution $\{f^{\epsilon}\}_{\epsilon \geq 0}$ verifies
\begin{enumerate}
\item{\bf (Propagation of regularity)} Fix $N\geq 4, l \geq 3N+2$, if additionally $\mathcal{E}^{N,l}(f_{0}) < \infty$ and $\epsilon$ is small enough, then
\ben \label{propagation}\sup_{t\ge0} \mathcal{E}^{N,l}(f^{\epsilon}({t}))\leq P_{N,l}\left(\mathcal{E}^{N,l}(f_{0})\right).\een
Here $P_{N,l}(\cdot)$ is a continuous and increasing function with $P_{N,l}(0)=0$.
\item{\bf (Global asymptotic formula)} Fix $N\geq 4, l \geq 3N+2$, assume $\mathcal{E}^{N+3,l+18}(f_{0}) < \infty$ and $\epsilon$ is small enough, then
\ben \label{error-function-uniform-estimate}\sup_{t\ge0} \mathcal{E}^{N,l}(f^{\epsilon}(t)-f(t)) \leq  |\ln \epsilon|^{-2} U_{N,l}(\mathcal{E}^{N+3,l+18}(f_{0})). \een
Here $U_{N,l}(\cdot)$ is a continuous and increasing function with $U_{N,l}(0)=0$.
\end{enumerate}
\end{thm}

Some remarks are in order.
\begin{rmk} The kernel studied in this work is the most singular one both in relative velocity and deviation angle, and is the borderline for the Boltzmann to be meaningful in the classical sense. To our best knowledge, Landau approximation in the inhomogeneous case has never been touched globally in time within classical solution setting. We manage to establish global-in-time asymptotic formula \eqref{error-function-uniform-estimate} with explicit accuracy order for the first time.
\end{rmk}

\begin{rmk} Our results are consistent with the results  in \cite{guo2002landau} and \cite{guo2012vlasov} when $\epsilon=0$. In particular,
the smallness assumption on initial data with finite regularity and finite weight is a universal constant, which is sufficient to prove propagation of regularity with arbitrary regularity and weight if $\epsilon$ is sufficiently small.
\end{rmk}

\begin{rmk} To keep the paper in a reasonable length, we only consider one species of particle, which enables us to focus more on operator analysis and a so-called linear-quasilinear method to close energy estimate. In the future, we will consider a more physical model: two species Vlasov-Possion-Boltzmann system with cutoff Rutherford cross-section, to derive the Vlasov-Possion-Landau system.
\end{rmk}
\begin{rmk}  Let us summarize the main difference between the Landau approximation  proposed in \cite{alexandre2004landau} and   Landau's original strategy in \cite{landau1936transport} as follows:
\begin{itemize}
\item[(1).] Landau approximation is based on the assumption that $\ln \Lambda$ is sufficiently large. However it is invalid in many physical situations where $\ln \Lambda$ is a relatively normal constant.
\item[(2).] The resulting equation derived from Landau  is different from that by  Landau approximation. We remind readers that the Coulomb logarithm $\ln \Lambda$ appears as a diffusive coefficient in \eqref{landauori} but it does not appear in \eqref{landau-equation}.
\item[(3).]  The error estimate between the solutions to Boltzmann and Landau equations via Landau approximation is   different from that by Landau's  strategy. One has logarithm accuracy  while the other has a high order accuracy thanks to \eqref{Ftaylor}.
\end{itemize}

\end{rmk}

\subsection{Main Difficulties} Boltzmann and Landau equations are well studied near Maxwellian(see \cite{duan2008cauchy,guo2002landau,guo2003classical,guo2012vlasov,gressman2011global,alexandre2011global,alexandre2012boltzmann}).
To explain the main difficulties and the new ideas of the paper, let us consider a typical kinetic equation near Maxwellian:
\beno \pa_t f+v\cdot \na_xf+\mathcal{L}f=\Gamma(f,f), \eeno where $\mathcal{L}$ and $\Gamma$ denote the linearized operator and the nonlinear term. We focus on propagation of regularity (or a priori estimate).
The general approach to prove propagation of regularity
for the equation can be  divided into four steps:

\begin{itemize}
	\item[Step 1:] This step is to describe the behavior of the linearized operator $\mathcal{L}$, including the spectral gap estimate and the coercivity estimate. Roughly speaking, they can be written as follows:
\beno
 \langle \mathcal{L} f, f\rangle_v  \gtrsim
|||(\mathbb{I}-\mathbb{P})f|||^2_{\mathrm{gap}}, \quad \langle \mathcal{L} f, f \rangle_v\gtrsim |||f|||^2_{\mathrm{coercivity}}-|||f|||^2_{\mathrm{gap}},
\eeno
where $\mathbb{P}$ is a projection operator that maps a function into the null space of $\mathcal{L}$, $\mathbb{I}$ is the identity operator. Here $|||\cdot|||_{\mathrm{gap}}$ and $|||\cdot|||_{\mathrm{coercivity}}$ are some explicit or implicit norms.
From these two estimates, one has
\beno  \langle \mathcal{L} f, f \rangle_v\gtrsim |||(\mathbb{I}-\mathbb{P})f|||_{\mathrm{coercivity}}^2.\eeno
\item[Step 2:] The second step is to use the norm $|||\cdot|||_{\mathrm{coercivity}}$ to give the upper bound for the nonlinear term $\Gamma(f,f)$. Roughly, the ideal estimate looks like:
\beno  \big|\langle  \Gamma(f,f), f\rangle_v\big|\lesssim \|f\|\,|||f|||^2_{\mathrm{coercivity}}.\eeno
Here $\|\cdot\|$ is the usual $L^{2}$ norm.
\item[Step 3:] The third step is to derive an evolution equation for $\mathbb{P}f$ and then to get some elliptic estimate for $\mathbb{P}f$ under the control of $|||(\mathbb{I}-\mathbb{P})f|||_{\mathrm{coercivity}}$. This is referred as Micro-Macro decomposition.

\item[Step 4:] The final step is to construct a proper energy functional to close the energy estimates  and then get propagation of regularity for the solution under smallness assumption.
	\end{itemize}

  Now we turn to our case (see \eqref{linearizedBE}) to explain the main difficulties in each step. Due to the definition of the kernel (see \eqref{DefScaledBker}), all the difficulties result from the  strong singularity: not only from the relative velocity $|v-v_*|^{-3}$ and but also from the angular function $b^{\epsilon}(\cos\theta)=|\ln \epsilon|^{-1}\sin^{-4}(\theta/2)\mathrm{1}_{\sin(\theta/2) \geq \epsilon}$. Loosely speaking, (i). on one hand, the high singularity from the relative velocity stops us to treat the equation like the cutoff Boltzmann equation; (ii). on the other hand, the high singularity from the deviation angle is the borderline case for non-cutoff Boltzmann equation in terms of finite  momentum transfer.

Mathematically, we face the following essential difficulties.
 \begin{itemize}
\item[{\bf (D1).}] The first one is concerned with the coercivity estimate of the linearized operator $\mathcal{L}^\epsilon$.
 Thanks to \cite{he2014well}, it was shown that there exists a characteristic function $W^\epsilon$ defined by
\ben\label{charicter function}
W^{\epsilon}(y) := \langle y\rangle\phi(y)+\langle y\rangle\bigg(1-\f{\ln|y|}{|\ln \epsilon|}+\f1{|\ln \epsilon|}\bigg)^{\f12}\big(\phi(\epsilon y)-\phi(y)\big)+\f1{|\ln \epsilon|^{\f12}\epsilon}\big(1-\phi(\epsilon y)\big),
\een
to catch the Sobolev regularity for the collision operator. Here $\phi$ is a smooth defined in \eqref{function-phi-psi}.
It is easily checked that the characteristic function $W^\epsilon$ behaves  quite different when $|y|\sim 1$ and $|y|\sim \epsilon^{-1}$. On one hand, it matches well the limiting operator $\mathcal{L}^L$ as $\epsilon$ goes to zero. On the other hand, it indicates that the behavior of $\mathcal{L}^\epsilon$ will be more complicated, in particular for the proof of  gain of weight and gain of anisotropic regularity (see Theorem \ref{coercivity-structure}). Also because of \eqref{charicter function}, we get stuck in the upper bound estimate of the operator.

\item[{\bf (D2).}]  Once the coercivity estimate is available, we still need to face the following spectral gap estimate: for any smooth function $f$,
\beno \langle \mathcal{L}^\epsilon f,f\rangle_v\ge C\|(\mathbb{I}-\mathbb{P})f\|_{L^2_{-3/2}}^2, \eeno
where $C>0$ is independent of $\epsilon$. The main difficulty here results from the singular factor $|\ln \epsilon|^{-1}$ in the angular function $b^\epsilon$ because it goes to zero when $\epsilon$ tends to zero. Thus to get the desired result we need a constructive proof. The easiest and also the clearest case is the Maxwellian molecules given in \cite{Chang-Uhlenbeck}, where the authors proved that the constant $C$ only depends on the first positive eigenvalue which can be computed explicitly as  \eqref{firstegenvaluemaxwellion}. Fortunately in our case, $b^\epsilon$ satisfies the condition \eqref{firstegenvalueLA} which motivates us to consider \eqref{firstegenvaluemaxwellion}.

\item[{\bf (D3).}]
As mentioned in {\bf (D1)}, \cite{he2014well} gives
\ben\label{CoerQe} \langle -Q^\epsilon(\mu,f),f  \rangle_v\gtrsim |W^\epsilon(D)f|_{L^2_{-3/2}}^2-|f|^2_{L^2_{-3/2}}. \een
Note that we gain one derivative only in the ``low frequency part''(that is, $|\xi|\lesssim 1/\epsilon$). Unfortunately because of the strong singularity of the relative velocity ($|v-v_*|^{-3}$ is a borderline case in 3-dimension), we at least need  one full derivative to give a uniformly upper bound for the collision operator with respect to $\epsilon$. Roughly speaking, \cite{he2014well} provides
\beno \big|\langle Q^\epsilon(g,h),f\rangle_v\big|\lesssim |g|_{H^1_{9/2}}|h|_{H^1}|W^\epsilon(D)f|_{L^2_{-3/2}}, \eeno
which indicates that what we gain from the coercivity is not enough to control the upper bound of the nonlinear term $\langle \Gamma^\epsilon(g,h),f\rangle_v$.
\end{itemize}

  The above three difficulties stop us to implement energy method and prove propagation of regularity, which forces us to figure out new techniques.

\subsection{Ideas and strategies} Before explaining our strategy to  overcome the above difficulties,
we begin with  basic facts on Micro-Macro decomposition and spherical harmonics.

 \begin{itemize}
 	\item {\bf Micro-Macro decomposition:}
Recall $\mathcal{N}= \mathrm{span}\{\sqrt{\mu}, \sqrt{\mu}v_1, \sqrt{\mu}v_2,\sqrt{\mu}v_3, \sqrt{\mu}|v|^2 \}$, an orthonormal basis of which can be chosen as $\{\sqrt{\mu}, \sqrt{\mu}v_1, \sqrt{\mu}v_2,\sqrt{\mu}v_3, \sqrt{\mu}(|v|^2-3) /\sqrt{6} \} := \{e_{j}\}_{1 \leq j \leq 5}$. The projection operator $\mathbb{P}$ on the null space $\mathcal{N}$ is defined as follows:
\ben\label{DefProj} \mathbb{P}f:=\sum_{j=1}^{5}\langle f, e_{j}\rangle e_{j} =(a+b\cdot v+c|v|^2)\sqrt{\mu}, \een
where for $1\le i\le 3$, \ben\label{Defabc}
 a=\int_{\R^3} (\frac{5}{2}-\frac{|v|^{2}}{2})\sqrt{\mu}fdv;\quad b_i=\int_{\R^3} v_i\sqrt{\mu}fdv;\quad c=\int_{\R^3} (\frac{|v|^2}{6}-\frac{1}{2})\sqrt{\mu}fdv.
\een
Generally we call $\mathbb{P}f$ and $(\mathbb{I}-\mathbb{P})f$ the macroscopic part and microscopic part of $f$ respectively.

\item {\bf Spherical harmonics:}
Let $Y_l^m$ with $l\in \mathbb{N}, m \in \mathbb{Z}, -l\le m\le l$ be real spherical harmonics. They are the eigenfunctions of the Laplace-Beltrami operator
$-\triangle_{\mathbb{S}^2}$. Mathematically,
\beno (-\triangle_{\mathbb{S}^2})Y_l^m=l(l+1)Y_l^m. \eeno
These functions are essential to help us to catch the anisotropic property of $\mathcal{L}^\epsilon$. We introduce the operator $W^{\epsilon}((-\Delta_{\mathbb{S}^{2}})^{1/2})$   defined by: if $v = r \sigma$, then
\ben\label{DeltaWe} (W^{\epsilon}((-\Delta_{\mathbb{S}^{2}})^{1/2})f)(v) := \sum_{l=0}^\infty\sum_{m=-l}^{l} W^{\epsilon}((l(l+1))^{1/2}) Y^{m}_{l}(\sigma)f^{m}_{l}(r),
 \een
where
$ f^{m}_{l}(r) = \int_{\mathbb{S}^{2}} Y^{m}_{l}(\sigma) f(r \sigma) d\sigma$.
\end{itemize}

We are ready to present our main ideas and strategies.

\subsubsection{\bf Idea on the coercivity estimates} This part is related to {\bf (D1)}. The coercivity estimate of the linearized operator $\mathcal{L}^{\epsilon}$ plays an essential role in studying \eqref{linearizedBE} and it reads

\begin{thm}\label{coercivity-structure}
There are two positive universal constants $\epsilon_0>0$ and $\nu_{0}>0$ such that for $0 \leq \epsilon \le\epsilon_0$ and any smooth function $f$, it holds
\begin{eqnarray}\label{uniforml2}
\langle \mathcal{L}^{\epsilon}f, f\rangle_v + |f|^{2}_{L^{2}_{-3/2}} \geq \nu_{0} |f|^{2}_{\epsilon,-3/2}.
\end{eqnarray}
Here for $l \in \mathbb{R}$ and $W_{l}(v) := (1+|v|^{2})^{l/2}$, we denote
\beno |f|^{2}_{\epsilon,l} := |W^{\epsilon}((-\Delta_{\mathbb{S}^{2}})^{1/2}) W_{l}f|^{2}_{L^{2}} + |W^{\epsilon}(D)W_{l}f|^2_{L^{2}} + |W^{\epsilon} W_{l}f|^{2}_{L^{2}}. \eeno
\end{thm}
\begin{rmk}   The characteristic  function $W^{\epsilon}$ indicates the gain of  the weight  simultaneously in phase space, frequency space and anisotropic space for $\mathcal{L}^{\epsilon}$. When $\epsilon$ goes to zero, it becomes to the following coercivity estimate for $\mathcal{L}^L$ which shows that our estimate is sharp.
\begin{eqnarray}\label{uniforml2on-landau}
\langle \mathcal{L}^{L}f, f\rangle_v + |f|^{2}_{L^{2}_{-3/2}} \geq \nu_{0} \left(|W((-\Delta_{\mathbb{S}^{2}})^{1/2}) W_{-3/2}f|^{2}_{L^{2}} + |W(D)W_{-3/2}f|^2_{L^{2}} + |W W_{-3/2}f|^{2}_{L^{2}} \right).
\end{eqnarray}
Here $W(v) := (1+|v|^{2})^{1/2}$.
\end{rmk}

\begin{rmk}   We  emphasize that gain of  regularity only happens in the ``low frequency part'' $(|\xi|\lesssim 1/\epsilon)$ while gain of weight in the phase space only happens in the ``big ball'' $(|v|\lesssim 1/\epsilon)$. In other words, $\mathcal{L}^\epsilon$ still keeps a hyperbolic structure due to the cutoff condition on the deviation angle.
\end{rmk}

The intuition behind Theorem \ref{coercivity-structure} comes from the knowledge that  the linearized Boltzmann collision operator without angular cutoff   corresponds to a unique characteristic function, which captures the key structure of the operator in phase, frequency and anisotropic spaces. One may check it in
\cite{he2018sharp}. In \cite{he2014well}, $W^\epsilon$ had been proved to be the symbol in the frequency space for $Q^\epsilon$ (see \eqref{CoerQe}). Therefore $W^\epsilon$ should be the characteristic function of $\mathcal{L}^\epsilon$.

Note that the behavior of $W^\epsilon$ changes on the region $\{|\cdot|\sim 1/\epsilon\}$ with logarithm correction. The key idea to catch this behavior lies in the following two aspects:
\begin{itemize}
\item
The first one is  the  ``geometric decomposition'' introduced in \cite{he2018sharp} resulting from the geometry inherent in an elastic collision event. It can be explained   as follows:
\beno
f(\xi) - f(\xi^{+}) &=& \big(f(\xi) - f(|\xi|\frac{\xi^{+}}{|\xi^{+}|})\big)+ \big(f(|\xi|\frac{\xi^{+}}{|\xi^{+}|}) - f(\xi^{+})\big),
\eeno
where $\xi^{+}:= \f{|\xi|\sigma+\xi}2$. The first difference $f(\xi) - f(|\xi|\frac{\xi^{+}}{|\xi^{+}|})$ is referred as the spherical part since
$\xi$ and $|\xi|\frac{\xi^{+}}{|\xi^{+}|}$ lie in the sphere centered at origin with radius $|\xi|$. The second difference $f(|\xi|\frac{\xi^{+}}{|\xi^{+}|}) - f(\xi^{+})$ is referred as the radial part since $|\xi|\frac{\xi^{+}}{|\xi^{+}|}$, $\xi^{+}$ and origin are on the same line.
We can extract anisotropic information from the spherical part for both the lower and upper bounds of the operator. One can see section 2.3 for more details.
 \item The second one is the development of some localization techniques: the dyadic decomposition in both phase and frequency spaces and also the partition of unity on the unit sphere. They play an essential role in capturing the leading term and the characteristic function. For more details, we refer readers to Section \ref{Coercivity-Estimate}, Section \ref{Upper-Bound-Estimate} and the Appendix.
\end{itemize}

\subsubsection{\bf   Idea on the spectral gap estimates} This sequel aims to overcome the difficulty raised in {\bf (D2)}.
Roughly speaking, the above Theorem \ref{coercivity-structure} successfully catches the leading term of $\langle \mathcal{L}^{\epsilon}f, f\rangle_{v}$. To eliminate the lower order term, we need a so-called spectral gap estimate for $\mathcal{L}^{\epsilon}$ which is required to hold uniformly in $\epsilon$.  We get the following result:

\begin{thm}\label{micro-dissipation}
There exist two positive universal constants $\epsilon_0$ and $\lambda$  such that for $0 \leq \epsilon \le\epsilon_0$ and any smooth function $f  $, it holds
\begin{eqnarray}\label{spectral-gap}
\langle \mathcal{L}^{\epsilon}f, f\rangle_v \geq \lambda|(\mathbb{I}-\mathbb{P})f|^{2}_{\epsilon,-3/2},
\end{eqnarray}
where $\lambda$ depends only on $\lambda_1^\epsilon$(see \eqref{firstegenvalueLA} for definition) and the constant $\nu_{0}$ appearing in Theorem \ref{coercivity-structure}.
\end{thm}

To give some illustration on the above theorem, we use  $\mathcal{L}^\gamma$  to denote the linearized Boltzmann collision operator associated to the general kernel
\ben\label{generalkernel} B(v-v_*,\sigma):= |v-v_*|^\gamma b(\cos\theta). \een

  The first spectral gap estimate is attributed to \cite{Chang-Uhlenbeck}. In fact, thanks to the simple structure of $\mathcal{L}^0$, which corresponds to the Maxwellian molecules (see \eqref{KernelMaxweM}), the authors explicitly constructed its eigenvalues and their corresponding eigenfunctions. Then the spectral gap estimate easily follows from the fact that $\mathcal{L}^0$ is a self-adjoint operator:
\ben\label{mxsgthm}
 \langle \mathcal{L}^0 f, f\rangle_v  \ge \lambda_1  |(\mathbb{I}-\mathbb{P})f|_{L^2}^2,
\een
where $\lambda_1$ is the first positive eigenvalue of $\mathcal{L}^0$ satisfying \eqref{firstegenvaluemaxwellion}.

Based on a proper decomposition of the operator and also the dyadic decomposition on the modulus of the relative velocity, the authors in \cite{bargermouhot,mouhot2006explicit}
extended the above result to the general case as follows,
\ben\label{mouhotresult} \langle \mathcal{L}^\gamma f, f\rangle_v  \gtrsim C_\gamma C_b  |(\mathbb{I}-\mathbb{P})f|_{L^2_{\gamma/2}}^2,
\een
where  \beno C_b=\inf_{\sigma_1,\sigma_2 \in \SS^2}\int_{\SS^2}\min\{b(\sigma_1\cdot\sigma_3),b(\sigma_2\cdot\sigma_3)\}d\sigma_3.\eeno
We remind readers that \eqref{mouhotresult} is more general but the estimate depends on $C_b$. As a result, \eqref{mouhotresult} cannot be applied directly if $C_b$ is not bounded from below, which unfortunately happens when the angular function $b$ concentrates on the grazing collisions. There are two typical examples:
\beno b(\cos\theta)=\epsilon^{2s-2}(\sin(\theta/2))^{-2-2s}\mathrm{1}_{\sin(\theta/2)\leq \epsilon},\quad\mbox{or}\quad b(\cos\theta)=b^\epsilon(
\cos\theta),\eeno
where $b^\epsilon(\cos \theta)$ is defined in  \eqref{DefScaled-b}. It is not difficult to check that in both cases $C_b$ tends to zero when $\epsilon$ goes to zero. We
fail to get the desired result through \eqref{mouhotresult}. On the other hand, \eqref{mxsgthm} works well for both cases since the following holds true uniformly in $\epsilon$,
\beno \int_0^\pi b(\cos\theta)(1-\cos\theta)\sin\theta d\theta\sim 1. \eeno

From the above short review on the spectral gap estimate, we need a new and  constructive proof for Theorem \ref{micro-dissipation}, which can be regarded as one of our main contributions in this paper.
The key step lies in finding a link  between the desired result \eqref{spectral-gap} and the well-known result \eqref{mxsgthm} by noting that \eqref{mxsgthm} is stable in the Landau approximation thanks to \eqref{firstegenvalueLA}.

Because of some technical restriction (which will be seen soon), for $-3 \leq \gamma \leq 0$,
we need to consider the general linearized collision operator  $\mathcal{L}^{\epsilon,\gamma}$ associated to the Boltzmann kernel
\ben\label{kernelBgamma} B^{\epsilon,\gamma}(v-v_*,\sigma) := \left|v-v_{*}\right|^{\gamma} b^\epsilon(\cos\theta).\een

Our strategy consists of two parts:
\begin{itemize}
\item  The first part utilizes the coercivity estimate in Theorem \ref{coercivity-structure}. More general than Theorem \ref{coercivity-structure} (see Theorem \ref{strong-coercivity}), for $-3 \leq \gamma \leq 0$, we derive
\ben\label{ideasg1}
\langle \mathcal{L}^{\epsilon,\gamma}f, f\rangle_v =\langle\mathcal{L}^{\epsilon}(\mathbb{I}-\mathbb{P})f, (\mathbb{I}-\mathbb{P})f\rangle_v  \ge \nu_{0} |(\mathbb{I}-\mathbb{P})f|^{2}_{\epsilon,\gamma/2}- |(\mathbb{I}-\mathbb{P})f|^{2}_{L^{2}_{\gamma/2}},
\een

\item The second part is to reduce the case of a general potential to the case of Maxwellian molecules with a small correction term. Roughly speaking, for any $0< \delta \leq 1$, we derive
\ben\label{ideasg2}  \langle \mathcal{L}^{\epsilon,\gamma}f, f\rangle_v\geq C_{1} \delta^{-\gamma}(|(\mathbb{I}-\mathbb{P})f|^{2}_{L^{2}_{\gamma/2}} - C_{2} \delta^{\gamma+2}|(\mathbb{I}-\mathbb{P})f|^{2}_{\epsilon,\gamma/2}). \een
 \end{itemize}

When $-2< \gamma \leq 0$, we can make a suitable combination of
\eqref{ideasg1} and \eqref{ideasg2} to get \eqref{spectral-gap1} in Theorem \ref{micro-dissipation1}. For $-3 \leq \gamma \leq -2$, we choose $-2< \alpha,\beta < 0$ such that $\gamma = \alpha+\beta$ and thus $B^{\epsilon,\gamma} = B^{\epsilon,\alpha} \left|v-v_{*}\right|^{\beta}$. Then $B^{\epsilon,\alpha}$ corresponds to $\mathcal{L}^{\epsilon,\alpha}$ which has spectral gap estimate since $\alpha>-2$. We next reduce $\left|v-v_{*}\right|^{\beta}$ to the  Maxwellian molecules $\left|v-v_{*}\right|^{0}$.
In summary, we can deal with $\gamma \in (-2,0)$ in the first stage, and $\gamma \in [-3,-2]$ in the second stage. For details, one can see the proof of Theorem \ref{micro-weight-dissipation} in Section \ref{Spectral-Gap-Estimate}.

We emphasize that the derivation of  \eqref{ideasg2} is very tricky. To this end, we introduce a proper decomposition on the modulus of the relative velocity and a special weight function $U_{\delta}(v) = (1+ \delta^{2}|v|^{2})^{1/2}$
 (see \eqref{specialweightfun}) to keep the symmetric structure of $\langle \mathcal{L}^{\epsilon,\gamma}f, f\rangle_{v}$.

More general than Theorem \ref{micro-dissipation}, we have
\begin{thm}\label{micro-dissipation1} Let $-3 \leq \gamma \leq 0$. There exist two positive universal constants $\epsilon_0$ and $\lambda$  such that for $0 \leq \epsilon \le\epsilon_0$ and any smooth function $f  $, it holds
\begin{eqnarray}\label{spectral-gap1}
\langle \mathcal{L}^{\epsilon,\gamma}f, f\rangle_{v} \geq \lambda|(\mathbb{I}-\mathbb{P})f|^{2}_{\epsilon,\gamma/2},
\end{eqnarray}
where $\lambda$ depends only on $\gamma, \lambda_1^\epsilon$(see \eqref{firstegenvalueLA} for definition) and the constant $\nu_{0}$ appearing in Theorem \ref{coercivity-structure}.
\end{thm}

\subsubsection{\bf Idea on the linear-quasilinear method} In this sequel, we deal with the difficulty raised in {\bf (D3).} As explained before, for \eqref{linearizedBE}, it seems that  the dissipation of $\mathcal{L}^\epsilon$ cannot prevail the nonlinear term $\Gamma^\epsilon(f,f)$, which is the biggest challenge for establishing global well-posedness. To overcome this obstacle, we introduce a so-called ``\,linear-quasilinear method\,''.

In order to explain the method,  we first introduce the truncation with respect to the modulus of the relative velocity. We associate $Q^{\epsilon,\gamma,\eta}$ with kernel $B^{\epsilon,\gamma,\eta} = B^{\epsilon,\gamma} 1_{|v-v_{*}| \geq \eta }$ and denote $\mathcal{L}^{\epsilon,\gamma,\eta}, \mathcal{L}^{\epsilon,\gamma,\eta}_{1}, \mathcal{L}^{\epsilon,\gamma,\eta}_{2}, \Gamma^{\epsilon,\gamma,\eta}$ correspondingly. To eradicate ambiguity, we define explicitly the Boltzmann operator $Q^{\epsilon,\gamma,\eta}$ as follows,
\ben \label{Q-ep-ga-geq-eta}
Q^{\epsilon,\gamma,\eta}(g,h)(v):=
\int_{\R^3}\int_{\mathbb{S}^{2}}B^{\epsilon,\gamma,\eta}(v-v_*,\sigma)\left(g(v'_*) h(v')-g(v_*)h(v)\right)dv_* d\sigma,
\een
where
\beno B^{\epsilon,\gamma,\eta}(v-v_*,\sigma) := \left|v-v_{*}\right|^{\gamma}\mathrm{1}_{|v-v_{*}|\geq \eta}
 b^{\epsilon}(\cos \theta) = \left|v-v_{*}\right|^{\gamma}\mathrm{1}_{|v-v_{*}|\geq \eta} |\ln \epsilon|^{-1}\sin^{-4}(\theta/2) \mathrm{1}_{\sin(\theta/2) \geq \epsilon}.\eeno
Similar to \eqref{DefLep}, we define
\ben\label{Gamma-ep-eta-ga}\Gamma^{\epsilon,\gamma,\eta}(g,h):= \mu^{-1/2} Q^{\epsilon,\gamma,\eta}(\mu^{1/2}g,\mu^{1/2}h),
\\ \label{L-ep-eta-ga}
\mathcal{L}^{\epsilon,\gamma,\eta}g:= -\Gamma^{\epsilon,\gamma,\eta}(\mu^{1/2},g) - \Gamma^{\epsilon,\gamma,\eta}(g, \mu^{1/2}),
\\ \label{L1-ep-eta-ga}
\mathcal{L}^{\epsilon,\gamma,\eta}_{1}g:=  -\Gamma^{\epsilon,\gamma,\eta}(\mu^{1/2},g),
\\ \label{L2-ep-eta-ga} \mathcal{L}^{\epsilon,\gamma,\eta}_{2}g:=  - \Gamma^{\epsilon,\gamma,\eta}(g, \mu^{1/2}).
\een
Then we set
\ben\label{etalower} \qquad B^{\epsilon,\gamma}_\eta:= B^{\epsilon,\gamma}-B^{\epsilon,\gamma,\eta},\, Q^{\epsilon,\gamma}_\eta:= Q^{\epsilon, \gamma}-Q^{\epsilon,\gamma,\eta},\, \mathcal{L}^{\epsilon,\gamma}_{\eta}:= \mathcal{L}^{\epsilon,\gamma} -\mathcal{L}^{\epsilon,\gamma,\eta},\, \Gamma^{\epsilon,\gamma}_{\eta}  := \Gamma^{\epsilon,\gamma}-\Gamma^{\epsilon,\gamma, \eta}. \een

With these notations, we rewrite \eqref{linearizedBE} as
\ben\label{linearizedBENew}
\pa_tf+v\cdot\na_xf+\mathcal{L}^{\epsilon,-3,\eta}f=(-\mathcal{L}^{\epsilon,-3}_{\eta}f+\Gamma^{\epsilon,-3}_{\eta}(f,f)) + \Gamma^{\epsilon,-3, \eta}(f,f).
\een
Now we are in a position to illustrate the linear-quasilinear   method in detail. Since
the standard $L^2$ energy estimate is employed to establish global well-posedness, the  method can be explained in terms of   the integration domain of $L^2$ inner product. In fact,   we separate the   integration domain into two parts: $|v-v_*|\le \eta$ and $|v-v_*|\ge \eta$, where $\eta$ is a parameter and principally it should be sufficiently small.  We call them  respectively the singular region and the regular region since $|v-v_*|^{-3}$ in the Boltzmann kernel $B^\epsilon$ is strongly singular near $0$. The spirit of the new method can be summarized as follows:
(i). In the regular region, we employ the standard linear method by showing that the dissipation of the linear term can dominate the nonlinear term; (ii). While in the singular region, we use the quasi-linear method by utilizing the non-negativity of the solution to eliminate the dangerous strong singularity. More precisely,

\begin{itemize}
\item In the regular region, we follow the linear method to show that the dissipation of  $\mathcal{L}^{\epsilon,-3,\eta}f$ dominates the nonlinear term $\Gamma^{\epsilon,-3, \eta}(f,f)$ under smallness assumption on $f$.
To this end,  technically we are forced to show that
\beno
\langle \mathcal{L}^{\epsilon,-3,\eta}f,f\rangle_v+|f|_{L^2_{-3/2}}^2\ge \nu_{0}|f|_{\epsilon,-3/2}^2\quad\mbox{and}\quad
|\langle \Gamma^{\epsilon,-3, \eta}(f,f),f\rangle_v|\lesssim C(\eta^{-1},f)|f|_{\epsilon,-3/2}^2.
 \eeno
That is, when $\eta$ is sufficiently small, $\mathcal{L}^{\epsilon,-3,\eta} $ yields the same dissipation as $\mathcal{L}^{\epsilon}$ in Theorem \ref{coercivity-structure}.
	\item In the singular region, we use the identity that
\beno &&\langle  -\mathcal{L}^{\epsilon,-3}_{\eta}f+\Gamma^{\epsilon,-3}_{\eta}(f,f), f\rangle_v=
\langle Q^{\epsilon,-3}_\eta(\mu+\mu^\f12f,\mu^{\f12}f),\mu^{-\f12}f\rangle_v+\langle \Gamma^{\epsilon,-3}_{\eta}(f, \mu^{\f12}),f\rangle_{v}\\
&&=\langle Q^{\epsilon,-3}_\eta(\mu+\mu^\f12f, f), f\rangle_v+\int B^{\epsilon,-3}_\eta (\mu^{\f12}+f)_*
(\mu_*'^{\f12}-\mu_*^{\f12})ff'd\sigma dv_*dv+\langle \Gamma^{\epsilon,-3}_{\eta}(f, \mu^{\f12}),f\rangle_{v}.
\eeno
In Theorem \ref{small-part-L+gamma}, we have
\beno \langle  -\mathcal{L}^{\epsilon,-3}_{\eta}f+\Gamma^{\epsilon,-3}_{\eta}(f,f), f\rangle_v \lesssim (\eta^{1/2}+\epsilon^{1/2})(1+|f|_{H^{2}})|f|_{L^2_{\epsilon,-3/2}}^2. \eeno
Let us explain the idea. There are three terms in the second line of the previous identity.
For the latter two terms, we make use of the regularity of $\mu^{\f12}$ to cancel the singularity.
For the first term we shall
use the quasi-linear method to give the estimate. In fact, thanks to the fact that $\mu+\mu^{\f12}f$ is a solution to the original Boltzmann equation,  it holds that $\mu+\mu^{\f12}f\ge0$ which implies that  the   coercivity type estimate \eqref{CoerQe}  holds for $\langle -Q^{\epsilon,-3}_\eta(\mu+\mu^{\f12}f, f), f\rangle_v$. Here we only use the good sign and Cancellation Lemma \ref{cancellation-lemma-general-gamma-minus3-mu} to get
\beno \langle Q^{\epsilon,-3}_\eta(\mu+\mu^{\f12} f, f), f\rangle_v \lesssim (\eta^{1/2}+\epsilon^{1/2})(1+|f|_{H^{2}})|f|_{L^2_{\epsilon,-3/2}}^2.\eeno
One can see subsection 4.3 for details.
\end{itemize}

Using such kind of treatment, we can deal with the highest order derivative in the energy estimates and
capture the highest order dissipation, which is crucial for the near Maxwellian framework. For other lower order derivatives, we have one order derivative to kill the strong singularity in $|v-v_{*}|^{-3}$ near $0$. In order to implement this plan, we need two types of upper bound estimates for the non-linear term, see Table \ref{ResultSummary} at the beginning of Section \ref{Upper-Bound-Estimate} for a summary.



\subsection{Notations, function spaces and organization of the paper}
We list notations and state the organization of the paper in this subsection.

\subsubsection{Notations} Here are the list:

\begin{itemize}
 \item We denote the multi-index $\alpha =(\alpha _1,\alpha _2,\alpha _3)$ with
$|\alpha |=\alpha _1+\alpha _2+\alpha _3$.

 \item We write $a\lesssim b$ to indicate that  there is a
universal constant $C,$ which may be different on different lines,
such that $a\leq Cb$.  We use the notation $a\sim b$ whenever $a\lesssim b$ and $b\lesssim
a$.

\item The notation $a^+$ means the maximum value of $a$ and $0$ and $[a]$ denotes the maximum integer which does not exceed $a$.

\item   The Japanese bracket $\langle \cdot\rangle$ is defined by $\langle v\rangle=(1+|v|^2)^{\frac{1}{2}}$. Then the weight function $W_l$ is defined by $W_l(v):= \langle v\rangle^l $. When $l=1$, we write $W(v):= W_{1}(v) = \langle v\rangle$.

\item   We denote $C(\lambda_1,\lambda_2,\cdots, \lambda_n)$ or $C_{\lambda_1,\lambda_2,\cdots, \lambda_n}$ by a constant depending on   parameters $\lambda_1,\lambda_2,\cdots, \lambda_n$.

\item The notations  $\langle f,g\rangle:= \int_{\R^3}f(v)g(v)dv$ and $(f,g):= \int_{\TT^3\times\R^3} fgdxdv$ are used to denote the inner products for $v$ variable and for $x,v$ variables respectively.

\item  The translator operator $T_u$ is defined by $(T_u f)(v):= f(u+v)$, for any $u,v \in \R^{3}$.

\item As usual, $\mathrm{1}_A$ is the characteristic function of a set $A$.

\item If $A,B$ are two operators, then their commutator $[A,B]:= AB-BA$.

\end{itemize}

\subsubsection{Function spaces} For simplicity, we set $\partial^{\alpha}:=\partial^{\alpha}_x, \partial_{\beta}:= \pa^{\beta}_v, \partial^{\alpha}_{\beta}:=\partial^{\alpha}_x\pa^{\beta}_v$. We will use the following spaces.

\begin{itemize}
\item  For   real number $n, l $, we define the weighted Sobolev space on $\R^3$
\begin{equation*}
H^{n}_l :=\bigg\{f(v)\big| |f|^2_{H^n_l}:= |\langle D\rangle^n W_{l} f|^{2}_{L^{2}} = \int_{\R^3} |(\langle D\rangle^n W_{l} f)(v)|^2 dv
 <+\infty\bigg\},
\end{equation*} Here $a(D)$ is a   differential operator with the symbol
$a(\xi)$ defined by
\beno  \big(a(D)f\big)(v):=\f1{(2\pi)^3}\int_{\R^3}\int_{\R^3} e^{i(v-y)\xi}a(\xi)f(y)dyd\xi.\eeno

\item
 For $n \in \N, l \in \R$, the weighted Sobolev space on $\R^3$ is defined by
\beno H^{n}_{l} := \bigg\{f(v) \big| |f|^{2}_{H^{n}_{l}} :=
\sum_{|\beta| \leq n}  |\pa_{\beta}f|^{2}_{L^{2}_l}< \infty  \bigg\},\eeno
where $|f|_{L^{2}_{l}} := |W_l f|_{L^{2}}$ is the usual $L^{2}$ norm with weight $W_l$.

\item For $n \in \N, l \in \R$, we denote the weighted pure order-$n$ space on $\R^3$ by
\ben \label{pure-order} \dot{H}^{n}_l := \bigg\{f(v) \big| |f|^{2}_{\dot{H}^{n}_{l}} :=
\sum_{|\beta| = n}  |\pa_{\beta} f|^{2}_{L^{2}_l}< \infty  \bigg\}.\een

\item For $m\in\N$, we denote the Sobolev space on $\mathbb{T}^{3}$ by
\begin{equation*} H^{m}_{x} :=\bigg\{f(x)\big| |f|^{2}_{H^{m}_{x}}:= \sum_{|\alpha | \leq m}|\partial^{\alpha} f|^{2}_{L^{2}_{x}}<\infty\bigg\}.
\end{equation*}

\item For $m,n \in \N, l \in \R$, the weighted Sobolev space on $\mathbb{T}^{3}\times\R^{3}$ is defined by
\beno H^{m}_xH^{n}_{l} := \bigg\{f(x,v) \big| \|f\|^{2}_{ H^{m}_xH^{n}_{l}} :=
\sum_{|\alpha| \leq m, |\beta| \leq n}  || \partial^{\alpha}_{\beta} f|_{L^{2}_l} |^{2}_{L^{2}_{x}} < \infty\bigg\}.\eeno
We write $\|f\|_{H^{m}_{x}L^{2}_{l}} := \|f\|_{ H^{m}_xH^{0}_{l} }$ if $n=0$ and   $\|f\|_{L^{2}_{x}L^{2}_{l}} := \|f\|_{ H^{0}_xH^{0}_{l} }$ if $m=n=0$. The space $H^{m}_x\dot{H}^{n}_{l}$ can be similarly defined.

\end{itemize}

\subsubsection{Dyadic decompositions} We now introduce the dyadic
decomposition. Let $B_{4/3} := \{x\in\R^{3}: |x| \leq 4/3\}$ and $C := \{x\in\R^{3}: 3/4 \leq |x| \leq 8/3\}$.  Then one may introduce two
radial functions $\phi \in C_{0}^{\infty}(B_{4/3})$ and $\psi \in C_{0}^{\infty}(C)$ which satisfy
\ben \label{function-phi-psi} 0\leq \phi, \psi \leq 1, \text{ and } \phi(x) + \sum_{j \geq 0} \psi(2^{-j}x) =1, \text{ for all } x \in \R^{3}. \een
Now define $\varphi_{-1}(x) :=  \phi(x)$ and $\varphi_{j}(x) :=  \psi(2^{-j}x)$ for any $x \in \R^{3}$ and $j \geq 0$. Then one has the following dyadic decomposition
\ben \label{dyadic-decomposition} f =\sum_{j=-1}^\infty \mathcal{P}_jf:= \sum_{j=-1}^{\infty} \varphi_{j}f, \een
for any function $f$ defined on $\R^{3}$.
We will  use the notations
\ben\label{defphilh} \mathfrak{F}_{\phi} f:= \phi(\epsilon D) f,\quad   \mathfrak{F}^{\phi} f:=(1-\phi(\epsilon D))f.\een

\subsubsection{Function spaces related to the collision operator}

Now we  introduce some spaces related to the coercivity estimate of $\mathcal{L}^\epsilon$.

\begin{itemize}
\item
{\it Space $L^2_{\epsilon,l}$.} For functions  defined on $\R^3$,
the space $L^2_{\epsilon,l}$ with $l\in\R$ is defined by
\beno
L^2_{\epsilon,l}:=\bigg\{f(v)\big||f|^{2}_{L^2_{\epsilon,l}}:= |W^{\epsilon}((-\Delta_{\mathbb{S}^{2}})^{1/2})W_{l}f|^{2}_{L^{2}} + |W^{\epsilon}(D)W_{l}f|^{2}_{L^{2}} + |W^{\epsilon}W_{l}f|^{2}_{L^{2}}<\infty\bigg\}.
\eeno

\item {\it Space  $H^{m}_xH^{n}_{\epsilon,l}$.} For functions defined on $\mathbb{T}^{3}\times\R^{3}$, the space $H^{m}_xH^{n}_{\epsilon,l}$   with $m,n\in\N$ is defined  by
\beno  H^{m}_xH^{n}_{\epsilon,l}:=\bigg\{f(x,v)\big|
\|f\|^{2}_{H^{m}_xH^{n}_{\epsilon,l}} := \sum_{|\alpha| \leq m, |\beta| \leq n} || \partial^{\alpha}_{\beta} f|_{L^{2}_{\epsilon,l}} |^{2}_{L^{2}_{x}} <\infty\bigg\}.
\eeno
\end{itemize}
We set $\|f\|_{H_x^{m}L^2_{\epsilon,l}}:= \|f\|_{H^{m}_xH^{0}_{\epsilon,l}}$ if $n=0$ and $\|f\|_{L^{2}_{x}L^2_{\epsilon,l}}:= \|f\|_{H^{0}_xH^{0}_{\epsilon,l}}$ if $m=n=0$. Again, the space $H^{m}_x\dot{H}^{n}_{\epsilon,l}$ can be defined accordingly.


\subsubsection{Organization of the paper} In Section \ref{Coercivity-Estimate}, we first give some elementary results and then endeavor to prove the coercivity estimate in Theorem \ref{coercivity-structure}. Section \ref{Spectral-Gap-Estimate} is devoted to the spectral gap estimate in Theorem \ref{micro-dissipation} and \ref{micro-dissipation1}. In Section \ref{Upper-Bound-Estimate}, the upper bound estimates of the nonlinear term $\Gamma^\epsilon$ are provided.
Some commutator estimates are given in Section \ref{Commutator-Estimate}.
In Section \ref{Energy-Estimate}, we prove our main Theorem \ref{asymptotic-result}. In the Appendix, we give some useful results for the sake of completeness.

\setcounter{equation}{0}

\section{Coercivity   estimate} \label{Coercivity-Estimate}
In this section, we will prove Theorem \ref{coercivity-structure}.  In fact, we want to capture the coercivity estimate of $\mathcal{L}^{\epsilon,\gamma,\eta}$ for $-3 \leq \gamma \leq 0$ and small $\eta \geq 0$, which is more general than  Theorem \ref{coercivity-structure}.

Our strategy  relies on the following relation (see \eqref{L-dominate-lower-bound} in the proof of Theorem \ref{strong-coercivity}):
\ben \label{equivalence-relation} \langle \mathcal{L}^{\epsilon,\gamma,\eta} f,f \rangle+|f|_{L^2_{\gamma/2}}^2 \gtrsim \mathcal{N}^{\epsilon,\gamma,\eta}(\mu^{1/2},f) + \mathcal{N}^{\epsilon,\gamma,\eta}(f,\mu^{1/2}), \een
where
\ben  \label{definition-of-N-ep-ga-eta}
\mathcal{N}^{\epsilon,\gamma,\eta}(g,h) := \int_{\mathbb{S}^2 \times \R^6} b^{\epsilon}(\cos\theta)| v-v_{*} |^{\gamma}
\mathrm{1}_{|v-v_{*}|\geq \eta}
g^{2}_{*} (h^{\prime}-h)^{2} d\sigma dv dv_{*}.\een
We remind that   $\eta \geq 0$ is an additional parameter to deal with the high singularity $|v-v_{*}|^{-3}$.
Thanks to \eqref{equivalence-relation},
to get the coercivity estimate of $\mathcal{L}^{\epsilon,\gamma,\eta}$, it suffices to estimate from below the two functionals
$\mathcal{N}^{\epsilon,\gamma,\eta}(\mu^{1/2},f)$  and $ \mathcal{N}^{\epsilon,\gamma,\eta}(f,\mu^{1/2})$.
 We will study $ \mathcal{N}^{\epsilon,\gamma,\eta}(f,\mu^{1/2})$ in  subsection 2.2 and $\mathcal{N}^{\epsilon,\gamma,\eta}(\mu^{1/2},f)$ in subsection 2.3.
The coercivity estimate is obtained in subsection 2.4 by utilizing \eqref{equivalence-relation}.

In the following, we will omit the range of some frequently used variables in the integrals. Usually, $\sigma \in \mathbb{S}^{2}, v, v_{*}, u, \xi \in \mathbb{R}^{3}$.  For example, we set $\int (\cdots) d\sigma := \int_{\mathbb{S}^{2}} (\cdots) d\sigma,  \int (\cdots) d\sigma dv dv_{*} := \int_{\mathbb{S}^{2} \times \mathbb{R}^{3} \times \mathbb{R}^{3}} (\cdots) d\sigma dv dv_{*}$. Integration w.r.t. other variables should be understood in a similar way. Whenever a new variable appears, we will specify its range once and then omit it thereafter.

We begin with some elementary results which will be used frequently throughout the paper.
\subsection{Elementary results}
We collect some properties of the function $W^{\epsilon}$ defined in \eqref{charicter function}.
Note that $W^{\epsilon}$ is a radial function defined on $\mathbb{R}^{3}$.
\begin{lem}\label{property-of-symbol}
For any $ x, y \in \mathbb{R}^{3}$, we have
\ben  
\label{upper-bound-cf} W^{\epsilon}(x) \leq  \min \{\langle x \rangle, |\ln \epsilon|^{-1/2}\epsilon^{-1}\}.
\\
\label{lower-bound-when-small-cf} W^{\epsilon}(x) \gtrsim   \phi(\epsilon^{1/2} x)\langle x \rangle.
\\
\label{lower-bound-when-large-cf} W^{\epsilon}(x) \gtrsim   (1- \phi(\epsilon^{1/2} x))\epsilon^{-1/2}.
\\
\label{low-frequency-lb-cf}  W^{\epsilon}(x) \geq  \phi(\epsilon x) |\ln \epsilon|^{-1/2} \langle x \rangle.
\\
\label{high-frequency-lb-cf}  W^{\epsilon}(x) \gtrsim \left(1-\phi(\epsilon x)\right) |\ln \epsilon|^{-1/2} \epsilon^{-1}.
\\
\label{separate-into-2-cf} W^{\epsilon}(x-y) \leq  W^{\epsilon}(x)W^{\epsilon}(y).
\een
\end{lem}

The following is an interpolation result.
\begin{lem}\label{interWsd} Let $m \geq 0$. For any $\eta>0$, one has
\beno |f|_{H^m}^2\lesssim (\eta+\epsilon)|W^\epsilon(D)f|_{H^m}^2+C(\eta,m)|f|_{L^2}^2. \eeno
\end{lem}
\begin{proof} Let $f_{\epsilon} := \phi(\epsilon^{1/2} D)f, f^{\epsilon} := f-f_{\epsilon}$. That is, $f_{\epsilon}$ and $f^{\epsilon}$ are the low and high frequency parts respectively. Note that the threshold is $|\xi| \sim \epsilon^{-1/2}$.
By interpolation inequality, it is easy to check that
\beno |f|_{H^m}^2\lesssim |f^{\epsilon}|_{H^m}^2+|f_{\epsilon}|_{H^m}^2\lesssim |f^{\epsilon}|_{H^m}^2+\eta |f_{\epsilon}|_{H^{m+1}_l}^{2}+C(\eta,m)|f_{\epsilon}|_{L^2}^2.\eeno
Then the lemma follows from \eqref{lower-bound-when-small-cf} and \eqref{lower-bound-when-large-cf} in Lemma \ref{property-of-symbol}.
\end{proof}
We remark that $W^{\epsilon}$  is a symbol of type $S^{1}_{1,0}$ (see Definition \ref{psuopde} for the definition of a symbol). From which together with
Lemma \ref{operatorcommutator1}, we have
\begin{lem}\label{equivalence} Let $l, m \in \R$, then
\beno|W^\epsilon(D)W_{l}f|_{H^m}^2 \sim |W_{l} W^\epsilon(D)f|_{H^m}^2. \eeno
\end{lem}
Thanks to Lemma \ref{equivalence}, we interchangeably use $|W^\epsilon(D)W_{l}f|_{H^m}^2$ and $|W^\epsilon(D)f|_{H^m_{l}}^2$ in the rest of the paper.

We collect some properties of the translation operator $T_{u}$ defined through $(T_{u}f)(v) := f(u+v)$.
\begin{lem}Thanks to \eqref{separate-into-2-cf}, for $u \in \mathbb{R}^{3}$, we have
\ben \label{translation-out-cf}
|W^{\epsilon}T_{u}f|_{L^{2}} \lesssim W^{\epsilon}(u) |W^{\epsilon}f|_{L^{2}}.
\een
For $u \in \mathbb{R}^{3}, l \in \mathbb{R}$, one has $
(T_{u} W_{l})(v) = \langle v + u \rangle^{l} \lesssim C(l)\langle u \rangle^{|l|} \langle v \rangle^{l}
$
 and thus
\ben \label{translation-out-weight}
|T_{u}f|_{L^{2}_{l}} \lesssim C(l) \langle u \rangle^{|l|} |f|_{L^{2}_{l}}.
\een
\end{lem}

Let us prepare some integrals regrading to the angular function $b^{\epsilon}$ over the unit sphere $\mathbb{S}^{2}$.
\begin{lem} \label{integral-angular-function} If $\epsilon \leq 1/4$, then
\ben \label{order-0}
2\pi |\ln \epsilon|^{-1}\epsilon^{-2} \leq \int b^{\epsilon}  d\sigma \leq 4\pi |\ln \epsilon|^{-1}\epsilon^{-2}.
\\ \label{order-1}
4\pi |\ln \epsilon|^{-1}\epsilon^{-1} \leq \int b^{\epsilon}\sin(\theta/2)  d\sigma \leq 8\pi |\ln \epsilon|^{-1}\epsilon^{-1}.
 \\ \label{order-2}  4\pi \leq \int b^{\epsilon} \sin^{2}(\theta/2) d\sigma \leq 8\pi
.
\een
\end{lem}
\begin{proof}
By Assumption {\bf  (A2)}, we can write
$b^{\epsilon}(\cos\theta) = |\ln \epsilon|^{-1} \sin^{-4}(\theta/2) \mathrm{1}_{\sqrt{2}/2 \geq \sin(\theta/2) \geq \epsilon}$. Note that $d\sigma =\sin\theta d\theta d\varphi= 4 \sin (\theta/2) d \sin (\theta/2) d\varphi$, we have
\beno \int b^{\epsilon}  d\sigma &=&  4 |\ln \epsilon|^{-1}\int_{0}^{\pi} \int_{0}^{2\pi} \mathrm{1}_{\sqrt{2}/2 \geq \sin(\theta/2)\geq \epsilon}\sin^{-3}(\theta/2)  d \sin (\theta/2) d\varphi
\\&=&8\pi |\ln \epsilon|^{-1} \int_{\epsilon}^{\sqrt{2}/2} t^{-3} dt
=4\pi |\ln \epsilon|^{-1} (\epsilon^{-2}-2).
\eeno
If $\epsilon \leq 1/2$, then $ \epsilon^{-2}\geq 4$, which gives \eqref{order-0}. Similarly, we have
\beno \int b^{\epsilon} \sin(\theta/2) d\sigma
= 8\pi |\ln \epsilon|^{-1} \int_{\epsilon}^{\sqrt{2}/2} t^{-2} dt
=8\pi |\ln \epsilon|^{-1} (\epsilon^{-1}-\sqrt{2}).
\eeno
If $\epsilon \leq 1/4$, then $ \epsilon^{-1}\geq 4$, which gives \eqref{order-1}. Similarly, we have
\beno \int b^{\epsilon} \sin^{2}(\theta/2) d\sigma
= 8\pi |\ln \epsilon|^{-1} \int_{\epsilon}^{\sqrt{2}/2} t^{-1} dt =8\pi |\ln \epsilon|^{-1} (|\ln \epsilon| - \frac{1}{2}\ln 2) .
\eeno
If $\epsilon \leq 1/2$, then $\ln 2 \leq |\ln \epsilon|$ which gives \eqref{order-2}.
\end{proof}

Cancellation lemma is very important especially when we need to shift regularity between the three functions $g,h,f$ of the inner product $\langle Q(g, h), f \rangle$. Depending on the parameter $\gamma$ in the kernel $B(v-v_{*}, \sigma)=|v-v_{*}|^{\gamma}b(\cos\theta)$, there are two types
of cancellation lemma.  For the case $\gamma>-3$, one may refer \cite{alexandre2000entropy}.
We mainly concern the case $\gamma=-3$ and recall the following cancellation lemma in \cite{alexandre2002boltzmann}.

\begin{lem} [Cancellation Lemma]
Recalling from Proposition 3 in \cite{alexandre2002boltzmann}, we define
\beno
J^{\epsilon}(z) := \mathrm{1}_{|z|\leq 1}\frac{2\pi}{|z|^{3}} \int_{2 \cos^{-1}(|z|)}^{\pi/2} b^{\epsilon}(\cos\theta) \sin\theta d\theta,
\eeno
whose $L^{1}$ norm is  bounded uniformly in $\epsilon$,
\ben \label{uniform-l1-estimate}
|J^{\epsilon}|_{L^{1}} = -8\pi^{2} \int_{0}^{\pi/2} b^{\epsilon}(\cos\theta) \left(\ln\cos(\theta/2)\right)\sin\theta d\theta \lesssim 1.
\een
Associated to $B^{\epsilon,\gamma,\delta} = \mathrm{1}_{|v-v_{*}|\geq \delta} |v-v_{*}|^{-3}b^{\epsilon}(\cos\theta)$, the convolution kernel is  $S^{\epsilon}_{\delta}(z) := \delta^{-3}J^{\epsilon}(z/\delta)$. That is,
\ben \label{cancellation-result-convolution}
\int B^{\epsilon,\gamma,\delta} g_{*}(h^{\prime}-h)d\sigma dv dv_{*} = \int S^{\epsilon}_{\delta}(v-v_{*}) g_{*}h dv dv_{*}.
\een
\end{lem}

For our purpose, we derive the following
various types of cancellation effect.
\begin{lem} [Cancellation Lemma Continued] \label{cancellation-lemma-general-gamma-minus3-mu} Let $p,q\ge1$ satisfying $1/p+1/q=1$.
\begin{itemize}
\item  Let $\delta=0, a \in \mathbb{R}$, then
\ben \label{classical-case} |\int B^{\epsilon,\gamma,0} g_{*}(h^{\prime}-h)d\sigma dv dv_{*}| \lesssim |\mu^{-a}g|_{L^{p}}|\mu^{a}h|_{L^{q}}.
\een
\item Let $1 \geq \delta \geq 0, a \geq 0$, then
\ben \label{geq-delta-mu} |\int B^{\epsilon,\gamma,\delta} g_{*}(h^{\prime}-h)d\sigma dv dv_{*}| \lesssim \mathrm{e}^{a}|\mu^{-2a}g|_{L^{p}}|\mu^{a}h|_{L^{q}}.
\een
\item Let $1 \geq \eta>\delta \geq 0, a \geq 0$, set $B^{\epsilon,\gamma,\delta}_{\eta} := B^{\epsilon,\gamma}\mathrm{1}_{\delta \leq |v-v_{*}|< \eta}$,
then
\ben \label{0-delta-eta-mu} &&|\int B^{\epsilon,\gamma,\delta}_{\eta} g_{*}(h^{\prime}-h)d\sigma dv dv_{*}| \lesssim  \mathrm{e}^{a}|\mu^{-2a}g|_{L^{p}}|\mu^{a}h|_{L^{q}},
\\ \label{0-delta-eta-small-eta-epsilon-mu} &&|\int B^{\epsilon,\gamma,\delta}_{\eta} g_{*}(h^{\prime}-h)d\sigma dv dv_{*}| \lesssim  (\eta + \epsilon^{1/2})(1+a)\mathrm{e}^{3a}|W^{\epsilon}(D)\mu^{-2a}g|_{L^{2}}|\mu^{a}h|_{L^{2}},
\\ \label{0-delta-eta-small-eta-epsilon-hf-mu} &&|\int B^{\epsilon,\gamma,\delta}_{\eta} g_{*}\left((h f)^{\prime}-hf\right)d\sigma dv dv_{*}|
\\ &\lesssim&  (\eta + \epsilon^{1/2})(1+a)\mathrm{e}^{3a}|\mu^{-2a}g|_{L^{\infty}}(|W^{\epsilon}(D)\mu^{a/2}h|_{L^{2}}|\mu^{a/2}f|_{L^{2}}
+|\mu^{a/2}h|_{L^{2}}|W^{\epsilon}(D)\mu^{a/2}f|_{L^{2}}). \nonumber
\een
\end{itemize}
\end{lem}
\begin{proof} Recalling from \eqref{cancellation-result-convolution} and $S^{\epsilon}_{\delta}(z) = \delta^{-3}J^{\epsilon}(z/\delta)$, we have
\ben \label{cancellation-result-convolution-2} \int B^{\epsilon,\gamma,\delta} g_{*}(h^{\prime}-h)d\sigma dv dv_{*}  = \int J^{\epsilon}(u)g(\delta u +v)h(v) dv du.
\een
We set to prove \eqref{classical-case}. By \eqref{cancellation-result-convolution-2}, H\"{o}lder's inequality and \eqref{uniform-l1-estimate}, we have
\beno |\int B^{\epsilon,\gamma,0} g_{*}(h^{\prime}-h) d\sigma dv dv_{*}| = |\lim_{\delta \rightarrow 0} \int J^{\epsilon}(u)g(\delta u +v)h(v) dv du| \leq |J^{\epsilon}|_{L^{1}}|g h|_{L^{1}} \lesssim |\mu^{-a}g|_{L^{p}}|\mu^{a}h|_{L^{q}}.
\eeno

We now go to prove \eqref{geq-delta-mu}. It is easy to check that for any $ 0 \leq  \alpha < 1$,  there holds
\ben  \label{basic-inequality-1}
|x|^{2} \geq \alpha |y|^{2}- \frac{\alpha}{1-\alpha}|x-y|^{2}.
\een
Taking $\alpha = 1/2$ in \eqref{basic-inequality-1}, if $|x-y| \leq 1$, we have $|x|^{2} \geq \frac{1}{2} |y|^{2}-|x-y|^{2} \geq \frac{1}{2} |y|^{2}-1$, and thus
\ben \label{to-1-2} |x-y| \leq 1 \Rightarrow \mu(x) = (2\pi)^{-3/2} e^{-|x|^{2}/2} \leq \mathrm{e}^{1/2}\mu^{1/2}(y). \een
From which, we get $ \mu^{a}(v) \leq e^{a/2} \mu^{a/2}(v+\delta u)$ for $a \geq 0$.
That is, $1 \leq  e^{a} \mu^{-2a}(v) \mu^{a}(v+\delta u)$. By symmetry, $1 \leq  e^{a} \mu^{a}(v) \mu^{-2a}(v+\delta u)$.
Together with H\"{o}lder's inequality, we get for any $|u| \leq 1, \delta \leq 1$,
\ben \label{l-infty-estimate} |\int_{\R^{3}} g(\delta u +v)h(v) dv| \leq \mathrm{e}^{a} \int_{\R^{3}} \mu^{a}(v) \mu^{-2a}(v+\delta u) |g(\delta u +v)| |h(v)| d v \leq \mathrm{e}^{a}|\mu^{-2a}g|_{L^{p}}|\mu^{a}h|_{L^{q}}. \een
Recalling $J^{\epsilon}(u)$ has support in $|u|\leq 1$ and the parameter $ \delta \leq 1$,  plugging \eqref{l-infty-estimate}  into \eqref{cancellation-result-convolution-2},
 together with \eqref{uniform-l1-estimate},
we get \eqref{geq-delta-mu}.

Since $ B^{\epsilon,\gamma,\delta}_{\eta} =  B^{\epsilon,\gamma,\delta}- B^{\epsilon,\gamma,\eta}$, we have
\beno  |\int B^{\epsilon,\gamma,\delta}_{\eta} g_{*}(h^{\prime}-h) d\sigma dv dv_{*}| \leq  |\int B^{\epsilon,\gamma,\delta}g_{*}(h^{\prime}-h)d\sigma dv dv_{*}| +|\int B^{\epsilon,\gamma,\eta} g_{*}(h^{\prime}-h)d\sigma dv dv_{*}| .
\eeno
Then we get \eqref{0-delta-eta-mu} by using \eqref{geq-delta-mu}.

We now turn to \eqref{0-delta-eta-small-eta-epsilon-mu}.
For notational brevity, we set $G = \mu^{-2a} g, H = \mu^{a} h$ and then have
\beno \int B^{\epsilon,\gamma,\delta}_{\eta} g_{*}(h^{\prime}-h)d\sigma dv dv_{*} &=& \int  J^{\epsilon}(u) \left(g(\delta u +v) - g(\eta u +v)\right) h(v) dv du
\\&=& \int J^{\epsilon}(u) \left(  (\mu^{2a}G)(\delta u +v) - (\mu^{2a}G)(\eta u +v)\right) (\mu^{-a} H)(v) dv du
:= \mathcal{I}(G,H).
\eeno
We decompose $G = G_{\epsilon} + G^{\epsilon}$, where
\ben \label{low-high-separation-ep-1/2}
G_{\epsilon}:= \phi(\epsilon^{1/2} D) G, G^{\epsilon}:= G- G_{\epsilon} = (1- \phi(\epsilon^{1/2} D))G. \een
From which we have
$\mathcal{I}(G,H) = \mathcal{I}(G_{\epsilon},H) + \mathcal{I}(G^{\epsilon},H).$
Then by Taylor expansion, we get
\ben \label{taylor-expansion-G-epsilon}
(\mu^{2a}G_{\epsilon})(\delta u +v) - (\mu^{2a}G_{\epsilon})(\eta u +v) = (\delta-\eta)\int_{0}^{1} (\nabla (\mu^{2a}G_{\epsilon}) )(v(\kappa)) \cdot u d\kappa.
\een
where $v(\kappa) := \left(\kappa \delta + (1-\kappa)\eta\right)u + v$. Noting that $(a|v|+1) \mu^{a/2} \leq C(1+a)$ for some universal constant $C$, we have
\ben\label{derivative-estimate}
 |\nabla \left( \mu^{2a}G_{\epsilon} \right)|   \leq (a|v|+1) \mu^{2a}(|G_{\epsilon}| + |\nabla G_{\epsilon} |) \leq C(1+a)\mu^{3a/2}(|G_{\epsilon}| + |\nabla G_{\epsilon} |).\een
Plugging \eqref{taylor-expansion-G-epsilon} and \eqref{derivative-estimate} into the definition of $\mathcal{I}(G_{\epsilon},H)$, we get
\ben \label{inter-G-epsilon}
|\mathcal{I}(G_{\epsilon},H)|&\leq&  C(1+a)(\eta-\delta)\int J^{\epsilon}(u) (|G_{\epsilon}(v(\kappa)) | + |(\nabla G_{\epsilon}) (v(\kappa))|)
\nonumber \\&&\times \mu^{3a/2}(v(\kappa)) \mu^{-a}(v)|H(v)|dudvd\kappa.
\een
Taking $\alpha = 3/4$ in \eqref{basic-inequality-1}, if $|x-y| \leq 1$, we have $|x|^{2} \geq \frac{3}{4} |y|^{2}-3|x-y|^{2} \geq \frac{3}{4} |y|^{2}-3$, and thus
\ben \label{to-3-4} |x-y| \leq 1 \Rightarrow  \mu(x) \leq \mathrm{e}^{3/2}\mu^{3/4}(y). \een
Since $|v(\kappa)-v| = |\left(\kappa \delta + (1-\kappa)\eta\right)u| \leq \eta|u| \leq 1$, by \eqref{to-3-4},
we get $ \mu^{3a/2}(v(\kappa)) \mu^{-a}(v) \leq \mathrm{e}^{3a}$ for $a \geq 0$.
Plugging which into \eqref{inter-G-epsilon}, we get
\ben \label{eliminate-weight}
|\mathcal{I}(G_{\epsilon},H)| \leq  C(1+a) \mathrm{e}^{3a}(\eta-\delta) \int J^{\epsilon}(u) (|G_{\epsilon}(v(\kappa))| + |(\nabla G_{\epsilon}) (v(\kappa))|)|H(v)|dudvd\kappa.
\een
By the change of variable $v \rightarrow v(\kappa)$ and Cauchy-Schwartz inequality, we get
\beno |\mathcal{I}(G_{\epsilon},H)| \leq C(1+a)\mathrm{e}^{3a} (\eta-\delta)|J^{\epsilon}|_{L^{1}}|G_{\epsilon}|_{H^{1}}|H|_{L^{2}}.
\eeno
By \eqref{lower-bound-when-small-cf}, one has $|G_{\epsilon}|^{2}_{H^{1}} \lesssim |W^{\epsilon}(D)G|^{2}_{L^{2}}. $
From which together with \eqref{uniform-l1-estimate}, we have
\ben \label{G-lower-frequency}
|\mathcal{I}(G_{\epsilon},H)| \lesssim (1+a)\mathrm{e}^{3a} \eta |W^{\epsilon}(D)G|_{L^{2}}|H|_{L^{2}}.
\een
By \eqref{lower-bound-when-large-cf}, one has
$|G^{\epsilon}|_{L^{2}} \lesssim \epsilon^{1/2} |W^{\epsilon}(D)G|_{L^{2}}. $
From which together with \eqref{0-delta-eta-mu}, we get
\ben\label{G-higher-frequency}
 |\mathcal{I}(G^{\epsilon},H)| \lesssim   \mathrm{e}^{a}|G^{\epsilon}|_{L^{2}}|H|_{L^{2}} \lesssim \epsilon^{1/2}\mathrm{e}^{a}|W^{\epsilon}(D)G|_{L^{2}}|H|_{L^{2}}. \een
Patching  together \eqref{G-lower-frequency} and  \eqref{G-higher-frequency}, we get \eqref{0-delta-eta-small-eta-epsilon-mu}.

In the last, we go to prove \eqref{0-delta-eta-small-eta-epsilon-hf-mu}. Thanks to the convolution structure in \eqref{cancellation-result-convolution-2}, we have
\beno \int B^{\epsilon,\gamma,\delta}_{\eta} g_{*}\left((h f)^{\prime}-h f\right) d\sigma dv dv_{*} = \int J^{\epsilon}(u) \left( (h f)(v-\delta u) - (h f)(v-\eta u)\right) g(v) dv du.
\eeno
In this part, we set $G = \mu^{-2a} g, H = \mu^{a/2} h, F = \mu^{a/2} f$ and thus have
\beno \int B^{\epsilon,\gamma,\delta}_{\eta} g_{*}\left((h f)^{\prime}-h f\right) d\sigma dv dv_{*}&=&\int J^{\epsilon}(u) \left((\mu^{-a} H F)(v-\delta u) - (\mu^{-a}H F)(v-\eta u)\right) (\mu^{2a}G)(v) dv du
\\&:=&\mathcal{J}(H,F).
\eeno
Recalling from \eqref{low-high-separation-ep-1/2} for the definition of $H_{\epsilon}, H^{\epsilon}, F_{\epsilon}, F^{\epsilon}$.
Decomposing $H = H_{\epsilon} + H^{\epsilon}, F = F_{\epsilon} + F^{\epsilon}$, we get
\beno \mathcal{J}(H,F)= \mathcal{J}(H_{\epsilon},F_{\epsilon})+\mathcal{J}(H_{\epsilon},F^{\epsilon})+\mathcal{J}(H^{\epsilon},F_{\epsilon})
+\mathcal{J}(H^{\epsilon},F^{\epsilon}).
\eeno
For $\mathcal{J}(H_{\epsilon},F_{\epsilon})$, we
apply Taylor expansion to $\mu^{-a}H_{\epsilon}F_{\epsilon}$. Similar to \eqref{taylor-expansion-G-epsilon}, \eqref{derivative-estimate}, \eqref{inter-G-epsilon} and \eqref{eliminate-weight},
we get
\beno |\mathcal{J}(H_{\epsilon},F_{\epsilon})| &\leq&  C(|a|+1)\mathrm{e}^{3a} (\eta-\delta) \int J^{\epsilon}(u) (|(H_{\epsilon} F_{\epsilon})(v(\kappa)) | + |( F_{\epsilon} \nabla H_{\epsilon} ) (v(\kappa))  |+ |( H_{\epsilon} \nabla F_{\epsilon} ) (v(\kappa))|)
\\&&\times|G(v)|dudvd\kappa
\\&\leq&   C(|a|+1)\mathrm{e}^{3a} (\eta-\delta)|J^{\epsilon}|_{L^{1}}|G|_{L^{\infty}}(|H_{\epsilon}|_{L^{2}}|F_{\epsilon}|_{L^{2}}
+|H_{\epsilon}|_{H^{1}}|F_{\epsilon}|_{L^{2}}+|H_{\epsilon}|_{L^{2}}|F_{\epsilon}|_{H^{1}})
\\&\lesssim&   (|a|+1)\mathrm{e}^{3a}\eta|G|_{L^{\infty}}(|W^{\epsilon}(D)H|_{L^{2}}|F|_{L^{2}}
+|H|_{L^{2}}|W^{\epsilon}(D)F|_{L^{2}}).
\eeno
where we take out $|G|_{L^{\infty}}$ and use the change of variable $v \rightarrow v(\kappa)$, Cauchy-Schwartz inequality and \eqref{lower-bound-when-small-cf}. Using \eqref{to-1-2} to deal with the $\mu$-type weight, taking out $|G|_{L^{\infty}}$, applying Cauchy-Schwartz inequality and using \eqref{lower-bound-when-large-cf}, we get
\beno |\mathcal{J}(H_{\epsilon},F^{\epsilon})+ \mathcal{J}(H^{\epsilon},F_{\epsilon})+\mathcal{J}(H^{\epsilon},F^{\epsilon})| &\leq&  2\mathrm{e}^{a}|J^{\epsilon}|_{L^{1}}|G|_{L^{\infty}}(|H_{\epsilon}|_{L^{2}}|F^{\epsilon}|_{L^{2}}
+|H^{\epsilon}|_{L^{2}}|F_{\epsilon}|_{L^{2}}+|H^{\epsilon}|_{L^{2}}|F^{\epsilon}|_{L^{2}})
\\&\leq&  2^{4}\epsilon^{1/2} \mathrm{e}^{a}|J^{\epsilon}|_{L^{1}}|G|_{L^{\infty}}(|W^{\epsilon}(D)H|_{L^{2}}|F|_{L^{2}}
+|H|_{L^{2}}|W^{\epsilon}(D)F|_{L^{2}}).
\eeno
Patching together the estimates of $|\mathcal{J}(H_{\epsilon},F_{\epsilon})|$ and $|\mathcal{J}(H_{\epsilon},F^{\epsilon})+ \mathcal{J}(H^{\epsilon},F_{\epsilon})+\mathcal{J}(H^{\epsilon},F^{\epsilon})|$, we get \eqref{0-delta-eta-small-eta-epsilon-hf-mu}.
\end{proof}

\subsection{Gain of weight from $\mathcal{N}^{\epsilon,\gamma,\eta}(f,\mu^{1/2}) $}
The functional  $\mathcal{N}^{\epsilon,\gamma,\eta}(f,\mu^{1/2})$ produces gain of the weight $W^{\epsilon}$ in the phase space.
\begin{lem}\label{lowerboundpart1} Let $-3 \leq \gamma \leq 0$.
There exists $\epsilon_{0} >0$ such that for  $0< \epsilon \leq \epsilon_{0}$ and $0 \leq  \eta \leq 1$,
\beno
\mathcal{N}^{\epsilon,\gamma,\eta}(f,\mu^{1/2}) + |f|^{2}_{L^{2}_{\gamma/2}} \geq  C_{\gamma}|W^{\epsilon}f|^{2}_{L^{2}_{\gamma/2}},
\eeno
where $C_{\gamma}$ is a positive constant depending only on $\gamma$.
\end{lem}
\begin{proof} If $0 \leq  \eta \leq 1$, one has $\mathcal{N}^{\epsilon,\gamma,1}(f,\mu^{1/2}) \leq \mathcal{N}^{\epsilon,\gamma,\eta}(f,\mu^{1/2})$. Therefore it suffices to consider the lower bound of $\mathcal{N}^{\epsilon,\gamma,1}(f,\mu^{1/2})$. The proof is divided into four steps. Let $0 <\delta <1 \leq R$ be two constants which will be determined later.

{\it{Step 1: $16/\pi \leq |v_{*}| \leq \delta/\epsilon.$}}
One has $\nabla \mu^{1/2} = -\frac{\mu^{1/2}}{2} v$ and $\nabla^{2} \mu^{1/2} = \frac{\mu^{1/2}}{4} (-2I_{3}+v \otimes v)$. By Taylor expansion, we have
\beno
\mu^{1/2}(v^{\prime}) - \mu^{1/2}(v) = -\frac{\mu^{1/2}(v)}{2} v \cdot (v^{\prime}-v) + \int_{0}^{1} \frac{1-\kappa}{2} (\nabla^{2} \mu^{1/2}) (v(\kappa)):(v^{\prime}-v)\otimes(v^{\prime}-v) d \kappa,
\eeno
where $v(\kappa) := v+\kappa (v^{\prime}-v)$.
Thanks to the fact $(a-b)^{2} \geq \frac{a^{2}}{2} - b^{2}$, we have
\ben \label{main-term-and-lower-order}
(\mu^{1/2}(v^{\prime}) - \mu^{1/2}(v))^{2} \geq \frac{\mu(v)}{8} |v \cdot (v^{\prime}-v)|^{2} - \frac{1}{4}\int_{0}^{1} (1-\kappa)^{2} |(\nabla^{2} \mu^{1/2}) (v(\kappa))|^{2}|v^{\prime}-v|^{4} d \kappa.
\een

For $\epsilon \leq \pi\delta/16$,  we set $A(\epsilon, \delta) = \{(v_{*},v,\sigma): 16/\pi \leq |v_{*}| \leq \delta/\epsilon, |v| \leq 8/\pi, \epsilon \leq \sin(\theta/2) \leq 4\delta|v_{*}|^{-1}\}$. It is easy to check that $A(\epsilon, \delta) $ is  non-empty. Thus we have
\begin{eqnarray}\label{vsmallvstarsmall}
\mathcal{N}^{\epsilon,\gamma,1}(f,\mu^{1/2}) &\geq& \int B^{\epsilon,\gamma,1} \mathrm{1}_{A(\epsilon, \delta)} f_{*}^{2} (\mu^{\prime 1/2}-\mu^{1/2})^{2} d\sigma dv dv_{*}
\nonumber \\&\geq& \frac{1}{8}\int B^{\epsilon,\gamma,1} \mathrm{1}_{A(\epsilon, \delta)} \mu(v)|v \cdot (v^{\prime}-v)|^{2} f_{*}^{2}  d\sigma dv dv_{*} \nonumber
\\&& - \frac{1}{4} \int B^{\epsilon,\gamma,1} \mathrm{1}_{A(\epsilon, \delta)} |(\nabla^{2} \mu^{1/2}) (v(\kappa))|^{2}|v^{\prime}-v|^{4} f_{*}^{2}  d\sigma dv dv_{*} d\kappa \nonumber
\\&:=& \frac{1}{8}\mathcal{I}_{1}^{\epsilon,\gamma,1} (\delta) - \frac{1}{4}\mathcal{I}_{2}^{\epsilon,\gamma,1} (\delta).
\end{eqnarray}

\underline{Estimate of $\mathcal{I}_{1}^{\epsilon,\gamma,1} (\delta)$.} For fixed $v, v_*$, we introduce an orthogonal basis $(h^{1}_{v,v_{*}},h^{2}_{v,v_{*}}, \frac{v-v_{*}}{|v-v_{*}|})$ such that $d\sigma= \sin\theta d\theta d\varphi$. Then one has
\beno
\frac{v^{\prime}-v}{|v^{\prime}-v|} = \cos\frac{\theta}{2}\cos\varphi h^{1}_{v,v_{*}} + \cos\frac{\theta}{2}\sin\varphi h^{2}_{v,v_{*}} -\sin\frac{\theta}{2} \frac{v-v_{*}}{|v-v_{*}|},
\\
\frac{v}{|v|} = a_{1} h^{1}_{v,v_{*}} + a_{2} h^{2}_{v,v_{*}} + a_{3} \frac{v-v_{*}}{|v-v_{*}|},
\eeno
where $a_{3}=\frac{v}{|v|}\cdot \frac{v-v_{*}}{|v-v_{*}|}$ and $a_{1}, a_{2}$ are  constants independent of $\theta$ and $\varphi$. Then we have
\beno
|\frac{v}{|v|} \cdot \frac{v^{\prime}-v}{|v^{\prime}-v|}|^{2}  &=& |a_{1}\cos\frac{\theta}{2}\cos\varphi + a_{2}\cos\frac{\theta}{2}\sin\varphi - a_{3}\sin\frac{\theta}{2}|^{2}
\\&=& a^{2}_{1}\cos^{2}\frac{\theta}{2}\cos^{2}\varphi + a^{2}_{2}\cos^{2}\frac{\theta}{2}\sin^{2}\varphi + a^{2}_{3}\sin^{2}\frac{\theta}{2}
\\&& + 2a_{1}a_{2}\cos^{2}\frac{\theta}{2}\cos\varphi\sin\varphi - 2a_{3}\cos\frac{\theta}{2}\sin\frac{\theta}{2}(a_{1}\cos\varphi + a_{2}\sin\varphi).
\eeno
Integrating with respect to $\sigma$, we have
\ben \label{key-quantity}
\int b^{\epsilon}(\cos\theta)\mathrm{1}_{A(\epsilon, \delta)}|v \cdot (v^{\prime}-v)|^{2}d\sigma &=& \int_{0}^{\pi}\int_{0}^{2\pi}b^{\epsilon}(\cos\theta)\sin\theta \mathrm{1}_{A(\epsilon, \delta)}|v \cdot (v^{\prime}-v)|^{2} d\theta d\varphi
\\ &\geq& \pi(a^{2}_{1}+a^{2}_{2})|v|^{2}|v-v_{*}|^{2}
\int_{0}^{\pi} b^{\epsilon}(\cos\theta)\sin\theta \cos^{2}\frac{\theta}{2} \sin^{2}\frac{\theta}{2}\mathrm{1}_{A(\epsilon, \delta)}d\theta. \nonumber
\een
Let $B(\epsilon, \delta) := \{(v_{*},v): 16/\pi \leq |v_{*}| \leq \delta/\epsilon, |v| \leq 8/\pi\}$. If $(v,v_{*}) \in B(\epsilon, \delta)$,
direct computation gives
\beno &&\int_{0}^{\pi} b^{\epsilon}(\cos\theta)\sin\theta \cos^{2}\frac{\theta}{2} \sin^{2}\frac{\theta}{2}\mathrm{1}_{\epsilon \leq \sin(\theta/2) \leq 4\delta|v_{*}|^{-1}}d\theta
\\&=& 4 |\ln \epsilon|^{-1} \int_{0}^{\pi}  \cos^{2}\frac{\theta}{2} \sin^{-1}\frac{\theta}{2}\mathrm{1}_{\epsilon \leq \sin(\theta/2) \leq 4\delta|v_{*}|^{-1}}d(\sin\frac{\theta}{2})
= 4 |\ln \epsilon|^{-1} \int_{0}^{1}  (1-t^{2}) t^{-1}\mathrm{1}_{\epsilon \leq t \leq 4\delta|v_{*}|^{-1}}dt
\\&\geq& 2 |\ln \epsilon|^{-1} \int_{\epsilon}^{4\delta|v_{*}|^{-1}}   t^{-1} dt
= 2 |\ln \epsilon|^{-1} (\ln (4\delta)-\ln |v_*|-\ln \epsilon),  \eeno
where we use $t \leq 4 \delta |v_{*}|^{-1} \leq  \sqrt{2}/2$ and $1-t^{2} \geq 1/2$. Back to \eqref{key-quantity}, we get
\beno
&& \int b^{\epsilon}(\cos\theta)\mathrm{1}_{A(\epsilon, \delta)}|v \cdot (v^{\prime}-v)|^{2}d\sigma
\\&\geq& 2\pi(a^{2}_{1}+a^{2}_{2})|v|^{2}|v-v_{*}|^{2}\mathrm{1}_{B(\epsilon, \delta)}|\ln \epsilon|^{-1}(\ln (4\delta)-\ln |v_*|-\ln \epsilon).
\eeno
If $(v,v_{*}) \in B(\epsilon, \delta)$, then $|v-v_{*}| \geq 8/\pi \geq  1$, which gives
\beno
\mathcal{I}_{1}^{\epsilon,\gamma,1} (\delta) &\geq&  2\pi|\ln \epsilon|^{-1} \int (a^{2}_{1}+a^{2}_{2})|v-v_{*}|^{\gamma+2}|v|^{2}\mathrm{1}_{B(\epsilon, \delta)}(\ln (4\delta)-\ln |v_*|-\ln \epsilon) \mu(v) f_{*}^{2}dv dv_{*}
\\&=& 2\pi|\ln \epsilon|^{-1}\int (1-(\frac{v}{|v|}\cdot\frac{v_{*}}{|v_{*}|})^{2})|v_{*}|^{2}|v-v_{*}|^{\gamma}|v|^{2}(\ln (4\delta)-\ln |v_*|-\ln \epsilon) \mathrm{1}_{B(\epsilon, \delta)} \mu(v) f_{*}^{2}dv dv_{*},
\eeno
where we use the fact $a_1^2+a_2^2+a_3^2=1$ and the law of sines $$(1-(\frac{v}{|v|}\cdot\frac{v_{*}}{|v_{*}|})^{2})^{-1}|v-v_{*}|^{2}= (1-a_3^2)^{-1}|v_{*}|^{2}.$$
If $(v,v_{*}) \in B(\epsilon, \delta)$, one has $ | v_* |/2 \leq |v-v_{*}| \leq 3| v_* |/2$, which yields
\ben \label{I-1-lower-bound}
\mathcal{I}_{1}^{\epsilon,\gamma,1} (\delta) &\geq&
2\pi (3/2)^{\gamma}|\ln \epsilon|^{-1}\int (1-(\frac{v}{|v|}\cdot\frac{v_{*}}{|v_{*}|})^{2})|v_{*}|^{\gamma+2} |v|^{2}(\ln (4\delta)-\ln |v_*|-\ln \epsilon) \mathrm{1}_{B(\epsilon, \delta)} \mu(v) f_{*}^{2}dv dv_{*}
\nonumber \\&=& 2\pi (3/2)^{\gamma} c_{1}
|\ln \epsilon|^{-1} \int (\ln (4\delta)-\ln |v_*|-\ln \epsilon) | v_{*} |^{\gamma+2}\mathrm{1}_{16/\pi \leq |v_{*}|\leq \delta/\epsilon}  f_{*}^{2} dv_{*}
\een
where we denote $c_{1} := \int (1-(\frac{v}{|v|}\cdot\frac{v_{*}}{|v_{*}|})^{2}) |v|^{2} \mu(v) \mathrm{1}_{|v|\leq 8/\pi}dv$.

\underline{Estimate $\mathcal{I}_{2}^{\epsilon,\gamma,1} (\delta)$.} Recalling $B^{\epsilon,\gamma,1} = \mathrm{1}_{|v-v_{*}| \geq 1, \sin(\theta/2) \geq \epsilon} |\ln \epsilon|^{-1}|v-v_{*}|^{\gamma} \sin^{-4}(\theta/2), |v^{\prime}-v| = |v-v_{*}|\sin(\theta/2)$, we have
\beno
\mathcal{I}_{2}^{\epsilon,\gamma,1} (\delta) &=& \int B^{\epsilon,\gamma,1} \mathrm{1}_{A(\epsilon, \delta)} |(\nabla^{2} \mu^{1/2}) (v(\kappa))|^{2}|v^{\prime}-v|^{4} f_{*}^{2}  d\sigma dv dv_{*} d\kappa
\\&=& |\ln \epsilon|^{-1}\int \mathrm{1}_{A(\epsilon, \delta)} |(\nabla^{2} \mu^{1/2}) (v(\kappa))|^{2}|v-v_{*}|^{\gamma+4} f_{*}^{2}  d\sigma dv dv_{*} d\kappa
\\&\leq& (3/2)^{\gamma+4} |\ln \epsilon|^{-1} \int \mathrm{1}_{A(\epsilon, \delta)} |(\nabla^{2} \mu^{1/2}) (v(\kappa))|^{2}|v_{*}|^{\gamma+4} f_{*}^{2}  d\sigma dv dv_{*} d\kappa.
\eeno
Given $\kappa \in [0,1], v_{*} \in \R^{3}$, we will use the following change of variables:
\ben \label{change-exact-formula}
\left( \sigma = (\theta, \varphi), v \right) \rightarrow \left( \sigma = (\theta(\kappa), \varphi), v(\kappa) \right)
\een
Here $\theta(\kappa)$ is the angle between $v(\kappa)-v_{*}$ and $\sigma$.
In the change, one has
\ben \label{relation-in-the-change}
\theta/2 \leq \theta(\kappa) \leq \theta, \sin\theta \leq 2\sin\theta(\kappa), |v-v_{*}|/\sqrt{2} \leq |v(\kappa)-v_{*}| \leq |v-v_{*}|.
\een
From which together with $|\frac{\partial (v(\kappa), \theta(\kappa))}{\partial (v, \theta)}|^{-1} \leq  (1-\frac{\kappa}{2})^{-5} \leq 32$, we have
\ben \label{change-Jacobean-bound}
 d\sigma dv = \sin \theta d\theta d\varphi dv \leq 2^{6} \sin\theta(\kappa) d\theta(\kappa) d\varphi dv(\kappa) = 2^{6} d\sigma d v(\kappa),
\een
and $\mathrm{1}_{A(\epsilon, \delta)} \leq \mathrm{1}_{16/\pi \leq |v_{*}| \leq \delta/\epsilon, \sin(\theta(\kappa)/2) \leq 4\delta|v_{*}|^{-1}} := \mathrm{1}_{C(\epsilon, \delta)}$. After the change, we get
\ben \label{to-be-refered-I2}
\mathcal{I}_{2}^{\epsilon,\gamma,1} (\delta)
\leq 2^{6} (3/2)^{\gamma+4} |\ln \epsilon|^{-1}  \int \mathrm{1}_{C(\epsilon, \delta)} |(\nabla^{2} \mu^{1/2}) (v(\kappa))|^{2}|v_{*}|^{\gamma+4} f_{*}^{2} d\sigma dv(\kappa) dv_{*} d\kappa.
\een
In the region $C(\epsilon, \delta)$, one has $\sin(\theta(\kappa)/2) \leq 4\delta|v_{*}|^{-1}$, then we have
\beno
\int  \mathrm{1}_{A(\epsilon, \delta)}   d \sigma  \leq 2\pi
\int_{0}^{\pi}  \mathrm{1}_{16/\pi \leq |v_{*}| \leq \delta/\epsilon}  \mathrm{1}_{\sin(\theta(\kappa)/2) \leq 4\delta|v_{*}|^{-1}} \sin\theta(\kappa) d\theta(\kappa)
=
2^{6}\pi \times \mathrm{1}_{16/\pi \leq |v_{*}| \leq \delta/\epsilon} \delta^{2}|v_{*}|^{-2}.
\eeno
Plugging which into \eqref{to-be-refered-I2}, we get
\ben \label{I-2-upper-bound}
\mathcal{I}_{2}^{\epsilon,\gamma,1} (\delta)
&\leq& 2^{12}\pi(3/2)^{\gamma+4} \delta^{2}|\ln \epsilon|^{-1}  \int \mathrm{1}_{16/\pi \leq |v_{*}| \leq \delta/\epsilon}|(\nabla^{2} \mu^{1/2}) (v(\kappa))|^{2}|v_{*}|^{\gamma+2} f_{*}^{2}  dv(\kappa) dv_{*} d\kappa
\nonumber \\&\leq&
2^{12}\pi(3/2)^{\gamma+4}c_{2} \delta^{2}|\ln \epsilon|^{-1}  \int \mathrm{1}_{16/\pi \leq |v_{*}| \leq \delta/\epsilon}|v_{*}|^{\gamma+2} f_{*}^{2} dv_{*},
\een
where $c_{2} = \int |(\nabla^{2} \mu^{1/2}) (v)|^{2} dv$. Plugging \eqref{I-1-lower-bound} and \eqref{I-2-upper-bound} into \eqref{vsmallvstarsmall}, thanks to $(\ln (4\delta)-\ln |v_*|-\ln \epsilon) \geq \ln4$ when $|v_{*}| \leq \delta/\epsilon$, we get
\beno
\mathcal{N}^{\epsilon,\gamma,1}(f,\mu^{1/2}) &\geq& \left(2^{-2}\pi (3/2)^{\gamma} c_{1} - 2^{10}\pi(3/2)^{\gamma+4}c_{2} (\ln4)^{-1} \delta^{2}\right)
\\&& \times |\ln \epsilon|^{-1} \int (\ln (4\delta)-\ln |v_*|-\ln \epsilon) \mathrm{1}_{16/\pi \leq |v_{*}| \leq \delta/\epsilon}|v_{*}|^{\gamma+2} f_{*}^{2} dv_{*}.
\eeno
For brevity, let $C_{1}=2^{-2}\pi (3/2)^{\gamma} c_{1}, C_{2}=2^{10}\pi(3/2)^{\gamma+4}c_{2} (\ln4)^{-1}$. We choose $\delta$ such that $C_{2} \delta^{2} = C_{1}/2$ and thus
\ben \label{lowerboundvstarsmall}
\mathcal{N}^{\epsilon,\gamma,1}(f,\mu^{1/2}) \geq (C_{1}/2)
|\ln \epsilon|^{-1} \int (\ln (4\delta)-\ln |v_*|-\ln \epsilon)\mathrm{1}_{16/\pi \leq |v_{*}| \leq \delta/\epsilon}|v_{*}|^{\gamma+2} f_{*}^{2} dv_{*}.
\een

{\it{Step 2: $|v_{*}| \geq R/\epsilon.$}} Here $R \geq 1, \epsilon \leq 1/2$. By direct computation, we have
\ben \label{large-part-lower-bound}
\mathcal{N}^{\epsilon,\gamma,1}(f,\mu^{1/2}) &=& \int B^{\epsilon,\gamma,1} f_{*}^{2} (\mu^{\prime 1/2}-\mu^{1/2})^{2} d\sigma dv dv_{*}
\nonumber \\&\geq& \int B^{\epsilon,\gamma,1} \mathrm{1}_{|v_{*}|\geq R/\epsilon} \mathrm{1}_{|v|\leq 1} f_{*}^{2} (\mu^{\prime 1/2}-\mu^{1/2})^{2} d\sigma dv dv_{*}
\nonumber \\&=& \int b^{\epsilon} |v-v_{*}|^{\gamma} \mathrm{1}_{|v_{*}|\geq R/\epsilon} \mathrm{1}_{|v|\leq 1} f_{*}^{2} (\mu^{\prime 1/2}-\mu^{1/2})^{2} d\sigma dv dv_{*}
\nonumber \\&\geq&  \int b^{\epsilon} |v-v_{*}|^{\gamma} \mathrm{1}_{|v_{*}|\geq R/\epsilon} \mathrm{1}_{|v|\leq 1} f_{*}^{2} \mu d\sigma dv dv_{*}
\nonumber \\&&-2\int b^{\epsilon} |v-v_{*}|^{\gamma} \mathrm{1}_{|v_{*}|\geq R/\epsilon} \mathrm{1}_{|v|\leq 1} f_{*}^{2}  \mu^{\prime 1/2}\mu^{1/2} d\sigma dv dv_{*} :=\mathcal{J}_{1}^{\epsilon,\gamma}(R) - \mathcal{J}_{2}^{\epsilon,\gamma}(R).
\een
By \eqref{order-0}, we have
\ben \label{J-1-lower-bound}
\mathcal{J}_{1}^{\epsilon,\gamma}(R) &\geq& 2\pi |\ln \epsilon|^{-1}\epsilon^{-2} \int |v-v_{*}|^{\gamma} \mathrm{1}_{|v_{*}|\geq R/\epsilon} \mathrm{1}_{|v|\leq 1} f_{*}^{2} \mu d\sigma dv dv_{*}
\nonumber \\&\geq& 2\pi (3/2)^{\gamma} |\ln \epsilon|^{-1}\epsilon^{-2}\int |v_{*}|^{\gamma} \mathrm{1}_{|v_{*}|\geq R/\epsilon} \mathrm{1}_{|v|\leq 1} f_{*}^{2} \mu  dv dv_{*}
\nonumber \\&=& 2\pi  (3/2)^{\gamma} c_{3}
|\ln \epsilon|^{-1}\epsilon^{-2}\int |v_{*}|^{\gamma} \mathrm{1}_{|v_{*}|\geq R/\epsilon} f_{*}^{2}  dv_{*}.
\een
where $c_{3} = \int \mathrm{1}_{|v|\leq 1} \mu  dv$ and we
use the following relation for $|v_{*}|\geq R/\epsilon \geq 2, |v| \leq 1$,
\ben \label{v-vstar-is-vstar} |v_{*}|/2 \leq |v-v_{*}| \leq 3|v_{*}|/2.\een

Since $\sin(\theta/2) \geq \epsilon$, there holds $|v^{\prime}|+|v| \geq |v^{\prime}-v| = \sin\frac{\theta}{2}|v-v_{*}|\geq\epsilon|v-v_{*}|\geq \epsilon(|v_{*}|-|v|)$, and then $|v^{\prime}|+(1+\epsilon)|v| \geq \epsilon|v_{*}| \geq R$. From which we have
$R^{2} \leq (|v^{\prime}|+2|v|)^{2} \leq 8 (|v^{\prime}|^{2}+|v|^{2})$ and
\ben \label{relax-cross-term}
\mu^{\prime 1/2}\mu^{1/2} = (2\pi)^{-3/2} e^{-\frac{|v^{\prime}|^{2}+|v|^{2}}{4}} \leq (2\pi)^{-3/2}e^{-\frac{|v|^{2}}{8}}e^{-\frac{R^{2}}{2^{6}}}.
\een
Then by \eqref{relax-cross-term}, the upper bound in \eqref{order-0}, the lower bound in \eqref{v-vstar-is-vstar},
we have
\ben \label{J-2-upper-bound}
\mathcal{J}_{2}^{\epsilon,\gamma}(R) &=&2\int b^{\epsilon} |v-v_{*}|^{\gamma} \mathrm{1}_{|v_{*}|\geq R/\epsilon} \mathrm{1}_{|v|\leq 1} f_{*}^{2}  \mu^{\prime 1/2}\mu^{1/2} d\sigma dv dv_{*}
\nonumber \\ &\leq&2 (2\pi)^{-3/2} e^{-\frac{R^{2}}{2^{6}}} \int b^{\epsilon} |v-v_{*}|^{\gamma} \mathrm{1}_{|v_{*}|\geq R/\epsilon} \mathrm{1}_{|v|\leq 1} f_{*}^{2}  e^{-\frac{|v|^{2}}{8}} d\sigma dv dv_{*}
\nonumber \\ &\leq& 8\pi (2\pi)^{-3/2} e^{-\frac{R^{2}}{2^{6}}}  |\ln \epsilon|^{-1}\epsilon^{-2}
\int |v-v_{*}|^{\gamma} \mathrm{1}_{|v_{*}|\geq R/\epsilon} \mathrm{1}_{|v|\leq 1} f_{*}^{2}  e^{-\frac{|v|^{2}}{8}} dv dv_{*}
\nonumber \\ &\leq& 8\pi (2\pi)^{-3/2}e^{-\frac{R^{2}}{2^{6}}} (1/2)^{\gamma} |\ln \epsilon|^{-1}\epsilon^{-2}
\int |v_{*}|^{\gamma} \mathrm{1}_{|v_{*}|\geq R/\epsilon} \mathrm{1}_{|v|\leq 1} f_{*}^{2}  e^{-\frac{|v|^{2}}{8}} dv dv_{*}
\nonumber \\ &=& 8\pi (2\pi)^{-3/2}e^{-\frac{R^{2}}{2^{6}}} (1/2)^{\gamma}c_{4}|\ln \epsilon|^{-1}\epsilon^{-2}
\int |v_{*}|^{\gamma} \mathrm{1}_{|v_{*}|\geq R/\epsilon} f_{*}^{2} dv_{*},
\een
where $c_{4} = \int \mathrm{1}_{|v|\leq 1}  e^{-\frac{|v|^{2}}{8}} dv$.
Plugging \eqref{J-1-lower-bound} and \eqref{J-2-upper-bound} into \eqref{large-part-lower-bound}, we arrive at for any $\epsilon \leq 1/2, R \geq 1$,
\ben \label{general-result-large}
\mathcal{N}^{\epsilon,\gamma,1}(f,\mu^{1/2}) \geq  (C_{3} - C_{4} e^{-\frac{R^{2}}{2^{6}}})|\ln \epsilon|^{-1} \epsilon^{-2}\int | v_{*} |^{\gamma}\mathrm{1}_{|v_{*}|\geq R/\epsilon}  f_{*}^{2} dv_{*},
\een
where $C_{3}=2\pi  (3/2)^{\gamma} c_{3}, C_{4}= 8\pi (2\pi)^{-3/2} (1/2)^{\gamma}c_{4}$.

{\it{Step 3: $|v_{*}| \geq \delta/\epsilon.$}}
Note that the above estimate \eqref{general-result-large} is valid for any $R \geq 1$ and $\epsilon \leq 1/2$.
For the fixed $\delta>0$ in {\it Step 1}, we choose $N$ large enough such that $N\delta \geq 1$ and $C_{3} - C_{4} e^{-\frac{(N\delta)^{2}}{2^{6}}} \geq C_{3}/2$.
Taking $R=N\delta$ in \eqref{general-result-large}, when $\epsilon$ is small such that $N\epsilon \leq 1/2$,
we have
\ben \label{geq-delta-ep-to-minus-1}
\mathcal{N}^{\epsilon,\gamma,1}(f,\mu^{1/2}) \geq |\ln \epsilon|^{-1} |\ln (N\epsilon)|
\mathcal{N}^{N\epsilon,\gamma,1}(f,\mu^{1/2})
\geq (C_{3}/2)N^{-2}|\ln \epsilon|^{-1} \epsilon^{-2}\int | v_{*} |^{\gamma}\mathrm{1}_{|v_{*}|\geq \delta/\epsilon}  f_{*}^{2} dv_{*}.
\een

{\it{Step 4: gain of the weight $W^{\epsilon}$.}} When  $|v_{*}| \geq 16/\pi \geq 4$, we get $|v_{*}|^{2} \leq 1+ |v_{*}|^{2} \leq \frac{17}{16}|v_{*}|^{2}$, which gives
\beno
| v_{*} |^{\gamma+2} \geq \min \{ 1, (17/16)^{-\gamma/2-1}\} \langle v_{*} \rangle^{\gamma+2},  | v_{*} |^{\gamma} \geq \min \{ 1, (17/16)^{-\gamma/2}\} \langle v_{*} \rangle^{\gamma}
\eeno
For simplicity, set $C_{1}(\gamma)=\min \{ 1, (17/16)^{-\gamma/2-1}\}, C_{2}(\gamma)= \{ 1, (17/16)^{-\gamma/2}\}$.
It is obvious
$|f|^{2}_{L^{2}_{\gamma/2}} \geq  \int \mathrm{1}_{|v_{*}|\leq 16/\pi}  \langle v_{*} \rangle^{\gamma} f_{*}^{2} dv_{*}.$
From these facts together with
\eqref{lowerboundvstarsmall}, \eqref{geq-delta-ep-to-minus-1},
we arrive at
\ben \label{N-1-lower-bound}
\mathcal{N}^{\epsilon,\gamma,1}(f,\mu^{1/2}) + |f|^{2}_{L^{2}_{\gamma/2}}&\geq&
\int \mathrm{1}_{|v_{*}|\leq 16/\pi}  \langle v_{*} \rangle^{\gamma}  f_{*}^{2} dv_{*}
\nonumber \\&& +(C_{1}/4)C_{1}(\gamma)|\ln \epsilon|^{-1}\int (\ln (4\delta)-\ln |v_*|-\ln \epsilon) \langle v_{*} \rangle^{\gamma+2}\mathrm{1}_{16/\pi \leq |v_{*}|\leq \delta/\epsilon}  f_{*}^{2} dv_{*}
\nonumber  \\&& +(C_{3}/4)N^{-2}C_{2}(\gamma)|\ln \epsilon|^{-1}\epsilon^{-2}\int \langle v_{*} \rangle^{\gamma}\mathrm{1}_{|v_{*}|\geq \delta/\epsilon}  f_{*}^{2} dv_{*}.
\een

By \eqref{upper-bound-cf}, we have $\mathrm{1}_{|v_{*}| \leq 16/\pi}W^{\epsilon}(v_{*}) \leq (1+16^{2}/\pi^{2})^{1/2},$ which gives
\ben \label{reduce-to-the-exact-form-near-origin}  \mathrm{1}_{|v_{*}| \leq 16/\pi}  \geq  (1+16^{2}/\pi^{2})^{-1}\mathrm{1}_{|v_{*}| \leq 16/\pi} (W^{\epsilon})^{2}(v_{*}).\een
In the middle region $16/\pi \leq |v_{*}|\leq \delta/\epsilon$, we have
\ben \label{reduce-to-the-exact-form}  \langle v_{*} \rangle^{2}|\ln \epsilon|^{-1}(\ln (4\delta)-\ln |v_*|-\ln \epsilon)
\geq \frac{\ln 4}{1-\ln
\delta} \langle v_{*} \rangle^{2}|\ln \epsilon|^{-1}(1-\ln |v_*|-\ln \epsilon)
\geq \frac{\ln 4}{1-\ln
\delta}(W^{\epsilon})^{2}(v_{*}). \een
In the large velocity region $|v_{*}|\geq \delta/\epsilon$, by \eqref{upper-bound-cf}, we have
\ben \label{reduce-to-the-exact-form-far-origin} |\ln \epsilon|^{-1}\epsilon^{-2} \geq (W^{\epsilon})^{2}(v_{*}). \een
Plugging \eqref{reduce-to-the-exact-form-near-origin}, \eqref{reduce-to-the-exact-form} and \eqref{reduce-to-the-exact-form-far-origin} into \eqref{N-1-lower-bound}, we get
\beno
\mathcal{N}^{\epsilon,\gamma,1}(f,\mu^{1/2}) + |f|^{2}_{L^{2}_{\gamma/2}}
\geq C(\gamma)|W^{\epsilon}f|^{2}_{L^{2}_{\gamma/2}},
\eeno
where $C(\gamma) = \min\{(1+16^{2}/\pi^{2})^{-1}, (C_{1}/4)C_{1}(\gamma)\frac{\ln 4}{1-\ln
\delta}, (C_{3}/4)N^{-2}C_{2}(\gamma)\}$ is a positive constant depending only on $\gamma$.
The proof of the lemma is complete.
\end{proof}

We next show that the lower bound in Lemma \ref{lowerboundpart1} is sharp in the following sense.
\begin{lem}\label{upperboundpart1} Let $-3 \leq \gamma \leq 0$. There exists $\epsilon_{0} >0$ such that for  $0< \epsilon \leq \epsilon_{0}$ and $ \eta\ge0$,
\beno
\mathcal{N}^{\epsilon,\gamma,\eta}(f,\mu^{1/2}) \lesssim |W^{\epsilon}f|^{2}_{L^{2}_{\gamma/2}}.
\eeno\end{lem}
\begin{proof} It is obvious that $\mathcal{N}^{\epsilon,\gamma,\eta}(f,\mu^{1/2}) \leq \mathcal{N}^{\epsilon,\gamma,0}(f,\mu^{1/2})$. Therefore it suffices to consider the upper bound of $\mathcal{N}^{\epsilon,\gamma,0}(f,\mu^{1/2})$.
First we have
\beno
\mathcal{N}^{\epsilon,\gamma,0}(f,\mu^{1/2}) &\lesssim& \int B^{\epsilon,\gamma} f_{*}^{2} (\mu^{\prime 1/4}-\mu^{1/4})^{2}(\mu^{\prime 1/2}+\mu^{1/2}) d\sigma dv dv_{*}
\\&\lesssim& \int  B^{\epsilon,\gamma} f_{*}^{2} (\mu^{\prime 1/4}-\mu^{1/4})^{2}\mu^{\prime 1/2} d\sigma dv dv_{*} + \int  B^{\epsilon,\gamma} f_{*}^{2} (\mu^{\prime 1/4}-\mu^{1/4})^{2}\mu^{1/2} d\sigma dv dv_{*}
\\&:=& \mathcal{K}^{\epsilon,\gamma}_{1}(f) + \mathcal{K}^{\epsilon,\gamma}_{2}(f).
\eeno
By Taylor expansion, one has
$(\mu^{\prime 1/4} - \mu^{1/4})^{2} \lesssim \min\{1,|v-v_{*}|^{2}\theta^{2}\} \sim \min\{1,|v^{\prime}-v_{*}|^{2}\theta^{2}\}.$
By Proposition \ref{symbol}, we have
$
\int b^{\epsilon}(\cos\theta) \min\{1,|v-v_{*}|^{2}\theta^{2}\} d\sigma \sim |v-v_{*}|^{2} \mathrm{1}_{|v-v_{*}|\leq 2} + (W^{\epsilon})^2(|v-v_*|) \mathrm{1}_{|v-v_{*}|\ge 2}.
$
From which together with \eqref{separate-into-2-cf}, we have
\beno
\mathcal{K}^{\epsilon,\gamma}_{2}(f) &\lesssim& \int f_{*}^{2} \mathrm{1}_{|v-v_{*}|\leq 2}|v-v_{*}|^{\gamma+2}  \mu^{1/2} dv dv_{*} \\&&+
\int f_{*}^{2} \mathrm{1}_{|v-v_{*}|\geq 2} |v-v_{*}|^{\gamma}(W^{\epsilon})^{2}(v)(W^{\epsilon})^{2}(v_{*})  \mu^{1/2} dv dv_{*} \lesssim  |W^\epsilon f|_{L^2_{\gamma/2}}^2.
\eeno
The term $\mathcal{K}^{\epsilon,\gamma}_{1}(f)$ can be similarly estimated by the change of variable $v \rightarrow v^{\prime}$ (take $\kappa=1$ in \eqref{change-exact-formula}). Indeed, one has $\mathcal{K}^{\epsilon,\gamma}_{1}(f) \lesssim \int  b^{\epsilon}(\cos(2\theta^{\prime})) |v^{\prime}-v_{*}|^{\gamma} f_{*}^{2} (\mu^{\prime 1/4}-\mu^{1/4})^{2}\mu^{\prime 1/2} d\sigma dv^{\prime} dv_{*},$
where $\theta^{\prime}$ is the angle between $v^{\prime}-v_{*}$ and $\sigma$. With the fact $\theta^{\prime} = \theta/2$, we also have
\beno
\int b^{\epsilon}(\cos(2\theta^{\prime})) \min\{1,|v^{\prime}-v_{*}|^{2}(\theta^{\prime})^{2}\} d\sigma \lesssim  \mathrm{1}_{|v^{\prime}-v_{*}|\leq 2} |v^{\prime}-v_{*}|^{2}  +\mathrm{1}_{|v^{\prime}-v_{*}|\ge 2}(W^{\epsilon})^{2}(v^{\prime})(W^{\epsilon})^{2}(v_{*}).
\eeno
Thus by exactly the same argument as that for $\mathcal{K}^{\epsilon,\gamma}_{2}(f)$, we have $\mathcal{K}^{\epsilon,\gamma}_{1}(f) \lesssim |W^\epsilon f|_{L^2_{\gamma/2}}^2$.
The proof of the lemma is complete.
\end{proof}

\subsection{Gain of  regularity from $\mathcal{N}^{\epsilon,\gamma,\eta}(g,f)$ }
In what follows, we will focus on gain of regularity from  $\mathcal{N}^{\epsilon,\gamma,\eta}(g,f)$. Our strategy can be concluded as follows:
\begin{enumerate}
\item gain of regularity from $\mathcal{N}^{\epsilon,0,0}(g,f)$.

\item  gain of regularity from $\mathcal{N}^{\epsilon,0,\eta}(g,f)$ by reducing to  $\mathcal{N}^{\epsilon,0,0}(g,f)$.

\item  gain of regularity from $\mathcal{N}^{\epsilon,\gamma,\eta}(g,f)$ by reducing  to $\mathcal{N}^{\epsilon,0,\eta}(g,f)$.
\end{enumerate}

\subsubsection{Gain of  regularity from $\mathcal{N}^{\epsilon,0,0}(g,f)$ } We first show gain of Sobolev regularity. We recall from  \cite{alexandre2000entropy}
that  for $g\ge0$ with $|g|_{L^1}\ge\delta>0$ and $|g|_{L^1_1}\le \lambda$,
\beno \int b(\cos\theta)g_*(f'-f)^2d\sigma dv_*dv+|f|_{L^2}^2\ge C(\delta, \lambda)|a(D)f|_{L^2}^2, \eeno
where $a(\xi):= \int b(\f{\xi}{|\xi|}\cdot \sigma)\min\{ |\xi|^2\sin^2(\theta/2),1\} d\sigma + 1$.
Thanks to Proposition \ref{symbol}, we get
\begin{lem}\label{lowerboundpart1-general-g}
Let $g$ be a function such that $|g^2|_{L^{1}} \geq \delta >0, |g^2|_{L^{1}_{1}} \leq \lambda < \infty$, then there exists a constant $C(\delta, \lambda)$ such that
\ben\label{sobolev-regularity-general-g}
\mathcal{N}^{\epsilon,0,0}(g,f)+ |f|^{2}_{L^{2}} \geq  C(\delta, \lambda)|W^{\epsilon}(D)f|^{2}_{L^{2}} .
\een
\end{lem}

Next we want to derive gain of anisotropic regularity from $\mathcal{N}^{\epsilon,0,0}(g,f)$.   In this part, we derive anisotropic regularity from  $\mathcal{N}^{\epsilon,\gamma,\eta}(\mu^{1/2},f)$. To this end, our strategy is to apply the geometric decomposition in the frequency space. More precisely, we  will use the following decomposition (see \eqref{geo-deco-frequency-space-another} in the proof of Lemma \ref{lowerboundpart2-general-g})
\ben
\hat{f}(\xi) - \hat{f}(\xi^{+}) &=& \left(\hat{f}(\xi) - \hat{f}(|\xi|\frac{\xi^{+}}{|\xi^{+}|})\right)+ \left(\hat{f}(|\xi|\frac{\xi^{+}}{|\xi^{+}|}) - \hat{f}(\xi^{+})\right) \nonumber
\\ \label{sphere-radius}&=& \text{ spherical part } + \text{ radial part}.
\een
We shall show that the ``spherical part'' produces anisotropic regularity. Namely,
\begin{lem}\label{spherical-part} Let $\mathcal{A}^{\epsilon}(f) := \int b^{\epsilon}(\frac{\xi}{|\xi|} \cdot \sigma)|\hat{f}(\xi) - \hat{f}(|\xi|\frac{\xi^{+}}{|\xi^{+}|})|^{2} d\sigma d\xi $ where $\xi^{+} = \frac{\xi + |\xi|\sigma}{2}$, then
\beno
\mathcal{A}^{\epsilon}(f) +|f|_{L^2}^2
\sim |W^{\epsilon}((-\Delta_{\mathbb{S}^{2}})^{1/2})f|^{2}_{L^{2}}+|f|^{2}_{L^{2}}.
\eeno
\end{lem}
\begin{proof}
Let $ r= |\xi|, \tau = \xi/|\xi|$ and $\varsigma = \frac{\tau+\sigma}{|\tau+\sigma|}$, then $\frac{\xi}{|\xi|} \cdot \sigma = 2(\tau\cdot\varsigma)^{2} - 1$ and $|\xi|\frac{\xi^{+}}{|\xi^{+}|} = r \varsigma$. For the change of variable $(\xi, \sigma) \rightarrow (r, \tau, \varsigma)$, one has
$
d\xi d\sigma = 4  (\tau\cdot\varsigma) r^{2} dr d \tau d \varsigma.
$ Let $\theta$ be the angle between $\tau$ and $\sigma$, then $2 \sin\frac{\theta}{2} = |\tau-\sigma|, |\tau - \varsigma| = 2(1-\cos \frac{\theta}{2})$ and thus
$\sin\frac{\theta}{2} = \frac{1}{2} |\tau-\sigma| \leq |\tau - \varsigma| \leq |\tau-\sigma| = 2\sin\frac{\theta}{2}$.
Therefore
\beno
|\ln \epsilon |^{-1} |\tau - \varsigma|^{-4} \mathrm{1}_{|\tau - \varsigma| \geq 2\epsilon} \leq  b^{\epsilon}(\cos\theta) \leq  |\ln \epsilon |^{-1} 2^{4} |\tau - \varsigma|^{-4} \mathrm{1}_{|\tau - \varsigma| \geq \epsilon}.
\eeno
With the help of \eqref{similarlemma5.6} in Proposition \ref{anisotropic-key-pro}, we have
\beno
\mathcal{A}^{\epsilon}(f) +|f|_{L^2}^2 &=& 4 \int_{0}^{\infty} \int_{\mathbb{S}^2 \times \mathbb{S}^2} b^{\epsilon}(\cos\theta)|\hat{f}(r\tau) - \hat{f}(r\varsigma)|^{2} (\tau\cdot\varsigma) r^{2} dr d \tau d \varsigma+|f|_{L^2}^2
\\&\lesssim& |\ln \epsilon|^{-1} \int |\hat{f}(r\tau) - \hat{f}(r\varsigma)|^{2}|\tau - \varsigma|^{-4}\mathrm{1}_{|\tau-\varsigma| \geq \epsilon}  r^{2} dr d \tau d \varsigma+|f|_{L^2}^2
\\&\lesssim& |W^{\epsilon}((-\Delta_{\mathbb{S}^{2}})^{1/2})\hat{f}|^{2}_{L^{2}}+ |\hat{f}|^{2}_{L^{2}}
= |W^{\epsilon}((-\Delta_{\mathbb{S}^{2}})^{1/2})f|^{2}_{L^{2}}+|f|^{2}_{L^{2}},
\eeno
where we use Lemma \ref{comWep} and Plancherel's  theorem in the last line. Thanks to \eqref{similarlemma5.6} in Proposition \ref{anisotropic-key-pro} and Remark \ref{2-epsilon-statement}, we similarly get the $\gtrsim$ direction.
\end{proof}

The ``radial part'' in \eqref{sphere-radius} is controllable by gain of $W^{\epsilon}$   in the phase and frequency space. Namely,

\begin{lem}\label{gammanonzerotozero}
	Let $
	\mathcal{Z}^{\epsilon,\gamma}(f) := \int b^{\epsilon}(\frac{\xi}{|\xi|}\cdot\sigma)\langle \xi\rangle^{\gamma} |\hat{f}(|\xi|\frac{\xi^{+}}{|\xi^{+}|}) - \hat{f}(\xi^{+})|^{2} d\sigma d\xi
	$ with $\xi^{+} = \frac{\xi + |\xi|\sigma}{2}$ and $\gamma\le0$. Then
	\beno
	\mathcal{Z}^{\epsilon,\gamma}(f) \lesssim |W^{\epsilon}(D)W_{\gamma/2}f|^{2}_{L^{2}} + |W^{\epsilon}W_{\gamma/2}f|^{2}_{L^{2}}.
	\eeno
\end{lem}
\begin{proof} We divide the proof into two steps.
	
{\it Step 1: $\gamma=0$.}
By the change of variable $(\xi, \sigma) \rightarrow (r, \tau, \varsigma)$ with $\xi=r\tau$ and $\varsigma=\f{\sigma+\tau}{|\sigma+\tau|}$, we have
	\beno
	 \mathcal{Z}^{\epsilon, 0}(f) = 4 \int b^{\epsilon}(2(\tau\cdot\varsigma)^{2} - 1)|f(r\varsigma) - f((\tau\cdot\varsigma)r\varsigma)|^{2} (\tau\cdot\varsigma) r^{2} dr d \tau d \varsigma.
	\eeno
	Let $u = r\varsigma$, and $\theta$ be the angle between $\tau$ and $\varsigma$. Since $b^{\epsilon}(2(\tau\cdot\varsigma)^{2} - 1) = b^{\epsilon}(\cos2\theta) \lesssim |\ln \epsilon|^{-1} \theta^{-4}$, and $r^{2} dr d \tau d \varsigma = \sin\theta du d\theta d\mathbb{S}$, we have
	\beno
	\mathcal{Z}^{\epsilon,0}(f) &\lesssim&  |\ln \epsilon|^{-1} \int_{\R^{3}} \int_{\epsilon/2}^{\pi/4} \theta^{-3}|\hat{f}(u) - \hat{f}(u\cos\theta)|^{2} du d\theta
	\\&\lesssim& |\ln \epsilon|^{-1} \int \theta^{-3}|f(u) - f(u/\cos\theta)|^{2} du d\theta
	\lesssim |W^{\epsilon}(D) f|^{2}_{L^{2}} + |W^{\epsilon} f|^{2}_{L^{2}},
	\eeno
where we use Plancherel's  theorem and Lemma \ref{a-technical-lemma} in the last line.

{\it Step 2: general cases.}	We  reduce the general case to the special case $\gamma = 0$. For simplicity, we denote $w = |u|\frac{u^{+}}{|u^{+}|}$, then $W_{\gamma}(u) = W_{\gamma}(w)$. From which we have
	\beno
	\langle u\rangle^{\gamma} (f(w)-f(u^{+}))^{2} &=& \left((W_{\gamma/2}f)(w)-(W_{\gamma/2}f)(u^{+})
	 + (W_{\gamma/2}f)(u^{+})(1-W_{\gamma/2}(w)W_{-\gamma/2}(u^{+}))\right)^{2}
	\\&\leq& 2 |(W_{\gamma/2}f)(u^{+})-(W_{\gamma/2}f)(w)|^{2}
	+ 2 |(W_{\gamma/2}f)(u^{+})|^{2}|1-W_{\gamma/2}(w)W_{-\gamma/2}(u^{+})|^{2}.
	\eeno
	Thus we have
	\beno
	\mathcal{Z}^{\epsilon,\gamma}(f) &\lesssim& \mathcal{Z}^{\epsilon,0}(W_{\gamma/2}f)
	+  \int b^{\epsilon}(\frac{u}{|u|}\cdot\sigma)|(W_{\gamma/2}f)(u^{+})|^{2}|1-W_{\gamma/2}(w)W_{-\gamma/2}(u^{+})|^{2}  d\sigma du
	\\&:=& \mathcal{Z}^{\epsilon,0}(W_{\gamma/2}f) + \mathcal{A}.
	\eeno
	By noticing that
	$
	|W_{\gamma/2}(w)W_{-\gamma/2}(u^{+}) - 1| \lesssim \theta^{2},
	$
	we have
	$
	|\mathcal{A}| \lesssim  |W_{\gamma/2}f|^{2}_{L^{2}},
	$
	where the change of variable $u \rightarrow u^{+}$ is used. Together with the estimate in {\it Step 1}, we get
	the desired result.
\end{proof}

Now we are in a position to get $|W^{\epsilon}((-\Delta_{\mathbb{S}^{2}})^{1/2})f|^{2}_{L^{2}_{\gamma/2}}$ from $\mathcal{N}^{\epsilon,0,0 }(g,f)$.

\begin{lem}\label{lowerboundpart2-general-g}
For any smooth functions $g, f$, the follow two estimates hold true.
\ben\label{anisotropic-regularity-general-g} &&
\mathcal{N}^{\epsilon,0,0}(g,f) + |g|^{2}_{L^{2}_{1}}|W^{\epsilon}(D)f|^{2}_{L^{2}}+ |g|^{2}_{L^{2}} |W^{\epsilon}f|^{2}_{L^{2}} \gtrsim  |g|^{2}_{L^{2}} |W^{\epsilon}((-\Delta_{\mathbb{S}^{2}})^{1/2})f|^{2}_{L^{2}},\\
\label{anisotropic-regularity-general-g-up-bound}&&
\mathcal{N}^{\epsilon,0,0}(g,f) \lesssim  |g|^{2}_{L^{2}} |W^{\epsilon}((-\Delta_{\mathbb{S}^{2}})^{1/2})f|^{2}_{L^{2}} + |g|^{2}_{L^{2}_{1}}|W^{\epsilon}(D)f|^{2}_{L^{2}}+ |g|^{2}_{L^{2}} |W^{\epsilon}f|^{2}_{L^{2}}.
\een
\end{lem}
\begin{proof}
By Bobylev's formula, we have
\beno
\mathcal{N}^{\epsilon,0,0}(g,f) &=& \frac{1}{(2\pi)^{3}}\int  b^{\epsilon}(\frac{\xi}{|\xi|} \cdot \sigma)(\widehat{g^{2}}(0)|\hat{f}(\xi) - \hat{f}(\xi^{+})|^{2} + 2\Re((\widehat{g^{2}}(0) - \widehat{g^{2}}(\xi^{-}))\hat{f}(\xi^{+})\bar{\hat{f}}(\xi)) d\sigma d\xi
\\ &:=& \frac{|g|^{2}_{L^{2}}}{(2\pi)^{3}}\mathcal{I}_{1} + \frac{2}{(2\pi)^{3}}\mathcal{I}_{2},
\eeno
where $\xi^{+} = \frac{\xi+|\xi|\sigma}{2}$ and $\xi^{-} = \frac{\xi-|\xi|\sigma}{2}$.
Thanks to the fact $\widehat{g^{2}}(0) - \widehat{g^{2}}(\xi^{-}) = \int (1-\cos(v \cdot \xi^{-}))g^{2}(v) dv$, we have
\beno
|\mathcal{I}_{2}| &=& |\int b^{\epsilon}(\frac{\xi}{|\xi|} \cdot \sigma) (1-\cos(v \cdot \xi^{-}))g^{2}(v) \Re(\hat{f}(\xi^{+})\bar{\hat{f}}(\xi)) d\sigma d\xi dv |
\\ &\leq& \big(\int b^{\epsilon}(\frac{\xi}{|\xi|} \cdot \sigma) (1-\cos(v \cdot \xi^{-}))g^{2}(v) |\hat{f}(\xi^{+})|^{2} d\sigma d\xi dv \big)^{1/2}
\\&& \times \big(\int b^{\epsilon}(\frac{\xi}{|\xi|} \cdot \sigma) (1-\cos(v \cdot \xi^{-}))g^{2}(v) |\bar{\hat{f}}(\xi)|^{2} d\sigma d\xi dv\big)^{1/2}.
\eeno
Observe that
$
1-\cos(v \cdot \xi^{-}) \lesssim |v|^{2}|\xi^{-}|^{2} = \frac{1}{4}|v|^{2}|\xi|^{2}|\frac{\xi}{|\xi|} - \sigma|^{2} \sim |v|^{2}|\xi^{+}|^{2}|\frac{\xi^{+}}{|\xi^{+}|} - \sigma|^{2},
$
thus
$
1-\cos(v \cdot \xi^{-}) \lesssim \min\{|v|^{2}|\xi|^{2}|\frac{\xi}{|\xi|} - \sigma|^{2},1\} \sim  \min\{|v|^{2}|\xi^{+}|^{2}|\frac{\xi^{+}}{|\xi^{+}|} - \sigma|^{2},1\}.
$
By the fact
$
\frac{\xi}{|\xi|} \cdot \sigma  = 2(\frac{\xi^{+}}{|\xi^{+}|} \cdot \sigma)^{2} - 1,
$
and the change of variable $\xi \rightarrow \xi^{+}$, and the property $W^{\epsilon}(|v||\xi|) \lesssim W^{\epsilon}(|v|)W^{\epsilon}(|\xi|)$, we have
\ben \label{upper-I-2-another}
|\mathcal{I}_{2}| \lesssim \int (W^{\epsilon})^{2}(|v||\xi|)|\hat{f}(\xi)|^{2}g^{2}(v)dvd\xi
\lesssim |W^{\epsilon}g|^{2}_{L^{2}} |W^{\epsilon}(D)f|^{2}_{L^{2}} \lesssim |g|^{2}_{L^{2}_{1}} |W^{\epsilon}(D)f|^{2}_{L^{2}}.
\een
Now we set to investigate $\mathcal{I}_{1}$. By the geometric decomposition
\ben \label{geo-deco-frequency-space-another}
\hat{f}(\xi) - \hat{f}(\xi^{+}) = \hat{f}(\xi) - \hat{f}(|\xi|\frac{\xi^{+}}{|\xi^{+}|})+ \hat{f}(|\xi|\frac{\xi^{+}}{|\xi^{+}|}) - \hat{f}(\xi^{+}),
\een
we have
\beno
\mathcal{I}_{1} = \int b^{\epsilon}(\frac{\xi}{|\xi|} \cdot \sigma)|\hat{f}(\xi) - \hat{f}(\xi^{+})|^{2} d\sigma d\xi
&\geq& \frac{1}{2} \int b^{\epsilon}(\frac{\xi}{|\xi|} \cdot \sigma)|\hat{f}(\xi) - \hat{f}(|\xi|\frac{\xi^{+}}{|\xi^{+}|})|^{2} d\sigma d\xi
\\&&- \int b^{\epsilon}(\frac{\xi}{|\xi|} \cdot \sigma)|\hat{f}(|\xi|\frac{\xi^{+}}{|\xi^{+}|}) - \hat{f}(\xi^{+})|^{2} d\sigma d\xi
:= \frac{1}{2}\mathcal{I}_{1,1} - \mathcal{I}_{1,2}.
\eeno
By Lemma \ref{spherical-part}, we have
\ben \label{upper-lower-I-11-another}
\mathcal{I}_{1,1}+|f|_{L^2}^2 \sim |W^{\epsilon}((-\Delta_{\mathbb{S}^{2}})^{1/2})f|^{2}_{L^{2}}+|f|^{2}_{L^{2}}.
\een
By Lemma \ref{gammanonzerotozero} in the special case of $\gamma=0$, there holds
\ben \label{upper-I-12-another} \mathcal{I}_{1,2} \lesssim |W^{\epsilon}(D)f|^{2}_{L^{2}} + |W^{\epsilon}f|^{2}_{L^{2}}.\een
Patching together \eqref{upper-I-2-another}, \eqref{upper-I-12-another} and \eqref{upper-lower-I-11-another}, we get \eqref{anisotropic-regularity-general-g} and \eqref{anisotropic-regularity-general-g-up-bound}.
\end{proof}

\subsubsection{Gain of  regularity from $\mathcal{N}^{\epsilon,0,\eta}(g,f)$}
We first introduce some notations. Let $\chi$ be a smooth function such that $0 \leq \chi \leq 1, \chi=1$ on $B_{1}$ and $\mathrm{Supp} \chi \subset B_{2}$. Here $B_{1}$ is the ball centered at origin with radius $1$. $B_{2}$ is interpreted in a similar way.
Let $\chi_{R}(v):=\chi(v/R), \chi_{r,u}(v) := \chi_{r}(v-u)$ and $\phi_{R,r,u} := \chi_{7R} -\chi_{3r,u}$ for some $r,R>0$ and $u \in \mathbb{R}^{3}$. The following lemma
bounds $\mathcal{N}^{\epsilon,0,\eta}(g,f)$ by $\mathcal{N}^{\epsilon,0,0}(g,f)$ from below once the distance between supports of $g$ and $f$ is larger than some multiples of $\eta$.

\begin{lem}\label{reduce-eta-to-0} For  $0 \leq \eta \leq r \leq 1 \leq R,  u \in B_{7R}$, the following two estimates hold true.
\ben \label{eta-to-0-part1}
\mathcal{N}^{\epsilon,0,\eta}(g,f) + |g|^{2}_{L^{2}}|f|^{2}_{L^{2}}  \gtrsim \mathcal{N}^{\epsilon,0,0}(\chi_{R}g,(1-\chi_{3R})f).
\\ \label{eta-to-0-part2}
\mathcal{N}^{\epsilon,0,\eta}(g,f) + r^{-2}R^{2}|g|^{2}_{L^{2}}|f|^{2}_{L^{2}}  \gtrsim \mathcal{N}^{\epsilon,0,0}(\phi_{R,r,u}g,\chi_{r,u}f).
\een
\end{lem}
\begin{proof} We proceed in the spirit of \cite{he2014well}. If $|v_{*}| \leq 2R$ and $|v| \geq 3R$, then $|v-v_{*}| \geq R \geq \eta$. Then we have
\beno  \mathcal{N}^{\epsilon,0,\eta}(g,f) &=& \int b^{\epsilon}(\cos\theta) \textrm{1}_{|v-v_{*}|\geq \eta} g^{2}_{*} (f^{\prime}-f)^{2} d\sigma dv dv_{*}
\\
&\geq& \int b^{\epsilon}(\cos\theta)  (\chi_{R}g)^{2}_{*} (f^{\prime}-f)^{2} (1-\chi_{3R})^{2} d\sigma dv dv_{*}
\\&\geq& \frac{1}{2} \int b^{\epsilon}(\cos\theta)  (\chi_{R}g)^{2}_{*}  (((1-\chi_{3R})f)^{\prime}-(1-\chi_{3R})f)^{2} d\sigma dv dv_{*}
\\&&-\int b^{\epsilon}(\cos\theta)   (\chi_{R}g)^{2}_{*} (f^{\prime})^{2} (\chi_{3R}^{\prime}-\chi_{3R})^{2} d\sigma dv dv_{*}
:= \frac{1}{2}  \mathcal{I}_{1} - \mathcal{I}_{2}.
\eeno
Observe that $\mathcal{I}_{1}  = \mathcal{N}^{\epsilon,0,0}(\chi_{R}g,(1-\chi_{3R})f)$. Since $|\nabla \chi_{3R}|_{L^{\infty}} \lesssim R^{-1} |\nabla \chi|_{L^{\infty}} \lesssim R^{-1}$, we get $(\chi_{R})^{2}_{*}(\chi_{3R}^{\prime}-\chi_{3R})^{2} \lesssim R^{-2}|v^{\prime}-v|^{2} = R^{-2}|v-v_{*}|^{2}\sin^{2}(\theta/2)$. If $|v_{*}| \leq 2R, |v|\geq 20R, \theta \leq \pi/2$, we have
\beno
|v^{\prime}-v_{*}|=\cos(\theta/2)|v-v_{*}|\geq \cos(\theta/2)(|v|-|v_{*}|) \geq  9 \sqrt{2}R.
\eeno
Then we have
$ |v^{\prime}| \geq |v^{\prime}-v_{*}| -|v_{*}| \geq 9 \sqrt{2}R -2R \geq 6R,
$
which gives $\chi_{3R}^{\prime} =0$. Since $\theta \leq \pi/2$, we have
\beno
(\chi_{R})^{2}_{*}(\chi_{3R}^{\prime}-\chi_{3R})^{2} \leq \mathrm{1}_{|v|\leq 20R, |v_{*}|\leq 2R} R^{-2}|\nabla \chi|^{2}_{L^{\infty}}|v-v_{*}|^{2}\sin^{2}(\theta/2) \lesssim \sin^{2}(\theta/2).
\eeno
By the change of variable $(v,\theta) \rightarrow (v^{\prime},\theta/2)$ and using \eqref{order-2}, we get
\beno  \mathcal{I}_{2} \lesssim  \int g^{2}_{*}  f^{2} dv dv_{*}\lesssim |g|^{2}_{L^{2}}|f|^{2}_{L^{2}}.\eeno
From which together with the fact $\mathcal{I}_{1}  = \mathcal{N}^{\epsilon,0,0}(\chi_{R}g,(1-\chi_{3R})f)$, we get \eqref{eta-to-0-part1}.

If $v \in \mathrm{supp} \chi_{r,u}, v_{*} \in \mathrm{supp}\phi_{R,r,u}$, then $|v-v_{*}| \geq r \geq \eta$, which gives
\beno  \mathcal{N}^{\epsilon,0,\eta}(g,f) &=& \int b^{\epsilon}(\cos\theta) \textrm{1}_{|v-v_{*}|\geq \eta} g^{2}_{*} (f^{\prime}-f)^{2} d\sigma dv dv_{*}
\\&\geq& \int b^{\epsilon}(\cos\theta)   (\phi_{R,r,u}g)^{2}_{*} (f^{\prime}-f)^{2}  \chi^{2}_{r,u} d\sigma dv dv_{*}
\\&\geq& \frac{1}{2} \int b^{\epsilon}(\cos\theta)  (\phi_{R,r,u}g)^{2}_{*}  ((\chi_{r,u}f)^{\prime}-\chi_{r,u}f)^{2} d\sigma dv dv_{*}
\\&&-\int b^{\epsilon}(\cos\theta)   (\phi_{R,r,u}g)^{2}_{*} (f^{\prime})^{2} (\chi_{r,u}^{\prime}-\chi_{r,u})^{2} d\sigma dv dv_{*}
:= \frac{1}{2}\mathcal{J}_{1} - \mathcal{J}_{2}.
\eeno
Observe that $\mathcal{J}_{1}  =  \mathcal{N}^{\epsilon,0,0}(\phi_{R,r,u}g,\chi_{r,u}f)$.
Since $|\nabla \chi_{r,u} (v)| \lesssim r^{-1} |\nabla \chi|_{L^{\infty}}\mathrm{1}_{r \leq |v-u|\leq 2r} \lesssim r^{-1}$, together with Taylor expansion, we get \beno
|\chi_{r,u}^{\prime}-\chi_{r,u}|^{2} = |\int_{0}^{1} \nabla \chi_{r,u} \left(v(\kappa)\right)\cdot(v^{\prime}-v) d\kappa|^{2} \lesssim r^{-2}|v-v_{*}|^{2}\sin^{2}(\theta/2) \int_{0}^{1} \mathrm{1}_{r \leq |v(\kappa)-u|\leq 2r} d\kappa.\eeno
For $u \in B_{7R}, |v_{*}| \leq 14R, r \leq |v(\kappa)-u|\leq 2r$, we have
\beno|v-v_{*}| \leq \sqrt{2}|v(\kappa)-v_{*}| \leq \sqrt{2}|v(\kappa)-u| +\sqrt{2}|u-v_{*}| \leq 2\sqrt{2}r+\sqrt{2}(7R+14R) \leq 23\sqrt{2}R,\eeno
 and thus
\beno (\phi_{R,r,u})^{2}_{*}(\chi_{r,u}^{\prime}-\chi_{r,u})^{2} \lesssim r^{-2}|v-v_{*}|^{2}\mathrm{1}_{|v(\kappa)-u|\sim r}\theta^{2} \lesssim r^{-2}R^{2}\sin^{2}(\theta/2).\eeno
By the change of variable $v \rightarrow v^{\prime}$ and \eqref{order-2}, we get
\beno  \mathcal{J}_{2} \lesssim  r^{-2}R^{2} \int g^{2}_{*}  f^{2} dv dv_{*} \lesssim r^{-2}R^{2}|g|^{2}_{L^{2}}|f|^{2}_{L^{2}}.\eeno
From which together with the fact $\mathcal{J}_{1}  = \mathcal{N}^{\epsilon,0,0}(\phi_{R,r,u}g,\chi_{r,u}f)$, we get \eqref{eta-to-0-part2}.
\end{proof}

\subsubsection{Gain of  regularity from $\mathcal{N}^{\epsilon,\gamma,\eta}(g,f)$}
To reduce $\mathcal{N}^{\epsilon,\gamma,\eta}$ to $\mathcal{N}^{\epsilon,0,\eta}$, we introduce an intermediate quantity
\beno \tilde{\mathcal{N}}^{\epsilon,\gamma,\eta}(g,h) := \int b^{\epsilon}(\cos\theta)\textrm{1}_{|v-v_{*}|\geq \eta} \langle v-v_{*} \rangle^{\gamma} g^{2}_{*} (h^{\prime}-h)^{2} d\sigma dv dv_{*}.\eeno

The next lemma reduces $\tilde{\mathcal{N}}^{\epsilon,\gamma,\eta}$ to $\mathcal{N}^{\epsilon,0,\eta}$.
\begin{lem}\label{reduce-gamma-to-0-no-sigularity} For  $\gamma \leq 0 \leq \eta$, one has
\begin{eqnarray}\label{gamma-to-0-no-sigularity}
 &&\frac{1}{2}C_{1}\mathcal{N}^{\epsilon,0,\eta}(W_{\gamma/2}g,W_{\gamma/2}f) - C_{3}|g|^{2}_{L^{2}_{|\gamma/2+1|}}|f|^{2}_{L^{2}_{\gamma/2}}
 \\&\leq& \tilde{\mathcal{N}}^{\epsilon,\gamma,\eta}(g,f)  \leq  2C_{2}\mathcal{N}^{\epsilon,0,\eta}(W_{-\gamma/2}g,W_{\gamma/2}f) + 2C_{3}|g|^{2}_{L^{2}_{|\gamma/2+1|}}|f|^{2}_{L^{2}_{\gamma/2}}, \nonumber
\end{eqnarray}
where $C_{1},C_{2},C_{3}$ are constants depending only on $\gamma$. They are universally bounded if $-3 \leq \gamma \leq 0$.
\end{lem}
\begin{proof} Set $F= W_{\gamma/2}f$. By definition, we have
\beno  \tilde{\mathcal{N}}^{\epsilon,\gamma,\eta}(g,f) &=& \int b^{\epsilon}(\cos\theta) \textrm{1}_{|v-v_{*}|\geq \eta}\langle v-v_{*} \rangle^{\gamma}  g^{2}_{*} ((W_{-\gamma/2}F)^{\prime}-W_{-\gamma/2}F)^{2} d\sigma dv dv_{*}.
\eeno
We make the following decomposition
\beno
(W_{-\gamma/2}F)^{\prime}-W_{-\gamma/2}F = (W_{-\gamma/2})^{\prime} (F^{\prime}-F) + F(W_{-\gamma/2}^{\prime}-W_{-\gamma/2}) :=  A + B.
\eeno
By $\frac{1}{2}A^{2} - B^{2} \leq (A+B)^{2} \leq 2A^{2}+2B^{2}$, we get
\beno  \frac{1}{2}\mathcal{I}_{1}- \mathcal{I}_{2} \leq \mathcal{N}^{\epsilon,\gamma,\eta}(g,f) \leq 2(\mathcal{I}_{1} + \mathcal{I}_{2}),
\eeno
where
\beno
\mathcal{I}_{1} &:=& \int b^{\epsilon}(\cos\theta) \textrm{1}_{|v-v_{*}|\geq \eta}\langle v-v_{*} \rangle^{\gamma}  g^{2}_{*} W_{-\gamma}^{\prime} (F^{\prime}-F)^{2} d\sigma dv dv_{*},
\\
\mathcal{I}_{2} &:=& \int b^{\epsilon}(\cos\theta) \textrm{1}_{|v-v_{*}|\geq \eta}\langle v-v_{*} \rangle^{\gamma}  g^{2}_{*}F^{2} (W_{-\gamma/2}^{\prime}-W_{-\gamma/2})^{2} d\sigma dv dv_{*}.
\eeno
Since $ \langle v_{*} \rangle^{\gamma} \lesssim \langle v_{*}-v^{\prime} \rangle^{\gamma} \langle v^{\prime} \rangle^{-\gamma} \lesssim \langle v_{*} \rangle^{-\gamma}$, we get
\beno \mathcal{N}^{\epsilon,0,\eta}(W_{\gamma/2}g,W_{\gamma/2}f) \lesssim \mathcal{I}_{1} \lesssim  \mathcal{N}^{\epsilon,0,\eta}(W_{-\gamma/2}g,W_{\gamma/2}f).\eeno
By Taylor expansion, one has $(W_{-\gamma/2}^{\prime}-W_{-\gamma/2})^{2} \lesssim \int \langle v(\kappa) \rangle^{-\gamma-2}|v-v_{*}|^{2}\sin^{2}(\theta/2) d\kappa$. Note that
\beno\langle v-v_{*} \rangle^{\gamma} |v-v_{*}|^{2} \langle v(\kappa) \rangle^{-\gamma-2} \lesssim
 \langle v-v_{*} \rangle^{\gamma+2} \langle v(\kappa) \rangle^{-\gamma-2} \lesssim
 \langle v(\kappa)-v_{*} \rangle^{\gamma+2} \langle v(\kappa) \rangle^{-\gamma-2} \lesssim \langle v_{*} \rangle^{|\gamma+2|}.\eeno
From which together with \eqref{order-2}, we get
\beno  \mathcal{I}_{2} \lesssim  \int g^{2}_{*}\langle v_{*} \rangle^{|\gamma+2|} F^{2} dv dv_{*} \lesssim |g|^{2}_{L^{2}_{|\gamma/2+1|}}|F|^{2}_{L^{2}}.\eeno
Patching  together the above estimates for $\mathcal{I}_{1}$ and $\mathcal{I}_{2}$, we get the lemma.
\end{proof}
If $\gamma \leq  0$, then $|v-v_{*}|^{\gamma} \geq \langle v-v_{*} \rangle^{\gamma}$, and thus $\mathcal{N}^{\epsilon,\gamma,\eta}(g,f) \geq \tilde{\mathcal{N}}^{\epsilon,\gamma,\eta}(g,f)$. From which together with Lemma \ref{reduce-gamma-to-0-no-sigularity}, we have
\begin{lem}\label{reduce-gamma-to-0} For  $\gamma  \leq 0 \leq \eta$, one has
\begin{eqnarray}\label{gamma-to-0}
\mathcal{N}^{\epsilon,\gamma,\eta}(g,f) + |g|^{2}_{L^{2}_{|\gamma/2+1|}}|f|^{2}_{L^{2}_{\gamma/2}}  \geq  C\mathcal{N}^{\epsilon,0,\eta}(W_{\gamma/2}g,W_{\gamma/2}f),
\end{eqnarray}
where $C$ is a constant depending only on $\gamma$. The constant is universally bounded if $-3 \leq \gamma \leq 0$.
\end{lem}

We are now ready to derive gain of regularity from $\mathcal{N}^{\epsilon,\gamma,\eta}(\mu^{1/2},f)$.

\begin{lem}\label{lowerboundpart1-gamma-eta} For $-3 \leq \gamma \leq 0 \leq \eta \leq  6^{-1} 2^{-7/6} e^{-1/6} := r_{0}$, we have
\ben \label{sobolev-regularity-gamma-eta} \mathcal{N}^{\epsilon,\gamma,\eta}(\mu^{1/2},f) +|f|_{L^2_{\gamma/2}}^2\gtrsim |W^\epsilon(D)W_{\gamma/2}f|_{L^2}^2.
\\ \label{anisotropic-regularity-gamma-eta} \mathcal{N}^{\epsilon,\gamma,\eta}(\mu^{1/2},f) +  |W^\epsilon(D)W_{\gamma/2}f|_{L^2}^2 + |W^\epsilon W_{\gamma/2}f|_{L^2}^2\gtrsim |W^{\epsilon}((-\Delta_{\mathbb{S}^{2}})^{1/2})W_{\gamma/2}f|^{2}_{L^{2}}. \een
Making suitable combination, we have
\ben \label{full-regularity-gamma-eta} \mathcal{N}^{\epsilon,\gamma,\eta}(\mu^{1/2},f) + |W^\epsilon W_{\gamma/2}f|_{L^2}^2\gtrsim |W^{\epsilon}((-\Delta_{\mathbb{S}^{2}})^{1/2})W_{\gamma/2}f|^{2}_{L^{2}} +  |W^\epsilon(D)W_{\gamma/2}f|_{L^2}^2. \een
\end{lem}
\begin{proof} Taking $g=\mu^{1/2}$ in Lemma \ref{reduce-gamma-to-0}, we have
\beno
\mathcal{N}^{\epsilon,\gamma,\eta}(\mu^{1/2},f) +|f|^{2}_{L^{2}_{\gamma/2}}  \gtrsim  \mathcal{N}^{\epsilon,0,\eta}(W_{\gamma/2}\mu^{1/2},W_{\gamma/2}f) \geq \mathcal{N}^{\epsilon,0,\eta}(W_{-3/2}\mu^{1/2},W_{\gamma/2}f).
\eeno
Taking $g=W_{-3/2}\mu^{1/2}, f = W_{\gamma/2}f = F$ in Lemma \ref{reduce-eta-to-0}, we have for $\eta \leq r \leq 1 \leq R,  u \in B_{7R}$,
\begin{eqnarray}\label{eta-to-0-part1-special-g}
\mathcal{N}^{\epsilon,0,\eta}(W_{-3/2}\mu^{1/2}, F) + |F|^{2}_{L^{2}}  \gtrsim \mathcal{N}^{\epsilon,0,0}(\chi_{R}W_{-3/2}\mu^{1/2},(1-\chi_{3R})F).
\\ \label{eta-to-0-part2-special-g}
\mathcal{N}^{\epsilon,0,\eta}(W_{-3/2}\mu^{1/2},F) + r^{-2}R^{2}|F|^{2}_{L^{2}}  \gtrsim \mathcal{N}^{\epsilon,0,0}(\phi_{R,r,u}W_{-3/2}\mu^{1/2},\chi_{r,u}F).
\end{eqnarray}
Taking $R = 1$, then $\chi_{R} = \chi$, we get
\beno |\chi_{R}W_{-3/2}\mu^{1/2}|^{2}_{L^{2}} = \int \chi^{2}W_{-3}\mu dv \geq \frac{4 \pi}{3} 2^{-3/2} (2\pi)^{-\frac{3}{2}}e^{-1/2} :=\delta^{2}_{*}. \eeno
Recalling $\phi_{R,r,u} = \chi_{7R} -\chi_{3r,u}$ and $\chi_{7R} \geq \chi_{R}$, we have
\beno  \int \phi^{2}_{R,r,u}W_{-3}\mu dv \geq \frac{1}{2} \int \chi^{2}_{7R}W_{-3}\mu dv
-  \int \chi^{2}_{3r,u}W_{-3}\mu dv \geq \frac{1}{2}\delta^{2}_{*}
-  \int \chi^{2}_{3r,u}W_{-3}\mu dv. \eeno
Note that $\int \chi^{2}_{3r,u}W_{-3}\mu dv \leq \frac{4\pi}{3}(6r)^{3} (2\pi)^{-\frac{3}{2}}  := C r^{3}$.
By choosing $r$ such that $C r^{3} = \frac{1}{4}\delta^{2}_{*}$, we get
\beno |\phi_{R,r,u}W_{-3/2}\mu^{1/2}|^{2}_{L^{2}} \geq \delta^{2}_{*}/4. \eeno
It is easy to check $r = 6^{-1} 2^{-7/6} e^{-1/6}$.
Therefore we have
\ben \label{lower-bound-mu-l2}\min\{|\phi_{R,r,u}W_{-3/2}\mu^{1/2}|_{L^{2}}, |\chi_{R}W_{-3/2}\mu^{1/2}|_{L^{2}}\} \geq \delta_{*}/2.\een
On the other hand, it is obvious to see
\ben \label{upper-bound-mu-l21}
\max\{|\chi_{R}W_{-3/2}\mu^{1/2}|_{L^{2}_{1}},|\phi_{R,r,u}W_{-3/2}\mu^{1/2}|_{L^{2}_{1}}\}  \leq |\mu|_{L^{1}_{2}} := \lambda_{*}. \een

Thanks to \eqref{lower-bound-mu-l2} and \eqref{upper-bound-mu-l21}, by Lemma \ref{lowerboundpart1-general-g},
we get
\begin{eqnarray}\label{special-g-r-R-large-v} \mathcal{N}^{\epsilon,0,0}(\chi_{R}W_{-3/2}\mu^{1/2},(1-\chi_{3R})F) + |(1-\chi_{3R})F|^{2}_{L^{2}} \geq  C(\delta_{*}/2, \lambda_{*})|W^{\epsilon}(D)(1-\chi_{3R})F|^{2}_{L^{2}} .
\\ \label{special-g-r-R-small-v} \mathcal{N}^{\epsilon,0,0}(\phi_{R,r,u}W_{-3/2}\mu^{1/2},\chi_{r,u}F) + |\chi_{r,u}F|^{2}_{L^{2}} \geq  C(\delta_{*}/2, \lambda_{*})|W^{\epsilon}(D)\chi_{r,u}F|^{2}_{L^{2}}.
\end{eqnarray}
There is a finite cover of $B_{6R}$ with open ball $B_{r}(u_{j})$ for $u_{j} \in B_{6R}$. More precisely,
there exists $\{u_{j}\}_{j=1}^{N} \subset B_{6R}$ such that $B_{6R} \subset \cup_{j=1}^{N}B_{r}(u_{j})$, where $N \sim \frac{1}{r^{3}}$ is a universal constant. We then have $\chi_{3R} \leq \sum_{j=1}^{N}\chi_{r,u_{j}}$ and thus
$
|W^{\epsilon}(D)\chi_{3R}F|^{2}_{L^{2}} \leq N \sum_{j=1}^{N} |W^{\epsilon}(D)\chi_{r,u_{j}}F|^{2}_{L^{2}}.
$
From which together with \eqref{eta-to-0-part1-special-g}, \eqref{eta-to-0-part2-special-g}, \eqref{special-g-r-R-large-v}, \eqref{special-g-r-R-small-v},
we get for any $0 \leq \eta \leq r$,
\beno \mathcal{N}^{\epsilon,\gamma,\eta}(\mu^{1/2},f) +|f|_{L^2_{\gamma/2}}^2 \gtrsim r^{8} |W^\epsilon(D)W_{\gamma/2}f|_{L^2}^2. \eeno
Since $r$ is a universal constant, we get \eqref{sobolev-regularity-gamma-eta}.

Thanks to \eqref{lower-bound-mu-l2} and \eqref{upper-bound-mu-l21},
by \eqref{anisotropic-regularity-general-g} in Lemma \ref{lowerboundpart2-general-g}, we get
\beno
\mathcal{N}^{\epsilon,0,0}(\chi_{R}W_{-3/2}\mu^{1/2},(1-\chi_{3R})F) + \lambda^{2}_{*}(|W^{\epsilon}(D)(1-\chi_{3R})F|^{2}_{L^{2}}+|W^{\epsilon}(1-\chi_{3R})F|^{2}_{L^{2}}) \\ \gtrsim  \delta^{2}_{*}|W^{\epsilon}((-\Delta_{\mathbb{S}^{2}})^{1/2}) (1-\chi_{3R})F|^{2}_{L^{2}}, \\
\mathcal{N}^{\epsilon,0,0}(\phi_{R,r,u}W_{-3/2}\mu^{1/2},\chi_{r,u}F) + \lambda^{2}_{*}(|W^{\epsilon}(D)\chi_{r,u}F|^{2}_{L^{2}}+|W^{\epsilon}\chi_{r,u}F|^{2}_{L^{2}}) \gtrsim \delta^{2}_{*}|W^{\epsilon}((-\Delta_{\mathbb{S}^{2}})^{1/2}) \chi_{r,u}f|^{2}_{L^{2}} .
\eeno
Then a similar argument yields
\eqref{anisotropic-regularity-gamma-eta}.
\end{proof}


\subsection{Lower bound of $\langle \mathcal{L}^{\epsilon,\gamma,\eta}f,f\rangle$}
Thanks to Lemma \ref{lowerboundpart1} and  \eqref{full-regularity-gamma-eta} in Lemma \ref{lowerboundpart1-gamma-eta},  we get

\begin{lem}\label{two-parts-together} Let $-3 \leq \gamma \leq 0 \leq \eta \leq r_{0}$ where $r_{0}$ is the constant in  Lemma \ref{lowerboundpart1-gamma-eta}. We have
\ben \label{two-parts-together-N-N}
&& \mathcal{N}^{\epsilon,\gamma,\eta}(\mu^{1/2},f) + \mathcal{N}^{\epsilon,\gamma,\eta}(\mu^{1/2},f) + |f|^{2}_{L^{2}_{\gamma/2}} \\&\gtrsim&  |W^{\epsilon}((-\Delta_{\mathbb{S}^{2}})^{1/2})W_{\gamma/2}f|_{L^2}^2+ |W^{\epsilon}(D)W_{\gamma/2}f|_{L^2}^2 + |W^{\epsilon}W_{\gamma/2}f|_{L^2}^2=|f|_{\epsilon,\gamma/2}^2 . \nonumber
\een
\end{lem}
Now we are ready to prove the following coercivity estimate, which a stronger version of
Theorem \ref{coercivity-structure}.
\begin{thm}\label{strong-coercivity} Let $-3 \leq \gamma \leq 0 \leq \eta \leq r_{0}$. We have
\beno
\langle \mathcal{L}^{\epsilon,\gamma,\eta}f, f\rangle + |f|^{2}_{L^{2}_{\gamma/2}} \gtrsim  |f|_{\epsilon,\gamma/2}^2.
\eeno
\end{thm}
\begin{proof}
We recall that $\mathcal{N}^{\epsilon,\gamma,\eta}(\mu^{1/2},f) + \mathcal{N}^{\epsilon,\gamma,\eta}(f,\mu^{1/2})$ corresponds to
the anisotropic norm $|||\cdot|||$ introduced in \cite{alexandre2012boltzmann}.
By the proof of Proposition 2.16 in \cite{alexandre2012boltzmann}, there holds
\beno
\langle \mathcal{L}^{\epsilon,\gamma,\eta}_{1} f,f \rangle \geq \frac{1}{10}(\mathcal{N}^{\epsilon,\gamma,\eta}(\mu^{1/2},f) + \mathcal{N}^{\epsilon,\gamma,\eta}(f,\mu^{1/2})) - \frac{3}{10} \big|\int B^{\epsilon,\gamma,\eta} \mu_{*} (f^{2} -f^{\prime 2}) d\sigma dv dv_{*} \big|.
\eeno
If $\gamma=-3$, by \eqref{geq-delta-mu} in the Cancellation Lemma \ref{cancellation-lemma-general-gamma-minus3-mu}
with $a=1/2, p=\infty, q=1$, we  have
\beno  \big|\int B^{\epsilon,\gamma,\eta} \mu_{*} (f^{2} -f^{\prime 2}) d\sigma dv dv_{*} \big| \leq C |\mu^{1/2}f^{2}|_{L^{1}}
\leq C |f|^{2}_{L^{2}_{\gamma/2}}.
\eeno
If $\gamma>-3$, referring to \cite{alexandre2000entropy}, we have
\beno \big|\int B^{\epsilon,\gamma,\eta} \mu_{*} (f^{2} -f^{\prime 2}) d\sigma dv dv_{*} \big| \leq C \int|v-v_{*}|^{\gamma}\mu_{*} f^{2} dv dv_{*}
\leq C |f|^{2}_{L^{2}_{\gamma/2}}.
\eeno
Therefore we have
\ben \label{L-1-dominate}
\langle \mathcal{L}^{\epsilon,\gamma,\eta}_{1} f,f \rangle \geq \frac{1}{10}(\mathcal{N}^{\epsilon,\gamma,\eta}(\mu^{1/2},f) + \mathcal{N}^{\epsilon,\gamma,\eta}(f,\mu^{1/2})) - C |f|^{2}_{L^{2}_{\gamma/2}}.
\een
By Lemma \ref{l2-full-estimate-geq-eta}, we have
\ben \label{L-2-lower}
|\langle \mathcal{L}^{\epsilon,\gamma,\eta}_{2} f,f \rangle| \lesssim |\mu^{1/8}f|_{L^2}^2 \lesssim |f|_{L^2_{\gamma/2}}^2.
\een
Patching \eqref{L-1-dominate} and \eqref{L-2-lower}, we arrive at \eqref{equivalence-relation}.
\ben \label{L-dominate-lower-bound}
\langle \mathcal{L}^{\epsilon,\gamma,\eta} f,f \rangle + |f|_{L^2_{\gamma/2}}^2
\gtrsim \mathcal{N}^{\epsilon,\gamma,\eta}(\mu^{1/2},f) + \mathcal{N}^{\epsilon,\gamma,\eta}(f,\mu^{1/2}).
\een
From which together with Lemma \ref{two-parts-together}, we finish the proof.
\end{proof}

\setcounter{equation}{0}

\section{Spectral gap estimate} \label{Spectral-Gap-Estimate}
In this section, we will consider the  spectral gap estimate of $\mathcal{L}^{\epsilon,\gamma,\eta}$. As we explained in the introduction, it will yield  Theorem \ref{micro-dissipation} and Theorem   \ref{micro-dissipation1}. In order to get \eqref{ideasg2},
we first prove the smallness of $\langle  \mathcal{L}^{\epsilon,0}_{\eta}f, f\rangle$ when $\eta$ is small.
\begin{lem}\label{gamma-0-eat-lb} Let $0 \leq \eta \leq 1$, then
\beno |\langle  \mathcal{L}^{\epsilon,0}_{\eta}f, f\rangle| \lesssim \eta^{3} |W^{\epsilon}(D)(\mathbb{I}-\mathbb{P})f|^{2}_{L^{2}}. \eeno
\end{lem}
\begin{proof} The null space of $\mathcal{L}^{\epsilon,0}_{\eta}$ is $\mathcal{N}$, and $\mathcal{L}^{\epsilon,0}_{\eta}$ is a self-joint operator. Therefore it suffices to consider $f \in \mathcal{N}^{\perp}$.
 Similar to \eqref{definition-of-N-ep-ga-eta}, we define $\mathcal{N}^{\epsilon,\gamma}_{\eta}(g,h)$ with the restriction
$|v-v_{*}|\leq \eta$. It is easy to check
 $\langle \mathcal{L}^{\epsilon,0}_{\eta} f,f \rangle \leq 2\mathcal{N}^{\epsilon,0}_{\eta}(\mu^{1/2},f) + 2\mathcal{N}^{\epsilon,0}_{\eta}(f,\mu^{1/2})$. We divide the proof into two steps.

\noindent{\it Step 1: Estimate of $\mathcal{N}^{\epsilon,0}_{\eta}(f,\mu^{1/2})$.}  Recall
\beno \mathcal{N}^{\epsilon,0}_{\eta}(f,\mu^{1/2}) = \int b^{\epsilon}(\cos\theta)
\mathrm{1}_{|v-v_{*}|\leq \eta}
f^{2}_{*} ((\mu^{1/2})^{\prime}-\mu^{1/2})^{2} d\sigma dv dv_{*}. \eeno
By Taylor expansion,
 $|(\mu^{1/2})^{\prime}-\mu^{1/2}| \leq |\nabla \mu^{1/2}|_{L^\infty}|v^{\prime}-v| \lesssim |v-v_{*}|\theta$.
 From which together with \eqref{order-2}, we have
\beno \mathcal{N}^{\epsilon,0}_{\eta}(f,\mu^{1/2}) &\lesssim&  \int b^{\epsilon}(\cos\theta)
\mathrm{1}_{|v-v_{*}|\leq \eta}
f^{2}_{*} |v-v_{*}|^{2}\theta^{2} d\sigma dv dv_{*}
\\&\lesssim&  \int
\mathrm{1}_{|v-v_{*}|\leq \eta}
f^{2}_{*} |v-v_{*}|^{2} dv dv_{*} \lesssim \eta^{5} |f|^{2}_{L^{2}}. \eeno

\noindent{\it Step 2: Estimate of $\mathcal{N}^{\epsilon,0}_{\eta}(\mu^{1/2},f)$.}  Recalling the decomposition $f=\mathfrak{F}_{\phi}f+\mathfrak{F}^{\phi}f$ in \eqref{defphilh}, we have
\beno \mathcal{N}^{\epsilon,0}_{\eta}(\mu^{1/2},f) &=& \int b^{\epsilon}(\cos\theta)
\mathrm{1}_{|v-v_{*}|\leq \eta}
\mu_{*} (f^{\prime}-f)^{2} d\sigma dv dv_{*}
\\&\leq& 2\int b^{\epsilon}(\cos\theta)
\mathrm{1}_{|v-v_{*}|\leq \eta}
\mu_{*} ((\mathfrak{F}_{\phi}f)^{\prime}-\mathfrak{F}_{\phi}f)^{2} d\sigma dv dv_{*}
\\&&+ 2\int b^{\epsilon}(\cos\theta)
\mathrm{1}_{|v-v_{*}|\leq \eta}
\mu_{*} ((\mathfrak{F}^{\phi}f)^{\prime}-\mathfrak{F}^{\phi}f)^{2} d\sigma dv dv_{*}
:= 2I_{\mathrm{low}}  + 2 I_{\mathrm{high}}.
\eeno

 {\it Step 2.1: Estimate of   $I_{\mathrm{high}}$.} By  $((\mathfrak{F}^{\phi}f)^{\prime}-\mathfrak{F}^{\phi}f)^{2} \leq  2 \left(((\mathfrak{F}^{\phi}f)^{\prime})^{2} + (\mathfrak{F}^{\phi}f)^{2}\right)$, the change of variable $v \rightarrow v^{\prime}$,
 the estimate \eqref{order-0} and the fact \eqref{high-frequency-lb-cf},
 we have
\beno I_{\mathrm{high}}&\lesssim& \int b^{\epsilon}(\cos\theta)
\mathrm{1}_{|v-v_{*}|\leq \eta}
\mu_{*} (\mathfrak{F}^{\phi}f)^{2} d\sigma dv dv_{*}
\\&\lesssim&  |\ln \epsilon|^{-1} \epsilon^{-2} \int
\mathrm{1}_{|v-v_{*}|\leq \eta}
\mu_{*} (\mathfrak{F}^{\phi}f)^{2} dv dv_{*}
\lesssim \eta^{3}  |\ln \epsilon|^{-1} \epsilon^{-2} |\mathfrak{F}^{\phi}f|^{2}_{L^{2}} \lesssim \eta^{3}  |W^{\epsilon}(D)f|^{2}_{L^{2}}.
\eeno

 {\it Step 2.2: Estimate of   $I_{\mathrm{low}}$.}
   Observe $((\mathfrak{F}_{\phi}f)^{\prime}-\mathfrak{F}_{\phi}f)^{2} = ((\mathfrak{F}_{\phi}f)^{\prime})^{2} - (\mathfrak{F}_{\phi}f)^{2} + 2 \mathfrak{F}_{\phi}f (\mathfrak{F}_{\phi}f - (\mathfrak{F}_{\phi}f)^{\prime})$. We get
\beno I_{\mathrm{low}}&=& \int b^{\epsilon}(\cos\theta)
\mathrm{1}_{|v-v_{*}|\leq \eta}
\mu_{*} \left(((\mathfrak{F}_{\phi}f)^{\prime})^{2} - (\mathfrak{F}_{\phi}f)^{2}\right) d\sigma dv dv_{*}
\\&&+ 2\int b^{\epsilon}(\cos\theta)
\mathrm{1}_{|v-v_{*}|\leq \eta}
\mu_{*} \mathfrak{F}_{\phi}f \left(\mathfrak{F}_{\phi}f - (\mathfrak{F}_{\phi}f)^{\prime}\right) d\sigma dv dv_{*}
:=  I_{\mathrm{low,cancell}} +2 I_{\mathrm{low,dydadic}}.
\eeno

\underline{Estimate of $I_{\mathrm{low,cancell}}$.} Thanks to \eqref{order-2},
by cancellation lemma in \cite{alexandre2000entropy},
we get
\ben \label{cancellation-gamma-0}
I_{\mathrm{low,cancell}} \lesssim \int
\mathrm{1}_{|v-v_{*}|\leq \eta}
\mu_{*} (\mathfrak{F}_{\phi}f)^{2} dv dv_{*}
\lesssim \eta^{3} |\mathfrak{F}_{\phi}f|^{2}_{L^{2}} \lesssim \eta^{3}  |f|^{2}_{L^{2}}.
\een

\underline{Estimate of $I_{\mathrm{low,dydadic}}$.}  For simplicity, we define
\beno \mathcal{Y}(g,h):= \int b^{\epsilon}(\cos\theta)
\mathrm{1}_{|v-v_{*}|\leq \eta}
\mu_{*} g \left(h^{\prime}-h\right) d\sigma dv dv_{*}. \eeno
By the dyadic decomposition, we have
\ben \label{low-dydadic} I_{\mathrm{low,dydadic}} &=& -\mathcal{Y}(\mathfrak{F}_{\phi}f,\mathfrak{F}_{\phi}f)
= -\sum_{j,k=-1}^{\infty} \mathcal{Y}(\mathfrak{F}_{j}\mathfrak{F}_{\phi}f, \mathfrak{F}_{k}\mathfrak{F}_{\phi}f)
\\&=& -\sum_{-1\leq j\leq k \lesssim |\ln \epsilon| }\mathcal{Y}(\mathfrak{F}_{j}\mathfrak{F}_{\phi}f, \mathfrak{F}_{k}\mathfrak{F}_{\phi}f)
-\sum_{-1\leq k < j \lesssim |\ln \epsilon| } \mathcal{Y}(\mathfrak{F}_{j}\mathfrak{F}_{\phi}f, \mathfrak{F}_{k}\mathfrak{F}_{\phi}f).  \nonumber
\een
For simplicity, we set $F_{k}=\mathfrak{F}_{k}\mathfrak{F}_{\phi}f$.

\underline{\it Case 1: $k<j$.}
Note that
\beno
\mathcal{Y}(F_{j},F_{k}) &=& \int B^{\epsilon,0}_{\eta} \phi(\sin(\theta/2)/2^{k}) \mu_{*} F_{j} ((F_{k})^{\prime}-F_{k}) d\sigma dv dv_{*}
\\&&+ \int B^{\epsilon,0}_{\eta} (1-\phi(\sin(\theta/2)/2^{k})) \mu_{*}  F_{j} ((F_{k})^{\prime}-F_{k}) d\sigma dv dv_{*}
:=\mathcal{Y}_{1}(F_{j},F_{k}) +\mathcal{Y}_{2}(F_{j},F_{k}).
\eeno
Let us first consider $\mathcal{Y}_{1}(F_{j},F_{k})$ in which $\epsilon \leq \sin(\theta/2) \leq \frac{4}{3} \times 2^{-k} \leq 2^{-k+1}$.
By Taylor expansion,
\beno (F_{k})^{\prime}-F_{k} = (\nabla F_{k}) \cdot (v^{\prime}-v) + \int_{0}^{1} \frac{1-\kappa}{2}  (\nabla^{2}F_{k})(v(\kappa)):(v^\prime - v)\otimes (v^\prime - v)
 d\kappa. \eeno
We separate $\mathcal{Y}_{1}(F_{j},F_{k}) = \mathcal{Y}_{1,1}(F_{j},F_{k}) + \mathcal{Y}_{1,2}(F_{j},F_{k})$ according to the previous expansion with
\beno
\mathcal{Y}_{1,1}(F_{j},F_{k}) := \int B^{\epsilon,0}_{\eta} \phi(\sin(\theta/2)/2^{k}) \mu_{*} F_{j} (\nabla F_{k}) \cdot (v^{\prime}-v) d\sigma dv dv_{*},
\\
\mathcal{Y}_{1,2}(F_{j},F_{k}) := \int_{0}^{1} \int B^{\epsilon,0}_{\eta} \phi(\sin(\theta/2)/2^{k}) \mu_{*} F_{j} \int_{0}^{1} \frac{1-\kappa}{2}  (\nabla^{2}F_{k})(v(\kappa)):(v^\prime - v)\otimes (v^\prime - v)
d\sigma dv dv_{*} d\kappa.
\eeno
We first estimate $\mathcal{Y}_{1,1}(F_{j},F_{k})$. Note that
\ben \label{around-mode-k}
|\int b^{\epsilon}\phi\left(\sin(\theta/2)/2^{k}\right) (v^{\prime}-v) d\sigma| &=& |\int b^{\epsilon}\phi\left(\sin(\theta/2)/2^{k}\right)  \sin^{2}\frac{\theta}{2} d\sigma (v_{*}-v)| \\&\lesssim& |\ln \epsilon|^{-1}(|\ln \epsilon|-k \ln 2 + \ln2) |v_{*}-v|, \nonumber
\een
which yields
\beno
|\mathcal{Y}_{1,1}(F_{j},F_{k})|&\lesssim& |\ln \epsilon|^{-1}(|\ln \epsilon|-k \ln 2 + \ln2) \int 1_{|v - v_{*}|\leq \eta} |v - v_{*}|\mu_{*} |F_{j} (\nabla F_{k})|dv dv_{*}
\\&\lesssim& \eta^{4}|\ln \epsilon|^{-1}(|\ln \epsilon|-k \ln 2 + \ln2)2^{k}|F_{j}|_{L^{2}}|F_{k}|_{L^{2}}.
\eeno
We go to estimate $\mathcal{Y}_{1,2}(F_{j},F_{k})$. By Cauchy-Schwartz inequality, and the change \eqref{change-exact-formula}-\eqref{change-Jacobean-bound},
we get
\beno
|\mathcal{Y}_{1,2}(F_{j},F_{k})|&\lesssim&  | \int b^{\epsilon} \phi\left(\sin(\theta/2)/2^{k}\right)\theta^{2} 1_{|v - v_{*}|\leq \eta} |v - v_{*}|^{2}\mu_{*} F_{j} |(\nabla^{2} F_{k})(v(\kappa))| d\sigma dv dv_{*} d\kappa
\\&\lesssim&  \left(  \int b^{\epsilon} \phi\left(\sin(\theta/2)/2^{k}\right)\theta^{2} 1_{|v - v_{*}|\leq \eta} |v - v_{*}|^{2}\mu_{*}
|F_{j}|^{2} d\sigma dv dv_{*}  \right)^{1/2}
\\&&\times \left(  \int b^{\epsilon} \phi\left(\sin(\theta/2)/2^{k}\right)\theta^{2} 1_{|v - v_{*}|\leq \eta} |v - v_{*}|^{2}\mu_{*}|\nabla^{2} F_{k}|^{2} d\sigma dv dv_{*} \right)^{1/2}
\\&\lesssim&  |\ln \epsilon|^{-1}(|\ln \epsilon|-k \ln 2 + \ln2) \eta^{5} |F_{j}|_{L^{2}}|F_{k}|_{H^{2}}
\\&\lesssim&  \eta^{5}|\ln \epsilon|^{-1}(|\ln \epsilon|-k \ln 2 + \ln2)2^{2k}|F_{j}|_{L^{2}}|F_{k}|_{L^{2}}.
\eeno
Patching together the estimates of $\mathcal{Y}_{1,1}(F_{j},F_{k})$ and $\mathcal{Y}_{1,2}(F_{j},F_{k})$, we have
\beno  |\mathcal{Y}_{1}(F_{j},F_{k})| \lesssim  \eta^{4}|\ln \epsilon|^{-1}(|\ln \epsilon|-k \ln 2 + \ln2)2^{2k}|F_{j}|_{L^{2}}|F_{k}|_{L^{2}}. \eeno

We now turn to $\mathcal{Y}_{2}(F_{j},F_{k})$ in which $\theta \gtrsim 2^{-k}$.
By Taylor expansion up to order 1, we have
\beno  |(F_{k})^{\prime}-F_{k}|  = |\int_{0}^{1}(\nabla (F_{k}))(v(\kappa))\cdot (v^{\prime}-v)d\kappa| \lesssim
\theta|v-v_{*}|\int_{0}^{1}|(\nabla (F_{k}))(v(\kappa))|d\kappa.\eeno
Plugging the above inequality into the definition of $\mathcal{Y}_{2}(F_{j},F_{k})$, by Cauchy-Schwartz inequality, the change \eqref{change-exact-formula}-\eqref{change-Jacobean-bound}, and the fact $\int_{2^{-k}}^{\pi/2} \theta^{-2} d\theta \lesssim 2^{k}$
, we get
\beno
|\mathcal{Y}_{2}(F_{j},F_{k})|&\lesssim&   \int b^{\epsilon} (1-\phi\left(\sin(\theta/2)/2^{k}\right))\theta 1_{|v - v_{*}|\leq \eta} |v - v_{*}|\mu_{*} | F_{j} (\nabla F_{k})(v(\kappa))|
d\sigma dv dv_{*} d\kappa
\\&\lesssim&  \left(  \int b^{\epsilon} \phi\left(\sin(\theta/2)/2^{k}\right)\theta 1_{|v - v_{*}|\leq \eta} |v - v_{*}| \mu_{*} (F_{j})^{2}  d\sigma dv dv_{*} \right)^{1/2}
\\&&\times \left(  \int b^{\epsilon} \phi\left(\sin(\theta/2)/2^{k}\right)\theta 1_{|v - v_{*}|\leq \eta} |v - v_{*}|\mu_{*}|\nabla F_{k}|^{2}  d\sigma dv dv_{*}  \right)^{1/2}
\\&\lesssim&  |\ln \epsilon|^{-1}2^{k}  |F_{j}|_{L^{2}}|F_{k}|_{H^{1}}
\lesssim  \eta^{4}|\ln \epsilon|^{-1}2^{2k} |F_{j}|_{L^{2}}|F_{k}|_{L^{2}}.
\eeno

Patching together the estimates of $\mathcal{Y}_{1}(F_{j},F_{k})$ and $\mathcal{Y}_{2}(F_{j},F_{k})$, we have when $k<j$,
\ben \label{k-less-j-low-low-gamma-0}  \mathcal{Y}(F_{j},F_{k}) \lesssim   \eta^{4}|\ln \epsilon|^{-1}(1 + |\ln \epsilon|-k \ln 2)2^{2k} |F_{j}|_{L^{2}}|F_{k}|_{L^{2}}. \een

\underline{\it Case 2: $j \leq  k$.} We have
\beno
\mathcal{Y}(F_{j},F_{k}) &=& \int B^{\epsilon,0}_{\eta}  \mu_{*} F_{j} ((F_{k})^{\prime}-F_{k}) d\sigma dv dv_{*}
\\&=& \int B^{\epsilon,0}_{\eta} \mu_{*} ((F_{j}F_{k})^{\prime}-F_{j}F_{k})d\sigma dv dv_{*}
+\int B^{\epsilon,0}_{\eta} \mu_{*} \left(F_{j}-(F_{j})^{\prime}\right)(F_{k})^{\prime} d\sigma dv dv_{*}
\\&:=& \mathcal{X}_{1}(F_{j},F_{k}) + \mathcal{X}_{2}(F_{j},F_{k}).
\eeno
Similar to
 \eqref{cancellation-gamma-0}, using cancellation lemma in \cite{alexandre2000entropy}, we get
\ben \label{with-eta-small-gamma-0} |\mathcal{X}_{1}(F_{j},F_{k})|  \lesssim \eta^{3}  |F_{j}|_{L^{2}}|F_{k}|_{L^{2}}.\een
Similar to the estimate of $\mathcal{Y}(F_{j},F_{k})$ in {\it {Case 1}} where $k<j$,  here we apply Taylor expansion to $F_{j}$, similar to \eqref{k-less-j-low-low-gamma-0}, we can get
\ben \label{without-eta-small-y2-gamma-0} |\mathcal{X}_{2}(F_{j},F_{k})| \lesssim \eta^{4}|\ln \epsilon|^{-1}(1 + |\ln \epsilon|-j \ln 2)2^{2j} |F_{j}|_{L^{2}}|F_{k}|_{L^{2}}. \een
Patching together, we get for $j \leq  k$,
\ben \label{with-small-eta-k-less-j-gamma-0} |\mathcal{Y}(F_{j},F_{k})| \lesssim \eta^{4}|\ln \epsilon|^{-1}(1 + |\ln \epsilon|-j \ln 2)2^{2j} |F_{j}|_{L^{2}}|F_{k}|_{L^{2}} + \eta^{3}  |F_{j}|_{L^{2}}|F_{k}|_{L^{2}}.\een

By \eqref{k-less-j-low-low-gamma-0} and \eqref{with-small-eta-k-less-j-gamma-0},  recalling \eqref{low-dydadic}, we have
\ben \label{with-eta-small-low-low-final-gamma-0} |I_{\mathrm{low,dydadic}}| &\lesssim& \eta^{4}\sum_{-1\leq k < j \lesssim |\ln \epsilon| } 2^{2k} \frac{|\ln \epsilon|- k\ln 2 +1}{|\ln \epsilon|} |F_{j}|_{L^{2}} |F_{k}|_{L^{2}} \nonumber
\\&&+\eta^{4}\sum_{-1\leq j\leq k \lesssim |\ln \epsilon| } 2^{2j} \frac{|\ln \epsilon|- j\ln 2 +1}{|\ln \epsilon|} |F_{j}|_{L^{2}} |F_{k}|_{L^{2}} \nonumber
\\&&+ \eta^{3}\sum_{-1\leq j \leq k \lesssim |\ln \epsilon| }|F_{j}|_{L^{2}}|F_{k}|_{L^{2}}
\lesssim   \eta^{3} |W^{\epsilon}(D)f|^{2}_{L^{2}}.
\een

The lemma follows by patching together all the estimates.
\end{proof}

Before giving the spectral gap result, we first introduce a special weight function $U_{\delta}$ defined by
\ben\label{specialweightfun} U_{\delta}(v) := (1+ \delta^{2}|v|^{2})^{1/2} \geq \max\{\delta|v|,1\}. \een
We remark that  $U_{\delta}$ plays an important role in deriving \eqref{ideasg2} and here $\delta$ is a sufficiently small parameter.
We recall the function $\chi$ and its dilation $\chi_{R}$ at the beginning of section 2.3.2 (right before Lemma \ref{reduce-eta-to-0}).
\begin{lem}\label{difference-term-complication} Set $X(\gamma,R,\delta):=\delta^{-\gamma}\left((\chi_{R})^{\prime}(\chi_{R})^{\prime}_{*}(U^{\gamma/2}_{\delta})^{\prime}(U^{\gamma/2}_{\delta})^{\prime}_{*}-
\chi_{R}(\chi_{R})_{*}U^{\gamma/2}_{\delta}(U^{\gamma/2}_{\delta})_{*}\right)^{2}$ with $\gamma \leq 0 < \delta \leq 1 \leq R$, then
\ben \label{chi-W-difference-part}&&X(\gamma,R,\delta)
\lesssim(\delta^{2}+R^{-2})\theta^{2}\langle v \rangle^{\gamma+2}\langle v_{*} \rangle^{2} \mathrm{1}_{|v|\leq 4R}.\een
\end{lem}
\begin{proof} Recall that $\chi_{R}$ has support in $|v| \leq 2R$.
If $|v|^{2}+|v_{*}|^{2} \geq 8 R^{2}$, then either $|v| \geq 2R$ or $|v_{*}| \geq 2R$, which implies $(\chi_{R})^{2}_{*} \chi^{2}_{R}=0$. Note that $|v|^{2}+|v_{*}|^{2}=|v^{\prime}|^{2}+|v^{\prime}_{*}|^{2}$, then $|v|^{2}+|v_{*}|^{2} \geq 8 R^{2}$ also implies $(\chi_{R})^{\prime}(\chi_{R})^{\prime}_{*}=0$. Therefore, we have
\ben \label{restriction-to-les-R} X(\gamma,R,\delta) =  X(\gamma,R,\delta) \mathrm{1}_{|v|^{2}+|v_{*}|^{2} \leq 8 R^{2}} \leq X(\gamma,R,\delta) \mathrm{1}_{|v|\leq 4R}.  \een
By adding and subtracting terms, we get
\beno X(\gamma,R,\delta) &\lesssim& \delta^{-\gamma}((\chi_{R})^{\prime}- \chi_{R})^{2}(\chi^{2}_{R})^{\prime}_{*}(U^{\gamma}_{\delta})^{\prime}(U^{\gamma}_{\delta})^{\prime}_{*}
+ \delta^{-\gamma}( (\chi_{R})^{\prime}_{*}- (\chi_{R})_{*})^{2}\chi^{2}_{R}(U^{\gamma}_{\delta})^{\prime}(U^{\gamma}_{\delta})^{\prime}_{*}
\\&&+ \delta^{-\gamma}\chi^{2}_{R}(\chi^{2}_{R})_{*}(U^{\gamma}_{\delta})^{\prime}_{*}\left((U^{\gamma/2}_{\delta})^{\prime}-(U^{\gamma/2}_{\delta})\right)^{2}
+\delta^{-\gamma}\chi^{2}_{R}(\chi^{2}_{R})_{*}U^{\gamma}_{\delta}\left((U^{\gamma/2}_{\delta})^{\prime}_{*}-(U^{\gamma/2}_{\delta})_{*}\right)^{2}
 \\&:=& A_{1} + A_{2} +  A_{3} + A_{4}. \eeno

 {\it  \underline{Estimate of $A_{1}$ and $A_{2}$}.} Since $\gamma \leq 0$ and $|v^{\prime}_{*}|^{2}+|v^{\prime}|^{2}=|v_*|^2+|v|^2$, we derive
\beno (U^{\gamma}_{\delta})^{\prime}(U^{\gamma}_{\delta})^{\prime}_{*}  =  (1+\delta^{2}|v^{\prime}|^{2}+\delta^{2}|v^{\prime}_{*}|^{2}+\delta^{4}|v^{\prime}|^{2}|v^{\prime}_{*}|^{2} )^{\gamma/2}
  \leq (1+\delta^{2}|v|^{2})^{\gamma/2},
 \eeno
which yields
\beno \delta^{-\gamma}(U^{\gamma}_{\delta})^{\prime}(U^{\gamma}_{\delta})^{\prime}_{*} \leq  (\delta^{-2}+|v|^{2})^{\gamma/2} \leq \langle v \rangle^{\gamma}.\eeno
Since $|\nabla \chi_{R}| \lesssim R^{-1}, |v^{\prime}-v|=|v^{\prime}_{*}-v_{*}| = |v-v_{*}| \sin(\theta/2)$, we get
\beno ((\chi_{R})^{\prime}- \chi_{R})^{2} + ( (\chi_{R})^{\prime}_{*}- (\chi_{R})_{*})^{2} \lesssim R^{-2}\theta^{2}|v-v_{*}|^{2}\lesssim R^{-2}\theta^{2}\langle v \rangle^{2}\langle v_{*} \rangle^{2}. \eeno
Therefore we deduce that
$ A_{1} + A_{2} \lesssim R^{-2}\theta^{2}\langle v \rangle^{\gamma+2}\langle v_{*} \rangle^{2}.$

 {\it \underline{Estimate of $A_{3}$ and $A_{4}$}.}
We now go to estimate $A_{3}$.
Noting that $|\nabla U^{\gamma/2}_{\delta}| \lesssim  \delta U^{\gamma/2}_{\delta}$, we get
\beno \big((U^{\gamma/2}_{\delta})^{\prime}-(U^{\gamma/2}_{\delta})\big)^{2}  = \big|\int_{0}^{1} (\nabla U^{\gamma/2}_{\delta})(v(\kappa))\cdot (v^{\prime}-v)d\kappa\big|^{2}
\lesssim\delta^{2} \theta^{2}|v-v_{*}|^{2}  \int_{0}^{1}  U^{\gamma}_{\delta}(v(\kappa))d\kappa.
\eeno
Thanks to $|v^{\prime}_{*}|^{2}+|v(\kappa)|^{2} \sim |v|^{2}+|v_{*}|^{2}$, we have
$\delta^{-\gamma}(U^{\gamma}_{\delta})^{\prime}_{*} U^{\gamma}_{\delta}(v(\kappa)) \lesssim \langle v \rangle^{\gamma}$, which gives
$ A_{3} \lesssim \delta^{2}\theta^{2}\langle v \rangle^{\gamma+2}\langle v_{*} \rangle^{2}.$
Similarly, we have $ A_{4} \lesssim \delta^{2}\theta^{2}\langle v \rangle^{\gamma+2}\langle v_{*} \rangle^{2}.$

Patching together the above estimates for $A_{1}, A_{2}, A_{3}, A_{4}$ and \eqref{restriction-to-les-R} , we arrive at \eqref{chi-W-difference-part}.
\end{proof}

Now we are in a position to prove the following spectral gap result.
\begin{thm}\label{micro-weight-dissipation} Let $-3 \leq \gamma \le 0$. There are three universal constants $\epsilon_{0}, \eta_{0}, \lambda_0>0$ ($\lambda_0$ is related to $\lambda_1^\epsilon$ in \eqref{firstegenvalueLA}), such that for any $0 \leq \epsilon \leq \epsilon_{0}, 0 \leq \eta \leq \eta_{0}$ and smooth function $g$, the following estimate holds true.
\ben \label{gap-estimate}
 \langle  \mathcal{L}^{\epsilon,\gamma,\eta}g, g\rangle \geq \lambda_{0}|(\mathbb{I}-\mathbb{P})g|^{2}_{\epsilon,\gamma/2}.
\een
\end{thm}

\begin{proof} Suppose  $\mathbb{P}g=0$ and then it suffices to   prove
$ \langle  \mathcal{L}^{\epsilon,\gamma,\eta}g, g\rangle \gtrsim  |g|^{2}_{\epsilon,\gamma/2}.
$
For brevity, we set
\beno J^{\epsilon,\gamma,\eta}(g) := 4 \langle  \mathcal{L}^{\epsilon,\gamma,\eta}g, g\rangle,\quad\mathbb{A}(f, g):=(f_{*}g + f g_{*} - f^{\prime}_{*}g^{\prime} - f^{\prime} g^{\prime}_{*}), \quad\mathbb{F}(f, g):= \mathbb{A}^{2}(f, g).\eeno With these notations, we have
 $J^{\epsilon,\gamma,\eta}(g) = \int B^{\epsilon,\gamma,\eta} \mathbb{F}(\mu^{1/2}, g) d\sigma dv dv_{*}.$ Our proof is divided into four steps.

\noindent{\it Step 1: Localization of $J^{\epsilon,\gamma,\eta}(g)$.}
Due to \eqref{specialweightfun} and the condition $\gamma\le0$,  we get \beno |v - v_{*}|^{-\gamma}
\leq  C_{\gamma} \delta^{\gamma} ((\delta|v|)^{-\gamma}+(\delta|v_{*}|)^{-\gamma})
 \leq 2C_{\gamma} \delta^{\gamma} U^{-\gamma}_{\delta}(v)U^{-\gamma}_{\delta}(v_{*}),  \eeno
which gives
$|v - v_{*}|^{\gamma} \gtrsim \delta^{-\gamma} U^{\gamma}_{\delta}(v)U^{\gamma}_{\delta}(v_{*})$
and thus
\beno J^{\epsilon,\gamma,\eta}(g) \gtrsim \delta^{-\gamma}\int b^{\epsilon} \mathrm{1}_{|v-v_{*}| \geq \eta} \chi^{2}_{R}(\chi^{2}_{R})_{*}U^{\gamma}_{\delta}(U^{\gamma}_{\delta})_{*}\mathbb{F}(\mu^{1/2}, g) d\sigma dv dv_{*}.
\eeno
We move the function $\chi^{2}_{R}(\chi^{2}_{R})_{*}U^{\gamma}_{\delta}(U^{\gamma}_{\delta})_{*}$ inside $\mathbb{F}(\mu^{1/2}, g)$, which leads to  $\mathbb{F}(\chi_{R}U^{\gamma/2}_{\delta}\mu^{1/2}, \chi_{R}U^{\gamma/2}_{\delta}g)$ with some correction terms.
For   simplicity, set $h=\chi_{R} U^{\gamma/2}_{\delta}, f=\mu^{1/2}$,
then
\ben \label{move-inside}
\chi^{2}_{R}(\chi^{2}_{R})_{*}U^{\gamma}_{\delta}(U^{\gamma}_{\delta})_{*}\mathbb{F}(\mu^{1/2}, g)  &=&
h^{2}_{*}h^{2}\mathbb{F}(f, g)
=
\left(h h_{*}\left(f_{*}g + f g_{*}\right) -
h h_{*}\left(f^{\prime}_{*}g^{\prime} + f^{\prime} g^{\prime}_{*}\right)\right)^{2}
\nonumber \\&=& \left(h h_{*}\left(f_{*}g + f g_{*}\right) -
h^{\prime} h^{\prime}_{*}\left(f^{\prime}_{*}g^{\prime} + f^{\prime} g^{\prime}_{*}\right)
+ \left(h^{\prime} h^{\prime}_{*}-
h h_{*}\right)
\left(f^{\prime}_{*}g^{\prime} + f^{\prime} g^{\prime}_{*}\right)
\right)^{2}
\nonumber \\  &\geq&  \frac{1}{2}  \left(h h_{*}\left(f_{*}g + f g_{*}\right) -
h^{\prime} h^{\prime}_{*}\left(f^{\prime}_{*}g^{\prime} + f^{\prime} g^{\prime}_{*}\right) \right)^{2}\nonumber -\left(h^{\prime} h^{\prime}_{*}-
h h_{*}\right)^{2}
\left(f^{\prime}_{*}g^{\prime} + f^{\prime} g^{\prime}_{*}\right)^{2}\nonumber
\\  &=&  \frac{1}{2}  \mathbb{F}(h f, h g)
-\left(h^{\prime} h^{\prime}_{*}-
h h_{*}\right)^{2}
\left(f^{\prime}_{*}g^{\prime} + f^{\prime} g^{\prime}_{*}\right)^{2}.
\een
From which we get
\ben \label{move-inside-sym} J^{\epsilon,\gamma,\eta}(g) &\gtrsim&
\frac{1}{2}\delta^{-\gamma}
\int b^{\epsilon} \mathrm{1}_{|v-v_{*}| \geq \eta} \mathbb{F}(\chi_{R}U^{\gamma/2}_{\delta}\mu^{1/2}, \chi_{R}U^{\gamma/2}_{\delta}g) d\sigma dv dv_{*}
\\  && - \delta^{-\gamma}\int b^{\epsilon} \left(h^{\prime} h^{\prime}_{*}-
h h_{*}\right)^{2}
\left(f^{\prime}_{*}g^{\prime} + f^{\prime} g^{\prime}_{*}\right)^{2} d\sigma dv dv_{*}. \nonumber
\een
We now move $\chi_{R}U^{\gamma/2}_{\delta}$ before $\mu^{1/2}$ out of $\mathbb{F}(\chi_{R}U^{\gamma/2}_{\delta}\mu^{1/2}, \chi_{R}U^{\gamma/2}_{\delta}g)$,
which leads to $\mathbb{F}(\mu^{1/2}, \chi_{R}U^{\gamma/2}_{\delta}g)$ with some correction terms. That is,
\ben \label{move-outside-asym} \mathbb{F}(\chi_{R}U^{\gamma/2}_{\delta}\mu^{1/2}, \chi_{R}U^{\gamma/2}_{\delta}g)  &=& \mathbb{A}^{2}(\chi_{R}U^{\gamma/2}_{\delta}\mu^{1/2}, \chi_{R}U^{\gamma/2}_{\delta}g)
\nonumber \\&=&\left( \mathbb{A}(\mu^{1/2}, \chi_{R}U^{\gamma/2}_{\delta}g)  - \mathbb{A}\big((1-\chi_{R}U^{\gamma/2}_{\delta})\mu^{1/2}, \chi_{R}U^{\gamma/2}_{\delta}g\big) \right)^{2}
\nonumber \\&\geq&\frac{1}{2} \mathbb{A}^{2}(\mu^{1/2}, \chi_{R}U^{\gamma/2}_{\delta}g)
-\mathbb{A}^{2}\big((1-\chi_{R}U^{\gamma/2}_{\delta})\mu^{1/2}, \chi_{R}U^{\gamma/2}_{\delta}g\big)
\nonumber \\&=&\frac{1}{2}\mathbb{F}(\mu^{1/2}, \chi_{R}U^{\gamma/2}_{\delta}g) -  \mathbb{F}((1-\chi_{R}U^{\gamma/2}_{\delta})\mu^{1/2}, \chi_{R}U^{\gamma/2}_{\delta}g).
\een
By symmetry, we have
\ben \label{sym-term}\int b^{\epsilon} \left(h^{\prime} h^{\prime}_{*}-
h h_{*}\right)^{2}
\left(f^{\prime}_{*}g^{\prime} + f^{\prime} g^{\prime}_{*}\right)^{2} d\sigma dv dv_{*} \leq 4\int b^{\epsilon} \left(h^{\prime} h^{\prime}_{*}-
h h_{*}\right)^{2}f^{2}_{*}g^{2} d\sigma dv dv_{*}.
 \een
Thanks  to \eqref{move-inside-sym}, \eqref{move-outside-asym} and \eqref{sym-term}, we get
\ben \label{separate-key-part} J^{\epsilon,\gamma,\eta}(g) &\gtrsim&
\frac{1}{4}\delta^{-\gamma}
\int b^{\epsilon} \mathrm{1}_{|v-v_{*}| \geq \eta}\mathbb{F}(\mu^{1/2}, \chi_{R}U^{\gamma/2}_{\delta}g) d\sigma dv dv_{*}
\\  && - \frac{1}{2}\delta^{-\gamma}\int b^{\epsilon} \mathbb{F}((1-\chi_{R}U^{\gamma/2}_{\delta})\mu^{1/2}, \chi_{R}U^{\gamma/2}_{\delta}g) d\sigma dv dv_{*}
\nonumber
\\  && - 4\delta^{-\gamma}\int b^{\epsilon} \left(h^{\prime} h^{\prime}_{*}-
h h_{*}\right)^{2} f^{2}_{*}g^{2} d\sigma dv dv_{*}
:= \frac{1}{4}J_{1} - \frac{1}{2}J_{2} -  4J_{3}. \nonumber
\een

\noindent{\it Step 2: Estimates of $J_i(i=1,2,3)$.} We will give the estimates term by term.

{\it \underline{Lower bound of $J_1$.}}
 We claim that for $\epsilon \leq 64^{-1}R^{-2}$ and some universal constant $C$,
\ben \label{estimate-J-1}  J_{1} \gtrsim \delta^{-\gamma}|g|^{2}_{L^{2}_{\gamma/2}}-C(\eta^{3}+\delta^{2}+R^{-2})|g|^{2}_{\epsilon,\gamma/2}.  \een
Thanks to \eqref{mxsgthm}, \eqref{firstegenvalueLA} and \eqref{order-2}, for any smooth function $F$, we have
\beno \langle  \mathcal{L}^{\epsilon,0,0}F, F\rangle \ge \lambda_1^\epsilon|(\mathbb{I}-\mathbb{P})F|^{2}_{L^{2}}\gtrsim |(\mathbb{I}-\mathbb{P})F|^{2}_{L^{2}}.\eeno
From which together with Lemma \ref{gamma-0-eat-lb}, for some universal constant $C$, we have
\ben \label{eta-pure-lower-bound}\langle  \mathcal{L}^{\epsilon,0,\eta}F, F\rangle =
\langle  \mathcal{L}^{\epsilon,0,0}F, F\rangle - \langle  \mathcal{L}^{\epsilon,0}_{\eta}F, F\rangle
\gtrsim |(\mathbb{I}-\mathbb{P})F|^{2}_{L^{2}} - C\eta^{3} |W^{\epsilon}(D)(\mathbb{I}-\mathbb{P})F|^{2}_{L^{2}}. \een
Applying \eqref{eta-pure-lower-bound} with $F=\chi_{R}U^{\gamma/2}_{\delta}g$,
and using $(a-b)^{2} \geq a^{2}/2-b^{2}$, we have
\ben \label{out-truncation-projection} J_{1}&=&\delta^{-\gamma}
\int b^{\epsilon} \mathrm{1}_{|v-v_{*}| \geq \eta}\mathbb{F}(\mu^{1/2}, \chi_{R}U^{\gamma/2}_{\delta}g) d\sigma dv dv_{*}
=4 \delta^{-\gamma}\langle  \mathcal{L}^{\epsilon,0,\eta}\chi_{R}U^{\gamma/2}_{\delta}g, \chi_{R}U^{\gamma/2}_{\delta}g\rangle
\\&\gtrsim& \delta^{-\gamma}|(\mathbb{I}-\mathbb{P})(\chi_{R}U^{\gamma/2}_{\delta}g)|^{2}_{L^{2}} -C\eta^{3} \delta^{-\gamma}|W^{\epsilon}(D)(\mathbb{I}-\mathbb{P})(\chi_{R}U^{\gamma/2}_{\delta}g)|^{2}_{L^{2}}
\nonumber  \\&\gtrsim& \frac{1}{2}\delta^{-\gamma}|\chi_{R}U^{\gamma/2}_{\delta}g|^{2}_{L^{2}} - \delta^{-\gamma}|\mathbb{P}(\chi_{R}U^{\gamma/2}_{\delta}g)|^{2}_{L^{2}} -C\eta^{3} \delta^{-\gamma}|W^{\epsilon}(D)(\mathbb{I}-\mathbb{P})(\chi_{R}U^{\gamma/2}_{\delta}g)|^{2}_{L^{2}}
\nonumber \\&\gtrsim& \frac{1}{4}\delta^{-\gamma}|U^{\gamma/2}_{\delta}g|^{2}_{L^{2}} - \frac{1}{2}
\delta^{-\gamma}|(1-\chi_{R})U^{\gamma/2}_{\delta}g|^{2}_{L^{2}}
-\delta^{-\gamma}|\mathbb{P}(\chi_{R}U^{\gamma/2}_{\delta}g)|^{2}_{L^{2}} \nonumber
\\&&-C\eta^{3} \delta^{-\gamma}|W^{\epsilon}(D)(\mathbb{I}-\mathbb{P})(\chi_{R}U^{\gamma/2}_{\delta}g)|^{2}_{L^{2}}
 := J_{1,1}- J_{1,2} - J_{1,3} - J_{1,4}. \nonumber
\een
\begin{itemize}
\item Since $\delta \leq 1$ and $\gamma \leq 0$, then $U^{\gamma/2}_{\delta} \geq W_{\gamma/2}$, which enables us to get the leading term
\ben \label{leading-term} J_{1,1} \gtrsim \delta^{-\gamma}|g|^{2}_{L^{2}_{\gamma/2}}.\een
\item
Thanks to the fact $\delta^{-\gamma}U^{\gamma}_{\delta} \leq W_{\gamma}$ and $1-\chi_{R}(v)=0$ when $|v|\leq R$,  we have
\ben \label{large-velocity-part-1} J_{1,2} &=& \frac{1}{2}
\delta^{-\gamma}|(1-\chi_{R})U^{\gamma/2}_{\delta}g|^{2}_{L^{2}}
\lesssim |(1-\chi_{R})W_{\gamma/2}g|^{2}_{L^{2}}
 \\&\lesssim&  |\mathrm{1}_{|v| \geq R} \phi(\epsilon^{1/2} \cdot) W_{\gamma/2}g|^{2}_{L^{2}}  +
 |(1- \phi(\epsilon^{1/2} \cdot))W_{\gamma/2}g|^{2}_{L^{2}} \nonumber
\\&\lesssim&  R^{-2}| \phi(\epsilon^{1/2} \cdot) W_{\gamma/2+1}g|^{2}_{L^{2}}  +
 \epsilon|(1- \phi(\epsilon^{1/2} \cdot)) \epsilon^{-1/2} W_{\gamma/2}g|^{2}_{L^{2}}
\lesssim (R^{-2}+\epsilon)|W^{\epsilon}W_{\gamma/2}g|^{2}_{L^{2}}, \nonumber\een
where we use \eqref{lower-bound-when-small-cf} and \eqref{lower-bound-when-large-cf} in the last inequality.
By the assumption $\epsilon \leq 64^{-1}R^{-2}$, we have
\ben \label{large-velocity-part-12} J_{1,2} \lesssim R^{-2}|W^{\epsilon}W_{\gamma/2}g|^{2}_{L^{2}}.\een
\item
We now estimate $J_{1,3}$. Recalling \eqref{DefProj} for the definition of $\mathbb{P}$
and  by the condition $\mathbb{P}g = 0$, we have
\beno  \mathbb{P}(\chi_{R}U^{\gamma/2}_{\delta}g) &=& \sum_{i=1}^{5} e_{i}\int e_{i}\chi_{R}U^{\gamma/2}_{\delta}g dv
= \sum_{i=1}^{5} e_{i}\int e_{i}(\chi_{R}U^{\gamma/2}_{\delta}-1)g dv .\eeno
Observing
\ben\label{lclfact1} 1-\chi_{R}U^{\gamma/2}_{\delta}\lesssim 1-\chi_R+\delta |v|\chi_{R}, \een
and thus $e_{i}(1-\chi_{R}U^{\gamma/2}_{\delta}) \lesssim (\delta+R^{-1}) \mu^{1/4}$, we have
$ \big|\int e_{i}(\chi_{R}U^{\gamma/2}_{\delta}-1)g dv\big| \lesssim (\delta+R^{-1}) |\mu^{1/8}g|_{L^{2}},$
which gives
\ben \label{large-velocity-part-13} J_{1,3} = \delta^{-\gamma}|\mathbb{P}(\chi_{R}U^{\gamma/2}_{\delta}g)|^{2}_{L^{2}} \lesssim (\delta^{2}+R^{-2}) |\mu^{1/8}g|^{2}_{L^{2}} \lesssim (\delta^{2}+R^{-2}) |g|^{2}_{L^{2}_{\gamma/2}}.  \een
\item Using Lemma \ref{operatorcommutator1} with  $\Phi=\delta^{-\gamma/2}\chi_{R}U^{\gamma/2}_{\delta}\in S^{\gamma/2}_{1,0}$ and $M=W^\epsilon\in S^1_{1,0}$, we have
\ben \label{large-velocity-part-14} J_{1,4} \le C\eta^{3}\delta^{-\gamma}|W^{\epsilon}(D)(\chi_{R}U^{\gamma/2}_{\delta}g)|^{2}_{L^{2}}
 \lesssim \eta^{3}|W^{\epsilon}(D)W_{\gamma/2}g|^{2}_{L^{2}}.  \een \end{itemize}
Patching together the estimates \eqref{leading-term},\eqref{large-velocity-part-12}, \eqref{large-velocity-part-13}, \eqref{large-velocity-part-14}, we get \eqref{estimate-J-1}.

{\it \underline{Upper bound of $J_2$.}}
For simplicity, setting $f_{\gamma}= (1-\chi_{R}U^{\gamma/2}_{\delta})\mu^{1/2}, g_{\gamma}=\chi_{R}U^{\gamma/2}_{\delta}g$, we get
\ben \label{J2-part}
J_{2} &=&
\delta^{-\gamma}\int b^{\epsilon} \mathbb{F}((1-\chi_{R}U^{\gamma/2}_{\delta})\mu^{1/2}, \chi_{R}U^{\gamma/2}_{\delta}g) d\sigma dv dv_{*}
=\delta^{-\gamma}\int b^{\epsilon} \mathbb{F}(f_{\gamma}, g_{\gamma}) d\sigma dv dv_{*}
\nonumber\\&\lesssim& \delta^{-\gamma}\int b^{\epsilon} (f_{\gamma}^{2})_{*}(g_{\gamma}^{\prime}-g_{\gamma})^{2} d\sigma dv dv_{*} + \delta^{-\gamma}\int b^{\epsilon} (g_{\gamma}^{2})_{*}(f_{\gamma}^{\prime}-f_{\gamma})^{2} d\sigma dv dv_{*}
:= J_{2,1} + J_{2,2}.
\een
Thanks to \eqref{lclfact1}, we have
\ben \label{f-gamma-upper}
(f_{\gamma}^{2})_{*} = ((1-\chi_{R}U^{\gamma/2}_{\delta})\mu^{1/2})^{2}_{*} \lesssim (\delta^{2}+R^{-2}) \mu^{1/2}_{*}.
\een
Plugging \eqref{f-gamma-upper} into $J_{2,1},$ we have
\ben \label{J21}
J_{2,1}&\lesssim& (\delta^{2}+R^{-2})\delta^{-\gamma}\int b^{\epsilon}  \mu^{1/2}_{*} (g_{\gamma}^{\prime}-g_{\gamma})^{2} d\sigma dv dv_{*}
\\&=& (\delta^{2}+R^{-2})\delta^{-\gamma} \mathcal{N}^{\epsilon,0,0}(\mu^{1/4},g_{\gamma})
\lesssim (\delta^{2}+R^{-2})\delta^{-\gamma}|\chi_{R}U^{\gamma/2}_{\delta}g|^{2}_{\epsilon,\gamma/2} \lesssim (\delta^{2}+R^{-2})|g|^{2}_{\epsilon,\gamma/2},\nonumber
\een
where we use \eqref{anisotropic-regularity-general-g-up-bound}
and  Lemma \ref{operatorcommutator1} with  $\Phi=\delta^{-\gamma/2}\chi_{R}U^{\gamma/2}_{\delta}\in S^{\gamma/2}_{1,0}$ and $M=W^\epsilon\in S^1_{1,0}$.

By Taylor expansion up to order 1,
$ f_{\gamma}^{\prime}-f_{\gamma}  = \int_{0}^{1} (\nabla f_{\gamma})(v(\kappa))\cdot (v^{\prime}-v)d\kappa.$ From which together with
\beno |\nabla f_{\gamma}| &=&| \nabla ((1-\chi_{R}U^{\gamma/2}_{\delta})\mu^{1/2})|
= |(1-\chi_{R}U^{\gamma/2}_{\delta}) \nabla \mu^{1/2}
- U^{\gamma/2}_{\delta}\mu^{1/2} \nabla \chi_{R} - \chi_{R} \mu^{1/2} \nabla U^{\gamma/2}_{\delta} |
\\&\lesssim& \mu^{1/8}(\delta + R^{-1}),
\eeno
we get
\ben \label{f-gamma-by-taylor}
|f_{\gamma}^{\prime}-f_{\gamma}|^{2}  \lesssim
(\delta^{2}+R^{-2})\theta^{2} \int_{0}^{1} \mu^{1/4}(v(\kappa))|(v(\kappa)-v_{*})|^{2}d\kappa.\een
Since $R \leq 8^{-1}\epsilon^{-1/2}$, by the change \eqref{change-exact-formula}-\eqref{change-Jacobean-bound}, and the fact \eqref{lower-bound-when-small-cf},
we have
\ben \label{J22}
J_{2,2}&\lesssim& (\delta^{2}+R^{-2})\delta^{-\gamma}\int b^{\epsilon} \theta^{2} (\chi_{R}U^{\gamma/2}_{\delta}g)^{2}_{*} \mu^{1/4}(v(\kappa))|v(\kappa)-v_{*}|^{2} d\sigma dv(\kappa) dv_{*} d\kappa
\nonumber\\&\lesssim& (\delta^{2}+R^{-2})|\chi_{R}W_{\gamma/2+1}g|^{2}_{L^{2}}
\lesssim (\delta^{2}+R^{-2})|W_{\gamma/2}W^{\epsilon}g|^{2}_{L^{2}}
\lesssim (\delta^{2}+R^{-2})|g|^{2}_{\epsilon,\gamma/2}.
\een
Plugging the estimates \eqref{J21} and \eqref{J22} into \eqref{J2-part}, we get
\ben \label{estimate-J-2}
J_{2} \lesssim (\delta^{2}+R^{-2})|g|^{2}_{\epsilon,\gamma/2}.
\een

{\it \underline{Upper bound of $J_3$.}}
By Lemma \ref{difference-term-complication}, we have
$
\delta^{-\gamma}\left(h^{\prime} h^{\prime}_{*}-
h h_{*}\right)^{2}
\lesssim(\delta^{2}+R^{-2})\theta^{2}\langle v \rangle^{\gamma+2}\langle v_{*} \rangle^{2}\mathrm{1}_{|v|\leq 4R}. $
Since $8R \leq \epsilon^{-1/2}$, by \eqref{lower-bound-when-small-cf},  we have,
\ben \label{estimate-J-3}
J_{3} &=& \delta^{-\gamma}\int b^{\epsilon} \left(h^{\prime} h^{\prime}_{*}-
h h_{*}\right)^{2} \mu_{*}g^{2} d\sigma dv dv_{*}
\lesssim (\delta^{2}+R^{-2})
\int b^{\epsilon} \theta^{2}\langle v_{*} \rangle^{2}\langle v \rangle^{\gamma+2}\mu_{*}\mathrm{1}_{|v|\leq 4R} g^{2} d\sigma dv dv_{*}
\nonumber\\&\lesssim& (\delta^{2}+R^{-2})|\mathrm{1}_{|\cdot|\leq 4R} W_{\gamma/2+1}g|^{2}_{L^{2}} \lesssim (\delta^{2}+R^{-2})|W_{\gamma/2}W^{\epsilon}g|^{2}_{L^{2}}
\leq (\delta^{2}+R^{-2})|g|^{2}_{\epsilon,\gamma/2}.
\een

{\it Step 3: Case $-2< \gamma <0$.}
Plugging the estimates of $J_{1}$ in \eqref{estimate-J-1},  $J_{2}$ in \eqref{estimate-J-2}, $J_{3}$ in \eqref{estimate-J-3} into \eqref{separate-key-part},  for $\epsilon \leq 64^{-1}R^{-2}, 0<\delta<1$, we get
\beno   J^{\epsilon,\gamma,\eta}(g) \gtrsim \delta^{-\gamma}|g|^{2}_{L^{2}_{\gamma/2}}-C(\eta^{3}+\delta^{2}+R^{-2})|g|^{2}_{\epsilon,\gamma/2}. \eeno
Choosing $R = \delta^{-1}, \eta = \delta^{2/3}$, for some universal constants $C_{1}, C_{2}$, we have
\ben \label{key-estimate-by-He}   J^{\epsilon,\gamma,\eta}(g) \geq C_{1}\delta^{-\gamma}|g|^{2}_{L^{2}_{\gamma/2}}-C_{2}\delta^{2}|g|^{2}_{\epsilon,\gamma/2}. \een
By the coercivity estimate in Theorem \ref{strong-coercivity}, for some universal constants $C_{3}, C_{4}$, we have
\ben \label{known-estimate}    J^{\epsilon,\gamma,\eta}(g) \geq C_{3}|g|^{2}_{\epsilon,\gamma/2}- C_{4}|g|^{2}_{L^{2}_{\gamma/2}}.\een
Multiplying \eqref{known-estimate} by $C_{5}\delta^{2}$ and adding the resulting inequality to \eqref{key-estimate-by-He}, we get
\ben \label{combination}  (1+C_{5}\delta^{2}) J^{\epsilon,\gamma,\eta}(g) \geq (
C_{1}\delta^{-\gamma}-C_{4}C_{5}\delta^{2})|g|^{2}_{L^{2}_{\gamma/2}}+(C_{3}C_{5}-C_{2})\delta^{2}|g|^{2}_{\epsilon,\gamma/2}.  \een
First take $C_{5}$ large enough such that $C_{3}C_{5}-C_{2} \geq C_{2}$, for example let $C_{5} = 2C_{2}/C_{3}$.
Then take $\delta$ small enough such that $C_{1}\delta^{-\gamma}-C_{4}C_{5}\delta^{2} \geq 0$, for example, let $\delta=\left(\frac{C_{1}}{C_{4}C_{5}}\right)^{1/(2+\gamma)} = \left(\frac{C_{1}C_{3}}{2C_{4}C_{2}}\right)^{1/(2+\gamma)}$. Then we get
\ben \label{conclusion1-gamma-minus2}   J^{\epsilon,\gamma,\eta}(g) \geq C_{2}\delta^{2}|g|^{2}_{\epsilon,\gamma/2} = C_{2}\left(\frac{C_{1}C_{3}}{2C_{4}C_{2}}\right)^{2/(2+\gamma)}|g|^{2}_{\epsilon,\gamma/2}, \een
for any $0 \leq \epsilon \leq 64^{-1}R^{-2}  = 64^{-1}\left(\frac{C_{1}C_{3}}{2C_{4}C_{2}}\right)^{2/(2+\gamma)}$ and $0 \leq \eta \leq \delta^{2/3} = \left(\frac{C_{1}C_{3}}{2C_{4}C_{2}}\right)^{2/(6+3\gamma)}$.

{\it Step 4: Case $-3\leq \gamma \leq -2$.}
In this case, we take $-2<\alpha,\beta<0$ such that $\alpha+\beta=\gamma$. Replacing $b^{\epsilon}$ by  $b^{\epsilon}|v-v_{*}|^{\alpha}$ and $\gamma$ by $\beta$, similar to \eqref{separate-key-part}, we get
\ben \label{separate-gamma-very-small} J^{\epsilon,\gamma,\eta}(g) &\gtrsim&
\frac{1}{4}\delta^{-\beta}
\int b^{\epsilon}|v-v_{*}|^{\alpha}\mathrm{1}_{|v-v_{*}| \geq \eta} \mathbb{F}(\mu^{1/2}, \chi_{R}U^{\beta/2}_{\delta}g) d\sigma dv dv_{*}
\\  && - \frac{1}{2}\delta^{-\beta}\int b^{\epsilon}|v-v_{*}|^{\alpha} \mathbb{F}((1-\chi_{R}U^{\beta/2}_{\delta})\mu^{1/2}, \chi_{R}U^{\beta/2}_{\delta}g) d\sigma dv dv_{*}
\nonumber\\  && -  4\delta^{-\beta}\int b^{\epsilon}|v-v_{*}|^{\alpha} \left(h^{\prime} h^{\prime}_{*}-
h h_{*}\right)^{2} \mu_{*}g^{2} d\sigma dv dv_{*}
:= \frac{1}{4} J^{\alpha,\beta}_{1}- \frac{1}{2} J^{\alpha,\beta}_{2} - 4 J^{\alpha,\beta}_{3},
\nonumber
\een
where $h :=\chi_{R} U^{\beta/2}_{\delta}$.

 {\it \underline{ Lower bound of $J^{\alpha,\beta}_{1}$.}}
Since $-2<\alpha<0$, we can use previous estimate \eqref{conclusion1-gamma-minus2} to get,
\beno J^{\alpha,\beta}_{1} = \delta^{-\beta} J^{\epsilon,\alpha,\eta}(\chi_{R}U^{\beta/2}_{\delta}g) \gtrsim \delta^{-\beta}
|W_{\alpha/2}(\mathbb{I}-\mathbb{P})(\chi_{R}U^{\beta/2}_{\delta}g)|^{2}_{L^{2}}.
\eeno
Using $(a-b)^{2} \geq a^{2}/2 - b^{2}$, we get
\ben \label{out-truncation-projection-2}   J^{\alpha,\beta}_{1}
&\gtrsim& \frac{1}{4}\delta^{-\beta}|W_{\alpha/2}U^{\beta/2}_{\delta}g|^{2}_{L^{2}} -
\frac{1}{2}\delta^{-\beta}|W_{\alpha/2}(1-\chi_{R})U^{\beta/2}_{\delta}g|^{2}_{L^{2}}
-\delta^{-\beta}|W_{\alpha/2}\mathbb{P}(\chi_{R}U^{\beta/2}_{\delta}g)|^{2}_{L^{2}}
 \nonumber \\&:=& J^{\alpha,\beta}_{1,1}- J^{\alpha,\beta}_{1,2} - J^{\alpha,\beta}_{1,3}.
\een
Thanks to $U_{\delta} \leq W$, one has $U^{\beta/2}_{\delta} \geq W_{\beta/2}$, and thus
\ben J^{\alpha,\beta}_{1,1} \label{J-al-be-11}
\gtrsim \delta^{-\beta}|W_{\alpha/2}W_{\beta/2}g|^{2}_{L^{2}} =  \delta^{-\beta}|g|^{2}_{L^{2}_{\gamma/2}}.
\een
Thanks to $\delta^{-\beta}U^{\beta}_{\delta} \leq W_{\beta}$, similar to \eqref{large-velocity-part-1}  and \eqref{large-velocity-part-12}, we have
\ben \label{J-al-be-12}
J^{\alpha,\beta}_{1,2} \lesssim |W_{\alpha/2}(1-\chi_{R})W_{\beta/2}g|^{2}_{L^{2}} \lesssim R^{-2}|W^{\epsilon}W_{\gamma/2}g|^{2}_{L^{2}}.\een
Similar to \eqref{large-velocity-part-13}, we get
\ben \label{J-al-be-13}
J^{\alpha,\beta}_{1,3} = \delta^{-\beta}|W_{\alpha/2}\mathbb{P}(\chi_{R}U^{\beta/2}_{\delta}g)|^{2}_{L^{2}}  \lesssim (\delta^{2}+R^{-2}) |\mu^{1/8}g|^{2}_{L^{2}} \lesssim (\delta^{2}+R^{-2}) |g|^{2}_{L^{2}_{\gamma/2}}.  \een
Plugging \eqref{J-al-be-11}, \eqref{J-al-be-12}, \eqref{J-al-be-13} into \eqref{out-truncation-projection-2},
we get
\ben \label{J-al-be-1}  J^{\alpha,\beta}_{1} \gtrsim \delta^{-\beta}|g|^{2}_{L^{2}_{\gamma/2}}-C(\delta^{2}+R^{-2})|g|^{2}_{\epsilon,\gamma/2}. \een

 {\it \underline{ Upper bound of $J^{\alpha,\beta}_{2}$.}} Now we analyze
\beno J^{\alpha,\beta}_{2}= \delta^{-\beta}\int b^{\epsilon}|v-v_{*}|^{\alpha} \mathbb{F}((1-\chi_{R}U^{\beta/2}_{\delta})\mu^{1/2}, \chi_{R}U^{\beta/2}_{\delta}g) d\sigma dv dv_{*}. \eeno
For simplicity, set $f_{\beta}= (1-\chi_{R}U^{\beta/2}_{\delta})\mu^{1/2}, g_{\beta}=\chi_{R}U^{\beta/2}_{\delta}g$, we get
\ben \label{J-al-be-2-into-2}
J^{\alpha,\beta}_{2} &\lesssim& \delta^{-\beta}\int b^{\epsilon} |v-v_{*}|^{\alpha} (f_{\beta}^{2})_{*}(g_{\beta}^{\prime}-g_{\beta})^{2} d\sigma dv dv_{*} + \delta^{-\beta}\int b^{\epsilon} |v-v_{*}|^{\alpha} (g_{\beta}^{2})_{*}(f_{\beta}^{\prime}-f_{\beta})^{2} d\sigma dv dv_{*}
\nonumber \\&:=& J^{\alpha,\beta}_{2,1} + J^{\alpha,\beta}_{2,2}.
\een
Similar to \eqref{f-gamma-upper}, we get
$ f_{\beta}^{2}=((1-\chi_{R}U^{\beta/2}_{\delta})\mu^{1/2})^{2} \lesssim (\delta^{2}+R^{-2}) \mu^{1/2}. $
From which together with $\delta^{-\beta}U^{\beta/2}_{\delta} \leq W_{\beta/2}$, we get
\ben \label{J-al-be-2-into-2-1}
J^{\alpha,\beta}_{2,1}&\lesssim& (\delta^{2}+R^{-2})\delta^{-\beta}\int b^{\epsilon}|v-v_{*}|^{\alpha}  \mu^{1/2}_{*} (g_{\beta}^{\prime}-g_{\beta})^{2} d\sigma dv dv_{*}
=(\delta^{2}+R^{-2})\delta^{-\beta}\mathcal{N}^{\epsilon,\alpha,0}(\mu^{1/4},g_{\beta})
\nonumber \\&\lesssim& (\delta^{2}+R^{-2})\delta^{-\beta}|\chi_{R}U^{\beta/2}_{\delta}g|^{2}_{\epsilon,\alpha/2} \lesssim (\delta^{2}+R^{-2})|g|^{2}_{\epsilon,\gamma/2},
\een
where we use Lemma \ref{operatorcommutator1} with $\Phi=W_{\alpha/2}\delta^{-\beta/2}U^{\beta/2}_{\delta}\chi_{R}$ and $M=W^\epsilon$.
Similar to \eqref{f-gamma-by-taylor}, we have
\beno  |f_{\beta}^{\prime}-f_{\beta}|^{2}  \lesssim
(\delta^{2}+R^{-2})\theta^{2} \int_{0}^{1} \mu^{1/4}(v(\kappa))|v(\kappa)-v_{*}|^{2}d\kappa.\eeno
Thanks to $|v-v_{*}| \sim |v(\kappa)-v_{*}|$, since $8 R \leq \epsilon^{-1/2}$, by the change \eqref{change-exact-formula}-\eqref{change-Jacobean-bound}, we have
\ben \label{J-al-be-2-into-2-2}
J^{\alpha,\beta}_{2,2}&\lesssim& (\delta^{2}+R^{-2})\delta^{-\beta}\int b^{\epsilon} \theta^{2} (\chi_{R}U^{\beta/2}_{\delta}g)^{2}_{*} \mu^{1/4}(v(\kappa))|v(\kappa)-v_{*}|^{2+\alpha} d\sigma dv(\kappa) dv_{*} d\kappa.
\nonumber \\&\lesssim& (\delta^{2}+R^{-2})|\chi_{R}W_{\gamma/2+1}g|^{2}_{L^{2}} \lesssim (\delta^{2}+R^{-2})|g|^{2}_{\epsilon,\gamma/2}.
\een
Plugging \eqref{J-al-be-2-into-2-1} and \eqref{J-al-be-2-into-2-2} into \eqref{J-al-be-2-into-2}, we get
\ben \label{J-al-be-2-into-2-final}
J^{\alpha,\beta}_{2} \lesssim (\delta^{2}+R^{-2})|g|^{2}_{\epsilon,\gamma/2}.
\een

 {\it \underline{ Upper bound of $J^{\alpha,\beta}_{3}$.}}
Recall $J^{\alpha,\beta}_{3} = \delta^{-\beta}\int b^{\epsilon}|v-v_{*}|^{\alpha} \left(h^{\prime} h^{\prime}_{*}-
h h_{*}\right)^{2} \mu_{*}g^{2} d\sigma dv dv_{*} $.
By Lemma \ref{difference-term-complication}, we have
\beno\delta^{-\beta}\left(h^{\prime} h^{\prime}_{*}-
h h_{*}\right)^{2} = X(\beta,R,\delta)
 \lesssim (\delta^{2}+R^{-2})\theta^{2}\langle v_{*} \rangle^{2}\langle v \rangle^{\beta+2}\mathrm{1}_{|v|\leq 4R}. \eeno
Thanks to $\int |v-v_{*}|^{\alpha}\langle v_{*} \rangle^{2} \mu_{*} d v_{*} \lesssim \langle v \rangle^{\alpha}$, since $8 R \leq \epsilon^{-1/2}$,
we get
\ben \label{J-al-be-3-up}
J^{\alpha,\beta}_{3} &\lesssim& (\delta^{2}+R^{-2}) \int b^{\epsilon}|v-v_{*}|^{\alpha} \theta^{2}\langle v_{*} \rangle^{2}\langle v \rangle^{\beta+2}\mu_{*} \mathrm{1}_{|v|\leq 4R} g^{2} d\sigma dv dv_{*}
\nonumber \\&\lesssim& (\delta^{2}+R^{-2})|\mathrm{1}_{|\cdot|\leq 4R} W_{\gamma/2+1}g|^{2}_{L^{2}} \lesssim (\delta^{2}+R^{-2})|g|^{2}_{\epsilon,\gamma/2}.\een

Plugging the estimates of $J^{\alpha,\beta}_{1}$ in \eqref{J-al-be-1},
$J^{\alpha,\beta}_{2}$ in \eqref{J-al-be-2-into-2-final},
$J^{\alpha,\beta}_{3}$ in \eqref{J-al-be-3-up} into \eqref{separate-gamma-very-small},
we get
\beno   J^{\epsilon,\gamma,\eta}(g) \gtrsim \delta^{-\beta}|g|^{2}_{L^{2}_{\gamma/2}}-C(\delta^{2}+R^{-2})|g|^{2}_{\epsilon,\gamma/2}. \eeno
Choosing $R = \delta^{-1}$, for some universal constants $C_{6},C_{7}$,
we get
\ben \label{key-estimate-by-He-2}   J^{\epsilon,\gamma,\eta}(g) \geq C_{6}\delta^{-\beta}|g|^{2}_{L^{2}_{\gamma/2}}-C_{7}\delta^{2}|g|^{2}_{\epsilon,\gamma/2}. \een
Together with coercivity estimate \eqref{known-estimate}, thanks to $-2< \beta <0$,
by a similar argument as in {\it Step 3}, similar to \eqref{conclusion1-gamma-minus2}, we get for $-3\leq \gamma \leq-2$,
\ben \label{conclusion1-gamma-minus3}   J^{\epsilon,\gamma,\eta}(g) \geq C_{7}\left(\frac{C_{6}C_{3}}{2C_{4}C_{7}}\right)^{2/(2+\beta)}|g|^{2}_{\epsilon,\gamma/2},\een
for any $0 \leq \epsilon \leq \min\bigg\{  64^{-1}\left(\frac{C_{1}C_{3}}{2C_{4}C_{2}}\right)^{2/(2+\alpha)}, 64^{-1}\left(\frac{C_{6}C_{3}}{2C_{4}C_{7}}\right)^{2/(2+\beta)}\bigg\}$  and $0 \leq \eta \leq  \left(\frac{C_{1}C_{3}}{2C_{4}C_{2}}\right)^{2/(6+3\alpha)}$.

One can trace the proof to settle down two universal constants $\epsilon_{0}, \eta_{0}>0$ such that \eqref{gap-estimate} holds true for any $0 \leq \epsilon \leq \epsilon_{0}$ and $0 \leq \eta \leq \eta_{0}$. One should not worry that the above constants could blow up if $\alpha, \beta \rightarrow -2$. Indeed, in {\it{Step 3}}, we can deal with $-7/4 \leq \gamma < 0$. Then in {\it{Step 4}}, we deal with $-3 \leq \gamma < -7/4$, where we can choose $ -7/4 \leq \alpha, \beta < 0$ such that $\alpha+\beta=\gamma$. In this way, all the constants are universally bounded.
\end{proof}

\setcounter{equation}{0}

\section{Upper bound estimate} \label{Upper-Bound-Estimate}
In this section, we will provide various upper bounds  on the nonlinear operator $\Gamma^{\epsilon}$ and linear operator $\mathcal{L}^{\epsilon }$. We recall the definition of $\Gamma^{\epsilon,\gamma,\eta}(g,h)$ from \eqref{Gamma-ep-eta-ga}, that is,
\beno
\Gamma^{\epsilon,\gamma,\eta}(g,h) &=& \mu^{-1/2} Q^{\epsilon,\gamma,\eta}(\mu^{1/2}g,\mu^{1/2}h)
= \int B^{\epsilon,\gamma,\eta}(v-v_*,\sigma) \mu^{1/2}_{*} \big(g^{\prime}_{*} h^{\prime}-g_{*} h\big) d\sigma dv_*
\\&=& \int B^{\epsilon,\gamma,\eta}(v-v_*,\sigma) \big((\mu^{1/2}g)^{\prime}_{*} h^{\prime}-(\mu^{1/2}g)_{*} h\big)d\sigma dv_*
\\&&+ \int B^{\epsilon,\gamma,\eta}(v-v_*,\sigma) \big(\mu^{1/2}_{*}-(\mu^{1/2})^{\prime}_{*}\big)g^{\prime}_{*} h^{\prime} d\sigma dv_*
= Q^{\epsilon,\gamma,\eta}(\mu^{1/2}g,h) + I^{\epsilon,\gamma,\eta}(g,h),
\eeno
where for notational brevity,  we set
\ben \label{I-ep-ga-geq-eta}
I^{\epsilon,\gamma,\eta}(g,h) :=\int B^{\epsilon,\gamma,\eta}(v-v_*,\sigma) \big(\mu^{1/2}_{*}-(\mu^{1/2})^{\prime}_{*}\big)g^{\prime}_{*} h^{\prime} d\sigma dv_{*} .
\een

We recall from \eqref{etalower} the operators $\Gamma^{\epsilon,\gamma}_{\eta}, Q^{\epsilon,\gamma}_{\eta}, \mathcal{L}^{\epsilon,\gamma}_{\eta}$ containing subscript $\eta$.
Similar to \eqref{L1-ep-eta-ga} and \eqref{L2-ep-eta-ga}, we can define  $\mathcal{L}^{\epsilon,\gamma}_{1,\eta}, \mathcal{L}^{\epsilon,\gamma}_{2,\eta}$ through $\Gamma^{\epsilon,\gamma}_{\eta}$.
We define $I^{\epsilon,\gamma}_{\eta}(g,h)$ using kernel $B^{\epsilon,\gamma}_{\eta}(v-v_*,\sigma)$ in \eqref{etalower} in the way as in \eqref{I-ep-ga-geq-eta}. When $\eta = 0$, we drop the superscript $\eta$ for brevity.
That is, $Q^{\epsilon,\gamma} := Q^{\epsilon,\gamma,0}, \Gamma^{\epsilon,\gamma}:= \Gamma^{\epsilon,\gamma,0}, I^{\epsilon,\gamma} := I^{\epsilon,\gamma,0}$. With these notations in hand, we have
\ben
\label{Gamma-ep-ga-into-IQ}
\Gamma^{\epsilon,\gamma}(g,h) &=& Q^{\epsilon,\gamma}(\mu^{1/2}g,h) + I^{\epsilon,\gamma}(g,h),
\\
\label{Gamma-ep-ga-geq-eta-into-IQ}
\Gamma^{\epsilon,\gamma,\eta}(g,h) &=& Q^{\epsilon,\gamma,\eta}(\mu^{1/2}g,h) + I^{\epsilon,\gamma,\eta}(g,h),
\\
\label{Gamma-ep-ga-leq-eta-into-IQ}
\Gamma^{\epsilon,\gamma}_{\eta}(g,h) &=& Q^{\epsilon,\gamma}_{\eta}(\mu^{1/2}g,h)+ I^{\epsilon,\gamma}_{\eta}(g,h),
\\
\label{Q-ep-ga-sep-eta}
Q^{\epsilon,\gamma}(g,h) &=& Q^{\epsilon,\gamma,\eta}(g,h) + Q^{\epsilon,\gamma}_{\eta}(g,h),
\\
\label{Gamma-ep-ga-sep-eta}
\Gamma^{\epsilon,\gamma}(g,h) &=& \Gamma^{\epsilon,\gamma,\eta}(g,h) + \Gamma^{\epsilon,\gamma}_{\eta}(g,h),
\\
\label{I-ep-ga-sep-eta}
I^{\epsilon,\gamma}(g,h) &=& I^{\epsilon,\gamma,\eta}(g,h) + I^{\epsilon,\gamma}_{\eta}(g,h),
\\
\label{L-ep-ga-sep-eta-l1-l2}
\mathcal{L}^{\epsilon,\gamma}_{\eta}g &=& \mathcal{L}^{\epsilon,\gamma}_{1,\eta}g + \mathcal{L}^{\epsilon,\gamma}_{2,\eta}g.
\een

Throughout this section, we assume $-3 \leq \gamma \leq 0$.
Our results on the upper bounds can be summarized in the following Table \ref{ResultSummary}.
\begin{table}[!htbp]
\centering
\caption{Results Summary} \label{ResultSummary}
\begin{tabular}{cc}
\hline
Functionals &Proposition or Theorem \\
\hline
$\langle Q^{\epsilon,\gamma,\eta}(g,h), f\rangle$ & Proposition \ref{ubqepsilonnonsingular}\\
$\langle I^{\epsilon,\gamma,\eta}(g,h), f\rangle$ & Proposition \ref{upforI-ep-ga-et}\\
$\langle Q^{\epsilon,\gamma}_{\eta}(g,h), f\rangle$ & Proposition \ref{ubqepsilon-singular}\\
$\langle I^{\epsilon,\gamma}_{\eta}(g,h), f\rangle$ & Proposition \ref{I-less-eta-upper-bound}\\
$\langle -\mathcal{L}^{\epsilon,\gamma}_{1,\eta}h + \Gamma^{\epsilon,\gamma}_{\eta}(f,h), h\rangle$ &Proposition \ref{less-eta-part-l1-Gamma}\\
$ \langle -\mathcal{L}^{\epsilon,\gamma}_{2,\eta}f , h\rangle$ & Proposition \ref{less-eta-part-l2}\\
$\langle \Gamma^{\epsilon,\gamma,\eta}(g,h), f\rangle$ & Theorem \ref{upGammagh-geq-eta}\\
$\langle Q^{\epsilon,\gamma}(g,h), f\rangle$ & Theorem \ref{Q-full-up-bound}\\
$\langle I^{\epsilon,\gamma}(g,h), f\rangle$ & Theorem \ref{upforI-total}\\
$\langle \Gamma^{\epsilon,\gamma}(g,h), f\rangle$ & Theorem \ref{Gamma-full-up-bound}\\
$\langle  \Gamma^{\epsilon,\gamma}_{\eta}(f,h)-\mathcal{L}^{\epsilon,\gamma}_{\eta}h, h\rangle$ & Theorem \ref{small-part-L+gamma}\\
\hline
\end{tabular}
\end{table}

It is easy to see that $\langle Q^{\epsilon,\gamma,\eta}(g,h), f\rangle$ and $\langle I^{\epsilon,\gamma,\eta}(g,h), f\rangle$ involve the
regular region $|v-v_{*}|\geq  \eta$, while $\langle Q^{\epsilon,\gamma}_{\eta}(g,h), f\rangle, \langle I^{\epsilon,\gamma}_{\eta}(g,h), f\rangle, \langle -\mathcal{L}^{\epsilon,\gamma}_{1,\eta}h + \Gamma^{\epsilon,\gamma}_{\eta}(f,h), h\rangle$ and $\langle -\mathcal{L}^{\epsilon,\gamma}_{2,\eta}f , h\rangle$  focus on the singular region $|v-v_{*}| \leq  \eta$. We provide two types of estimates on these functionals because we will meet two cases for the nonlinear term $\Gamma^\epsilon$ when the standard energy method is applied.  These two cases can be clarified as follows: $\langle \Gamma^\epsilon(f,\pa^\alpha f), \pa^\alpha f\rangle$ and $\langle \Gamma^\epsilon(\pa^{\alpha_1}f,\pa^{\alpha_2} f), \pa^\alpha f\rangle$, where $\alpha_1+\alpha_2=\alpha$ and $|\alpha_2| < |\alpha|$.
\begin{itemize}
	\item The first case corresponds to the highest order estimate of the solution. As we explained in section 1.4.3, the linear-quasilinear method will be employed. Technically we need to separate the integration domain into two regions: singular region and regular region. In this situation, all the upper bounds will depend on the parameter $\eta$.
	\item For the second case, since $|\alpha_2|<|\alpha|$, we have one more derivative freedom on the function $\pa^{\alpha_2}f$.
In this situation, all the upper bounds are independent of the parameter $\eta$ and allow more regularity.
\end{itemize}

\subsection{Upper bounds of $\langle Q^{\epsilon,\gamma,\eta}(g,h), f\rangle$ and $\langle I^{\epsilon,\gamma,\eta}(g,h), f\rangle$ }
 Thanks to \eqref{Gamma-ep-ga-geq-eta-into-IQ}, we have
\ben \label{Gamma-ep-ga-geq-eta-into-IQ-inner}
\langle \Gamma^{\epsilon,\gamma,\eta}(g,h), f\rangle =  \langle Q^{\epsilon,\gamma,\eta}(\mu^{1/2}g,h), f\rangle +
\langle I^{\epsilon,\gamma,\eta}(g,h), f\rangle.
\een
It suffices to consider $\langle Q^{\epsilon,\gamma,\eta}(g,h), f\rangle$ and $\langle I^{\epsilon,\gamma,\eta}(g,h), f\rangle$ individually.

\subsubsection{Upper bound of $\langle Q^{\epsilon,\gamma,\eta}(g,h), f\rangle$}
We begin with a lemma.
\begin{lem}\label{crosstermsimilar} Let $0< \eta \leq 1$,
	$
	\mathcal{Y}^{\epsilon,\gamma}(h,f) := \int b^{\epsilon}(\frac{u}{|u|}\cdot\sigma)|u|^\gamma\mathrm{1}_{|u|\geq \eta} h(u)[f(u^{+}) - f(|u|\frac{u^{+}}{|u^{+}|})] d\sigma du,
	$
	then
	\beno
	|\mathcal{Y}^{\epsilon,\gamma}(h,f)| \lesssim \eta^{\gamma-3} (|W^{\epsilon}W_{\gamma/2}h|_{L^{2}}+|W^{\epsilon}(D)W_{\gamma/2}h|_{L^{2}})
	(|W^{\epsilon}W_{\gamma/2}f|_{L^{2}}+|W^{\epsilon}(D)W_{\gamma/2}f|_{L^{2}}).
	\eeno
\end{lem}
\begin{proof}We divide the proof into two steps.
	
{\it Step 1: without the term $|u|^\gamma\mathrm{1}_{|u|\geq \eta}$.} For ease of notation, we denote $\mathcal{X}(h,f):=\int b^{\epsilon}(\frac{u}{|u|}\cdot\sigma)  h(u)[f(u^{+}) - f(|u|\frac{u^{+}}{|u^{+}|})] d\sigma du $. First applying dyadic decomposition in the phase space, we have
	\beno
	\mathcal{X}(h,f) = \sum_{k=-1}^{\infty}\int b^{\epsilon}(\frac{u}{|u|}\cdot\sigma) (\tilde{\varphi}_{k}h)(u) [(\varphi_{k}f)(u^{+})- (\varphi_{k}f)(|u|\frac{u^{+}}{|u^{+}|})] d\sigma du
	:= \sum_{k=-1}^{\infty} \mathcal{X}_{k}.
	\eeno
	where $\tilde{\varphi}_{k} = \sum_{|l-k|\leq 3} \varphi_{l}$.
	We split the proof into two cases:  $2^{k}\geq 1/\epsilon$ and $2^{k}\leq 1/\epsilon$.

	\underline{\it Case 1: $2^{k}\geq 1/\epsilon$. } We first have
	\beno
	|\mathcal{X}_{k}| \leq \bigg(\int b^{\epsilon}(\frac{u}{|u|}\cdot\sigma) |(\tilde{\varphi}_{k}h)(u)|^{2}  d\sigma du\bigg)^{\frac{1}{2}}
	  \bigg(\int b^{\epsilon}(\frac{u}{|u|}\cdot\sigma)  (|(\varphi_{k}f)(u^{+})|^{2} + |(\varphi_{k}f)(|u|\frac{u^{+}}{|u^{+}|})|^{2}) d\sigma du\bigg)^{\frac{1}{2}}.
	\eeno
	By the changes  $u \rightarrow u^{+}$ and $u \rightarrow |u|\frac{u^{+}}{|u^{+}|}$, the estimate \eqref{order-0},
 we have
	$|\mathcal{X}_{k}| \lesssim |\log \epsilon |^{-1} \epsilon^{-2}|\tilde{\varphi}_{k}h|_{L^{2}}|\varphi_{k}f|_{L^{2}}$, which gives
	\beno
	|\sum_{2^{k}\geq 1/\epsilon} \mathcal{X}_{k}| \lesssim \sum_{2^{k}\geq 1/\epsilon} |\log \epsilon |^{-1} \epsilon^{-2}|\tilde{\varphi}_{k}h|_{L^{2}}|\varphi_{k}f|_{L^{2}} \lesssim |W^{\epsilon}h|_{L^{2}}|W^{\epsilon}f|_{L^{2}}.
	\eeno

	\underline{\it Case 2: $2^{k}\leq 1/\epsilon$.} By Proposition \ref{fourier-transform-cross-term} and  the dyadic decomposition in the frequency space, we have
	\beno
	\mathcal{X}_{k} &=& \int b^{\epsilon}(\frac{\xi}{|\xi|}\cdot\sigma)  [\widehat{\tilde{\varphi}_{k}h}(\xi^{+})- \widehat{\tilde{\varphi}_{k}h}(|\xi|\frac{\xi^{+}}{|\xi^{+}|})] \overline{\widehat{\varphi_{k}f}}(\xi)  d\sigma d\xi
	\\&=& \sum_{l=-1}^{\infty} \int b^{\epsilon}(\frac{\xi}{|\xi|}\cdot\sigma)[({\varphi}_{l}\widehat{\tilde{\varphi}_{k}h})(\xi^{+})- ({\varphi}_{l}\widehat{\tilde{\varphi}_{k}h})(|\xi|\frac{\xi^{+}}{|\xi^{+}|})] (\tilde{\varphi}_{l}\overline{\widehat{\varphi_{k}f}})(\xi)  d\sigma d\xi
	:= \sum_{l=-1}^{\infty} \mathcal{X}_{k,l}.
	\eeno

	\underline{\it Case 2.1:  $2^{l}\geq 1/\epsilon$.}
	In this case, we have
	$
	|\mathcal{X}_{k,l}| \lesssim |\ln \epsilon|^{-1}\epsilon^{-2}|{\varphi}_{l}\widehat{\tilde{\varphi}_{k}h}|_{L^{2}}|\tilde{\varphi}_{l}\widehat{\varphi_{k}f}|_{L^{2}},
	$
	which yields
	\beno
	\sum_{2^{l}\geq 1/\epsilon}|\mathcal{X}_{k,l}| \lesssim \sum_{2^{l}\geq 1/\epsilon} |\ln \epsilon|^{-1}\epsilon^{-2}|{\varphi}_{l}\widehat{\tilde{\varphi}_{k}h}|_{L^{2}}|\tilde{\varphi}_{l}\widehat{\varphi_{k}f}|_{L^{2}}
	\lesssim |W^{\epsilon}(D)\tilde{\varphi}_{k}h|_{L^{2}}|W^{\epsilon}(D)\varphi_{k}f|_{L^{2}}.
	\eeno
	Then by Lemma \ref{piece-to-whole}, we have
	$
	\sum_{2^{k}\leq 1/\epsilon,2^{l}\geq 1/\epsilon}|\mathcal{X}_{k,l}| \lesssim |W^{\epsilon}(D)h|_{L^{2}}|W^{\epsilon}(D)f|_{L^{2}}.
	$

\underline{\it Case 2.2:  $2^{l}\le 1/\epsilon$.}  We have
	\beno
	\mathcal{X}_{k,l} &=& \int b^{\epsilon}(\frac{\xi}{|\xi|}\cdot\sigma)\mathrm{1}_{\theta \geq 2^{-\frac{k+l}{2}}}[({\varphi}_{l}\widehat{\tilde{\varphi}_{k}h})(\xi^{+})- ({\varphi}_{l}\widehat{\tilde{\varphi}_{k}h})(|\xi|\frac{\xi^{+}}{|\xi^{+}|})] (\tilde{\varphi}_{l}\overline{\widehat{\varphi_{k}f}})(\xi) d\sigma d\xi
	\\&& + \int b^{\epsilon}(\frac{\xi}{|\xi|}\cdot\sigma)\mathrm{1}_{\theta \leq 2^{-\frac{k+l}{2}}}[({\varphi}_{l}\widehat{\tilde{\varphi}_{k}h})(\xi^{+})- ({\varphi}_{l}\widehat{\tilde{\varphi}_{k}h})(|\xi|\frac{\xi^{+}}{|\xi^{+}|})] (\tilde{\varphi}_{l}\overline{\widehat{\varphi_{k}f}})(\xi) d\sigma d\xi
	:= \mathcal{X}_{k,l,1} + \mathcal{X}_{k,l,2}.
	\eeno
	By the similar argument as before, we have
	$
	|\mathcal{X}_{k,l,1}| \lesssim |\ln \epsilon|^{-1}2^{k+l}|{\varphi}_{l}\widehat{\tilde{\varphi}_{k}h}|_{L^{2}}|\tilde{\varphi}_{l}\widehat{\varphi_{k}f}|_{L^{2}}.
	$
	Therefore  we have
	\beno
	\sum_{2^{k}\leq 1/\epsilon,2^{l}\leq 1/\epsilon} |\mathcal{X}_{k,l,1}| &\leq& \bigg(\sum_{2^{k}\leq 1/\epsilon,2^{l}\leq 1/\epsilon} |\ln \epsilon|^{-1}2^{2l}|{\varphi}_{l}\widehat{\tilde{\varphi}_{k}h}|^{2}_{L^{2}}\bigg)^{1/2}\bigg(\sum_{2^{k}\leq 1/\epsilon,2^{l}\leq 1/\epsilon}|\ln \epsilon|^{-1}2^{2k}|\tilde{\varphi}_{l}\widehat{\varphi_{k}f}|^{2}_{L^{2}}\bigg)^{1/2}
\\&\lesssim&
\bigg(\sum_{2^{k}\leq 1/\epsilon} |W^{\epsilon}(D)\tilde{\varphi}_{k}h|^{2}_{L^{2}}\bigg)^{1/2}\bigg(\sum_{2^{k}\leq 1/\epsilon}|\ln \epsilon|^{-1}2^{2k}|\varphi_{k}f|^{2}_{L^{2}}\bigg)^{1/2}
	\\&\lesssim& |W^{\epsilon}(D)h|_{L^{2}}|W^{\epsilon}f|_{L^{2}}.
	\eeno
	By Taylor expansion,
	$
	({\varphi}_{l}\widehat{\tilde{\varphi}_{k}h})(\xi^{+})- ({\varphi}_{l}\widehat{\tilde{\varphi}_{k}h})(|\xi|\frac{\xi^{+}}{|\xi^{+}|}) = (1-\frac{1}{\cos\theta})\int_{0}^{1}
	(\nabla (\varphi_{l}\widehat{\tilde{\varphi}_{k}h}))(\xi^{+}(\kappa))\cdot \xi^{+} d\kappa,
	$
	where $\xi^{+}(\kappa) = (1-\kappa)|\xi|\frac{\xi^{+}}{|\xi^{+}|} + \kappa \xi^{+}$. From which we obtain
	\beno
	|\mathcal{X}_{k,l,2}| &=& | \int_{[0,1]\times \R^{3} \times \mathbb{S}^{2}}  b^{\epsilon}(\frac{\xi}{|\xi|}\cdot\sigma)(1-\frac{1}{\cos\theta})\mathrm{1}_{\epsilon \leq \theta \leq 2^{-\frac{k+l}{2}} }(\tilde{\varphi}_{l}\overline{\widehat{\varphi_{k}f}})(\xi)
(\nabla (\varphi_{l}\widehat{\tilde{\varphi}_{k}h})(\xi^{+}(\kappa)))\cdot \xi^{+} d\kappa d\sigma d\xi |
	\\&\lesssim& |\ln \epsilon|^{-1} (\int_{\epsilon}^{2^{-\frac{k+l}{2}}} \theta^{-1} |\tilde{\varphi}_{l}\widehat{\varphi_{k}f}(\xi)|^{2} d\theta d\xi)^{1/2}
	 (\int_{\epsilon}^{2^{-\frac{k+l}{2}}} \theta^{-1} |u|^{2}|\nabla (\varphi_{l}\widehat{\tilde{\varphi}_{k}h})(u)|^{2} d\theta du )^{1/2}
	\\&\lesssim& |\ln \epsilon|^{-1}(|\ln \epsilon| - \frac{k+l}{2} \ln 2)
|\tilde{\varphi}_{l}\widehat{\varphi_{k}f}|_{L^{2}}(\int |u|^{2}|\nabla (\varphi_{l}\widehat{\tilde{\varphi}_{k}h})(u)|^{2}du )^{1/2}
\\&\lesssim& |\ln \epsilon|^{-1}(|\ln \epsilon| - \frac{k+l}{2} \ln 2)
|\tilde{\varphi}_{l}\widehat{\varphi_{k}f}|_{L^{2}}(|\tilde{\varphi}_{l}\widehat{\tilde{\varphi}_{k}h}|_{L^{2}} + 2^{l}|\varphi_{l} \widehat{v\tilde{\varphi}_{k}h}|_{L^{2}}),
	\eeno
	where we use the change of variable $\xi \rightarrow u = \xi^{+}(\kappa)$ and the fact
	\beno
	\int |u|^{2}|\nabla (\varphi_{l}\widehat{\tilde{\varphi}_{k}h})(u)|^{2}du\ \lesssim |\tilde{\varphi}_{l}\widehat{\tilde{\varphi}_{k}h}|^2_{L^{2}} + 2^{2l}|\varphi_{l} \widehat{v\tilde{\varphi}_{k}h}|^{2}_{L^{2}}.
	\eeno
Since $|\ln \epsilon|^{-1}(|\ln \epsilon| - \frac{k+l}{2} \ln 2) \lesssim 1$, we have
\beno \sum_{2^{k}\leq 1/\epsilon,2^{l}\leq 1/\epsilon}|\ln \epsilon|^{-1}(|\ln \epsilon| - \frac{k+l}{2} \ln 2)|\tilde{\varphi}_{l}\widehat{\varphi_{k}f}|_{L^{2}}|\tilde{\varphi}_{l}\widehat{\tilde{\varphi}_{k}h}|_{L^{2}} \lesssim |f|_{L^{2}}|h|_{L^{2}}.\eeno
It is easy to check
  $|\ln \epsilon| - \frac{k+l}{2} \ln 2 \leq \left(  |\ln \epsilon| - k \ln 2 +2 \right)^{1/2} \left(  |\ln \epsilon| - l \ln 2 +2 \right)^{1/2}$ and thus
\beno  &&\sum_{ 2^{k}\leq 1/\epsilon, 2^{l}\leq 1/\epsilon}|\ln \epsilon|^{-1}(|\ln \epsilon| - \frac{k+l}{2} \ln 2) 2^{l}|\varphi_{l} \widehat{v\tilde{\varphi}_{k}h}|_{L^{2}}|\tilde{\varphi}_{l}\widehat{\varphi_{k}f}|_{L^{2}} \\&\lesssim& \bigg( \sum_{ 2^{k}\leq 1/\epsilon, 2^{l}\leq 1/\epsilon}|\ln \epsilon|^{-1}(|\ln \epsilon| - l \ln 2+2) 2^{2l}|\tilde{\varphi}_{l}\widehat{\varphi_{k}f}|^{2}_{L^{2}} \bigg)^{1/2}
\\&&\times \bigg( \sum_{ 2^{k}\leq 1/\epsilon,  2^{l}\leq 1/\epsilon}|\ln \epsilon|^{-1}(|\ln \epsilon| - k \ln 2+2) |\varphi_{l} \widehat{v\tilde{\varphi}_{k}h}|^{2}_{L^{2}}\bigg)^{1/2}
\\&\lesssim& \bigg( \sum_{ 2^{k}\leq 1/\epsilon}|W^{\epsilon}(D)\varphi_{k}f|^{2}_{L^{2}} \bigg)^{1/2} \bigg( \sum_{ 2^{k}\leq 1/\epsilon}|\ln \epsilon|^{-1}(|\ln \epsilon| - k \ln 2+2) |v\tilde{\varphi}_{k}h|^{2}_{L^{2}}\bigg)^{1/2}
\lesssim |W^{\epsilon}(D)f|_{L^{2}} |W^{\epsilon}h|_{L^{2}}.
\eeno
By the previous two results, we have
	$
	\sum_{2^{k}\leq 1/\epsilon,2^{l}\leq 1/\epsilon}|\mathcal{X}_{k,l,2}| \lesssim |W^{\epsilon}(D)f|_{L^{2}} |W^{\epsilon}h|_{L^{2}}.
	$
	Patching together all the above results, we conclude that
	\ben \label{gamma=0-clear}
	|\mathcal{X}(h,f)| \lesssim (|W^{\epsilon}h|_{L^{2}}+|W^{\epsilon}(D)h|_{L^{2}})(|W^{\epsilon}f|_{L^{2}}+|W^{\epsilon}(D)f|_{L^{2}}).
	\een

{\it Step 2: with the term $|u|^\gamma\mathrm{1}_{|u|\geq \eta}$.} Observe that $|u|^\gamma\mathrm{1}_{|u|\geq \eta}=|u|^\gamma(1-\phi(u))+|u|^\gamma(\phi(u)-\mathrm{1}_{|u| < \eta})$. From which, we separate $\mathcal{Y}^{\epsilon,\gamma}(h,f)$ into two parts $\mathcal{Y}_1^{\epsilon,\gamma}(h,f)$ and $\mathcal{Y}_2^{\epsilon,\gamma}(h,f)$ which correspond to $|u|^\gamma(1-\phi(u))$ and $|u|^\gamma(\phi(u)-\mathrm{1}_{|u| < \eta})$ respectively.

\underline{Estimate of $\mathcal{Y}_1^{\epsilon,\gamma}(h,f)$.}
 Set $H(u) :=h\langle u \rangle^{-\gamma} |u|^\gamma(1-\phi(u))$ and $w = |u|\frac{u^{+}}{|u^{+}|}$, then $W_{\gamma/2}(u) = W_{\gamma/2}(w)$ and
	\beno
	\langle u \rangle^{\gamma} H(u)[f(u^{+}) - f(w)] &=& (W_{\gamma/2}H)(u)[(W_{\gamma/2}f)(u^{+})-(W_{\gamma/2}f)(w)]
	\\&& +(W_{\gamma/2}H)(u)(W_{\gamma/2}f)(u^{+})(W_{\gamma/2}(w)W_{-\gamma/2}(u^{+}) - 1).
	\eeno
From which we have
	\beno
	\mathcal{Y}_1^{\epsilon,\gamma}(h,f) &=& \mathcal{X}(W_{\gamma/2}H, W_{\gamma/2}f)+ \int b^{\epsilon}(\frac{u}{|u|}\cdot\sigma)(W_{\gamma/2}H)(u) \\
	&&\times(W_{\gamma/2}f)(u^{+})(W_{\gamma/2}(w)W_{-\gamma/2}(u^{+}) - 1) d\sigma du
	:= \mathcal{X}(W_{\gamma/2}H, W_{\gamma/2}f) + \mathcal{A}.
	\eeno
	Observing that
	$
	|W_{\gamma/2}(u)W_{-\gamma/2}(u^{+}) - 1| \lesssim \theta^{2},
	$
	we have
	\beno
	|\mathcal{A}| &=&  \big(\int b^{\epsilon}(\frac{u}{|u|}\cdot\sigma)|W_{\gamma/2}H|^{2}(u)|W_{\gamma/2}(w)W_{-\gamma/2}(u^{+}) - 1|d\sigma du\big)^{1/2}
	\\&&\times\big(\int b^{\epsilon}(\frac{u}{|u|}\cdot\sigma)|W_{\gamma/2}f|^{2}(u^{+})|W_{\gamma/2}(w)W_{-\gamma/2}(u^{+}) - 1| d\sigma du\big)^{1/2}
	\lesssim  |W_{\gamma/2}H|_{L^{2}}|W_{\gamma/2}f|_{L^{2}},
	\eeno
	where the change of variable $u \rightarrow u^{+}$ is used. Thanks to the result \eqref{gamma=0-clear} in {\it Step 1} and Lemma \ref{operatorcommutator1} with $M=W^\epsilon$ and $\Phi=  |\cdot|^\gamma(1-\phi)(\cdot)$, we have
\beno |\mathcal{Y}_1^{\epsilon,\gamma}(h,f)|\lesssim (|W^{\epsilon}W_{\gamma/2}h|_{L^{2}}+|W^{\epsilon}(D)W_{\gamma/2}h|_{L^{2}})
	(|W^{\epsilon}W_{\gamma/2}f|_{L^{2}}+|W^{\epsilon}(D)W_{\gamma/2}f|_{L^{2}}).\eeno

 \underline{Estimate of $\mathcal{Y}_2^{\epsilon,\gamma}(h,f)$.}
 Since the support of $|u|^\gamma(\phi(u)-\mathrm{1}_{|u| < \eta})$ belongs to $\eta \lesssim u \lesssim 1$, we notice that
 \ben\label{Y2epgaestimate}  \mathcal{Y}_2^{\epsilon,\gamma}(h,f)=\int b^{\epsilon}(\frac{u}{|u|}\cdot\sigma)  \tilde{W}(u) \tilde{H}(u)[\tilde{F}(u^{+}) - \tilde{F}(|u|\frac{u^{+}}{|u^{+}|})] d\sigma du, \een
 where $\tilde{W}(u):=|u|^\gamma(\phi(u)-\mathrm{1}_{|u| < \eta}) $, $\tilde{\phi}(u):=\phi(u/4)$, $\tilde{H}:=\tilde{\phi}h$, $\tilde{F}:=\tilde{\phi}f$. By the result \eqref{gamma=0-clear} in {\it Step 1}, we derive that
 \beno|\mathcal{Y}_2^{\epsilon,\gamma}(h,f)|\lesssim (|W^{\epsilon}\tilde{W}\tilde{H}|_{L^{2}}+|W^{\epsilon}(D)\tilde{W}\tilde{H}|_{L^{2}})
	(|W^{\epsilon}\tilde{F}|_{L^{2}}+|W^{\epsilon}(D)\tilde{F}|_{L^{2}}). \eeno
First by $|\tilde{W}| \lesssim \phi(u)\eta^{\gamma}$, we have $|W^{\epsilon}\tilde{W}\tilde{H} |_{L^{2}} \lesssim \eta^{\gamma}|W^{\epsilon}W_{\gamma/2}h|_{L^{2}}$. Next, let us focus on $W^{\epsilon}(D)(\tilde{W}\tilde{H})$. Note that for any $\Psi\in L^2$, there holds
	\beno \int \Psi W^{\epsilon}(D)(\tilde{W}\tilde{H})  dv=\int \widehat{\Psi}(\xi)  W^\epsilon(\xi) \widehat{\tilde{W}}(v)\widehat{\tilde{H}}(\xi-v) dv d\xi.  \eeno
By \eqref{separate-into-2-cf}, Fubini's theorem, $|\cdot|_{L^{1}} \lesssim |\cdot|_{L^{2}_{2}}$ and
\eqref{upper-bound-cf}, we have
\beno \big|\int \Psi W^{\epsilon}(D)(\tilde{W}\tilde{H})  dv\big|&\lesssim& |W^\epsilon \widehat{\tilde{W}}|_{L^1}|W^\epsilon \widehat{\tilde{H}}|_{L^2}|\Psi|_{L^2}\lesssim  | \tilde{W}|_{H^3}|W^\epsilon(D)\tilde{H}|_{L^2}|\Psi|_{L^2}\\
	&\lesssim& \eta^{\gamma-3}|W^\epsilon(D)\tilde{H}|_{L^2}|\Psi|_{L^2}. \eeno
From which we infer that $|W^{\epsilon}(D)(\tilde{W}\tilde{H})|_{L^2}\lesssim \eta^{\gamma-3}|W^\epsilon(D)\tilde{H}|_{L^2}$. From which together with the support of $\tilde{H}$ and $\tilde{F}$,
and Lemma \ref{operatorcommutator1}, we finally have \beno
|\mathcal{Y}_2^{\epsilon,\gamma}(h,f)|\lesssim \eta^{\gamma-3}(|W^{\epsilon}W_{\gamma/2}h|_{L^{2}}+|W^{\epsilon}(D)W_{\gamma/2}h|_{L^{2}})
	(|W^{\epsilon}W_{\gamma/2}f|_{L^{2}}+|W^{\epsilon}(D)W_{\gamma/2}f|_{L^{2}}).\eeno
We conclude the desired result by patching together all the estimates.
\end{proof}

Now we are in a position to prove the following upper bound  of $ Q^{\epsilon,\gamma,\eta}$.
\begin{prop}\label{ubqepsilonnonsingular} Fix $0< \eta \leq 1$. For smooth functions $g, h$ and $f$, there holds
\beno
|\langle Q^{\epsilon,\gamma,\eta}(g,h), f\rangle| \lesssim \eta^{\gamma-3}|g|_{L^{1}_{|\gamma|+2}}|h|_{\epsilon,\gamma/2}|f|_{\epsilon,\gamma/2}.
\eeno
\end{prop}
\begin{proof} Recalling the translation operator $T_{v_{*}}$ defined by $(T_{v_{*}}f)(v) = f(v_{*}+v)$.
By geometric decomposition in the phase space, we have
$\langle Q^{\epsilon,\gamma,\eta}(g,h), f\rangle = \mathcal{D}_{1} + \mathcal{D}_{2}$,
where
\beno \mathcal{D}_{1} := \int b^{\epsilon}(\frac{u}{|u|}\cdot\sigma)|u|^{\gamma}\mathrm{1}_{|u|\geq \eta}g_{*} (T_{v_{*}}h)(u) ((T_{v_{*}}f)(u^{+})-(T_{v_{*}}f)(|u|\frac{u^{+}}{|u^{+}|})) d\sigma dv_{*} du, \\
\mathcal{D}_{2}:=\int b^{\epsilon}(\frac{u}{|u|}\cdot\sigma)|u|^{\gamma}\mathrm{1}_{|u|\geq \eta}g_{*} (T_{v_{*}}h)(u)((T_{v_{*}}f)(|u|\frac{u^{+}}{|u^{+}|})- (T_{v_{*}}f)(u)) d\sigma dv_{*} du.
\eeno
We remark that $\mathcal{D}_{1}$ represents the "radical" part and $\mathcal{D}_{2}$ stands for the "spherical" part.

{\it Step 1: Estimate of $\mathcal{D}_{1}$.}
By   Lemma \ref{crosstermsimilar}, we have
\beno
|\mathcal{D}_{1}| &\lesssim& \eta^{\gamma-3}\int |g_{*}| (|W^{\epsilon}W_{\gamma/2}T_{v_{*}}h|_{L^{2}}+|W^{\epsilon}(D)W_{\gamma/2}T_{v_{*}}h|_{L^{2}})
\\&&\times(|W^{\epsilon}W_{\gamma/2}T_{v_{*}}f|_{L^{2}}+|W^{\epsilon}(D)W_{\gamma/2}T_{v_{*}}f|_{L^{2}}) dv_{*}.
\eeno
By \eqref{translation-out-cf} and \eqref{translation-out-weight}, we have
\begin{eqnarray}\label{tvstartonovstar1}
|W^{\epsilon}W_{\gamma/2}T_{v_{*}}h|_{L^{2}} \lesssim W^{\epsilon}(v_{*})W_{|\gamma|/2}(v_{*})|W^{\epsilon}W_{\gamma/2}h|_{L^{2}}\lesssim W_{|\gamma|/2+1}(v_{*})|W^{\epsilon}W_{\gamma/2}h|_{L^{2}}.
\end{eqnarray}
Since $W^{\epsilon} \in S^{1}_{1,0}, W_{\gamma/2} \in S^{\gamma/2}_{1,0}$, by \eqref{translation-out-weight} and Lemma \ref{operatorcommutator1}, we have
\begin{eqnarray}\label{tvstartonovstar2}
|W^{\epsilon}(D)W_{\gamma/2}T_{v_{*}}h|_{L^{2}} &\lesssim& |W_{\gamma/2}W^{\epsilon}(D)T_{v_{*}}h|_{L^{2}} + |T_{v_{*}}h|_{H^{0}_{\gamma/2-1}}
= |W_{\gamma/2}T_{v_{*}}W^{\epsilon}(D)h|_{L^{2}} + |T_{v_{*}}h|_{H^{0}_{\gamma/2-1}}
\nonumber \\&\lesssim& W_{|\gamma|/2}(v_{*})(|W_{\gamma/2}W^{\epsilon}(D)h|_{L^{2}} + |h|_{L^{2}_{\gamma/2-1}})
\lesssim W_{|\gamma|/2}(v_{*})|W^{\epsilon}(D)W_{\gamma/2}h|_{L^{2}}.
\end{eqnarray}
By \eqref{tvstartonovstar1} and \eqref{tvstartonovstar2},
 we get
\beno
|\mathcal{D}_{1}| \lesssim \eta^{\gamma-3}| g|_{L^{1}_{|\gamma|+2}}( |W^{\epsilon}(D)W_{\gamma/2}h|_{L^{2}} + |W^{\epsilon}W_{\gamma/2}h|_{L^{2}})
( |W^{\epsilon}(D)W_{\gamma/2}f|_{L^{2}} + |W^{\epsilon}W_{\gamma/2}f|_{L^{2}}).
\eeno

{\it Step 2: Estimate of $\mathcal{D}_{2}$.}
Let $u = r \tau$ and $\varsigma = \frac{\tau+\sigma}{|\tau+\sigma|}$, then $\frac{u}{|u|} \cdot \sigma = 2(\tau\cdot\varsigma)^{2} - 1$ and $|u|\frac{u^{+}}{|u^{+}|} = r \varsigma$. By the change of variable $(u, \sigma) \rightarrow (r, \tau, \varsigma)$, one has
$
d\sigma du = 4  (\tau\cdot\varsigma) r^{2} dr d \tau d \varsigma.
$
Then
\beno
\mathcal{D}_{2} &=& 4 \int r^\gamma(1-\phi)(r)b^{\epsilon}(2(\tau\cdot\varsigma)^{2} - 1)(T_{v_*}h)(r\tau)\big((T_{v_*}f)(r\varsigma) - (T_{v_*}f) (r\tau)\big) (\tau\cdot\varsigma) r^{2} dr d \tau d \varsigma dv_{*}\\
&=& 2 \int r^\gamma(1-\phi)(r)b^{\epsilon}(2(\tau\cdot\varsigma)^{2} - 1)\big((T_{v_*}h)(r\tau) - (T_{v_*}h) (r\varsigma)\big)\\ &&\times \big((T_{v_*}f)(r\varsigma) - (T_{v_*}f) (r\tau)\big) (\tau\cdot\varsigma) r^{2} dr d \tau d \varsigma dv_{*}\\&=&
-\frac{1}{2}\int b^{\epsilon}(\frac{u}{|u|}\cdot\sigma)|u|^{\gamma}\mathrm{1}_{|u|\geq \eta}g_{*} ((T_{v_{*}}h)(|u|\frac{u^{+}}{|u^{+}|})-(T_{v_{*}}h)(u))\\ &&\times
 ((T_{v_{*}}f)(|u|\frac{u^{+}}{|u^{+}|})-(T_{v_{*}}f)(u)) d\sigma dv_{*} du.
\eeno
Then by Cauchy-Schwartz inequality and the fact $|u|^{\gamma}\mathrm{1}_{|u|\geq \eta} \lesssim \eta^{\gamma}\langle u \rangle^{\gamma}$, we have
\beno
|\mathcal{D}_{2}| &\lesssim& \eta^{\gamma}\big(\int b^{\epsilon}(\frac{u}{|u|}\cdot\sigma)\langle u \rangle^{\gamma}|g_{*}| ((T_{v_{*}}h)( |u|\frac{u^{+}}{|u^{+}|})-(T_{v_{*}}h)( u))^{2} d\sigma dv_{*} du\big)^{1/2}
\\&& \times \big(\int b^{\epsilon}(\frac{u}{|u|}\cdot\sigma)\langle u \rangle^{\gamma}|g_{*}|
((T_{v_{*}}f)(|u|\frac{u^{+}}{|u^{+}|})-(T_{v_{*}})f(u))^{2} d\sigma dv_{*} du\big)^{1/2}
:= \eta^{\gamma}(\mathcal{D}_{2,1})^{1/2}(\mathcal{D}_{2,2})^{1/2}.
\eeno
Note that $\mathcal{D}_{2,1}$ and $\mathcal{D}_{2,2}$ have exactly the same structure. It suffices to focus on  $\mathcal{D}_{2,2}$.
Since
\beno
((T_{v_{*}}f)(|u|\frac{u^{+}}{|u^{+}|})-(T_{v_{*}}f)(u))^{2} \leq 2 ((T_{v_{*}}f)(|u|\frac{u^{+}}{|u^{+}|})-(T_{v_{*}}f)(u^{+}))^{2} + 2 ((T_{v_{*}}f)(u^{+})-(T_{v_{*}}f)(u))^{2},
\eeno
we have
\beno
\mathcal{D}_{2,2} &\lesssim& \int b^{\epsilon}(\frac{u}{|u|}\cdot\sigma)\langle u \rangle^{\gamma}|g_{*}| ((T_{v_{*}}f)(|u|\frac{u^{+}}{|u^{+}|})-(T_{v_{*}}f)(u^{+}))^{2} d\sigma dv_{*} du
\\&&+\int b^{\epsilon}(\frac{u}{|u|}\cdot\sigma)\langle u \rangle^{\gamma}|g_{*}| ((T_{v_{*}}f)(u^{+})-(T_{v_{*}}f)(u))^{2} d\sigma dv_{*} du
:= \mathcal{D}_{2,2,1}+ \mathcal{D}_{2,2,2}.
\eeno
\begin{itemize}\item
By Lemma \ref{gammanonzerotozero}, and the facts (\ref{tvstartonovstar1}) and (\ref{tvstartonovstar2}), we have
\beno
\mathcal{D}_{2,2,1} \lesssim \int |g_{*}| \mathcal{Z}^{\epsilon,\gamma}(T_{v_{*}}f) dv_{*} \lesssim| g|_{L^{1}_{|\gamma|+2}}(|W^{\epsilon}(D)W_{\gamma/2}f|^{2}_{L^{2}}+|W^{\epsilon}W_{\gamma/2}f|^{2}_{L^{2}}).
\eeno
\item
Observe that
$\mathcal{D}_{2,2,2} =  \tilde{\mathcal{N}}^{\epsilon,\gamma,0}(\sqrt{|g|}, f).$
By Lemma \ref{reduce-gamma-to-0-no-sigularity}, we have
\beno
\tilde{\mathcal{N}}^{\epsilon,\gamma,0}(\sqrt{|g|},f)  \lesssim \mathcal{N}^{\epsilon,0,0}(W_{-\gamma/2} \sqrt{|g|},W_{\gamma/2}f) + |g|_{L^{1}_{|\gamma+2|}}|f|^{2}_{L^{2}_{\gamma/2}}. \nonumber
\eeno
By \eqref{anisotropic-regularity-general-g-up-bound} in Lemma \ref{lowerboundpart2-general-g}, we get
\beno
\mathcal{N}^{\epsilon,0,0}(W_{-\gamma/2} \sqrt{|g|},W_{\gamma/2}f) \lesssim  |W_{-\gamma/2} \sqrt{|g|}|^{2}_{L^{2}_{1}}|f|^{2}_{\epsilon,\gamma/2} \lesssim |g|_{L^{1}_{-\gamma+2}}|f|^{2}_{\epsilon,\gamma/2}.
\eeno
So we have
$\mathcal{D}_{2,2,2} \lesssim  |g|_{L^{1}_{|\gamma|+2}}|f|^{2}_{\epsilon,\gamma/2}.$
\end{itemize}
Patching together the estimates for  $\mathcal{D}_{2,2,1}$ and $\mathcal{D}_{2,2,2}$, we get
\beno
\mathcal{D}_{2,2} \lesssim |g|_{L^{1}_{|\gamma|+2}}|f|^{2}_{\epsilon,\gamma/2},
\eeno
which yields
$|\mathcal{D}_{2}| \lesssim \eta^{\gamma}(\mathcal{D}_{2,1})^{1/2}(\mathcal{D}_{2,2})^{1/2} \lesssim \eta^{\gamma}|g|_{L^{1}_{|\gamma|+2}}|h|_{\epsilon,\gamma/2}|f|_{\epsilon,\gamma/2}.$

We complete the proof by patching together the estimates of $\mathcal{D}_{1}$ and $\mathcal{D}_{2}$.
\end{proof}

\subsubsection{Upper bound of $\langle I^{\epsilon,\gamma,\eta}(g,h), f\rangle$}
To implement the energy estimates for the nonlinear equations, in this subsection, we will give the upper bound of $\langle I^{\epsilon,\gamma,\eta}(g,h;\beta), f\rangle$ where
\ben \label{I-ep-ga-geq-eta-beta}
I^{\epsilon,\gamma,\eta}(g,h;\beta) :=\int B^{\epsilon,\gamma,\eta}(v-v_*,\sigma) \big((\pa_{\beta}\mu^{1/2})_{*}-(\pa_{\beta}\mu^{1/2})^{\prime}_{*}\big)g^{\prime}_{*} h^{\prime} d\sigma dv_{*}.
\een

Let us deviate to explain why we consider the additional differential operator $\pa_{\beta}$.
By binomial expansion, we have
\ben \label{alpha-beta-on-Gamma} \pa^{\alpha}_{\beta}\Gamma^{\epsilon}(g,h) = \sum _{\beta_{0}+\beta_{1}+\beta_{2}= \beta,\alpha_{1}+\alpha_{2}=\alpha} C^{\beta_{0},\beta_{1},\beta_{2}}_{\beta} C^{\alpha_{1},\alpha_{2}}_{\alpha} \Gamma^{\epsilon}(\pa^{\alpha_{1}}_{\beta_{1}}g,\pa^{\alpha_{2}}_{\beta_{2}}h;\beta_{0}),\een
where
\ben \label{Gamma-beta}
 \Gamma^{\epsilon}(g,h;\beta)(v):=
\int_{\SS^{2} \times \R^3} B^{\epsilon}(v-v_*,\sigma)(\pa_{\beta}\mu^{1/2})_{*}(g'_*h'-g_*h)d\sigma dv_*.
\een
Note that
\ben \label{Gamma-ep-ga-geq-eta-into-IQ-inner-beta}
\Gamma^{\epsilon,\gamma,\eta}(g,h;\beta) =   Q^{\epsilon,\gamma,\eta}(g\partial_{\beta}\mu^{1/2},h) +
 I^{\epsilon,\gamma,\eta}(g,h;\beta).
\een
This explains why we consider the general version $ I^{\epsilon,\gamma,\eta}(g,h;\beta)$.

By writing $\pa_{\beta} \mu^\f12 =\mu^{\f12}P_{\beta}$ where $P_{\beta}$ is a polynomial, we observe that
\ben \label{nice-decomposition}
 (\mu^{\f12}P_{\beta})_{*}^{\prime} - (\mu^{\f12}P_{\beta})_*&=& ((\mu^{\f14})_*^{\prime} - (\mu^{\f14})_*)((\mu^{\f14}P_{\beta})_{*}^{\prime} - (\mu^{\f14}P_{\beta})_{*})
 + (\mu^{\f14})_*((\mu^{\f14}P_{\beta})_{*}^{\prime}   - (\mu^{\f14}P_{\beta})_{*}) \\&&+  (\mu^{1/4}P_{\beta})_{*}((\mu^{\f14})_*^{\prime} - (\mu^{\f14})_*). \nonumber
\een
Then we have
\beno
\langle I^{\epsilon,\gamma,\eta}(g,h;\beta), f\rangle  = \int B^{\epsilon,\gamma,\eta}((\mu^{1/8})_{*}^{\prime} + \mu_{*}^{1/8})   ((\mu^{\f18})_*^{\prime} - (\mu^{\f18})_*)((\mu^{\f14}P_{\beta})_{*}^{\prime} - (\mu^{\f14}P_{\beta})_{*})g_{*} h f^{\prime} d\sigma dv_{*} dv
\\+  \int B^{\epsilon,\gamma,\eta}[((\mu^{\f14}P_{\beta})_{*}^{\prime} - (\mu^{\f14}P_{\beta})_{*})(\mu^{1/4}g)_{*}+((\mu^{\f14})_*^{\prime} - (\mu^{\f14})_*)(\mu^{1/4}P_{\beta}g)_{*}](h-h^{\prime})f^{\prime} d\sigma dv_{*}dv
\\+   \int B^{\epsilon,\gamma,\eta}[((\mu^{\f14}P_{\beta})_{*}^{\prime} - (\mu^{\f14}P_{\beta})_{*})(\mu^{1/4}g)_{*}+((\mu^{\f14})_*^{\prime} - (\mu^{\f14})_*)(\mu^{1/4}P_{\beta}g)_{*}] h^{\prime} f^{\prime} d\sigma dv_{*}dv
:= \mathcal{I}_{1} + \mathcal{I}_{2} + \mathcal{I}_{3}.
\eeno

\begin{prop}\label{upforI-ep-ga-et}
For any  $0 <\eta \leq 1, 0 <\delta \leq 1/2$, and smooth functions $g,h$ and $f$, there holds for $s_{1}, s_{2} \geq 0$ with $s_{1}+s_{2} = 3/2+\delta$,
\beno |\langle I^{\epsilon,\gamma,\eta}(g,h; \beta), f\rangle|  \lesssim  \delta^{-1/2}|\mu^{1/12}g|_{H^{s_{1}}}|\mu^{1/12}h|_{H^{s_{2}}}|W^{\epsilon}f|_{L^{2}_{\gamma/2}}
+\eta^{\gamma-3}|g|_{L^{2}}|h|_{\epsilon,\gamma/2}|W^{\epsilon}f|_{L^{2}_{\gamma/2}}.\eeno
The $\lesssim$ constant could depend on $|\beta|$.
\end{prop}
\begin{proof} Let us consider the $\beta=0$ case since the following arguments also work when $|\beta|>0$. There are three steps in the proof.
We will indicate the main difference at the end of each step.

In the proof, we will
frequently use the following fact:
\ben \label{mu-prodcut-square}
((\mu^{1/8})_{*}^{\prime} - \mu_{*}^{1/8})^{2} \lesssim \min\{1,|v-v_{*}|^{2}\theta^{2}\} \sim \min\{1,|v^{\prime}-v_{*}|^{2}\theta^{2}\}\sim \min\{1,|v-v^{\prime}_{*}|^{2}\theta^{2}\}.
\een

{\it Step 1: Estimate of $\mathcal{I}_{1}$.}
When $|\beta|=0$, recall
\beno
\mathcal{I}_{1} = \int B^{\epsilon,\gamma,\eta}((\mu^{1/8})_{*}^{\prime} + \mu_{*}^{1/8})^{2}((\mu^{1/8})_{*}^{\prime} - \mu_{*}^{1/8})^{2}g_{*} h f^{\prime} d\sigma dv_{*} dv.
\eeno
Since $|v-v_{*}| \geq \eta$,  we have
\ben \label{v-minus-vstar-lower-no-singularity}
|v-v_{*}|^{\gamma} \mathrm{1}_{|v-v_{*}| \geq \eta} \lesssim \eta^{\gamma}\langle v -v_{*} \rangle^{\gamma}. \een
By \eqref{v-minus-vstar-lower-no-singularity} and Cauchy-Schwartz inequality, we have
\beno
|\mathcal{I}_{1}| &\lesssim& \eta^{\gamma}\big(\int b^{\epsilon}(\cos\theta)\langle v-v_{*}\rangle^{\gamma}((\mu^{1/8})_{*}^{\prime} + \mu_{*}^{1/8})^{2}((\mu^{1/8})_{*}^{\prime} - \mu_{*}^{1/8})^{2}g^{2}_{*} h^{2}  d\sigma dv_{*} dv \big)^{1/2}
\\&&\times \big(\int b^{\epsilon}(\cos\theta)\langle v-v_{*} \rangle^{\gamma}((\mu^{1/8})_{*}^{\prime} + \mu_{*}^{1/8})^{2}((\mu^{1/8})_{*}^{\prime} - \mu_{*}^{1/8})^{2} f^{\prime 2} d\sigma dv_{*} dv\big)^{1/2}
:= \eta^{\gamma}(\mathcal{I}_{1,1})^{1/2} (\mathcal{I}_{1,2})^{1/2}.
\eeno

\underline{Estimate of $\mathcal{I}_{1,1}$.}
We claim that
\ben \label{claim-1-L2} \mathcal{A} := \int b^{\epsilon}(\cos\theta)\langle v-v_{*}\rangle^{\gamma}((\mu^{1/8})_{*}^{\prime} + \mu_{*}^{1/8})^{2}((\mu^{1/8})_{*}^{\prime} - \mu_{*}^{1/8})^{2}  d\sigma \lesssim (W^{\epsilon})^{2}(v)\langle v\rangle^{\gamma},\een
which immediately gives
$\mathcal{I}_{1,1} \lesssim  |g|^{2}_{L^{2}}|W^{\epsilon}h|^{2}_{L^{2}_{\gamma/2}}$. It remains to prove \eqref{claim-1-L2}.
Since $((\mu^{1/8})_{*}^{\prime} + \mu_{*}^{1/8})^{2} \lesssim (\mu^{1/4})_{*}^{\prime} + \mu_{*}^{1/4}$, we have $\mathcal{A} \lesssim \mathcal{A}_{1}+\mathcal{A}_{2}$, where
\beno \mathcal{A}_{1} := \int b^{\epsilon}(\cos\theta)\langle v-v_{*}\rangle^{\gamma}\mu_{*}^{1/4}((\mu^{1/8})_{*}^{\prime} - \mu_{*}^{1/8})^{2}d\sigma,
\\  \mathcal{A}_{2} := \int b^{\epsilon}(\cos\theta)\langle v-v_{*}\rangle^{\gamma}(\mu^{1/4})_{*}^{\prime}((\mu^{1/8})_{*}^{\prime} - \mu_{*}^{1/8})^{2}  d\sigma.
\eeno
Thanks to \eqref{mu-prodcut-square}, Proposition \ref{symbol} and property \eqref{separate-into-2-cf},  one has
\ben \label{estimate-A-1}
\mathcal{A}_{1} \lesssim \langle v-v_{*}\rangle^{\gamma}\mu_{*}^{1/4}(W^{\epsilon})^{2}(v-v_{*}) \lesssim (W^{\epsilon})^{2}(v)\langle v\rangle^{\gamma}.\een
As for $\mathcal{A}_{2}$, thanks to $|v-v_{*}| \sim |v-v^{\prime}_{*}|$ and thus $\langle v-v_{*}\rangle^{\gamma} \lesssim \langle v-v^{\prime}_{*}\rangle^{\gamma} \lesssim \langle v\rangle^{\gamma}\langle v^{\prime}_{*}\rangle^{|\gamma|}$, we have
 \beno \mathcal{A}_{2} \lesssim \langle v\rangle^{\gamma} \int b^{\epsilon}(\cos\theta) (\mu^{1/8})_{*}^{\prime} \min\{1,|v-v_{*}|^{2}\theta^{2}\} d\sigma.\eeno
 \begin{itemize}
 	\item
If $|v-v_{*}|\geq 10|v|$, then there holds $|v^{\prime}_{*}| = |v^{\prime}_{*}-v+v|\geq |v^{\prime}_{*}-v| -|v| \geq (1/\sqrt{2} - 1/10)|v-v_{*}| \geq \frac{1}{5}|v-v_{*}|$, and thus $(\mu^{1/8})_{*}^{\prime} \lesssim \mu^{1/200}(v-v_{*})$,
which yields
\beno \mathcal{A}_{2} \lesssim \langle v\rangle^{\gamma}\mu^{1/200}(v-v_{*}) (W^{\epsilon})^{2}(v-v_{*}) \lesssim \langle v\rangle^{\gamma}.\eeno

\item If $|v-v_{*}|\leq 10|v|$, by Proposition \ref{symbol}, we have
\beno \mathcal{A}_{2} \lesssim \langle v\rangle^{\gamma} \int b^{\epsilon}(\cos\theta)  \min\{1,|v|^{2}\theta^{2}\} d\sigma  \lesssim (W^{\epsilon})^{2}(v)\langle v\rangle^{\gamma}.\eeno
\end{itemize}
Patching together the estimates of $\mathcal{A}_{1}$ and $\mathcal{A}_{2}$, we get the claim \eqref{claim-1-L2}.

\underline{Estimate of $\mathcal{I}_{1,2}$.}
By the   change of variable $(v, v_{*},\sigma)\rightarrow (v^{\prime}, v_{*}^{\prime}, \tau=(v-v_{*})/|v-v_{*}|)$, we have
\beno
\mathcal{I}_{1,2} &=& \int b^{\epsilon}(\cos\theta)\langle v-v_{*} \rangle^{\gamma}((\mu^{1/8})_{*}^{\prime} + \mu_{*}^{1/8})^{2}((\mu^{1/8})_{*}^{\prime} - \mu_{*}^{1/8})^{2} f^{2} d\sigma dv_{*} dv
\\&\leq& 2\int b^{\epsilon}(\cos\theta)\langle v-v_{*} \rangle^{\gamma}\mu_{*}^{1/4}((\mu^{1/8})_{*}^{\prime} - \mu_{*}^{1/8})^{2} f^{2} d\sigma dv_{*} dv
\\&& + 2\int b^{\epsilon}(\cos\theta)\langle v-v_{*} \rangle^{\gamma}(\mu^{1/4})_{*}^{\prime}((\mu^{1/8})_{*}^{\prime} - \mu_{*}^{1/8})^{2} f^{2} d\sigma dv_{*} dv
:= \mathcal{I}_{1,2,1} + \mathcal{I}_{1,2,2}.
\eeno
By \eqref{mu-prodcut-square}, Proposition \ref{symbol} and property \eqref{separate-into-2-cf}, we have
\beno
\mathcal{I}_{1,2,1} &\lesssim& \int b^{\epsilon}(\cos\theta)\langle v-v_{*} \rangle^{\gamma}\mu_{*}^{1/4}\min\{1,|v-v_{*}|^{2}\theta^{2}\} f^{2} d\sigma dv_{*} dv
\\&\lesssim&  \int \langle v-v_{*} \rangle^{\gamma}\mu_{*}^{1/8} (W^{\epsilon})^{2}(v)f^{2}  dv_{*} dv
\lesssim |W^{\epsilon}f|^{2}_{L^{2}_{\gamma/2}}.
\eeno
By the fact $|v-v_{*}|\sim|v-v_{*}^{\prime}|$, the change of variable $v_{*}\rightarrow v_{*}^{\prime}$, and the estimate \eqref{estimate-A-1} of $\mathcal{A}_{1}$, we have
\beno
\mathcal{I}_{1,2,2} \lesssim \int b^{\epsilon}(\cos\theta)\langle v-v_{*}^{\prime} \rangle^{\gamma}(\mu^{1/4})_{*}^{\prime}\min\{1,|v-v_{*}^{\prime}|^{2}\theta^{2}\} f^{2} d\sigma dv_{*}^{\prime} dv
\lesssim |W^{\epsilon}f|^{2}_{L^{2}_{\gamma/2}}.
\eeno
Therefore we have
$
\mathcal{I}_{1,2} \lesssim |W^{\epsilon}f|^{2}_{L^{2}_{\gamma/2}}.
$
Patching together the estimates  for $\mathcal{I}_{1,1}$ and $\mathcal{I}_{1,2}$, we have
\ben \label{I-1-1}
\mathcal{I}_{1} \lesssim \eta^{\gamma}|g|_{L^{2}}|W^{\epsilon}h|_{L^{2}_{\gamma/2}}|W^{\epsilon}f|_{L^{2}_{\gamma/2}}.
\een

Since \eqref{mu-prodcut-square} also holds for $(\mu^{\f14}P_{\beta})_{*}^{\prime} - (\mu^{\f14}P_{\beta})_{*}$, the above estimates in this step  are also valid for the $|\beta|>0$ case.

{\it Step 2: Estimate of $\mathcal{I}_{2}$.}  When $|\beta|=0$, by Cauchy-Schwartz inequality, we have
\ben \label{estimate-I-21}
\mathcal{I}_{2} &=& 2\int B^{\epsilon,\gamma,\eta}((\mu^{1/4})_{*}^{\prime} - \mu_{*}^{1/4})(\mu^{1/4}g)_{*}(h-h^{\prime})f^{\prime}
d\sigma dv_{*} dv
\nonumber \\&\lesssim&  \big(\int B^{\epsilon,\gamma,\eta}|(\mu^{1/4}g)_{*}|(h-h^{\prime})^{2}d\sigma dv_{*} dv\big)^{1/2}
\nonumber \\&&\times \big(\int B^{\epsilon,\gamma,\eta}((\mu^{1/4})_{*}^{\prime} - \mu_{*}^{1/4})^{2}|(\mu^{1/4}g)_{*}|f^{\prime 2}d\sigma dv_{*} dv\big)^{1/2}
:= (\mathcal{I}_{2,1})^{1/2}(\mathcal{I}_{2,2})^{1/2}.
\een

\underline{Estimate of $\mathcal{I}_{2,1}$.}
Since
$
(h-h^{\prime})^{2} = h^{\prime 2} - h^{2} - 2h(h^{\prime}-h),
$
we have
\beno
\mathcal{I}_{2,1} &=& \int B^{\epsilon,\gamma,\eta}(\mu^{1/4}g)_{*}(h^{\prime 2} - h^{2})d\sigma dv_{*} dv - 2 \langle Q^{\epsilon}(\mu^{1/4}g, h), h\rangle
\\&:=& \mathcal{I}_{2,1,1} - 2 \langle Q^{\epsilon,\gamma,\eta}(\mu^{1/4}g, h), h\rangle. \nonumber
\eeno
By \eqref{geq-delta-mu} with $a=1/6, 1/p+1/q=1$ in Cancellation Lemma \ref{cancellation-lemma-general-gamma-minus3-mu}, one has
\beno |\mathcal{I}_{2,1,1}| \lesssim |\mu^{1/12}g|_{L^{p}}|\mu^{1/6}h^{2}|_{L^{q}} = |\mu^{1/12}g|_{L^{p}}|\mu^{1/12}h|^{2}_{L^{2q}} \lesssim
\delta^{-1/2}|\mu^{1/12}g|_{H^{s_{1}}}|\mu^{1/12}h|^{2}_{H^{s_{2}}}, \eeno
where $s_{1}+s_{2} = 3/2 +\delta $ and we use the Sobolev imbedding $|\cdot|_{L^{\infty}} \lesssim \delta^{-1/2}|\cdot|_{H^{3/2+\delta}}$ with $\delta>0$ and $|\cdot|_{L^{p}} \lesssim |\cdot|_{H^{s}}$ with $s/3 = 1/2 - 1/p$.
By Proposition \ref{ubqepsilonnonsingular}, we have
\beno |\langle Q^{\epsilon,\gamma,\eta}(\mu^{1/4}g, h), h\rangle|  \lesssim \eta^{\gamma-3}|\mu^{1/4}g|_{L^{1}_{|\gamma|+2}}|h|^{2}_{\epsilon,\gamma/2} \lesssim \eta^{\gamma-3}|\mu^{1/8}g|_{L^{2}}|h|_{\epsilon,\gamma/2}^2.\eeno
Patching together the previous two results, we get
\ben \label{result-estimate-I-21}
|\mathcal{I}_{2,1}| \lesssim \delta^{-1/2}|\mu^{1/12}g|_{H^{s_{1}}}|\mu^{1/12}h|^{2}_{H^{s_{2}}}+\eta^{\gamma-3}|\mu^{1/8}g|_{L^{2}}|h|_{\epsilon,\gamma/2}^2.
\een

\underline{Estimate of $\mathcal{I}_{2,2}$.}
By the change of variable $v \rightarrow v^{\prime}$, and the estimate \eqref{estimate-A-1} of $\mathcal{A}_{1}$, we have
$
\mathcal{I}_{2,2} \lesssim \eta^{\gamma}|\mu^{1/8}g|_{L^{2}}|W^{\epsilon}f|^{2}_{L^{2}_{\gamma/2}}.
$

Patching the estimates for $\mathcal{I}_{2,1}$ and $\mathcal{I}_{2,2}$, we get
\beno
|\mathcal{I}_{2}| \lesssim
\delta^{-1/2}|\mu^{1/12}g|_{H^{s_{1}}}|\mu^{1/12}h|_{H^{s_{2}}}|W^{\epsilon}f|_{L^{2}_{\gamma/2}}
+\eta^{\gamma-3}| \mu^{1/8}g|_{L^{2}}|h|_{\epsilon,\gamma/2}|W^{\epsilon}f|_{L^{2}_{\gamma/2}}.
\eeno

We remark that the $|\beta|>0$ case can be dealt with in a similar way and there is no essential difference.

{\it Step 3: Estimate of $\mathcal{I}_{3}$.}
By the change of variables $(v,v_{*}) \rightarrow (v^{\prime},v_{*}^{\prime})$ and   $(v,v_{*},\sigma) \rightarrow (v_{*},v,-\sigma)$,
\beno
\mathcal{I}_{3} =  \int b^{\epsilon}(\cos\theta)|v-v_{*}|^{\gamma}(\mu^{ 1/4} - (\mu^{1/4})^{\prime})(\mu^{1/4}g)^{\prime} h_{*}f_{*} d\sigma dv_{*}dv.
\eeno
We separate the integration domain into three parts, $\mathbb{S}^{2} \times \mathbb{R}^{3} \times \mathbb{R}^{3} = E_{1} \cup E_{2} \cup E_{3}$, where
 $E_{1} = \{(\sigma, v_{*}, v): |v-v_{*}| \geq 1/\epsilon\}, E_{2} = \{(\sigma, v_{*}, v): |v-v_{*}| \leq 1/\epsilon, |v-v_{*}|^{-1}\leq \theta \leq \pi/2\}, E_{3} = \{(\sigma, v_{*}, v): |v-v_{*}| \leq 1/\epsilon, \epsilon \leq \theta \leq |v-v_{*}|^{-1}\}$. Then $\mathcal{I}_{3}  = \mathcal{I}_{3,1} + \mathcal{I}_{3,2} + \mathcal{I}_{3,3} $ where $\mathcal{I}_{3,i} = \int_{E_{1}}(\cdots) d\sigma dv_{*}dv$.

 \underline{Estimate of $\mathcal{I}_{3,1}$.}
  By the change of variable $v \rightarrow v^{\prime}$, the fact $|v^{\prime}-v_{*}|\geq |v-v_{*}| /\sqrt{2}$ and \eqref{order-0}, we have
\beno
|\mathcal{I}_{3,1}| &\lesssim& \int b^{\epsilon}(\cos\theta)|v^{\prime}-v_{*}|^{\gamma}\mathrm{1}_{|v^{\prime}-v_{*}|\geq (\sqrt{2}\epsilon)^{-1}}|(\mu^{1/4}g)^{\prime} h_{*}f_{*}| d\sigma dv_{*}dv^{\prime}
\\&\lesssim& |\ln \epsilon|^{-1}\epsilon^{-2} \int |v^{\prime}-v_{*}|^{\gamma}\mathrm{1}_{|v^{\prime}-v_{*}|\geq (\sqrt{2}\epsilon)^{-1}}|(\mu^{1/4}g)^{\prime} h_{*}f_{*}| dv_{*}dv^{\prime}.
\eeno
On one hand, by Cauchy-Schwartz inequality, we have
\begin{eqnarray}\label{lessthanep2s}
&&|\ln \epsilon|^{-1}\epsilon^{-2} \int |v^{\prime}-v_{*}|^{\gamma}\mathrm{1}_{|v^{\prime}-v_{*}|\geq (\sqrt{2}\epsilon)^{-1}}|(\mu^{1/4}g)^{\prime}|dv^{\prime}
\\&\leq& |\mu^{1/8}g|_{L^{2}} |\ln \epsilon|^{-1}\epsilon^{-2} \big(\int |v^{\prime}-v_{*}|^{2\gamma}\mathrm{1}_{|v^{\prime}-v_{*}|\geq (\sqrt{2}\epsilon)^{-1}}(\mu^{1/4})^{\prime}dv^{\prime}\big)^{1/2}
\lesssim |\mu^{1/8}g|_{L^{2}} |\ln \epsilon|^{-1}\epsilon^{-2} \langle v_{*} \rangle^{\gamma}, \nonumber
\end{eqnarray}
where we use
$|v^{\prime}-v_{*}|^{2\gamma}\mathrm{1}_{|v^{\prime}-v_{*}|\geq (\sqrt{2}\epsilon)^{-1}} \lesssim \langle v^{\prime}-v_{*} \rangle^{2\gamma} \mathrm{1}_{|v^{\prime}-v_{*}|\geq 1} \lesssim \langle v^{\prime} \rangle^{|2\gamma|}\langle v_{*} \rangle^{2\gamma}.$
On the other hand, we have
\begin{eqnarray}\label{lessthanvstar2s}
&&|\ln \epsilon|^{-1}\epsilon^{-2} \int |v^{\prime}-v_{*}|^{\gamma}\mathrm{1}_{|v^{\prime}-v_{*}|\geq (\sqrt{2}\epsilon)^{-1}}|(\mu^{1/4}g)^{\prime}|dv^{\prime}
\\&\lesssim& |\ln \epsilon|^{-1}\int |v^{\prime}-v_{*}|^{\gamma+2}\mathrm{1}_{|v^{\prime}-v_{*}|\geq (\sqrt{2}\epsilon)^{-1}}|(\mu^{1/4}g)^{\prime}|dv^{\prime} \nonumber
\\&\leq& |\ln \epsilon|^{-1}|\mu^{1/8}g|_{L^{2}}  \big(\int |v^{\prime}-v_{*}|^{2\gamma+4}\mathrm{1}_{|v^{\prime}-v_{*}|\geq (\sqrt{2}\epsilon)^{-1}}(\mu^{1/4})^{\prime}dv^{\prime}\big)^{1/2} \nonumber
\lesssim |\ln \epsilon|^{-1}|\mu^{1/8}g|_{L^{2}}  \langle v_{*} \rangle^{\gamma+2}. \nonumber
\end{eqnarray}
With estimates (\ref{lessthanep2s}) and  (\ref{lessthanvstar2s}) in hand, since
\ben \label{property-of-cf}
\min\{|\ln \epsilon|^{-1}\epsilon^{-2},|\ln \epsilon|^{-1} \langle v_{*} \rangle^{2}\} \lesssim (W^{\epsilon})^{2}(v_{*} ), \een
we have
$
|\mathcal{I}_{3,1}| \lesssim |\mu^{1/8}g|_{L^{2}}|W^{\epsilon}h|_{L^{2}_{\gamma/2}}|W^{\epsilon}f|_{L^{2}_{\gamma/2}}.
$

 \underline{Estimate of $\mathcal{I}_{3,2}$.}  By the change of variable $v \rightarrow v^{\prime}$ and the fact $|v^{\prime}-v_{*}| \leq |v-v_{*}|$, we get
\begin{eqnarray}\label{i32preliminary}
|\mathcal{I}_{3,2}| &\lesssim& \int b^{\epsilon}(\cos\theta)\mathrm{1}_{\theta \geq (\sqrt{2}|v^{\prime}-v_{*}|)^{-1} }|v^{\prime}-v_{*}|^{\gamma}\mathrm{1}_{|v^{\prime}-v_{*}|\leq 1/\epsilon}|(\mu^{1/4}g)^{\prime} h_{*}f_{*}| d\sigma dv_{*}dv^{\prime}
\\&\lesssim& |\ln \epsilon|^{-1}\int |v^{\prime}-v_{*}|^{\gamma+2}\mathrm{1}_{|v^{\prime}-v_{*}|\leq 1/\epsilon}|(\mu^{1/4}g)^{\prime} h_{*}f_{*}| dv_{*}dv^{\prime}, \nonumber
\end{eqnarray}
where $\int^{\pi/2}_{(\sqrt{2}|v^{\prime}-v_{*}|)^{-1}} \theta^{-3} d\theta \lesssim |v^{\prime}-v_{*}|^{2}$ is used.

On one hand, similar to the argument in (\ref{lessthanvstar2s}), we have
\begin{eqnarray}\label{i32lessthanvstar2s}
 |\ln \epsilon|^{-1}\int |v^{\prime}-v_{*}|^{\gamma+2}\mathrm{1}_{|v^{\prime}-v_{*}|\leq 1/\epsilon}|(\mu^{1/4}g)^{\prime}| dv^{\prime}
\lesssim |\ln \epsilon|^{-1}|\mu^{1/8}g|_{L^{2}}  \langle v_{*} \rangle^{\gamma+2}.
\end{eqnarray}
On the other hand, if $|v_{*}|\geq 2/\epsilon$, then $|v^{\prime}| \geq |v_{*}| - |v^{\prime}-v_{*}| \geq |v_{*}|/2 \geq 1/\epsilon$, which gives $\mu^{\prime} \lesssim \mu_{*}^{1/4} \lesssim e^{-1/2\epsilon^{2}}$. Then we deduce that
\begin{eqnarray}\label{i32lessthanep2s}
&&
|\ln \epsilon|^{-1}\mathrm{1}_{|v_*|\geq 2/ \epsilon}\int |v^{\prime}-v_{*}|^{\gamma+2}\mathrm{1}_{|v^{\prime}-v_{*}|\leq 1/\epsilon}|(\mu^{1/4}g)^{\prime}| dv^{\prime}
\\&\lesssim& |\ln \epsilon|^{-1}\mathrm{1}_{|v_*|\geq 2/ \epsilon}|\mu^{1/8}g|_{L^{2}}  \big(\int |v^{\prime}-v_{*}|^{2\gamma+4}\mathrm{1}_{|v^{\prime}-v_{*}|\leq 1/\epsilon}(\mu^{1/4})^{\prime}dv^{\prime}\big)^{1/2} \nonumber
\\&\lesssim& |\ln \epsilon|^{-1}\mathrm{1}_{|v_*|\geq 2/ \epsilon}|\mu^{1/8}g|_{L^{2}} \mu_{*}^{1/64} (\epsilon^{-1})^{\gamma+2+3/2} e^{-1/32\epsilon^{2}} \nonumber
\lesssim |\ln \epsilon|^{-1}\mathrm{1}_{|v_*|\geq 2/ \epsilon}|\mu^{1/8}g|_{L^{2}} \mu_{*}^{1/64}. \nonumber
\end{eqnarray}
With estimates (\ref{i32lessthanvstar2s}) and  (\ref{i32lessthanep2s}) in hand, recalling \eqref{property-of-cf}, we have
$
|\mathcal{I}_{3,2}| \lesssim |\mu^{1/8}g|_{L^{2}}|W^{\epsilon}h|_{L^{2}_{\gamma/2}}|W^{\epsilon}f|_{L^{2}_{\gamma/2}}.
$

\underline{Estimate of $\mathcal{I}_{3,3}$.}
By Taylor expansion, one has
\beno
\mu^{ 1/4} - (\mu^{1/4})^{\prime} = (\nabla \mu^{ 1/4})(v^{\prime})\cdot(v-v^{\prime}) + \frac{1}{2}\int_{0}^{1} (1-\kappa) [(\nabla^{2} \mu^{ 1/4})(v(\kappa)):(v-v^{\prime})\otimes(v-v^{\prime})] d\kappa,
\eeno
where $v(\kappa) = v^{\prime} + \kappa(v-v^{\prime})$.
For any fixed $v_{*}$, there holds
\beno
\int b^{\epsilon}(\cos\theta)|v-v_{*}|^{\gamma} \mathrm{1}_{|v-v_{*}| \leq 1/\epsilon, \epsilon \leq \theta \leq |v-v_{*}|^{-1}}(\nabla \mu^{ 1/4})(v^{\prime})\cdot(v-v^{\prime})(\mu^{1/4}g)^{\prime}  d\sigma dv = 0.
\eeno
By the change of variable $v \rightarrow v^{\prime}$, the fact $|v^{\prime}-v_{*}|\geq |v-v_{*}| /\sqrt{2}$ and $|\nabla^{2} \mu^{ 1/4}|_{L^{\infty}} \lesssim 1$, we have
\beno
|\mathcal{I}_{3,3}| &=&  \frac{1}{2}|\int_{E_{3}\times[0,1]} b^{\epsilon}(\cos\theta)|v-v_{*}|^{\gamma}\mathrm{1}_{|v-v_{*}| \leq 1/\epsilon, \epsilon \leq \theta \leq |v-v_{*}|^{-1}}
\\&&\times (1-\kappa)[(\nabla^{2} \mu^{ 1/4})(v(\kappa)):(v-v^{\prime})\otimes(v-v^{\prime})] (\mu^{1/4}g)^{\prime} h_{*}f_{*} d\kappa d\sigma dv_{*}dv|
\\&\lesssim& \int b^{\epsilon}(\cos\theta)|v^{\prime}-v_{*}|^{\gamma+2}\theta^{2}\mathrm{1}_{|v^{\prime}-v_{*}| \leq 1/\epsilon, \epsilon \leq \theta \leq |v^{\prime}-v_{*}|^{-1}}|(\mu^{1/4}g)^{\prime} h_{*}f_{*}|  d\sigma dv_{*}dv^{\prime}
\\&\lesssim& |\ln \epsilon|^{-1}\int \left( |\ln \epsilon| - \ln|v^{\prime}-v_{*}|
\right)|v^{\prime}-v_{*}|^{\gamma+2}\mathrm{1}_{|v^{\prime}-v_{*}| \leq 1/\epsilon}|(\mu^{1/4}g)^{\prime} h_{*}f_{*}|   dv_{*}dv^{\prime}.
\eeno
We claim that
\beno
|\ln \epsilon|^{-1}\int ( |\ln \epsilon| - \ln|v-v_{*}| ) |v-v_{*}|^{\gamma+2}\mathrm{1}_{|v-v_{*}| \leq 1/\epsilon}|\mu^{1/4}g| dv  \lesssim  (W^{\epsilon})^{2}_{*} \langle v_{*} \rangle^{\gamma} |\mu^{1/8}g|_{L^{2}},
\eeno
which immediately gives  $|\mathcal{I}_{3,3}| \lesssim|\mu^{1/8}g|_{L^{2}}|W^{\epsilon}h|_{L^{2}_{\gamma/2}}|W^{\epsilon}f|_{L^{2}_{\gamma/2}}$. By Cauchy-Schwartz inequality,
it suffices to prove
\ben \label{definition-K}
K(v_{*}):=|\ln \epsilon|^{-1} \big(\int ( |\ln \epsilon| - \ln|v-v_{*}| )^{2} |v-v_{*}|^{2\gamma+4}\mathrm{1}_{|v-v_{*}| \leq 1/\epsilon} \mu^{1/4} dv\big)^{1/2}  \lesssim  (W^{\epsilon})^{2}_{*} \langle v_{*} \rangle^{\gamma}.
\een

\underline{\it Case 1: $|v_{*}| \geq 2/\epsilon$.} Since $|v-v_{*}| \leq 1/\epsilon$, we have $|v-v_{*}| \leq |v_{*}|/2$ and thus $|v| \geq |v_{*}|-|v-v_{*}| \geq |v_{*}|/2$. Then we get $\mu \lesssim \mu^{1/4}_{*}$. On the other hand, since $|v|^{2} \geq \frac{1}{2}|v-v_{*}|^{2}-|v_{*}|^{2}$, we have
\ben \label{shift-to-v-minus-vStar}
\mu \lesssim \mu^{1/2}(v-v_{*}) \mu^{-1}_{*},
\een which gives
$
\mu^{1/4}\mathrm{1}_{|v_{*}| \geq 2/\epsilon, |v-v_{*}| \leq 1/\epsilon}
\lesssim \mu^{1/128}(v-v_{*}) \mu^{1/64}_{*}\mathrm{1}_{|v_{*}| \geq 2/\epsilon, |v-v_{*}| \leq 1/\epsilon}.
$
Plugging which into $K(v_{*})$, we get
\beno
K(v_{*})\mathrm{1}_{|v_{*}| \geq 2/\epsilon} &\lesssim& \mu^{1/128}_{*} |\ln \epsilon|^{-1} \big(\int ( |\ln \epsilon| - \ln|v-v_{*}| )^{2} |v-v_{*}|^{2\gamma+4}\mathrm{1}_{|v-v_{*}| \leq 1/\epsilon} \mu^{1/128}(v-v_{*})  dv\big)^{1/2}
\\&=&  \mu^{1/128}_{*} |\ln \epsilon|^{-1} \big(\int ( |\ln \epsilon| - \ln|u| )^{2} |u|^{2\gamma+4}\mathrm{1}_{|u| \leq 1/\epsilon} \mu^{1/128}(u)  du \big)^{1/2}.
\eeno
We separate the integration domain   into two regions: $|u| \leq 1$ and $|u| \geq 1$.
\begin{itemize}
	\item
 For the part $|u| \leq 1$, we have
\beno
\int ( |\ln \epsilon| - \ln|u| )^{2} |u|^{2\gamma+4}\mathrm{1}_{|u| \leq 1} \mu^{1/128}(u)  du
&\leq& 2 \int (|\ln \epsilon|^{2} + \ln^{2}|u|) |u|^{2\gamma+4}\mathrm{1}_{|u| \leq 1}  du
\\&\leq& 2 (C_{1}(\gamma)|\ln \epsilon|^{2}+ C_{2}(\gamma))
\lesssim |\ln \epsilon|^{2}.
\eeno
where $C_{1}(\gamma)  := \int  |u|^{2\gamma+4}\mathrm{1}_{|u| \leq 1}   du $ and
$ \int  (\ln|u|)^{2} |u|^{2\gamma+4}\mathrm{1}_{|u| \leq 1}   du \leq  C \int   |u|^{2\gamma+7/2}\mathrm{1}_{|u| \leq 1}:=C_{2}(\gamma)$.
 \item For the part $|u| \ge 1$, since $\ln|u| \geq 0$ and $\gamma<0$
we have
\beno
\int ( |\ln \epsilon| - \ln|u| )^{2} |u|^{2\gamma+4}\mathrm{1}_{1/\epsilon \geq |u| \geq 1} \mu^{1/128}(u)  du
\leq |\ln \epsilon|^{2} \int |u|^{4}\mathrm{1}_{|u| \geq 1} \mu^{1/128}(u)  du
\lesssim |\ln \epsilon|^{2}.
\eeno
\end{itemize}
By these two estimates, we get
$
\mathrm{1}_{|v_{*}| \geq 2/\epsilon}K(v_{*}) \lesssim \mu^{1/128}_{*}  \lesssim (W^{\epsilon})^{2}_{*} \langle v_{*} \rangle^{\gamma}.
$

\underline{\it Case 2: $1 \leq |v_{*}| \leq 2/\epsilon$.} We separate the integration domain into  two regions: $|v-v_{*}| \leq |v_{*}|/2 \leq 1/\epsilon$ and $|v-v_{*}| \geq |v_{*}|/2$. Using $\sqrt{A+B} \leq \sqrt{A}  + \sqrt{B}$, we get
\beno
\mathrm{1}_{1 \leq |v_{*}| \leq 2/\epsilon}K(v_{*})\leq \mathrm{1}_{1 \leq |v_{*}| \leq 2/\epsilon}
|\ln \epsilon|^{-1} \big(\int ( |\ln \epsilon| - \ln|v-v_{*}| )^{2} |v-v_{*}|^{2\gamma+4}\mathrm{1}_{|v-v_{*}| \leq |v_{*}|/2} \mu^{1/4} dv\big)^{1/2}
\\+\mathrm{1}_{1 \leq |v_{*}| \leq 2/\epsilon}
|\ln \epsilon|^{-1} \big(\int ( |\ln \epsilon| - \ln|v-v_{*}| )^{2} |v-v_{*}|^{2\gamma+4}\mathrm{1}_{|v_{*}|/2  \leq |v-v_{*}| \leq 1/\epsilon} \mu^{1/4} dv\big)^{1/2}
\\:= K_{1}(v_{*}) + K_{2}(v_{*}).
\eeno
When $|v-v_{*}| \leq |v_{*}|/2$, we can follow the computation in {\it Case 1}  to get
$ K_{1}(v_{*})  \lesssim \mu^{1/128}_{*}.$
When $|v-v_{*}| \geq |v_{*}|/2$, then $|\ln \epsilon| - \ln|v-v_{*}| \leq |\ln \epsilon| - \ln|v_{*}| + \ln 2$, we get
\beno
K_{2}(v_{*})&\leq&
|\ln \epsilon|^{-1} \big( |\ln \epsilon| - \ln|v_{*}| + \ln 2 \big) \big(\int  |v-v_{*}|^{2\gamma+4}\mathrm{1}_{|v_{*}|/2  \leq |v-v_{*}| \leq 1/\epsilon} \mu^{1/4} dv\big)^{1/2}
\\&\lesssim&
|\ln \epsilon|^{-1} \big( |\ln \epsilon| - \ln|v_{*}| + \ln 2 \big) \langle v_{*} \rangle^{\gamma+2}.
\eeno
By  $\mu^{1/128}_{*} \lesssim \langle v_{*} \rangle^{\gamma} $ and $\mathrm{1}_{1 \leq |v_{*}| \leq 2/\epsilon}|\ln \epsilon|^{-1} \big( |\ln \epsilon| - \ln|v_{*}| + \ln 2 \big) \langle v_{*} \rangle^{2} \lesssim (W^{\epsilon})^{2}_{*}$, we get
$
\mathrm{1}_{1 \leq |v_{*}| \leq 2/\epsilon} K(v_{*}) \lesssim (W^{\epsilon})^{2}_{*} \langle v_{*} \rangle^{\gamma}.
$

\underline{\it Case 3: $|v_{*}| \leq 1$.} By \eqref{shift-to-v-minus-vStar}, we have
$
\mu \lesssim \mu^{1/2}(v-v_{*}).
$
Plugging which into $K(v_{*})$, we have
\beno \mathrm{1}_{|v_{*}| \leq 1}
K(v_{*}) &\lesssim& |\ln \epsilon|^{-1} \big(\int ( |\ln \epsilon| - \ln|v-v_{*}| )^{2} |v-v_{*}|^{2\gamma+4}\mathrm{1}_{|v-v_{*}| \leq 1/\epsilon} \mu^{1/8}(v-v_{*})  dv\big)^{1/2}
\\&=&  |\ln \epsilon|^{-1} \big(\int ( |\ln \epsilon| - \ln|u| )^{2} |u|^{2\gamma+4}\mathrm{1}_{|u| \leq 1/\epsilon} \mu^{1/8}(u)  du \big)^{1/2}.
\eeno
By the computation in {\it Case 1}, we get
$
\mathrm{1}_{|v_{*}| \leq 1} K(v_{*}) \lesssim  \mathrm{1}_{|v_{*}| \leq 1} \lesssim \mathrm{1}_{|v_{*}| \leq 1}(W^{\epsilon})^{2}_{*} \langle v_{*} \rangle^{\gamma}.
$

Patching together the above three cases, we get \eqref{definition-K}.

By the above upper bounds of $\mathcal{I}_{3,1}, \mathcal{I}_{3,2}$ and $\mathcal{I}_{3,3}$, we have
\beno
|\mathcal{I}_{3}| \lesssim |\mu^{1/8}g|_{L^{2}}|W^{\epsilon}h|_{L^{2}_{\gamma/2}}|W^{\epsilon}f|_{L^{2}_{\gamma/2}}.
\eeno

In the $|\beta|>0$ case,  $\mathcal{I}_{3}$ contains two parts. The first part involving $((\mu^{\f14}P_{\beta})_{*}^{\prime} - (\mu^{\f14}P_{\beta})_{*})(\mu^{1/4}g)_{*}$, by the change of variables $(v,v_{*}) \rightarrow (v^{\prime},v_{*}^{\prime})$ and   $(v,v_{*},\sigma) \rightarrow (v_{*},v,-\sigma)$, gives
\beno
\mathcal{I}_{3} =  \int b^{\epsilon}(\cos\theta)|v-v_{*}|^{\gamma}(\mu^{1/4}P_{\beta} - (\mu^{1/4} P_{\beta})^{\prime})(\mu^{1/4}g)^{\prime} h_{*}f_{*} d\sigma dv_{*}dv.
\eeno
With the same decomposition as above according to $E_{1}, E_{2}, E_{3}$, we have $\mathcal{I}_{3} = \mathcal{I}_{3,1} + \mathcal{I}_{3,2} + \mathcal{I}_{3,3}$.
In $\mathcal{I}_{3,1}, \mathcal{I}_{3,2}$, we can use $|\mu^{1/4}P_{\beta}| \lesssim 1$.
In $\mathcal{I}_{3,3}$, we can use $|\nabla^{2}\mu^{1/4}P_{\beta}| \lesssim 1$.
The second part involving $P_{\beta}\mu^{1/4}g$ can be dealt with in the same way as the above for the case $|\beta|=0$.

We end the proof by patching together the above estimates of $\mathcal{I}_{1}, \mathcal{I}_{2}$ and $\mathcal{I}_{3}$.
\end{proof}

\subsubsection{Upper bound of $ \Gamma^{\epsilon,\gamma,\eta}(g,h)$}
We are now ready to give the upper bound for the inner product $\langle \Gamma^{\epsilon,\gamma,\eta}(g,h; \beta), f\rangle$.
\begin{thm}\label{upGammagh-geq-eta}
For any  $0 <\eta \leq 1, 0 <\delta \leq 1/2$, and smooth functions $g,h$ and $f$, there holds for $s_{1}, s_{2} \geq 0$ with $s_{1}+s_{2} = 3/2+\delta$,
\beno
|\langle \Gamma^{\epsilon,\gamma,\eta}(g,h; \beta), f\rangle| \lesssim  \delta^{-1/2}|\mu^{1/12}g|_{H^{s_{1}}}|\mu^{1/12}h|_{H^{s_{2}}}|W^{\epsilon}f|_{L^{2}_{\gamma/2}}
+\eta^{\gamma-3}|g|_{L^{2}}|h|_{\epsilon,\gamma/2}|f|_{\epsilon, \gamma/2}.
\eeno
\end{thm}
\begin{proof} Recalling from \eqref{Gamma-ep-ga-geq-eta-into-IQ-inner-beta}, we have
\beno
\langle \Gamma^{\epsilon,\gamma,\eta}(g,h; \beta), f\rangle = \langle Q^{\epsilon,\gamma,\eta}(P_{\beta}\mu^{1/2}g,h), f\rangle + \langle I^{\epsilon,\gamma,\eta}(g,h;\beta), f\rangle.
\eeno
The theorem follows directly from Proposition \ref{ubqepsilonnonsingular} and Proposition \ref{upforI-ep-ga-et}.
\end{proof}
Taking $\delta = 1/2, s_{1}=2, s_{2}=0$ in Theorem \ref{upGammagh-geq-eta}, we have
\begin{col}\label{upgammamuff1-geq-eta}
For any  $0 <\eta \leq 1,$ and smooth functions $h$ and $f$, there holds
\beno
 |\langle \Gamma^{\epsilon,\gamma,\eta}(\partial_{\beta_{1}}\mu^{1/2},h; \beta_{0}), f\rangle| \lesssim \eta^{\gamma-3}|h|_{\epsilon,\gamma/2}|f|_{\epsilon, \gamma/2}.
\eeno
\end{col}

\subsection{Upper bound of  $\langle Q^{\epsilon,\gamma}_{\eta}(g,h), f\rangle$ and $ \langle I^{\epsilon,\gamma}_{\eta}(g,h), f\rangle$} We will provide two estimates for each functional. One allows us to make use of the smallness of $\eta$ later,  and the other is independent of $\eta$.

\subsubsection{Upper bound of $Q^{\epsilon,\gamma}_{\eta}$}
We give the  upper bound  of $ Q^{\epsilon,\gamma}_{\eta}$ in the following proposition.
\begin{prop}\label{ubqepsilon-singular} Let $\delta \in(0,1/2], \eta\in(0,1],a\in[0,1]$ and $(s_{3},s_{4})=(1/2+\delta,0)$ or $ (0, 1/2+\delta)$. Then for any smooth functions $g, h$ and $f$, the following estimates are valid.
\ben \label{with-small-eta}
|\langle Q^{\epsilon,\gamma}_{\eta}(g,h), f\rangle| &\lesssim& \delta^{-1/2} (\eta^{\delta}+\epsilon^{1/2})|\mu^{-2a}g|_{H^{3/2+\delta}}|\mu^{a/2}h|_{H^{1}}|W^{\epsilon}(D)\mu^{a/2}f|_{L^{2}}.
\\ \label{without-small-eta}
|\langle Q^{\epsilon,\gamma}_{\eta}(g,h), f\rangle| &\lesssim&  \delta^{-1/2}|\mu^{-2a}g|_{H^{s_{3}}}|\mu^{a/2}h|_{H^{1+s_{4}}}|W^{\epsilon}(D)\mu^{a/2}f|_{L^{2}}.
\een
\end{prop}

\begin{proof}  Set $G =\mu^{-2a}g, F=\mu^{a/2}f, H=\mu^{a/2}h$. By the definition of $Q^{\epsilon,\gamma}_{\eta}$,  we have
\beno
\langle Q^{\epsilon,\gamma}_{\eta}(g,h), f\rangle = \int B^{\epsilon,\gamma}_{\eta} (\mu^{2a}G)_{*} \mu^{-a/2}H ((\mu^{-a/2}F)^{\prime}-\mu^{-a/2}F) d\sigma dv dv_{*}
:= \mathcal{Y}(G,H,F).
\eeno
By the decomposition $F = \mathfrak{F}_{\phi}F + \mathfrak{F}^{\phi}F$ and $H = \mathfrak{F}_{\phi}H + \mathfrak{F}^{\phi}H$, we have
\beno \mathcal{Y}(G,H,F) = \mathcal{Y}(G,\mathfrak{F}_{\phi}H,\mathfrak{F}_{\phi}F)+\mathcal{Y}(G,\mathfrak{F}^{\phi}H,\mathfrak{F}_{\phi}F)
+\mathcal{Y}(G,H,\mathfrak{F}^{\phi}F). \eeno

{\it{Step 1: $\mathcal{Y}(G,H,\mathfrak{F}^{\phi}F)$}.}
In order to transfer regularity from $\mu^{-a/2}\mathfrak{F}^{\phi}F$ to $\mu^{-a/2}H$, we rearrange
\beno
\mathcal{Y}(G,H,\mathfrak{F}^{\phi}F) &=& \int B^{\epsilon,\gamma}_{\eta}  (\mu^{2a}G)_{*} \mu^{-a/2}H \big((\mu^{-a/2}\mathfrak{F}^{\phi}F)^{\prime}-\mu^{-a/2}\mathfrak{F}^{\phi}F \big) d\sigma dv dv_{*}
\\&=& \int B^{\epsilon,\gamma}_{\eta} (\mu^{2a}G)_{*} \big((\mu^{-a/2}H\mu^{-a/2}\mathfrak{F}^{\phi}F)^{\prime}-\mu^{-a/2}H\mu^{-a/2}\mathfrak{F}^{\phi}F \big) d\sigma dv dv_{*}
\\&&+\int B^{\epsilon,\gamma}_{\eta} (\mu^{2a}G)_{*} \big(\mu^{-a/2}H-(\mu^{-a/2}H)^{\prime}\big)(\mu^{-a/2}\mathfrak{F}^{\phi}F)^{\prime} d\sigma dv dv_{*}
\\&:=& \mathcal{Y}_{1}(G,H,\mathfrak{F}^{\phi}F) + \mathcal{Y}_{2}(G,H,\mathfrak{F}^{\phi}F).
\eeno

\underline{Estimate of  $\mathcal{Y}_{1}(G,H,\mathfrak{F}^{\phi}F)$.} We will give two results on it.

\begin{enumerate}\item
Using \eqref{0-delta-eta-small-eta-epsilon-hf-mu}  in Lemma \ref{cancellation-lemma-general-gamma-minus3-mu}, we have
\ben\label{cancellation-less-eta} |\mathcal{Y}_{1}(G,H,\mathfrak{F}^{\phi}F)| &\lesssim& (\eta + \epsilon^{1/2})|G|_{L^{\infty}} (|W^{\epsilon}(D)H|_{L^{2}}|\mathfrak{F}^{\phi}F|_{L^{2}}+|H|_{L^{2}}|W^{\epsilon}(D)\mathfrak{F}^{\phi}F|_{L^{2}})
\\&\lesssim& {\delta}^{-1/2}(\eta + \epsilon^{1/2})|G|_{H^{3/2+\delta}} |W^{\epsilon}(D)H|_{L^{2}}|W^{\epsilon}(D)\mathfrak{F}^{\phi}F|_{L^{2}}.
\nonumber \een
\item
Using \eqref{geq-delta-mu} in Lemma \ref{cancellation-lemma-general-gamma-minus3-mu}, we have
\beno |\mathcal{Y}_{1}(G,H,\mathfrak{F}^{\phi}F)|\lesssim |G|_{L^{p}} |H\mathfrak{F}^{\phi}F|_{L^{q}} \lesssim |G|_{L^{p}} |H|_{L^{r}}|\mathfrak{F}^{\phi}F|_{L^{2}},
\eeno
where $\frac{1}{p} + \frac{1}{r} = \frac{1}{2}$.
Taking $p,r=2,\infty$ or $p,r=3,6$, by Sololev imbedding one has
\beno |G|_{L^{2}} |H|_{L^{\infty}} \lesssim \delta^{-1/2}|G|_{H^{0}} |H|_{H^{3/2+\delta}} \text{ or }  |G|_{L^{3}} |H|_{L^{6}} \lesssim |G|_{H^{1/2}} |H|_{H^{1}}.\eeno
Therefore we have for $(s_{3},s_{4})=(1/2+\delta,0)$ or $(s_{3},s_{4})=(0, 1/2+\delta)$,
\ben \label{geq-delta-no-small-eta} |\mathcal{Y}_{1}(G,H,\mathfrak{F}^{\phi}F)| \lesssim \delta^{-1/2}|G|_{H^{s_{3}}} |H|_{H^{1+s_{4}}}|\mathfrak{F}^{\phi}F|_{L^{2}}.
\een
\end{enumerate}

\underline{Estimate of $\mathcal{Y}_{2}(G,H,\mathfrak{F}^{\phi}F)$.}
By Taylor expansion up to order 1,
\beno  |(\mu^{-a/2}H)^{\prime}-\mu^{-a/2}H|  = |\int_{0}^{1}(\nabla (\mu^{-a/2}H))(v(\kappa))\cdot (v^{\prime}-v)d\kappa| \lesssim
\theta|v-v_{*}|\int_{0}^{1}|(\nabla (\mu^{-a/2}H))(v(\kappa))|d\kappa ,\eeno
and the fact
\ben \label{mu-H-nabla}
|\nabla (\mu^{-a/2}H)| &=& |\mu^{-a/2}\nabla H + H \nabla \mu^{-a/2}| \lesssim  (1+a|v|)\mu^{-a/2}(|\nabla H| + |H|)
\\&\lesssim&  (1+a)\mu^{-a}(|\nabla H| + |H|), \nonumber
 \een
we get
\beno \mathcal{Y}_{2}(G,H,\mathfrak{F}^{\phi}F) &\lesssim& (1+a)
\int B^{\epsilon,\gamma}_{\eta} \theta|v-v_{*}| |(\mu^{2a}G)_{*}|\mu^{-a}(v(\kappa))
\\&&\times (|\nabla H (v(\kappa))| + |H(v(\kappa))|)
|(\mu^{-a/2}\mathfrak{F}^{\phi}F)^{\prime}| d\sigma dv dv_{*} d\kappa.
 \eeno
Thanks to $|v_{*}-v(\kappa)| \leq 1$ and $|v_{*}-v^{\prime}| \leq 1$,
we can apply \eqref{to-3-4} to get
$ \mu^{2a}(v_{*}) \mu^{-a}(v(\kappa)) \mu^{-a/2}(v^{\prime}) \leq \mathrm{e}^{3a}, $
which gives
\beno \mathcal{Y}_{2}(G,H,\mathfrak{F}^{\phi}F) \leq C(1+a)\mathrm{e}^{3a}
\int B^{\epsilon,\gamma}_{\eta} \theta|v-v_{*}| |G_{*}|(|\nabla H (v(\kappa))| + |H(v(\kappa))|)
|(\mathfrak{F}^{\phi}F)^{\prime}| d\sigma dv dv_{*} d\kappa.
 \eeno
By Cauchy-Schwartz inequality, by the change \eqref{change-exact-formula}-\eqref{change-Jacobean-bound},  $\int_{\epsilon}^{\pi/2} \theta^{-2} d\theta \lesssim \epsilon^{-1}$, and $|v-v_{*}|^{\gamma+1} \leq |v-v_{*}|^{-2}$ when $|v-v_{*}| \leq 1$,
we get
\ben \label{cauchy-change-integration}
|\mathcal{Y}_{2}(G,H,\mathfrak{F}^{\phi}F)| &\lesssim& |\ln \epsilon|^{-1} \epsilon^{-1} \big(\int \mathrm{1}_{|v-v_{*}| \leq \eta} |v-v_{*}|^{-1-2\delta} G_{*}^2
(|\nabla H|^{2}+H^{2}) dv dv_{*}\big)^{1/2}
\\&& \times\big(\int \mathrm{1}_{|v-v_{*}| \leq \eta} |v-v_{*}|^{-3+2\delta}
|\mathfrak{F}^{\phi}F|^{2} dv dv_{*}\big)^{1/2}
.
\nonumber\een

It is easy to check that for $\delta\in(0,1/2)$, and for $(s_{3},s_{4})=(1/2+\delta,0)$ or $(s_{3},s_{4})=(0, 1/2+\delta)$,  by Hardy inequality and Hardy-Littlewood-Sobolev inequality,
\ben\label{Y2ghf1} \int_{|v-v_*|\le1}|v-v_{*}|^{-1-2\delta} G_{*}^2
(|\nabla H|^{2}+H^{2}) dv dv_{*}&\lesssim& |G|_{H^{s_3}}^2|H|_{H^{1+s_4}}^2,\\
\int \mathrm{1}_{|v-v_{*}| \leq \eta} |v-v_{*}|^{-3+2\delta}
|\mathfrak{F}^{\phi}F|^{2} dv dv_{*}&\lesssim& \delta^{-1}\eta^{2\delta}|\mathfrak{F}^{\phi}F|_{L^2}^2,\label{Y2ghf2}
\een
which yields
\ben \label{all-low-with-small-eta-Y2} |\mathcal{Y}_{2}(G,H,\mathfrak{F}^{\phi}F)| \lesssim
\delta^{-1/2} \eta^{\delta}|\ln \epsilon|^{-1/2}|G|_{H^{s_{3}}}|H|_{H^{1+s_{4}}}|W^{\epsilon}(D)F|_{L^{2}}
.\een

Patching \eqref{cancellation-less-eta} and \eqref{all-low-with-small-eta-Y2} together, we get
\ben \label{all-low-with-small-eta} |\mathcal{Y}(G,H,\mathfrak{F}^{\phi}F)|
&\lesssim& \delta^{-1/2} \eta^{\delta}|G|_{H^{s_{3}}}|H|_{H^{1+s_{4}}}|W^{\epsilon}(D)F|_{L^{2}} \\&&+ \delta^{-1/2} (\eta + \epsilon^{1/2})|G|_{H^{3/2+\delta}} |W^{\epsilon}(D)H|_{L^{2}}|W^{\epsilon}(D)F|_{L^{2}}
. \nonumber \een
Patching \eqref{geq-delta-no-small-eta} and \eqref{all-low-with-small-eta-Y2} together, for $(s_{3},s_{4})=(1/2+\delta,0)$ or $(s_{3},s_{4})=(0, 1/2+\delta)$, we get
\ben \label{all-low-without-small-eta} |\mathcal{Y}(G,H,\mathfrak{F}^{\phi}F)|
\lesssim \delta^{-1/2} |G|_{H^{s_{3}}}|H|_{H^{1+s_{4}}}|W^{\epsilon}(D)F|_{L^{2}}
. \een

{\it{Step 2: $\mathcal{Y}(G,\mathfrak{F}^{\phi}H,\mathfrak{F}_{\phi}F)$}.}  Recall
\beno
\mathcal{Y}(G,\mathfrak{F}^{\phi}H,\mathfrak{F}_{\phi}F) = \int B^{\epsilon,\gamma}_{\eta}  (\mu^{2a}G)_{*} (\mu^{-a/2}\mathfrak{F}^{\phi}H) ((\mu^{-a/2}\mathfrak{F}_{\phi}F)^{\prime}-\mu^{-a/2}\mathfrak{F}_{\phi}F) d\sigma dv dv_{*} .
\eeno
The analysis of term $\mathcal{Y}(G,\mathfrak{F}^{\phi}H,\mathfrak{F}_{\phi}F)$ is similar to that of $\mathcal{Y}_{2}(G,H,\mathfrak{F}^{\phi}F)$ in {\it Step 1}. In this step, we can apply Taylor expansion to function $\mu^{-a/2}\mathfrak{F}_{\phi}F$. Then similar to \eqref{cauchy-change-integration}, we will get
\ben \label{cauchy-change-integration-high-H-low-F}
|\mathcal{Y}(G,\mathfrak{F}^{\phi}H,\mathfrak{F}_{\phi}F)| &\lesssim& |\ln \epsilon|^{-1} \epsilon^{-1} \big(\int \mathrm{1}_{|v-v_{*}| \leq \eta} |v-v_{*}|^{-1-2\delta} G_{*}^2
|\mathfrak{F}^{\phi}H|^{2} dv dv_{*}\big)^{1/2}
\\&& \times\big(\int \mathrm{1}_{|v-v_{*}| \leq \eta} |v-v_{*}|^{-3+2\delta}   (|\nabla \mathfrak{F}_{\phi}F|^{2}+|\mathfrak{F}_{\phi}F|^{2}) dv dv_{*}\big)^{1/2}
.\nonumber\een

Thanks to \eqref{Y2ghf1} and \eqref{Y2ghf2},
 for $(s_{3},s_{4})=(1/2+\delta,0)$ or $(s_{3},s_{4})=(0, 1/2+\delta)$, we have
\ben \label{high-low-with-or-without-eta} |\mathcal{Y}(G,\mathfrak{F}^{\phi}H,\mathfrak{F}_{\phi}F)|
&\lesssim& \delta^{-1/2} \eta^{\delta}|\ln \epsilon|^{-1} \epsilon^{-1}|G|_{H^{s_{3}}}|\mathfrak{F}^{\phi}H|_{H^{s_{4}}}|\mathfrak{F}_{\phi}F|_{H^{1}}
\\ &\lesssim& \delta^{-1/2} \eta^{\delta}|G|_{H^{s_{3}}}|H|_{H^{1+s_{4}}}|W^{\epsilon}(D)F|_{L^{2}}, \nonumber
\een
where we use \eqref{high-frequency-lb-cf}, \eqref{low-frequency-lb-cf} and \eqref{upper-bound-cf} in the last line.

{\it{Step 3: $\mathcal{Y}(G,\mathfrak{F}_{\phi}H,\mathfrak{F}_{\phi}F)$}.} We  make dyadic decomposition in the  frequency space and get
\beno \mathcal{Y}(G,\mathfrak{F}_{\phi}H,\mathfrak{F}_{\phi}F) &=& \sum_{j,k=-1}^{\infty} \mathcal{Y}(G,\mathfrak{F}_{j}\mathfrak{F}_{\phi}H, \mathfrak{F}_{k}\mathfrak{F}_{\phi}F)
\\&=& \sum_{-1\leq j\leq k \lesssim |\ln \epsilon| }\mathcal{Y}(G,\mathfrak{F}_{j}\mathfrak{F}_{\phi}H, \mathfrak{F}_{k}\mathfrak{F}_{\phi}F) +
\sum_{-1\leq k < j \lesssim |\ln \epsilon| } \mathcal{Y}(G,\mathfrak{F}_{j}\mathfrak{F}_{\phi}H, \mathfrak{F}_{k}\mathfrak{F}_{\phi}F).
\eeno
For simplicity, set $H_{j}:=\mathfrak{F}_{j}\mathfrak{F}_{\phi}H, F_{k}:=\mathfrak{F}_{k}\mathfrak{F}_{\phi}F$.

\underline{\it Case 1: $k<j$.} We will apply Taylor expansion to $\mu^{-a/2}F_{k}$.
Note that
\beno
\mathcal{Y}(G,H_{j},F_{k}) &=& \int B^{\epsilon,\gamma}_{\eta} \phi(\sin(\theta/2)/2^{k}) (\mu^{2a}G)_{*} \mu^{-a/2} H_{j} ((\mu^{-a/2}F_{k})^{\prime}-\mu^{-a/2}F_{k}) d\sigma dv dv_{*}
\\&&+ \int B^{\epsilon,\gamma}_{\eta} (1-\phi(\sin(\theta/2)/2^{k})) (\mu^{2a}G)_{*} \mu^{-a/2} H_{j} ((\mu^{-a/2}F_{k})^{\prime}-\mu^{-a/2}F_{k}) d\sigma dv dv_{*}
\\&:=&\mathcal{X}_{1}(G,H_{j},F_{k}) +\mathcal{X}_{2}(G,H_{j},F_{k}).
\eeno

\underline{Estimate of $\mathcal{X}_{1}(G,H_{j},F_{k})$.}
We remind the reader that in this case $\epsilon \lesssim \theta \lesssim 2^{-k}$.
By Taylor expansion up to order 2,
\beno (\mu^{-a/2}F_{k})^{\prime}-\mu^{-a/2}F_{k} &=& (\nabla (\mu^{-a/2}F_{k})) \cdot (v^{\prime}-v)
\\&&+ \int_{0}^{1} \frac{1-\kappa}{2}  (\nabla^{2}(\mu^{-a/2}F_{k}))(v(\kappa)):(v^\prime - v)\otimes (v^\prime - v)
 d\kappa. \eeno
We have $\mathcal{X}_{1}(G,H_{j},F_{k})= \mathcal{X}_{1,1}(G,H_{j},F_{k}) + \mathcal{X}_{1,2}(G,H_{j},F_{k})$
according to the previous expansion with
\beno
\mathcal{X}_{1,1}(G,H_{j},F_{k})& :=& \int B^{\epsilon,\gamma}_{\eta} \phi(\sin(\theta/2)/2^{k}) (\mu^{2a}G)_{*} \mu^{-a/2} H_{j}
 (\nabla (\mu^{-a/2}F_{k})) \cdot (v^{\prime}-v) d\sigma dv dv_{*} ,
\\
\mathcal{X}_{1,2}(G,H_{j},F_{k})&:=& \int B^{\epsilon,\gamma}_{\eta} \phi(\sin(\theta/2)/2^{k}) (\mu^{2a}G)_{*} \mu^{-a/2} H_{j}
\\&&\times \frac{1-\kappa}{2} [ (\nabla^{2}(\mu^{-a/2}F_{k}))(v(\kappa)) :(v^\prime - v)\otimes (v^\prime - v) ]
 d\sigma dv dv_{*} d\kappa .
\eeno

{\it $\bullet$ Estimate of $\mathcal{X}_{1,1}(G,H_{j},F_{k})$.}
Plugging \eqref{around-mode-k} into $\mathcal{X}_{1,1}(G,H_{j},F_{k})$, we have
\beno
\mathcal{X}_{1,1}(G,H_{j},F_{k})&\lesssim& |\ln \epsilon|^{-1}(|\ln \epsilon|-k \ln 2 + 1) |\int 1_{|v - v_{*}|\leq \eta} |v - v_{*}|^{\gamma+1}(\mu^{2a}G)_{*} \mu^{-a/2} H_{j} |(\nabla \mu^{-a/2}F_{k})| dv dv_{*}
\\&\lesssim&  |\ln \epsilon|^{-1}(|\ln \epsilon|-k \ln 2 + 1) \int 1_{|v - v_{*}|\leq \eta}
|v - v_{*}|^{-2}G_{*}  H_{j} (|F_{k}|+|\nabla F_{k}|) dv dv_{*},
\eeno
where we use \eqref{mu-H-nabla} and \eqref{to-3-4}.  Thanks to \eqref{Y2ghf1} and \eqref{Y2ghf2}, for $(s_{3},s_{4})=(1/2+\delta,0)$ or $(s_{3},s_{4})=(0, 1/2+\delta)$, we have
\beno
\mathcal{X}_{1,1}(G,H_{j},F_{k}) \lesssim  |\ln \epsilon|^{-1}(|\ln \epsilon|-k \ln 2 + 1)2^{k} \delta^{-1/2}\eta^{\delta} |G|_{H^{s_3}}|H_{j}|_{H^{s_4}}|F_{k}|_{L^{2}}.
\eeno

{\it $\bullet$ Estimate of $\mathcal{X}_{1,2}(G,H_{j},F_{k})$.} By Cauchy-Schwartz inequality and the change \eqref{change-exact-formula}-\eqref{change-Jacobean-bound},
we get
\beno
\mathcal{X}_{1,2}(G,H_{j},F_{k})&\lesssim&  \int b^{\epsilon} \phi \left(\sin(\theta/2)/2^{k}\right)\theta^{2} 1_{|v - v_{*}|\leq \eta} |v - v_{*}|^{\gamma+2}
\\&&\times|(\mu^{2a}G)_{*} \mu^{-a/2} H_{j} (\nabla^{2} \mu^{-a/2}F_{k})(v(\kappa)) | d\sigma dv dv_{*} d\kappa
\\&\lesssim&  \big(   \int b^{\epsilon} \phi\left(\sin(\theta/2)/2^{k}\right)\theta^{2} 1_{|v - v_{*}|\leq \eta} |v - v_{*}|^{-1}(\mu^{2a}G)_{*} ( \mu^{-a/2} H_{j})^{2} d\theta dv dv_{*}  \big)^{1/2}
\\&&\times \big( \int b^{\epsilon} \phi\left(\sin(\theta/2)/2^{k}\right)\theta^{2} 1_{|v - v_{*}|\leq \eta} |v - v_{*}|^{-1}(\mu^{2a}G)_{*}|(\nabla^{2} \mu^{-a/2}F_{k})|^{2} d\theta dv dv_{*} \big)^{1/2}
\\&\lesssim&  |\ln \epsilon|^{-1}(|\ln \epsilon|-k \ln 2 + 1)
\big(   \int  1_{|v - v_{*}|\leq \eta} |v - v_{*}|^{-1-2\delta}|G_{*}|^2 H_{j}^{2} dv dv_{*} \big)^{1/2}
\\&&\times \big( \int 1_{|v - v_{*}|\leq \eta} |v - v_{*}|^{-1+2\delta}
(|\nabla^{2} F_{k}|^{2}+|\nabla F_{k}|^{2}+| F_{k}|^{2} ) dv dv_{*} \big)^{1/2},
\eeno
where we use the fact
$
|\nabla^{2} (\mu^{-a/2}F)|
\lesssim  (1+a^{2})\mu^{-a}(|\nabla^{2} F|+|\nabla F| + |F|)
$
and \eqref{to-3-4}.  Thanks to \eqref{Y2ghf1} and \eqref{Y2ghf2},
we get
\beno
\mathcal{X}_{1,2}(G,H_{j},F_{k}) \lesssim   |\ln \epsilon|^{-1}(|\ln \epsilon|-k \ln 2 + 1)2^{2k} \delta^{-1/2}\eta^{\delta} |G|_{H^{s_3}}|H_{j}|_{H^{s_4}}|F_{k}|_{L^{2}}.
\eeno

Patching together the estimates of $\mathcal{X}_{1,1}(G,H_{j},F_{k})$ and $\mathcal{X}_{1,2}(G,H_{j},F_{k}),$ we get
\beno  \mathcal{X}_{1}(G,H_{j},F_{k}) \lesssim   |\ln \epsilon|^{-1}(|\ln \epsilon|-k \ln 2 + 1)2^{2k}\delta^{-1/2}\eta^{\delta} |G|_{H^{s_{3}}}|H_{j}|_{H^{s_{4}}}|F_{k}|_{L^{2}}. \eeno

\underline{Estimate of $\mathcal{X}_{2}(G,H_{j},F_{k})$.} In this case, one has $\theta \gtrsim 2^{-k}$.
By Taylor expansion up to order 1,
\beno  |(\mu^{-a/2}F_{k})^{\prime}-\mu^{-a/2}F_{k}| &=& |\int_{0}^{1}(\nabla (\mu^{-a/2}F_{k}))(v(\kappa))\cdot (v^{\prime}-v)d\kappa|
\\&\lesssim&
\theta|v-v_{*}|\int_{0}^{1}|(\nabla (\mu^{-a/2}F_{k}))(v(\kappa))|d\kappa. \eeno
Plugging which into $\mathcal{X}_{2}(G,H_{j},F_{k})$, since $\int_{2^{-k}}^{\pi/2} \theta^{-2} d\theta \lesssim 2^{k}$, by similar computation as in \eqref{cauchy-change-integration}, we get
\beno
\mathcal{X}_{2}(G,H_{j},F_{k})&\lesssim&   |\ln \epsilon|^{-1}2^{k} \big(   \int  1_{|v - v_{*}|\leq \eta} |v - v_{*}|^{-1-2\delta}|G_{*}|^2 H_{j}^{2}  dv dv_{*}\big)^{1/2}
\\&&\times \big( \int 1_{|v - v_{*}|\leq \eta} |v - v_{*}|^{-3+2\delta} (|\nabla F_{k}|^{2}+| F_{k}|^{2} ) dv dv_{*} \big)^{1/2}.
\eeno
We conclude that  for $0 < \delta <1/2, (s_{3},s_{4})=(1/2+\delta,0)$ or $(s_{3},s_{4})=(0, 1/2+\delta)$,
\beno  \mathcal{X}_{2}(G,H_{j},F_{k}) \lesssim   |\ln \epsilon|^{-1}2^{2k} \delta^{-1/2}\eta^{\delta} |G|_{H^{s_{3}}}|H_{j}|_{H^{s_{4}}}|F_{k}|_{L^{2}}. \eeno

To summarize, we have when $k<j$,
\ben \label{k-less-j-low-low}  \mathcal{Y}(G,H_{j},F_{k}) \lesssim   |\ln \epsilon|^{-1}(1 + |\ln \epsilon|-k \ln 2)2^{2k} \delta^{-1/2}\eta^{\delta} |G|_{H^{s_{3}}}|H_{j}|_{H^{s_{4}}}|F_{k}|_{L^{2}}. \een

\underline{\it Case 2: $j \leq  k$.}
Note that
\beno
\mathcal{Y}(G,H_{j},F_{k}) &=& \int B^{\epsilon,\gamma}_{\eta}  (\mu^{2a}G)_{*} \mu^{-a/2}H_{j} \big((\mu^{-a/2}F_{k})^{\prime}-\mu^{-a/2}F_{k} \big) d\sigma dv dv_{*}
\\&=& \int B^{\epsilon,\gamma}_{\eta} (\mu^{2a}G)_{*} \big((\mu^{-a}H_{j}F_{k})^{\prime}-\mu^{-a}H_{j}F_{k} \big) d\sigma dv dv_{*}
\\&&+\int B^{\epsilon,\gamma}_{\eta} (\mu^{2a}G)_{*} \big(\mu^{-a/2}H_{j}-(\mu^{-a/2}H_{j})^{\prime}\big)(\mu^{-a/2}F_{k})^{\prime} d\sigma dv dv_{*}
\\&=& \mathcal{Y}_{1}(G,H_{j},F_{k}) + \mathcal{Y}_{2}(G,H_{j},F_{k}),
\eeno
where $\mathcal{Y}_{1}$ and $\mathcal{Y}_{2}$ are defined  in {\it Step 1}. Since  $\mathcal{Y}_{1}(G,H_{j},F_{k})$ is handled in {\it Step 1} and  $\mathcal{Y}_{2}(G,H_{j},F_{k})$ enjoys almost the same argument as that for $\mathcal{Y}(G,H_{j},F_{k})$ in the {\it Case 1} where  $k<j$, we conclude from \eqref{cancellation-less-eta}, \eqref{geq-delta-no-small-eta} and \eqref{k-less-j-low-low}
 that for
 $j \leq  k$,
\ben \label{with-small-eta-k-less-j} |\mathcal{Y}(G,H_{j},F_{k})| &\lesssim& |\ln \epsilon|^{-1}(1 + |\ln \epsilon|-j \ln 2)2^{2j} \delta^{-1/2}\eta^{\delta } |G|_{H^{s_{3}}}|H_{j}|_{H^{s_{4}}}|F_{k}|_{L^{2}} \\&&+ \delta^{-1/2}(\eta + \epsilon^{1/2})|G|_{H^{3/2+\delta}}|W^{\epsilon}(D)H_{j}|_{L^{2}}|W^{\epsilon}(D)F_{k}|_{L^{2}},  \nonumber  \\
 \label{without-small-eta-k-less-j}  |\mathcal{Y}(G,H_{j},F_{k})| &\lesssim& |\ln \epsilon|^{-1}(1 + |\ln \epsilon|-j \ln 2)2^{2j} \delta^{-1/2}\eta^{\delta } |G|_{H^{s_{3}}}|H_{j}|_{H^{s_{4}}}|F_{k}|_{L^{2}} \\&&+ \delta^{-1/2}|G|_{H^{s_{3}}} |H_{j}|_{H^{1+s_{4}}}|F_{k}|_{L^{2}}.\nonumber \een

By \eqref{k-less-j-low-low} and \eqref{with-small-eta-k-less-j}, we have
\ben \label{with-eta-small-low-low-final} |\mathcal{Y}(G,\mathfrak{F}_{\phi}H,\mathfrak{F}_{\phi}F) | &\lesssim& \delta^{-1/2}\eta^{\delta } |G|_{H^{s_{3}}}\sum_{-1\leq k < j \lesssim |\ln \epsilon| } 2^{2k} \frac{|\ln \epsilon|- k\ln 2 +1}{|\ln \epsilon|} |H_{j}|_{H^{s_{4}}} |F_{k}|_{L^{2}}
 \nonumber\\&&+\delta^{-1/2}\eta^{\delta } |G|_{H^{s_{3}}}\sum_{-1\leq j\leq k \lesssim |\ln \epsilon| } 2^{2j} \frac{|\ln \epsilon|- j\ln 2 +1}{|\ln \epsilon|} |H_{j}|_{H^{s_{4}}} |F_{k}|_{L^{2}} \nonumber
\\&&+ \delta^{-1/2}(\eta + \epsilon^{1/2})|G|_{H^{3/2+\delta}}\sum_{-1\leq j \leq k \lesssim |\ln \epsilon| } |W^{\epsilon}(D)H_{j}|_{L^{2}}|W^{\epsilon}(D)F_{k}|_{L^{2}} \nonumber
\\&\lesssim&  \delta^{-1/2}\eta^{\delta } |G|_{H^{s_{3}}}|W^{\epsilon}(D)H|_{H^{s_{4}}}|W^{\epsilon}(D)F|_{L^{2}}
\\&&+\delta^{-1/2}(\eta + \epsilon^{1/2})|G|_{H^{3/2+\delta}} |W^{\epsilon}(D)H|_{L^{2}}|W^{\epsilon}(D)F|_{L^{2}}. \nonumber
\een
Similarly, by \eqref{k-less-j-low-low} and \eqref{without-small-eta-k-less-j}, we have
\ben \label{without-eta-small-low-low-final} |\mathcal{Y}(G,\mathfrak{F}_{\phi}H,\mathfrak{F}_{\phi}F) | &\lesssim&  \delta^{-1/2}\eta^{\delta } |G|_{H^{s_{3}}}|W^{\epsilon}(D)H|_{H^{s_{4}}}|W^{\epsilon}(D)F|_{L^{2}} + \delta^{-1/2}|G|_{H^{s_{3}}} |H|_{H^{1+s_{4}}}|F|_{L^{2}} \nonumber
\\&\lesssim&  \delta^{-1/2}|G|_{H^{s_{3}}} |H|_{H^{1+s_{4}}}|W^{\epsilon}(D)F|_{L^{2}}.
\een

Patching all the estimates, we get the proposition. Indeed, patching together \eqref{all-low-with-small-eta}, \eqref{high-low-with-or-without-eta} and \eqref{with-eta-small-low-low-final}, we get \eqref{with-small-eta}.
Patching together \eqref{all-low-without-small-eta}, \eqref{high-low-with-or-without-eta} and \eqref{without-eta-small-low-low-final}, we get \eqref{without-small-eta}.
\end{proof}

\subsubsection{Upper bound of $Q^{\epsilon,\gamma}(g,h)$}
As a result of Proposition \ref{ubqepsilonnonsingular} and Proposition \ref{ubqepsilon-singular}, we get
\begin{thm}\label{Q-full-up-bound} Let $\delta \in(0,1/2], \eta\in(0,1],a\in[0,1]$ and $(s_{3},s_{4})=(1/2+\delta,0)$ or $ (0, 1/2+\delta)$. Then for any smooth functions $g, h$ and $f$, we have
\beno
 |\langle Q^{\epsilon,\gamma}(g,h), f\rangle| &\lesssim& \delta^{-1/2} (\eta^{\delta}+\epsilon^{1/2})|\mu^{-2a}g|_{H^{3/2+\delta}}|\mu^{a/2}h|_{H^{1}}|W^{\epsilon}(D)\mu^{a/2}f|_{L^{2}}
\\&&+\eta^{\gamma-3}|g|_{L^{1}_{|\gamma|+2}}|h|_{\epsilon,\gamma/2}|f|_{\epsilon,\gamma/2},
\\
|\langle Q^{\epsilon,\gamma}(g,h), f\rangle| &\lesssim& \delta^{-1/2}|\mu^{-2a}g|_{H^{s_{3}}}|\mu^{a/2}h|_{H^{1+s_{4}}}|W^{\epsilon}(D)\mu^{a/2}f|_{L^{2}}+
|g|_{L^{1}_{|\gamma|+2}}|h|_{\epsilon,\gamma/2}|f|_{\epsilon,\gamma/2}.
\eeno
\end{thm}

\subsubsection{Upper bound of $\langle I^{\epsilon,\gamma}_{\eta}(g,h), f\rangle$} We derive
\begin{prop} \label{I-less-eta-upper-bound}  Let $\delta\in(0,1/2], \eta\in(0,1]$ and $(s_{3},s_{4})=(1/2+\delta,0)$ or $ (0, 1/2+\delta)$. Then for any smooth functions $g, h$ and $f$, there holds
\beno \langle I^{\epsilon,\gamma}_{\eta}(g,h;\beta), f\rangle  \lesssim  \delta^{-1/2}\eta^{\delta}|\mu^{1/16}g|_{H^{s_{3}}}|W^{\epsilon}(D)\mu^{1/16}h|_{H^{s_{4}}}|W^{\epsilon}(D)\mu^{1/16}f|_{L^{2}}. \eeno
\end{prop}
\begin{proof}Let us consider the $\beta=0$ case since the following arguments also work when we replace $\mu^{1/2}$  with $P_{\beta}\mu^{1/2}$.
Recall
\ben \label{I-ep-ga-eta-ghf}
\langle I^{\epsilon,\gamma}_{\eta}(g,h), f\rangle =  \int B^{\epsilon,\gamma}_{\eta}((\mu^{1/2})_{*}^{\prime} - \mu_{*}^{1/2})g_{*} h f^{\prime} d\sigma dv_{*} dv.
\een
By setting $G = \mu^{1/16}g, H = \mu^{1/16}h, F = \mu^{1/16}f$, we have
\beno
\langle I^{\epsilon,\gamma}_{\eta}(g,h), f\rangle &=&  \int B^{\epsilon,\gamma}_{\eta}((\mu^{1/2})_{*}^{\prime} - \mu_{*}^{1/2})(\mu^{-1/16}G)_{*} \mu^{-1/16}H (\mu^{-1/16}F)^{\prime} d\sigma dv_{*} dv
\\&=&  \int B^{\epsilon,\gamma}_{\eta}((\mu^{1/2})_{*}^{\prime} - \mu_{*}^{1/2})(\mu^{-1/16}G)_{*} \big(\mu^{-1/16}H- (\mu^{-1/16}H)^{\prime}\big)(\mu^{-1/16}F)^{\prime} d\sigma dv_{*} dv
\\&&+  \int B^{\epsilon,\gamma}_{\eta}((\mu^{1/2})_{*}^{\prime} - \mu_{*}^{1/2})(\mu^{-1/16}G)_{*} (\mu^{-1/8}H F)^{\prime} d\sigma dv_{*} dv
:= \mathcal{A}(G, H, F) + \mathcal{B}(G, H, F)
.
\eeno

{\it Step 1: $\mathcal{A}(G, H, F)$.}
By the decomposition $F = \mathfrak{F}_{\phi}F + \mathfrak{F}^{\phi}F$ and $H = \mathfrak{F}_{\phi}H + \mathfrak{F}^{\phi}H$, we have
\beno \mathcal{A}(G, H, F) = \mathcal{A}(G,\mathfrak{F}_{\phi}H,\mathfrak{F}_{\phi}F) + \mathcal{A}(G,\mathfrak{F}^{\phi}H,\mathfrak{F}_{\phi}F) +\mathcal{A}(G,\mathfrak{F}_{\phi}H,\mathfrak{F}^{\phi}F) +\mathcal{A}(G,\mathfrak{F}^{\phi}H,\mathfrak{F}^{\phi}F).
\eeno
\underline{\it Step 1.1: low-high, high-low, high-high.}
By Taylor expansion, we get
\beno |(\mu^{1/2})^{\prime}_{*}-\mu^{1/2}_{*}| = |\int_{0}^{1} (\nabla \mu^{1/2})(v_{*}(\iota)) \cdot (v^{\prime}-v) d\iota| \lesssim
\theta|v-v_{*}| \int |\mu^{1/4}(v_{*}(\iota))| d\iota, \eeno
which yields
\beno |\mathcal{A}(G,H,F)| &\lesssim& \int B^{\epsilon,\gamma}_{\eta}\theta|v-v_{*}| \mu^{1/4}(v_{*}(\iota)) |(\mu^{-1/16}G)_{*} \big(\mu^{-1/16}H- (\mu^{-1/16}H)^{\prime}\big)(\mu^{-1/16}F)^{\prime}| d\sigma dv dv_{*} d\iota
\\&\lesssim&\int B^{\epsilon,\gamma}_{\eta}\theta|v-v_{*}| \mu^{1/4}(v_{*}(\iota)) \mu^{-1/16}(v_{*}) \mu^{-1/16}(v) \mu^{-1/16}(v^{\prime}) |G_{*}HF^{\prime}| d\sigma dv dv_{*} d\iota
\\&&+\int B^{\epsilon,\gamma}_{\eta}\theta|v-v_{*}|  \mu^{1/4}(v_{*}(\iota)) \mu^{-1/16}(v_{*}) \mu^{-1/8}(v^{\prime}) |G_{*}H^{\prime}F^{\prime}| d\sigma dv dv_{*} d\iota.
\eeno
Note that $|v-v_{*}| \leq \eta \leq 1$ yields $|v_{*}-v_{*}(\iota)| \leq 1, |v-v_{*}(\iota)| \leq 1, |v^{\prime}-v_{*}(\iota)| \leq 1$. Then by \eqref{to-3-4}, one has
\beno \mu^{1/4}(v_{*}(\iota)) \mu^{-1/16}(v_{*}) \mu^{-1/16}(v) \mu^{-1/16}(v^{\prime}) \lesssim 1, \mu^{1/4}(v_{*}(\iota)) \mu^{-1/16}(v_{*}) \mu^{-1/8}(v^{\prime}) \lesssim 1,\eeno
which yields
\beno |\mathcal{A}(G,H,F)| &\lesssim&\int B^{\epsilon,\gamma}_{\eta} \theta|v-v_{*}| |G_{*}|(|H|+|H^{\prime}|)|F^{\prime}| d\sigma dv dv_{*}
\\&\lesssim& \big(  \int B^{\epsilon,\gamma}_{\eta}\theta|v-v_{*}|^{2-2\delta} |G_{*}|^2 |H|^{2} d\sigma dv dv_{*}  \big)^{1/2}
\big( \int B^{\epsilon,\gamma}_{\eta}\theta|v-v_{*}|^{2\delta}|F|^{2} d\sigma dv dv_{*}  \big)^{1/2},\eeno
where we use Cauchy-Schwartz inequality and the change of variable $v \rightarrow v^{\prime}$. Since $\int_{\epsilon}^{\pi/2} \theta^{-2} d\theta \lesssim \epsilon^{-1}$,
thanks to \eqref{Y2ghf1} and  \eqref{Y2ghf2},   for $(s_{3},s_{4})=(1/2+\delta,0)$ or $(s_{3},s_{4})=(0, 1/2+\delta)$, we derive
\beno |\mathcal{A}(G,H,F)| \lesssim |\ln \epsilon|^{-1} \epsilon^{-1} \delta^{-1/2}\eta^{\delta}
|G|_{H^{s_3}} |H|_{H^{s_4}} |F|_{L^{2}}. \eeno
Taking $(H,F)=(\mathfrak{F}^{\phi}H,\mathfrak{F}_{\phi}F), $ or $(H,F)=(\mathfrak{F}_{\phi}H,\mathfrak{F}^{\phi}F),$ or $  (H,F)=(\mathfrak{F}^{\phi}H,\mathfrak{F}^{\phi}F), $ by \eqref{low-frequency-lb-cf} and \eqref{high-frequency-lb-cf},
we have
\beno | \mathcal{A}(G,\mathfrak{F}^{\phi}H,\mathfrak{F}_{\phi}F) +\mathcal{A}(G,\mathfrak{F}_{\phi}H,\mathfrak{F}^{\phi}F) +\mathcal{A}(G,\mathfrak{F}^{\phi}H,\mathfrak{F}^{\phi}F)| \lesssim  \delta^{-1/2}\eta^{\delta}|G|_{H^{s_3}} |W^{\epsilon}(D)H|_{H^{s_4}}|W^{\epsilon}(D)F|_{L^{2}}. \eeno

\underline{\it Step 1.2: low-low.} We  make dyadic decomposition in the  frequency space and get
\beno \mathcal{A}(G,\mathfrak{F}_{\phi}H,\mathfrak{F}_{\phi}F) &=& \sum_{j,k=-1}^{\infty} \mathcal{A}(G,\mathfrak{F}_{j}\mathfrak{F}_{\phi}H,\mathfrak{F}_{k}\mathfrak{F}_{\phi}F)
\\&=& \sum_{-1\leq j\leq k \lesssim |\ln \epsilon| } \mathcal{A}(G,\mathfrak{F}_{j}\mathfrak{F}_{\phi}H,\mathfrak{F}_{k}\mathfrak{F}_{\phi}F) +
\sum_{-1\leq k < j \lesssim |\ln \epsilon| } \mathcal{A}(G,\mathfrak{F}_{j}\mathfrak{F}_{\phi}H,\mathfrak{F}_{k}\mathfrak{F}_{\phi}F).
\eeno
For simplicity, let $H_{j}=\mathfrak{F}_{j}\mathfrak{F}_{\phi}H, F_{k}=\mathfrak{F}_{k}\mathfrak{F}_{\phi}F$.

\underline{{\it Case 1: $j\leq k$.}}
Let us first consider
$\mathcal{A}(G,H_{j},F_{k})$ for $-1\leq j\leq k \lesssim |\ln \epsilon| $. Recall
\beno \mathcal{A}(G,H_{j},F_{k}) = \int B^{\epsilon,\gamma}_{\eta}((\mu^{1/2})_{*}^{\prime} - \mu_{*}^{1/2})(\mu^{-1/16}G)_{*} \big(\mu^{-1/16}H_{j}- (\mu^{-1/16}H_{j})^{\prime}\big)(\mu^{-1/16}F_{k})^{\prime} d\sigma dv_{*} dv.
\eeno
By Taylor expansion up to order 1,
\beno |\mu^{-1/16}H_{j}-(\mu^{-1/16}H_{j})^{\prime}| = |\int (\nabla \mu^{-1/16}H_{j})(v(\kappa)) \cdot (v^{\prime}-v) d\kappa| \lesssim
\theta|v-v_{*}| \int |(\nabla \mu^{-1/16}H_{j})(v(\kappa))| d\kappa
 \\
 |(\mu^{1/2})^{\prime}_{*}-\mu^{1/2}_{*}| = |\int (\nabla \mu^{1/2})(v_{*}(\iota)) \cdot (v^{\prime}-v) d\iota| \lesssim
\theta|v-v_{*}| \int |(\nabla \mu^{1/2})(v_{*}(\iota))| d\iota. \eeno
From which together with $|\nabla \mu^{-1/16}H_{j}| \lesssim \mu^{-1/8}(|H_{j}|+|\nabla H_{j}|)$ and $|\nabla \mu^{1/2}| \lesssim \mu^{1/3}$, we get
\beno |\mathcal{A}(G,H_{j},F_{k})| &\lesssim& \int B^{\epsilon,\gamma}_{\eta}\theta^{2}|v-v_{*}|^{2} \mu^{1/3}(v_{*}(\iota))  \mu^{-1/8}(v(\kappa)) \mu^{-1/16}(v_{*})\mu^{-1/16}(v^{\prime})
\\&&\times|G_{*}|(|(\nabla H_{j})(v(\kappa))|+ |H_{j}(v(\kappa))|)|F_{k}^{\prime}| d\sigma dv_{*} dv d\kappa d\iota.
\eeno
By \eqref{to-3-4}, we have
$ \mu^{1/3}(v_{*}(\iota))  \mu^{-1/8}(v(\kappa)) \mu^{-1/16}(v_{*})\mu^{-1/16}(v^{\prime}) \lesssim 1, $ which gives
\beno |\mathcal{A}(G,H_{j},F_{k})| &\lesssim& \int B^{\epsilon,\gamma}_{\eta}\theta^{2}|v-v_{*}|^{2}
|G_{*}|(|(\nabla H_{j})(v(\kappa))|+ |H_{j}(v(\kappa))|)|F_{k}^{\prime}| d\sigma dv_{*} dv d\kappa
\\&\lesssim& \big(  \int B^{\epsilon,\gamma}_{\eta}\theta^{2}|v-v_{*}|^{2-2\delta} |G_{*}| (|\nabla H_{j}|^{2}+
|H_{j}|^{2}) d\sigma dv_{*} dv \big)^{1/2}
\\&&\times
\big( \int B^{\epsilon,\gamma}_{\eta}\theta^{2}|v-v_{*}|^{2+2\delta} |F_{k}|^{2} d\sigma dv_{*} dv \big)^{1/2}.
\eeno
where we use Cauchy-Schwartz inequality and the change \eqref{change-exact-formula}-\eqref{change-Jacobean-bound}.
Thanks to \eqref{Y2ghf1} and \eqref{Y2ghf2},  for $(s_{3},s_{4})=(1/2+\delta,0)$ or $(s_{3},s_{4})=(0, 1/2+\delta)$, we derive that
 \beno |\mathcal{A}(G,H_{j},F_{k})| \lesssim  \delta^{-1/2}\eta^{\delta} 2^j|G|_{H^{s_3}} | H_{j}|_{H^{s_4}} |F_{k}|_{L^{2}},\eeno
from which together with the fact
 $ 2^{j} \lesssim 2^{2j} \frac{|\ln \epsilon|- j\ln 2 +1}{|\ln \epsilon|}$,
  we arrive at
\ben \label{j-less-k-low-low-2}|\mathcal{A}(G,H_{j},F_{k})| \lesssim 2^{2j} \frac{|\ln \epsilon|- j\ln 2 +1}{|\ln \epsilon|} \delta^{-1/2}\eta^{\delta}|G|_{H^{s_3}} |H_{j}|_{H^{s_4}} |F_{k}|_{L^{2}}.\een

\underline{{\it Case 2: $ k < j$.}}
Let us now consider $\mathcal{A}(G,H_{j},F_{k})$ for $-1 \leq k < j \lesssim |\ln \epsilon| $.
Note that
\beno \mathcal{A}(G,H_{j},F_{k}) &=& \int B^{\epsilon,\gamma}_{\eta}((\mu^{1/2})_{*}^{\prime} - \mu_{*}^{1/2})(\mu^{-1/16}G)_{*} \big(\mu^{-1/16}H_{j}- (\mu^{-1/16}H_{j})^{\prime}\big)(\mu^{-1/16}F_{k})^{\prime} d\sigma dv_{*} dv
\\&=& \int B^{\epsilon,\gamma}_{\eta} (\mu^{-1/16}G)_{*}(\mu^{-1/8}H_{j}F_{k}-(\mu^{-1/8}H_{j}F_{k})^{\prime})((\mu^{1/2})^{\prime}_{*}-\mu^{1/2}_{*}) d\sigma dv_{*} dv
\\&&+ \int B^{\epsilon,\gamma}_{\eta} (\mu^{-1/16}G)_{*} \mu^{-1/16}H_{j}((\mu^{-1/16}F_{k})^{\prime} - \mu^{-1/16}F_{k})((\mu^{1/2})^{\prime}_{*}-\mu^{1/2}_{*}) d\sigma dv_{*} dv
\\&:=& \mathcal{A}_{1}(g,H_{j},F_{k})+\mathcal{A}_{2}(g,H_{j},F_{k}).
\eeno

{\it $\bullet$ Estimate of  $\mathcal{A}_{1}(g,H_{j},F_{k})$.} By Taylor expansion, one has
\ben \label{taylor-order-2-mu} (\mu^{1/2})^{\prime}_{*}-\mu^{1/2}_{*} = (\nabla \mu^{1/2})_{*} \cdot (v^{\prime}_{*}-v_{*}) + \int_{0}^{1} \frac{1-\kappa}{2}  (\nabla^{2}\mu^{1/2})(v_{*}(\kappa)):(v^\prime_{*} - v_{*})\otimes (v^\prime_{*} - v_{*}) d\kappa. \een
where $v_{*}(\kappa)  = v_{*} + \kappa (v^{\prime}_{*}-v_{*})$. Then $\mathcal{A}_{1}(g,H_{j},F_{k})= \mathcal{A}_{1,1}(g,H_{j},F_{k}) +\mathcal{A}_{1,2}(g,H_{j},F_{k})$,  where
\beno
\mathcal{A}_{1,1}(g,H_{j},F_{k})& := &\int B^{\epsilon,\gamma}_{\eta} (\mu^{-1/16}G)_{*}(\mu^{-1/8}H_{j}F_{k}-(\mu^{-1/8}H_{j}F_{k})^{\prime}) (\nabla \mu^{1/2})_{*} \cdot (v^{\prime}_{*}-v_{*}),
\\
\mathcal{A}_{1,2}(g,H_{j}{},F_{k})&:=& \int B^{\epsilon,\gamma}_{\eta} (\mu^{-1/16}G)_{*}(\mu^{-1/8}H_{j}F_{k}-(\mu^{-1/8}H_{j}F_{k})^{\prime}) \int_{0}^{1} \frac{1-\kappa}{2}  (\nabla^{2}\mu^{1/2})(v_{*}(\kappa))\\&&:(v^\prime_{*} - v_{*})\otimes (v^\prime_{*} - v_{*}) d\kappa.
\eeno
Using the fact $v^{\prime}_{*}-v_{*} = v - v^{\prime}$ and the identities (see \cite{alexandre2002boltzmann})
\ben &&\int  B^{\epsilon,\gamma}_{\eta}  f^{\prime}  (v^{\prime}-v)d\sigma dv =0,\label{dispear1} \\
 &&\int  B^{\epsilon,\gamma}_{\eta}  (v^{\prime}-v)d\sigma =\big(\int  B^{\epsilon,\gamma}_{\eta} \sin^2(\theta/2) d\sigma\big) (v_*-v), \label{dispear2}\een
we have
\beno |\mathcal{A}_{1,1}(g,H_{j},F_{k})|  &=& |\int B^{\epsilon,\gamma}_{\eta} (\mu^{-1/16}G)_{*}(\mu^{-1/8}H_{j}F_{k}) (\nabla \mu^{1/2})_{*} \cdot (v^{\prime}-v) d\sigma dv_{*} dv|
\\&=& |\int B^{\epsilon,\gamma}_{\eta} \sin^{2}\frac{\theta}{2}(\mu^{-1/16}G)_{*}(\mu^{-1/8}H_{j}F_{k})(\nabla \mu^{1/2})_{*} \cdot (v_{*}-v) d\sigma dv_{*} dv|
\\&\lesssim& \int  1_{|v - v_{*}|\leq \eta}|v - v_{*}|^{-2} |G_{*}H_{j}F_{k}| dv_{*} dv
 \lesssim  \delta^{-1/2}\eta^{\delta} | G|_{H^{s_3}}|H_{j}|_{H^{s_4}} |F_{k}|_{L^{2}},
\eeno
where we use \eqref{Y2ghf1} and \eqref{Y2ghf2}.  Similar to the estimate of $\mathcal{A}$ in {\it Case 1},
we get that
\beno |\mathcal{A}_{1,2}(g,H_{j},F_{k}) | &\lesssim& \int B^{\epsilon,\gamma}_{\eta}\theta^{2} |v-v_{*}|^{2} |(\mu^{-1/16}G)_{*}| (|\mu^{-1/8}H_{j}F_{k}| + |(\mu^{-1/8}H_{j}F_{k})^{\prime}|)
\\&&\times|(\nabla^{2}\mu^{1/2})(v_{*}(\kappa))| d\sigma dv_{*} dv d\kappa
\\&\lesssim& \int B^{\epsilon,\gamma}_{\eta}\theta^{2} |v-v_{*}|^{2} |G_{*}| (|H_{j}F_{k}| + |(H_{j}F_{k})^{\prime}|) d\sigma dv_{*} dv
\\&\lesssim& \int  1_{|v - v_{*}|\leq \eta}|v - v_{*}|^{-1} |G_{*}H_{j}F_{k}| dv_{*} dv
\lesssim \delta^{-1/2}\eta^{\delta}|G|_{H^{s_3}} |H_{j}|_{H^{s_4}}|F_{k}|_{L^{2}}.
\eeno
Patching together the previous two estimates, we have
$ |\mathcal{A}_{1}(g,H_{j},F_{k})|  \lesssim  \delta^{-1/2}\eta^{\delta} | G|_{H^{s_3}}|H_{j}|_{H^{s_4}} |F_{k}|_{L^{2}}.
$

{\it $\bullet$ Estimate of  $\mathcal{A}_{2}(g,H_{j},F_{k})$.}  Similar to the idea used to estimate $\mathcal{A}(G,F_{k},H_{j})$ in {\it Case 1},  we apply Taylor expansion to
$\mu^{-1/16}F_{k}$ and get
\beno |\mathcal{A}_{2}(g,H_{j},F_{k})| \lesssim 2^{2k} \frac{|\ln \epsilon|- k\ln 2 +1}{|\ln \epsilon|} \delta^{-1/2}\eta^{\delta}|G|_{H^{s_3}} |H_{j}|_{H^{s_4}} |F_{k}|_{L^{2}}. \eeno

Patching together the estimates of $\mathcal{A}_{1}(g,H_{j},F_{k})$ and $\mathcal{A}_{2}(g,H_{j},F_{k})$, we have for $-1 \leq k < j \lesssim |\ln \epsilon| $,
\ben \label{k-less-j-low-low-2} |\mathcal{A}(G,H_{j},F_{k})| &\lesssim& 2^{2k} \frac{|\ln \epsilon|- k\ln 2 +1}{|\ln \epsilon|} \delta^{-1/2}\eta^{\delta}|G|_{H^{s_3}} |H_{j}|_{H^{s_4}} |F_{k}|_{L^{2}}
\\&&+\delta^{-1/2}\eta^{\delta} | G|_{H^{s_3}}|H_{j}|_{H^{s_4}} |F_{k}|_{L^{2}}. \nonumber \een
Patching together \eqref{j-less-k-low-low-2} and \eqref{k-less-j-low-low-2}, we have
\beno |\mathcal{A}(G,\mathfrak{F}_{\phi}H,\mathfrak{F}_{\phi}F)| &\lesssim& \delta^{-1/2}\eta^{\delta} \sum_{-1\leq j\leq k \lesssim |\ln \epsilon| }2^{2j} \frac{|\ln \epsilon|- j\ln 2 +1}{|\ln \epsilon|} |G|_{H^{s_3}} |H_{j}|_{H^{s_4}} |F_{k}|_{L^{2}}\\&&+
\delta^{-1/2}\eta^{\delta} \sum_{-1\leq k < j \lesssim |\ln \epsilon| } 2^{2k} \frac{|\ln \epsilon|- k\ln 2 +1}{|\ln \epsilon|} |G|_{H^{s_3}} |H_{j}|_{H^{s_4}} |F_{k}|_{L^{2}}
\\&&+
\delta^{-1/2}\eta^{\delta} \sum_{-1\leq k < j \lesssim |\ln \epsilon| }   | G|_{H^{s_3}}|H_{j}|_{H^{s_4}} |F_{k}|_{L^{2}}
\\&\lesssim&   \delta^{-1/2}\eta^{\delta}|G|_{H^{s_{3}}} |W^{\epsilon}(D)H|_{H^{s_{4}}}|W^{\epsilon}(D)F|_{L^{2}},
\eeno
for $(s_{3},s_{4})=(1/2+\delta,0)$ or $(s_{3},s_{4})=(0, 1/2+\delta)$.

{\it Step 2: Estimate of $\mathcal{B}(G, H, F)$.} Recalling \eqref{taylor-order-2-mu},
thanks to \eqref{dispear1}, similar to the estimate of $\mathcal{A}$ in {\it Case 1}, by the change of variable $v \rightarrow v^{\prime}$, we get
\beno |\mathcal{B}(G, H, F)| &\leq&  |\int B^{\epsilon,\gamma}_{\eta} (\mu^{-1/16}G)_{*}(\mu^{-1/8}H F)^{\prime}   (\nabla^{2}\mu^{1/2})(v_{*}(\kappa)):(v^\prime - v)\otimes (v^\prime - v)
d\sigma dv_{*} dv d\kappa|
\\&\lesssim&  \int B^{\epsilon,\gamma}_{\eta} |v - v_{*}|^{2}\theta^{2} |G_{*}(HF)^{\prime}| d\sigma dv_{*} dv
\lesssim \eta^{1/2}|G|_{H^{s_3}} |H|_{H^{s_4}}|F|_{L^{2}}.
 \eeno

Patching together the estimates in {\it Step 1} and {\it Step 2}, we finish the proof.
\end{proof}

\subsubsection{Upper bound of $\langle I^{\epsilon,\gamma}(g,h) , f \rangle$}
By Proposition \ref{upforI-ep-ga-et} and  Proposition \ref{I-less-eta-upper-bound}, we get
\begin{thm}\label{upforI-total}
 Let $\delta\in(0,1/2], \eta \in(0,1]$ and $(s_{3},s_{4})=(1/2+\delta,0)$ or $ (0, 1/2+\delta)$. Then for any smooth functions $g, h$ and $f$, the following two estimates are valid.
\beno
|\langle I^{\epsilon,\gamma}(g,h; \beta) , f \rangle|  &\lesssim& \eta^{\gamma-3}|g|_{L^{2}}|h|_{\epsilon,\gamma/2}|W^{\epsilon}f|_{L^{2}_{\gamma/2}}
+\delta^{-1/2}\eta^{\delta}|\mu^{1/16}g|_{H^{1/2+\delta}}|\mu^{1/16}h|_{H^{1}}|f|_{\epsilon,\gamma/2}\\&& + \delta^{-1/2}|\mu^{1/16}g|_{H^{3/2+\delta}}|\mu^{1/16}h|_{L^2}|f|_{\epsilon,\gamma/2},\\
|\langle I^{\epsilon,\gamma}(g,h; \beta) , f \rangle|  &\lesssim& |g|_{L^{2}}|h|_{\epsilon,\gamma/2}|W^{\epsilon}f|_{L^{2}_{\gamma/2}} + \delta^{-1/2}|\mu^{1/16}g|_{H^{s_{3}}}|\mu^{1/16}h|_{H^{1+s_{4}}}|f|_{\epsilon,\gamma/2}.\eeno
\end{thm}

\subsubsection{Upper bound of $\Gamma^{\epsilon,\gamma}$}
Recalling \eqref{Gamma-ep-ga-geq-eta-into-IQ-inner-beta}, as a result of Theorem \ref{Q-full-up-bound} with $a=1/8$ and  Theorem \ref{upforI-total}, we get
\begin{thm}\label{Gamma-full-up-bound}   Let $\delta\in(0,1/2], \eta \in(0,1]$ and $(s_{3},s_{4})=(1/2+\delta,0)$ or $ (0, 1/2+\delta)$.  Then for any smooth functions $g, h$ and $f$,  the following two estimates are valid.
\beno
|\langle \Gamma^{\epsilon,\gamma}(g,h;\beta), f\rangle| &\lesssim& \eta^{\gamma-3}|g|_{L^{2}}|h|_{\epsilon,\gamma/2}|f|_{\epsilon,\gamma/2}
+\delta^{-1/2}(\eta^{\delta}+\epsilon^{1/2})|\mu^{1/16}g|_{H^{3/2+\delta}}|\mu^{1/16}h|_{H^{1}}|f|_{\epsilon,\gamma/2}\\&& + \delta^{-1/2}|\mu^{1/16}g|_{H^{3/2+\delta}}|\mu^{1/16}h|_{L^2}|f|_{\epsilon,\gamma/2},\\
|\langle \Gamma^{\epsilon,\gamma}(g,h;\beta), f\rangle| &\lesssim& |g|_{L^{2}}|h|_{\epsilon,\gamma/2}|f|_{\epsilon,\gamma/2} + \delta^{-1/2}|\mu^{1/16}g|_{H^{s_{3}}}|\mu^{1/16}h|_{H^{1+s_{4}}}|f|_{\epsilon,\gamma/2}.
\eeno
\end{thm}
Taking $\delta = 1/2$ in Theorem \ref{Gamma-full-up-bound}, we have
\begin{col}\label{upgammamuff1-full} For any smooth functions $h$ and $f$ and any $\eta \in(0,1]$, there holds
\beno
|\langle \Gamma^{\epsilon,\gamma}(\partial_{\beta_{1}}\mu^{1/2},h;\beta_{0}), f\rangle| \lesssim (\eta^{\gamma-3}|h|_{\epsilon,\gamma/2} +(\eta^{1/2}+\epsilon^{1/2})|\mu^{1/16}h|_{H^{1}})|f|_{\epsilon,\gamma/2}.
\eeno
\end{col}

\subsection{Upper bound of $\langle \Gamma^{\epsilon,\gamma}_{\eta}(f,h)-\mathcal{L}^{\epsilon,\gamma}_{\eta}h, h\rangle$}
This is the core part of the linear-quasilinear method. Observe
\beno
\langle  \Gamma^{\epsilon,\gamma}_{\eta}(f,h)-\mathcal{L}^{\epsilon,\gamma}_{\eta}h, h\rangle =\langle -\mathcal{L}^{\epsilon,\gamma}_{1,\eta}h + \Gamma^{\epsilon,\gamma}_{\eta}(f,h), h\rangle + \langle -\mathcal{L}^{\epsilon,\gamma}_{2,\eta}h , h\rangle.
\eeno
We begin with
\subsubsection{Upper bound of $\langle -\mathcal{L}^{\epsilon,\gamma}_{1,\eta}h + \Gamma^{\epsilon,\gamma}_{\eta}(f,h), h\rangle$} We have
\begin{prop}\label{less-eta-part-l1-Gamma} Let $\delta \in (0, 1/2], \eta\in (0, 1]$. For any smooth functions $h$ and $f$ with $\mu^{1/2}+ f \geq 0$, there holds
\beno \langle -\mathcal{L}^{\epsilon,\gamma}_{1,\eta}h + \Gamma^{\epsilon,\gamma}_{\eta}(f,h), h\rangle  \lesssim
\delta^{-1/2}\eta^{\delta}|\mu^{1/16}(\mu^{1/2}+ f)|_{H^{1/2+\delta}}|W^{\epsilon}(D)\mu^{1/16}h|^{2}_{L^{2}}
\\+ \delta^{-1/2}(\eta + \epsilon^{1/2}) |\mu^{1/4}(\mu^{1/2}+f)|_{H^{3/2+\delta}}|W^{\epsilon}(D)\mu^{1/8}h|^{2}_{L^{2}}.
\eeno
\end{prop}
\begin{proof} Set  $F=\mu +\mu^{1/2}f, g = \mu^{1/2}+ f$, then $\langle -\mathcal{L}^{\epsilon,\gamma}_{1,\eta}h + \Gamma^{\epsilon,\gamma}_{\eta}(f,h), h\rangle = \langle \mu^{-1/2} Q^{\epsilon,\gamma}_{\eta}(F,\mu^{1/2}h), h\rangle$ and $F = \mu^{1/2}g \geq 0$.
We make the decomposition
\beno \langle \mu^{-1/2} Q^{\epsilon,\gamma}_{\eta}(F,\mu^{1/2}h), h\rangle &=& \int B^{\epsilon,\gamma}_{\eta} F_{*}\mu^{1/2}h((\mu^{-1/2}h)^{\prime}-\mu^{-1/2}h) d\sigma dv_{*} dv
\\&=& \int B^{\epsilon,\gamma}_{\eta} F_{*}h(h^{\prime}-h) d\sigma dv_{*} dv +\int B^{\epsilon,\gamma}_{\eta} g_{*}hh^{\prime}((\mu^{1/2})^{\prime}_{*}-\mu^{1/2}_{*}) d\sigma dv_{*} dv
\\&:=&I_{1}+I_{2}.
\eeno
By the inequality $2 h(h^{\prime}-h) \leq \left((h^{\prime})^{2}-h^{2}\right)$ and the condition $F \geq 0$, we have
\beno I_{1}   \leq \frac{1}{2}\int B^{\epsilon,\gamma}_{\eta} F_{*}\left((h^{\prime})^{2}-h^{2}\right) d\sigma dv_{*} dv .
\eeno
By Lemma \ref{cancellation-lemma-general-gamma-minus3-mu}, taking $a=1/8$  in \eqref{0-delta-eta-small-eta-epsilon-hf-mu}, we have
\beno |\int B^{\epsilon,\gamma}_{\eta} F_{*}\left((h^{\prime})^{2}-h^{2}\right) d\sigma dv_{*} dv|  \lesssim (\eta + \epsilon^{1/2}) |\mu^{-1/4}F|_{L^{\infty}}|W^{\epsilon}(D)\mu^{1/8}h|_{L^{2}}|\mu^{1/8}h|_{L^{2}}.
\eeno
Recalling $F = \mu^{1/2}(\mu^{1/2}+f)$, one has $\mu^{-1/4}F= \mu^{1/4}(\mu^{1/2}+f)$ and thus
\beno  I_{1} \lesssim \delta^{-1/2}(\eta + \epsilon^{1/2}) |\mu^{1/4}(\mu^{1/2}+f)|_{H^{3/2+\delta}}|W^{\epsilon}(D)\mu^{1/8}h|_{L^{2}}|\mu^{1/8}h|_{L^{2}}.
\eeno
Observe that
$ I_{2} = \int B^{\epsilon,\gamma}_{\eta} g_{*}hh^{\prime}((\mu^{1/2})^{\prime}_{*}-\mu^{1/2}_{*}) d\sigma dv_{*} dv  = \langle I^{\epsilon,\gamma}_{\eta}(g,h), h \rangle.
$
Then by Proposition \ref{I-less-eta-upper-bound}, we have
\beno |I_{2}| \lesssim \delta^{-1/2}\eta^{\delta}|\mu^{1/16}g|_{H^{1/2+\delta}}|W^{\epsilon}(D)\mu^{1/16}h|^{2}_{L^{2}}.\eeno
Patching together the previous two inequalities, we finish the proof.
\end{proof}

\subsubsection{ Upper bound of  $\langle -\mathcal{L}^{\epsilon,\gamma}_{2,\eta}f , h\rangle$ } We have
\begin{prop}\label{less-eta-part-l2} Fix $0 < \eta \leq 1$. For any smooth functions $f$ and $h$, there holds
\beno \langle -\mathcal{L}^{\epsilon,\gamma}_{2,\eta}f , h\rangle \lesssim  (\eta^{1/2}+\epsilon^{1/2})|W^{\epsilon}(D)\mu^{1/8}f|_{L^{2}}|\mu^{1/8}h|_{L^{2}}.  \eeno
\end{prop}
\begin{proof}
Note that
\beno \langle -\mathcal{L}^{\epsilon,\gamma}_{2,\eta}f , h\rangle &=&
\langle \mu^{-1/2} Q^{\epsilon,\gamma}_{\eta}(\mu^{1/2}f,\mu), h\rangle
= \int B^{\epsilon,\gamma}_{\eta} (\mu^{1/2}f)_{*}\mu((\mu^{-1/2}h)^{\prime}-\mu^{-1/2}h) d\sigma dv_{*} dv
\\&=& \int B^{\epsilon,\gamma}_{\eta} (\mu^{1/2}f)_{*} \mu^{1/2} (h^{\prime}-h) d\sigma dv_{*} dv +\int B^{\epsilon,\gamma}_{\eta} f_{*}\mu^{1/2}h^{\prime}((\mu^{1/2})^{\prime}_{*}-\mu^{1/2}_{*}) d\sigma dv_{*} dv
\\&:=&Y_{1}+Y_{2}.
\eeno
We first estimate $Y_{1}$. Observe that
\beno Y_{1}&=& \int B^{\epsilon,\gamma}_{\eta} (\mu^{1/2}f)_{*} \mu^{1/2} (h^{\prime}-h) d\sigma dv_{*} dv
\\&=&\int B^{\epsilon,\gamma}_{\eta} (\mu^{1/2}f)_{*}  ((\mu^{1/2}h)^{\prime}-\mu^{1/2}h) d\sigma dv_{*} dv+\int B^{\epsilon,\gamma}_{\eta} (\mu^{1/2}f)_{*} (\mu^{1/2}- (\mu^{1/2})^{\prime})h^{\prime} d\sigma dv_{*} dv
\\&=& Y_{1,1} +Y_{1,2}.
\eeno
For $Y_{1,1}$, use
\eqref{0-delta-eta-small-eta-epsilon-mu} with $\delta=a=0$ in the Lemma \ref{cancellation-lemma-general-gamma-minus3-mu}  to get
\beno |Y_{1,1}|\lesssim (\eta+\epsilon^{1/2})|W^{\epsilon}(D)\mu^{1/2}f|_{L^{2}}|\mu^{1/2}h|_{L^{2}}.
\eeno
Similar to the estimate of $\mathcal{B}(G, H, F)$ in the proof of Proposition \ref{I-less-eta-upper-bound}, we have
\beno |Y_{1,2}| \lesssim  \eta^{1/2}|\mu^{1/8}f|_{L^{2}}|\mu^{1/8}h|_{L^{2}}.
\eeno
We turn to $Y_2$. Note that
\beno Y_{2} &=& \int B^{\epsilon,\gamma}_{\eta} f_{*}\mu^{1/2}h^{\prime}((\mu^{1/2})^{\prime}_{*}-\mu^{1/2}_{*}) d\sigma dv_{*} dv
\\&=& \int B^{\epsilon,\gamma}_{\eta} f_{*}(\mu^{1/2}-(\mu^{1/2})^{\prime})h^{\prime}((\mu^{1/2})^{\prime}_{*}-\mu^{1/2}_{*}) d\sigma dv_{*} dv  +\int B^{\epsilon,\gamma}_{\eta} f_{*}(\mu^{1/2})^{\prime}h^{\prime}((\mu^{1/2})^{\prime}_{*}-\mu^{1/2}_{*}) d\sigma dv_{*} dv
\\&:=& Y_{2,1} +Y_{2,2}.
\eeno
Similar to the estimate of $\mathcal{B}(G, H, F)$ in the proof of Proposition \ref{I-less-eta-upper-bound}, we get
\beno |Y_{2,2}| \lesssim  \eta^{1/2}|\mu^{1/8}f|_{L^{2}}|\mu^{1/8}h|_{L^{2}}.
\eeno
 By Cauchy-Schwartz inequality, we get
\beno |Y_{2,1}| &\leq&  \big(\int B^{\epsilon,\gamma}_{\eta} f^{2}_{*}(\mu^{1/2}-(\mu^{1/2})^{\prime})^{2}d\sigma dv_{*} dv \big)^{1/2}\big(\int B^{\epsilon,\gamma}_{\eta}(h^{\prime})^{2}((\mu^{1/2})^{\prime}_{*}-\mu^{1/2}_{*})^{2}d\sigma dv_{*} dv \big)^{1/2}
  \\&=&\big(\int B^{\epsilon,\gamma}_{\eta} f^{2}_{*}(\mu^{1/2}-(\mu^{1/2})^{\prime})^{2}d\sigma dv_{*} dv \big)^{1/2}\big(\int B^{\epsilon,\gamma}_{\eta}h_{*}^{2}(\mu^{1/2}-(\mu^{1/2})^{\prime})^{2}d\sigma dv_{*} dv \big)^{1/2}.
\eeno
By Taylor expansion up to order 1,
$(\mu^{1/2})^{\prime}-\mu^{1/2}  = \int_{0}^{1}(\nabla \mu^{1/2})(v(\kappa))\cdot (v^{\prime}-v)d\kappa,$
and the change \eqref{change-exact-formula}-\eqref{change-Jacobean-bound}, and \eqref{to-3-4},  we get
\beno \int B^{\epsilon,\gamma}_{\eta} f^{2}_{*}(\mu^{1/2}-(\mu^{1/2})^{\prime})^{2} d\sigma dv_{*} dv  \lesssim
\int  f^{2}_{*} \mu^{1/4}_{*} |v-v_{*}|^{-1}\mathrm{1}_{|v-v_{*}|\leq \eta} dv dv_{*}\lesssim
 \eta^{2} |\mu^{1/8}f|^{2}_{L^{2}},\eeno
which gives
$ |Y_{2,1}| \lesssim \eta^{2} |\mu^{1/8}f|_{L^{2}} |\mu^{1/8}h|. $

Patching together the above estimates, we finish the proof.
\end{proof}

\subsubsection{Quasilinear estimate} As a result of Proposition \ref{less-eta-part-l1-Gamma} and Proposition \ref{less-eta-part-l2}, we have
\begin{thm}\label{small-part-L+gamma} Let $\delta \in(0,1/2],\eta\in(0,1]$. For any smooth functions $h$ and $f$ with $\mu^{1/2}+ f \geq 0$, there holds
\beno \langle  \Gamma^{\epsilon,\gamma}_{\eta}(f,h)-\mathcal{L}^{\epsilon,\gamma}_{\eta}h, h\rangle \lesssim
\delta^{-1/2}(\eta^{\delta}+\epsilon^{1/2})(1+|\mu^{1/16}f|_{H^{3/2+\delta}})|W^{\epsilon}(D)\mu^{1/16}h|^{2}_{L^{2}}.
\eeno
\end{thm}

\subsubsection{Byproducts} In this part, we give some byproducts of previous results.
We define
\ben \label{beta-version-L}
\mathcal{L}^{\epsilon,\gamma,\eta,\beta_{0},\beta_{1}} g = \mathcal{L}^{\epsilon,\gamma,\eta,\beta_{0},\beta_{1}}_{1} g + \mathcal{L}^{\epsilon,\gamma,\eta, \beta_{0},\beta_{1}}_{2} g,
\een
where
\ben \label{beta-version-L-1-2}
\mathcal{L}^{\epsilon,\gamma,\eta,\beta_{0},\beta_{1}}_{1} g := - \Gamma^{\epsilon,\gamma,\eta}(\pa_{\beta_{1}}\mu^{1/2}, g;\beta_{0}),
\mathcal{L}^{\epsilon,\gamma,\eta,\beta_{0},\beta_{1}}_{2} g := - \Gamma^{\epsilon,\gamma,\eta}(g, \pa_{\beta_{1}}\mu^{1/2};\beta_{0}).
\een

\begin{lem} \label{l2-full-estimate-geq-eta} For any $\eta \geq 0$, there holds
\ben \label{full-L2} |\langle \mathcal{L}^{\epsilon,\gamma,\eta,\beta_{0},\beta_{1}}_{2}f , h\rangle| \lesssim |\mu^{1/8}f|_{L^{2}}|\mu^{1/8}h|_{L^{2}}.  \een
\end{lem}
\begin{proof} The proof is similar to Proposition \ref{less-eta-part-l2}, so we omit the details.
\end{proof}

By Corollary \ref{upgammamuff1-geq-eta} and Lemma \ref{l2-full-estimate-geq-eta}, we have the following lemma for upper bound of $\mathcal{L}^{\epsilon,\gamma,\eta,\beta_{0},\beta_{1}}$.

\begin{lem} \label{l-full-estimate-geq-eta} For any $0<\eta \leq 1$, there holds
\ben \label{full-L-geq-eta} |\langle \mathcal{L}^{\epsilon,\gamma,\eta,\beta_{0},\beta_{1}}f , h\rangle| \lesssim \eta^{\gamma-3} |h|_{\epsilon,\gamma/2}|f|_{\epsilon,\gamma/2}.  \een
\end{lem}

By Corollary \ref{upgammamuff1-full} and Lemma \ref{l2-full-estimate-geq-eta}, we have the following lemma for upper bound of $\mathcal{L}^{\epsilon,\gamma,\beta_{0},\beta_{1}} =\mathcal{L}^{\epsilon,\gamma,0,\beta_{0},\beta_{1}}$.
\begin{lem} \label{l-full-estimate} For any $0<\eta \leq 1$,  there holds
\ben \label{full-L-upper-bound} |\langle\mathcal{L}^{\epsilon,\gamma,\beta_{0},\beta_{1}}f , h\rangle| \lesssim  \eta^{\gamma-3}|f|_{\epsilon,\gamma/2}|h|_{\epsilon,\gamma/2}+(\eta^{1/2}+\epsilon^{1/2})
|\mu^{1/16}f|_{H^{1}}|h|_{\epsilon,\gamma/2}.  \een
\end{lem}

\setcounter{equation}{0}

\section{Commutator estimates}\label{Commutator-Estimate}
This section is devoted to the estimate of the  commutator estimates between $ \Gamma^{\epsilon}(g,\cdot)$ and $W_{l}$, which are necessary for energy estimates in weighted Sobolev space. In this section, unless indicated otherwise, $-3 \leq \gamma \leq 0$ and $g,h,f$ are suitable smooth functions.

\subsection{Commutator estimates for $Q^{\epsilon,\gamma,\eta}$} We first have
\begin{prop}\label{commutatorQepsilon}
	Let $0<\eta \leq 1, l \geq 2$, there holds
\beno |\langle Q^{\epsilon,\gamma,\eta}(\mu^{1/2}g,W_{l}h)-W_{l}Q^{\epsilon,\gamma,\eta}(\mu^{1/2}g,h), f\rangle| \lesssim \eta^{\gamma-3} C_{l}|\mu^{1/16}g|_{L^{2}}|h|_{L^{2}_{l+\gamma/2}}|f|_{\epsilon,\gamma/2}. \eeno
\end{prop}
\begin{proof}
We observe that
\beno &&\langle Q^{\epsilon,\gamma,\eta}(\mu^{1/2}g,W_{l}h)-W_{l}Q^{\epsilon,\gamma,\eta}(\mu^{1/2}g,h), f\rangle  = \int B^{\epsilon,\gamma,\eta}(W_{l}-W^{\prime}_{l})\mu^{1/2}_{*}g_{*} h f^{\prime} d\sigma dv_{*} dv
\\&=& \int B^{\epsilon,\gamma,\eta}(W_{l}-W^{\prime}_{l})\mu^{1/2}_{*}g_{*} h (f^{\prime}-f) d\sigma dv_{*} dv
 +\int B^{\epsilon,\gamma,\eta}(W_{l}-W^{\prime}_{l})\mu^{1/2}_{*}g_{*} h f d\sigma dv_{*} dv
:=\mathcal{A}_{1} + \mathcal{A}_{2}. \eeno

{\it Step 1: Estimate of $\mathcal{A}_{1}$.}
By Cauchy-Schwartz inequality, we have
\beno |\mathcal{A}_{1}| &\leq& \big(\int B^{\epsilon,\gamma,\eta} \mu^{1/2}_{*}(f^{\prime}-f)^{2} d\sigma dv_{*} dv\big)^{1/2}
\big(\int B^{\epsilon,\gamma,\eta}(W_{l}-W^{\prime}_{l})^{2}\mu^{1/2}_{*}g^{2}_{*} h^{2}  d\sigma dv_{*} dv\big)^{1/2}
\\&:=&(\mathcal{A}_{1,1})^{1/2}(\mathcal{A}_{1,2})^{1/2}. \eeno
Note that $\mathcal{A}_{1,1}$ has the same structure as $\mathcal{I}_{2,1}$ in \eqref{estimate-I-21}. Taking $\delta=1/2, s_{1}=2, s_{2}=0$ in \eqref{result-estimate-I-21}, we have $\mathcal{A}_{1,1} \lesssim \eta^{\gamma-3}|f|^{2}_{\epsilon,\gamma/2}$.
It is easy to derive $ \int b^{\epsilon}(W_{l}-W^{\prime}_{l})^{2}d\sigma \lesssim |v-v_{*}|^{2}\langle v \rangle^{2l-2}\langle v_{*} \rangle^{2l-2}, $
which gives \beno \mathcal{A}_{1,2} \lesssim \int \mathrm{1}_{|v-v_{*}|\geq \eta} |v-v_{*}|^{\gamma+2}\langle v \rangle^{2l-2}\langle v_{*} \rangle^{2l-2}\mu^{1/2}_{*}g^{2}_{*} h^{2}   dv_{*} dv.\eeno
If $\gamma+2 \geq 0$, there holds $ \mathcal{A}_{1,2} \lesssim |\mu^{1/16}g|^{2}_{L^{2}}|h|^{2}_{L^{2}_{l+\gamma/2}}.$
If $\gamma+2 \leq 0$,   we get
\beno\mathcal{A}_{1,2} \lesssim \eta^{\gamma+2}\int \langle v-v_{*} \rangle^{\gamma+2}\langle v \rangle^{2l-2}\langle v_{*} \rangle^{2l-2}\mu^{1/2}_{*}g^{2}_{*} h^{2}   dv_{*} dv \lesssim \eta^{\gamma+2} |\mu^{1/16}g|^{2}_{L^{2}}|h|^{2}_{L^{2}_{l+\gamma/2}}.\eeno
Patching together the estimates of $\mathcal{A}_{1,1}$  and $\mathcal{A}_{1,2}$, we get
$ |\mathcal{A}_{1}| \lesssim \eta^{\gamma-3} |\mu^{1/16}g|_{L^{2}}|h|_{L^{2}_{l+\gamma/2}}|f|_{\epsilon,\gamma/2}.
$

{\it Step 2: Estimate of $\mathcal{A}_{2}$.}
By Taylor expansion, one has
\beno W^{\prime}_{l} - W_{l} = (\nabla W_{l})(v)\cdot(v^{\prime}-v) +\frac{1}{2}\int_{0}^{1}(1-\kappa)(\nabla^{2}W_{l})(v(\kappa)):(v^{\prime}-v)\otimes(v^{\prime}-v)d\kappa, \eeno
where $v(\kappa) = v + \kappa (v^{\prime}-v)$. Thus we have
\beno \mathcal{A}_{2} &=& -\int B^{\epsilon,\gamma,\eta}(\nabla W_{l})(v)\cdot(v^{\prime}-v)\mu^{1/2}_{*}g_{*} h f d\sigma dv_{*} dv
\\&&-\frac{1}{2}\int B^{\epsilon,\gamma,\eta}(1-\kappa)(\nabla^{2}W_{l})(v(\kappa)):(v^{\prime}-v)\otimes(v^{\prime}-v)\mu^{1/2}_{*}g_{*} h f d\kappa d\sigma dv_{*} dv
:= \mathcal{A}_{2,1}+\mathcal{A}_{2,2}.\eeno

\underline{Estimate of $\mathcal{A}_{2,1}$.} Thanks to the fact \eqref{dispear2}, using $|(\nabla W_{l})(v)| \lesssim \langle v \rangle^{l-1}$,
we have
\beno |\mathcal{A}_{2,1}|  \lesssim   \int \mathrm{1}_{|v-v_{*}|\geq \eta} |v-v_{*}|^{\gamma+1}\langle v \rangle^{l-1}\mu^{1/2}_{*}|g_{*} h f| d\sigma dv_{*} dv.\eeno
If $\gamma+1 \geq 0$, there holds $ |\mathcal{A}_{2,1}| \lesssim |\mu^{1/16}g|_{L^{2}}|h|_{L^{2}_{l+\gamma/2}}|f|_{L^{2}_{\gamma/2}}.$
If $\gamma+1 \leq 0$,  we have
\beno |\mathcal{A}_{2,1}| \lesssim \eta^{\gamma+1}\int \langle v-v_{*} \rangle^{\gamma+1}\langle v \rangle^{l-1}\mu^{1/2}_{*}|g_{*} h f|   dv_{*} dv \lesssim \eta^{\gamma+1} |\mu^{1/16}g|_{L^{2}}|h|_{L^{2}_{l+\gamma/2}}|f|_{L^{2}_{\gamma/2}}.\eeno
Patching together the two cases, we get
$ |\mathcal{A}_{2,1}| \lesssim \eta^{(\gamma+1) \wedge 0} |\mu^{1/16}g|_{L^{2}}|h|_{L^{2}_{l+\gamma/2}}|f|_{L^{2}_{\gamma/2}}.
$

\underline{Estimate of $\mathcal{A}_{2,2}$.} Since $|(\nabla^{2}W_{l})(v(\kappa))| \lesssim \langle v(\kappa) \rangle^{l-2} \lesssim \langle v \rangle^{l-2}\langle v_{*} \rangle^{l-2}$ and $|v^{\prime}-v|^{2} \lesssim \theta^{2} |v-v_{*}|^{2}$, we have
\beno |\mathcal{A}_{2,2}| &\lesssim& \int b^{\epsilon}(\cos\theta)\theta^{2}\mathrm{1}_{|v-v_{*}|\geq \eta}|v-v_{*}|^{\gamma+2}\langle v \rangle^{l-2}\langle v_{*} \rangle^{l-2}\mu^{1/2}_{*}|g_{*} h f |d\sigma dv_{*} dv
\\&\lesssim& \int \mathrm{1}_{|v-v_{*}|\geq \eta}|v-v_{*}|^{\gamma+2}\langle v \rangle^{l-2}\mu_{*}^{1/8}|g_{*} h f | dv_{*} dv.
\eeno
Similar as in the estimate of $\mathcal{A}_{2,1}$, we have
$ |\mathcal{A}_{2,2}| \lesssim \eta^{(\gamma+2) \wedge 0} |\mu^{1/16}g|_{L^{2}}|h|_{L^{2}_{l+\gamma/2}}|f|_{L^{2}_{\gamma/2}}.
$
Patching together the estimates of $\mathcal{A}_{2,1}$ and $\mathcal{A}_{2,2}$, we have
\beno |\mathcal{A}_{2}| \lesssim \eta^{(\gamma+1) \wedge 0} |\mu^{1/16}g|_{L^{2}}|h|_{L^{2}_{l+\gamma/2}}|f|_{L^{2}_{\gamma/2}}.
\eeno

The proposition follows by patching together the estimates of $\mathcal{A}_{1}$ and $\mathcal{A}_{2}$.
\end{proof}

Observe that
\ben \label{singular-Q-ep-ga-eta}
\langle Q^{\epsilon,\gamma}_\eta(\mu^{1/2}g,W_{l}h)-W_{l}Q^{\epsilon,\gamma}_\eta(\mu^{1/2}g,h), f\rangle  = \int B^{\epsilon,\gamma}_\eta(W_{l}-W^{\prime}_{l})\mu^{1/2}_{*}g_{*} h f^{\prime} d\sigma dv_{*} dv. \een
Comparing \eqref{singular-Q-ep-ga-eta} with \eqref{I-ep-ga-eta-ghf}, we find that they enjoy almost the same structure. Thus following the argument there and using the fact $\mu_{*}^{1/2}\mu^{3/4} \mathrm{1}_{|v-v_*|\leq 1} \lesssim \mu_{*}\mathrm{1}_{|v-v_*|\leq 1} \lesssim \mu_{*}^{1/2}\mu^{1/4} \mathrm{1}_{|v-v_*|\leq 1}$, we get
\begin{prop} \label{commutatorQepsilon-less-eta}Let $0 < \eta \leq 1, l \geq 2, 0< \delta \leq 1/2 $, $(s_{3},s_{4})=(1/2+\delta,0)$ or $(s_{3},s_{4})=(0, 1/2+\delta)$, there holds
\beno |\langle Q^{\epsilon,\gamma}_{\eta}(\mu^{1/2}g,W_{l}h)-W_{l}Q^{\epsilon,\gamma}_{\eta}(\mu^{1/2}g,h), f\rangle|  &\lesssim&  C_{l}\delta^{-1/2}\eta^{\delta}|\mu^{1/16}g|_{H^{s_{3}}}|W^{\epsilon}(D)\mu^{1/16}h|_{H^{s_{4}}}|W^{\epsilon}(D)\mu^{1/16}f|_{L^{2}}. \eeno
\end{prop}

\subsection{Commutator estimates for $I^{\epsilon,\gamma,\eta}$} We have
\begin{prop}\label{commutatorforI} Let $ \eta \geq 0, l \geq 1$, there holds
\beno
|\langle I^{\epsilon,\gamma,\eta}(g,W_{l}h;\beta) - W_{l} I^{\epsilon,\gamma,\eta}(g,h;\beta), f\rangle| \lesssim C_{l}|g|_{L^{2}}|h|_{L^{2}_{l+\gamma/2}}|W^{\epsilon}f|_{L^{2}_{\gamma/2}}.
\eeno
\end{prop}
\begin{proof}
Let us consider the $\beta=0$ case since the following arguments also work when we replace $\mu^{1/2}$  with $P_{\beta}\mu^{1/2}$ by using the decomposition \eqref{nice-decomposition}. There are two steps in the proof. We will indicate the main difference at the end of each step.

By the definition \eqref{I-ep-ga-geq-eta} of $I^{\epsilon,\gamma,\eta}(g,h)$ and the fact $((\mu^{1/2})_{*}^{\prime} - \mu_{*}^{1/2}) =((\mu^{1/4})_{*}^{\prime} - \mu_{*}^{1/4})^{2}+2\mu_{*}^{1/4}((\mu^{1/4})_{*}^{\prime} - \mu_{*}^{1/4})$, we have
\beno
\langle I^{\epsilon,\gamma,\eta}(g,W_{l}h) - W_{l} I^{\epsilon,\gamma,\eta}(g,h), f\rangle &=& \int B^{\epsilon,\gamma,\eta}((\mu^{1/2})_{*}^{\prime} - \mu_{*}^{1/2}) (W_{l}-W^{\prime}_{l})g_{*} h f^{\prime} d\sigma dv_{*} dv
\\&=&  \int B^{\epsilon,\gamma,\eta}((\mu^{1/4})_{*}^{\prime} - \mu_{*}^{1/4})^{2}(W_{l}-W^{\prime}_{l})g_{*} h f^{\prime} d\sigma dv_{*} dv
\\&&
 + 2 \int B^{\epsilon,\gamma,\eta}\mu_{*}^{1/4}((\mu^{1/4})_{*}^{\prime} - \mu_{*}^{1/4})(W_{l}-W^{\prime}_{l})g_{*} h f^{\prime} d\sigma dv_{*} dv
\\&:=& \mathcal{A}_{1} + 2\mathcal{A}_{2}.
\eeno

{\it Step 1: Estimate of $\mathcal{A}_{1}$.}  By Cauchy-Schwartz inequality, we have
\beno |\mathcal{A}_{1}| &\leq& \big(\int B^{\epsilon,\gamma,\eta} ((\mu^{1/4})_{*}^{\prime} - \mu_{*}^{1/4})^{2} f^{\prime 2} d\sigma dv_{*} dv\big)^{1/2}
\\&&\times\big(\int B^{\epsilon,\gamma,\eta}((\mu^{1/4})_{*}^{\prime} - \mu_{*}^{1/4})^{2}(W_{l}-W^{\prime}_{l})^{2}g^{2}_{*} h^{2}  d\sigma dv_{*} dv\big)^{1/2}
:=(\mathcal{A}_{1,1})^{1/2}(\mathcal{A}_{1,2})^{1/2}. \eeno
By the change of variables $(v,v_{*}) \rightarrow (v_{*}^{\prime},v^{\prime})$ and Lemma \ref{upperboundpart1} (the result still holds with $\mu^{1/2}$ replaced by $\mu^{1/4}$), we have
\beno \mathcal{A}_{1,1} = \int B^{\epsilon,\gamma,\eta} ((\mu^{1/4})^{\prime} - \mu^{1/4})^{2} f^{2}_{*} d\sigma dv_{*} dv = \mathcal{N}^{\epsilon,\gamma,\eta}(f,\mu^{1/4})
\lesssim |W^{\epsilon}f|^{2}_{L^{2}_{\gamma/2}}.\eeno
Thanks to  $((\mu^{1/4})_{*}^{\prime} - \mu_{*}^{1/4})^{2} = ((\mu^{1/8})_{*}^{\prime} + \mu_{*}^{1/8})^{2}((\mu^{1/8})_{*}^{\prime} - \mu_{*}^{1/8})^{2} \leq 2 ((\mu^{1/4})_{*}^{\prime} + \mu_{*}^{1/4})((\mu^{1/8})_{*}^{\prime} - \mu_{*}^{1/8})^{2}$, we have
\beno \mathcal{A}_{1,2} &\lesssim& \int B^{\epsilon,\gamma,\eta}\mu_{*}^{1/4}((\mu^{1/8})_{*}^{\prime} - \mu_{*}^{1/8})^{2}(W_{l}-W^{\prime}_{l})^{2}g^{2}_{*} h^{2}  d\sigma dv_{*} dv \\&&+ \int B^{\epsilon,\gamma,\eta}(\mu^{1/4})_{*}^{\prime}((\mu^{1/8})_{*}^{\prime} - \mu_{*}^{1/8})^{2}(W_{l}-W^{\prime}_{l})^{2}g^{2}_{*} h^{2}  d\sigma dv_{*} dv
:= \mathcal{A}_{1,2,1} + \mathcal{A}_{1,2,2}.\eeno
We first estimate $\mathcal{A}_{1,2,2}$.  We recall $|v-v^{\prime}_{*}| \sim |v-v_{*}|$ and thus
\ben\label{roughaboutwl}
(W_{l}-W^{\prime}_{l})^{2} \lesssim \min\{\theta^{2}|v-v_{*}^{\prime}|^{2}\langle v \rangle^{2l-2} \langle v_{*}^{\prime} \rangle^{2l-2}, \theta^{2}\langle v \rangle^{2l} \langle v_{*}^{\prime} \rangle^{2l}\},
\\ \label{roughaboutmu}
((\mu^{1/8})_{*}^{\prime} - \mu_{*}^{1/8})^{2} \lesssim \min\{ \theta^{2}|v-v_{*}^{\prime}|^{2}, 1\}.
\een
We set to prove
\begin{eqnarray}\label{kernelestimate2}
\mathcal{B} := \int B^{\epsilon,\gamma,\eta}(\mu^{1/4})_{*}^{\prime}((\mu^{1/8})_{*}^{\prime} - \mu_{*}^{1/8})^{2}(W_{l}-W^{\prime}_{l})^{2}  d\sigma \lesssim \langle v \rangle^{2l+\gamma},
\end{eqnarray}
which immediately gives
$\mathcal{A}_{1,2,2} \lesssim |g|^{2}_{L^{2}}|h|^{2}_{L^{2}_{l+\gamma/2}}.$

\underline{\it Case 1: $|v-v_{*}|\leq 1$.} By (\ref{roughaboutwl}) and (\ref{roughaboutmu}), we have
\beno \mathcal{B}  \lesssim \int b^{\epsilon}(\cos\theta)\theta^{4}|v-v^{\prime}_{*}|^{\gamma+4}(\mu^{1/4})_{*}^{\prime} \langle v \rangle^{2l-2} \langle v_{*}^{\prime} \rangle^{2l-2} d\sigma. \eeno
Since $|v-v_{*}|\leq 1$, there holds $|v-v_{*}^{\prime}|\leq 1$, $|v-v^{\prime}_{*}|^{\gamma+4}\leq 1$ and $\langle v \rangle \sim \langle v^{\prime}_{*} \rangle$, thus $\langle v \rangle^{2l-2}  \lesssim \langle v \rangle^{2l+\gamma} \langle v^{\prime}_{*} \rangle^{-2-\gamma}$, which implies
\beno \mathcal{B}  \lesssim \int b^{\epsilon}(\cos\theta)\theta^{4}(\mu^{1/4})_{*}^{\prime} \langle v \rangle^{2l+\gamma} \langle v_{*}^{\prime} \rangle^{2l-4-\gamma} d\sigma \lesssim \int b^{\epsilon}(\cos\theta)\theta^{4} \langle v \rangle^{2l+\gamma} d\sigma \lesssim \langle v \rangle^{2l+\gamma}.  \eeno

\underline{\it Case 2: $|v-v_{*}|\geq 1$.} By (\ref{roughaboutwl}) and (\ref{roughaboutmu}), we have
$\mathcal{B}  \lesssim \int b^{\epsilon}(\cos\theta)\theta^{2}|v-v^{\prime}_{*}|^{\gamma}(\mu^{1/4})_{*}^{\prime} \langle v \rangle^{2l} \langle v_{*}^{\prime} \rangle^{2l} d\sigma. $
Since $|v-v_{*}|\geq 1$, there holds   $|v-v^{\prime}_{*}|^{\gamma}\sim \langle v-v^{\prime}_{*} \rangle^{\gamma} \lesssim \langle v \rangle^{\gamma}\langle v^{\prime}_{*} \rangle^{|\gamma|}$, which implies
\beno \mathcal{B}  \lesssim \int b^{\epsilon}(\cos\theta)\theta^{2} (\mu^{1/4})_{*}^{\prime} \langle v \rangle^{2l+\gamma} \langle v_{*}^{\prime} \rangle^{2l+|\gamma|} d\sigma  \lesssim \int b^{\epsilon}(\cos\theta)\theta^{2} \langle v \rangle^{2l+\gamma} d\sigma \lesssim \langle v \rangle^{2l+\gamma}. \eeno
We get \eqref{kernelestimate2} by patching together the two cases.

We then go to estimate $\mathcal{A}_{1,2,1}$.
Thanks to
$
(W_{l}-W^{\prime}_{l})^{2} \lesssim \min\{\theta^{2}|v-v_{*}|^{2}\langle v \rangle^{2l-2} \langle v_{*} \rangle^{2l-2}, \theta^{2}\langle v \rangle^{2l} \langle v_{*} \rangle^{2l}\},
$
and  $
((\mu^{1/8})_{*}^{\prime} - \mu_{*}^{1/8})^{2} \lesssim \min\{ \theta^{2}|v-v_{*}|^{2}, 1\},
$
similar to (\ref{kernelestimate2}), we can prove
\begin{eqnarray}\label{kernelestimate1}
\int B^{\epsilon,\gamma,\eta}\mu_{*}^{1/4}((\mu^{1/8})_{*}^{\prime} - \mu_{*}^{1/8})^{2}(W_{l}-W^{\prime}_{l})^{2}  d\sigma \lesssim \langle v \rangle^{2l+\gamma}\mu_{*}^{1/8}.
\end{eqnarray}
Plugging  \eqref{kernelestimate1} into $\mathcal{A}_{1,2,1}$, we get $ \mathcal{A}_{1,2,1} \lesssim |\mu^{1/16}g|^{2}_{L^{2}}|h|^{2}_{L^{2}_{l+\gamma/2}}.$
Patching together the upper bound estimates of $\mathcal{A}_{1,2,1}$ and $\mathcal{A}_{1,2,2}$, we arrive at
 $\mathcal{A}_{1,2} \lesssim |g|^{2}_{L^{2}}|h|^{2}_{L^{2}_{l+\gamma/2}}. $
Patching together the estimates of $\mathcal{A}_{1,1}$ and $\mathcal{A}_{1,2}$, we conclude
$|\mathcal{A}_{1}| \lesssim |g|_{L^{2}}|h|_{L^{2}_{l+\gamma/2}}|W^{\epsilon}f|_{L^{2}_{\gamma/2}}.$

In the $|\beta|>0$ case, by recalling \eqref{alpha-beta-on-Gamma} and \eqref{nice-decomposition},
changes only happen in $\mathcal{A}_{1,1}$, in which $((\mu^{1/4})^{\prime} - \mu^{1/4})^{2}$ is replaced with $(P_{\beta}\mu^{1/4})^{\prime} - P_{\beta}\mu^{1/4})^{2}$. Then Lemma \ref{upperboundpart1} also holds since it only utilizes the condition \eqref{mu-prodcut-square}.

{\it Step 2: Estimate of $\mathcal{A}_{2}$.}
By Cauchy-Schwartz inequality, we have
\beno |\mathcal{A}_{2}| &\leq& \big(\int B^{\epsilon,\gamma,\eta}\mu_{*}^{1/4}((\mu^{1/4})_{*}^{\prime} - \mu_{*}^{1/4})^{2}g_{*} f^{\prime 2} d\sigma dv_{*} dv\big)^{1/2}
\\&\times&\big(\int B^{\epsilon,\gamma,\eta}\mu_{*}^{1/4}(W_{l}-W^{\prime}_{l})^{2}g_{*} h^{2}  d\sigma dv_{*} dv\big)^{1/2}
:=(\mathcal{A}_{2,1})^{1/2}(\mathcal{A}_{2,2})^{1/2}. \eeno

\underline{Estimate of $\mathcal{A}_{2,1}$.} By the change of variable $v \rightarrow v^{\prime}$, we have
\beno \mathcal{A}_{2,1} \lesssim \int b^{\epsilon}(\cos\theta)|v^{\prime}-v_{*}|^{\gamma}\mu_{*}^{1/4}((\mu^{1/4})_{*}^{\prime} - \mu_{*}^{1/4})^{2}g_{*} f^{\prime 2} d\sigma dv_{*} dv^{\prime}.\eeno
By Proposition \ref{symbol}, one has
 $ \int b^{\epsilon}(\cos\theta) ((\mu^{1/4})_{*}^{\prime} - \mu_{*}^{1/4})^{2} d\sigma \lesssim  |v'-v_*|^2\mathrm{1}_{|v'-v_*|\le2}+(W^{\epsilon})^{2}(v^{\prime}-v_{*})\mathrm{1}_{|v'-v_*|\ge2},$
which gives
\beno\mathcal{A}_{2,1} &\lesssim& \int |v^{\prime}-v_{*}|^{\gamma}\mu_{*}^{1/4} (|v'-v_*|^2\mathrm{1}_{|v'-v_*|\le2}+(W^{\epsilon})^{2}(v^{\prime}-v_{*})\mathrm{1}_{|v'-v_*|\ge2}) g_{*} f^{\prime 2} dv_{*} dv^{\prime}\\
&\lesssim& |\mu^{1/8}g|_{L^2}|W^\epsilon f|_{L^2_{\gamma/2}}^2. \eeno

 \underline{Estimate of $\mathcal{A}_{2,2}$.}
By Taylor expansion, when $l \geq 1$, it is easy to check
\beno (W_{l}-W^{\prime}_{l})^{2} \lesssim \theta^{2}|v-v_{*}|^{2}(\langle v \rangle^{2l-2} + \langle v_{*} \rangle^{2l-2}) \lesssim \theta^{2}|v-v_{*}|^{2}\langle v \rangle^{2l-2} \langle v_{*} \rangle^{2l-2}\eeno
which gives
\beno \mathcal{A}_{2,2} &\lesssim& \int b^{\epsilon}(\cos\theta) \theta^{2}|v-v_{*}|^{\gamma+2}\langle v \rangle^{2l-2} \langle v_{*} \rangle^{2l-2}\mu_{*}^{1/4} g_{*} h^{ 2}  d\sigma dv_{*} dv
\\&\lesssim& \int |v-v_{*}|^{\gamma+2}\langle v \rangle^{2l-2} \langle v_{*} \rangle^{2l-2}\mu_{*}^{1/4} g_{*} h^{ 2} dv_{*} dv.
\eeno
Noting that \beno \int |v-v_{*}|^{\gamma+2} \langle v_{*} \rangle^{2l-2}\mu_{*}^{1/4} g_{*} dv_{*} &\leq& \big(\int |v-v_{*}|^{2\gamma+4}  \mu_{*}^{1/4}  dv_{*}\big)^{1/2}
\big(\int \langle v_{*} \rangle^{4l-4}\mu_{*}^{1/4} g^{2}_{*} dv_{*} \big)^{1/2}
\\&\lesssim& \langle v \rangle^{\gamma+2} |\mu^{1/16}g|_{L^{2}}, \eeno
which gives
$ \mathcal{A}_{2,2} \lesssim |\mu^{1/16}g|_{L^{2}}|h|^{2}_{L^{2}_{l+\gamma/2}}. $
Putting together the estimates of $\mathcal{A}_{2,1}$ and $\mathcal{A}_{2,2}$, we arrive at
\beno |\mathcal{A}_{2}| \lesssim |\mu^{1/16}g|_{L^{2}}|h|_{L^{2}_{l+\gamma/2}}|W^{\epsilon}f|_{L^{2}_{\gamma/2}}.\eeno

In the $|\beta|>0$ case, by recalling \eqref{alpha-beta-on-Gamma} and \eqref{nice-decomposition},
$\mu_{*}^{1/4}((\mu^{1/4})_{*}^{\prime} - \mu_{*}^{1/4})^{2}$ is replaced by $(P_{\beta}\mu^{1/4})_{*}((\mu^{1/4})_{*}^{\prime} - \mu_{*}^{1/4})^{2}$ or $\mu_{*}^{1/4}((P_{\beta}\mu^{1/4})_{*}^{\prime} - (P_{\beta}\mu^{1/4})_{*})^{2}$.
The above arguments also work. In the former, just replace $(P_{\beta}\mu^{1/4})_{*}$ with $(\mu^{1/8})_{*}$.
In the latter, $((P_{\beta}\mu^{1/4})_{*}^{\prime} - (P_{\beta}\mu^{1/4})_{*})^{2}$ enjoys the condition \eqref{mu-prodcut-square}.

The proposition follows   the estimates of $\mathcal{A}_{1}$ and $\mathcal{A}_{2}$.
\end{proof}

\subsection{Applications of previous results}  We first have
\begin{thm}\label{commutator-Gamma-geq-eta}Let $0 < \eta \leq 1,l \geq 2$. There holds
\ben \label{Gamma-commutator-geq-eta} |\langle \Gamma^{\epsilon,\gamma,\eta}(g,W_{l}h;\beta)-W_{l}\Gamma^{\epsilon,\gamma,\eta}(g,h;\beta), f\rangle| \lesssim \eta^{\gamma-3}C_{l}|g|_{L^{2}}|h|_{L^{2}_{l+\gamma/2}}|f|_{\epsilon,\gamma/2}. \een
Let $0< \delta \leq 1/2, (s_{3},s_{4})=(1/2+\delta,0)$ or $(s_{3},s_{4})=(0, 1/2+\delta)$, there holds
\ben \label{Gamma-commutator-full} |\langle \Gamma^{\epsilon,\gamma}(g,W_{l}h;\beta)-W_{l}\Gamma^{\epsilon,\gamma}(g,h;\beta), f\rangle| \lesssim  \eta^{\gamma-3}C_{l}|g|_{L^{2}}|h|_{L^{2}_{l+\gamma/2}}|f|_{\epsilon,\gamma/2} \\+ C_{l}\delta^{-1/2}\eta^{\delta}|\mu^{1/16}g|_{H^{s_{3}}}|W^{\epsilon}(D)\mu^{1/16}h|_{H^{s_{4}}}|W^{\epsilon}(D)\mu^{1/16}f|_{L^{2}}. \nonumber\een
\end{thm}
\begin{proof} By Proposition \ref{commutatorQepsilon} and Proposition \ref{commutatorforI}, we have \eqref{Gamma-commutator-geq-eta}. By Proposition \ref{commutatorQepsilon} and Proposition \ref{commutatorQepsilon-less-eta} and Proposition \ref{commutatorforI}, we get \eqref{Gamma-commutator-full}.
\end{proof}

%

Theorem \ref{Gamma-full-up-bound} and Theorem \ref{commutator-Gamma-geq-eta} together give the following upper bound estimate with weight.
\begin{col}\label{Gamma-full-up-bound-with-weight} Let $0< \delta \leq 1/2, (s_{3},s_{4})=(1/2+\delta,0)$ or $(s_{3},s_{4})=(0, 1/2+\delta)$, there holds
\ben
|\langle W_{l}\Gamma^{\epsilon,\gamma}(g,h;\beta), f\rangle| &\lesssim& |g|_{L^{2}}|h|_{\epsilon,l+\gamma/2}|f|_{\epsilon,\gamma/2} +
C_{l}|g|_{L^{2}}|h|_{L^{2}_{l+\gamma/2}}|f|_{\epsilon,\gamma/2}
\nonumber \\&&+C_{l}\delta^{-1/2}|\mu^{1/16}g|_{H^{s_{3}}}|\mu^{1/32}h|_{H^{1+s_{4}}}|f|_{\epsilon,\gamma/2}. \label{version-2-dissipation-Boltzmann}
\\ \label{version-2-dissipation-landau}
|\langle W_{l}\Gamma^{\epsilon,\gamma}(g,h;\beta), f\rangle| &\lesssim& |g|_{H^{2}}|h|_{0,l+\gamma/2}|f|_{\epsilon,\gamma/2}
+ C_{l}|g|_{H^{2}}|h|_{L^{2}_{l+\gamma/2}}|f|_{\epsilon,\gamma/2}.
\een
\end{col}

As an application of Theorem \ref{commutator-Gamma-geq-eta}, we have
\begin{col}\label{commutator-L-geq-eta}Let $0<\eta \leq 1, l \geq 2$,  there holds
\beno  \langle [\mathcal{L}^{\epsilon,\gamma,\beta_{0},\beta_{1}},W_{l}]g, W_{l}f\rangle| \lesssim \eta^{\gamma-3}C_{l}|g|_{L^{2}_{l+\gamma/2}}|f|_{\epsilon,l+\gamma/2} + C_{l}\eta^{1/2}|W^{\epsilon}(D)\mu^{1/16}g|_{L^{2}}|f|_{\epsilon,\gamma/2}.\eeno
\end{col}
\begin{proof} Recall from \eqref{beta-version-L-1-2}, $\mathcal{L}^{\epsilon,\gamma,\beta_{0},\beta_{1}}_{1} g = - \Gamma^{\epsilon,\gamma}(\pa_{\beta_{1}}\mu^{1/2}, g;\beta_{0})$.
Taking $\delta=1/2, s_{3}= 1, s_{4}=0$ in \eqref{Gamma-commutator-full}, we get
\beno \langle [\mathcal{L}^{\epsilon,\gamma,\beta_{0},\beta_{1}}_{1},W_{l}]g, W_{l}f\rangle| \lesssim  \eta^{\gamma-3}C_{l}|g|_{L^{2}_{l+\gamma/2}}|f|_{\epsilon,l+\gamma/2} + C_{l}\eta^{1/2}|W^{\epsilon}(D)\mu^{1/16}g|_{L^2}|W^{\epsilon}(D)\mu^{1/16}W_{l}f|_{L^2}.\eeno
Recall from \eqref{beta-version-L-1-2}, $
\mathcal{L}^{\epsilon,\gamma,\beta_{0},\beta_{1}}_{2} g := - \Gamma^{\epsilon,\gamma}(g, \pa_{\beta_{1}}\mu^{1/2};\beta_{0}).
$
Taking $\delta=1/2, s_{3}= 0, s_{4}=1$ in \eqref{Gamma-commutator-full}, we get
\beno \langle [\mathcal{L}^{\epsilon,\gamma,\beta_{0},\beta_{1}}_{2},W_{l}]g, f\rangle| \lesssim \eta^{\gamma-3}C_{l}|g|_{L^{2}}|f|_{\epsilon,l+\gamma/2}+ C_{l}\eta^{1/2}|\mu^{1/16}g|_{L^2}|W^{\epsilon}(D)\mu^{1/16}W_{l}f|_{L^2}.\eeno
Patching together the above two estimates, recalling \eqref{beta-version-L}, thanks to the fact $|W^{\epsilon}(D)\mu^{1/16}W_{l}f|_{L^2} \lesssim C_{l}|f|_{\epsilon,\gamma/2}$, we finish the proof.
\end{proof}

When $\gamma=-3$, recall the notation $\mathcal{L}^{\epsilon,\beta_{0},\beta_{1}} = \mathcal{L}^{\epsilon,-3,\beta_{0},\beta_{1}}$. As a special of case of Corollary \ref{commutator-L-geq-eta}, we have
\begin{col}\label{commutator-L-gamma-minus3} Let $0<\eta \leq 1, l \geq 2$, there holds
\beno \langle [\mathcal{L}^{\epsilon,\beta_{0},\beta_{1}},W_{l}]g, W_{l}f\rangle| \leq  \eta^{-6}C_{l}|g|_{L^{2}_{l+\gamma/2}}|f|_{\epsilon,l+\gamma/2} + \eta^{1/2}C_{l}|g|_{\epsilon,\gamma/2}|f|_{\epsilon,\gamma/2}. \eeno
\end{col}


\setcounter{equation}{0}

\section{Energy estimate and asymptotic formula} \label{Energy-Estimate}
In this section, we will give the proof to Theorem \ref{asymptotic-result}. We divide the proof into three subsections. The first subsection is devoted to the {\it a priori} estimates for the linear equation \eqref{lBE-general}.
In subsection 4.2, we consider the global well-posedness \eqref{uniform-controlled-by-initial} and regularity propagation \eqref{propagation} of the linearized Boltzmann equation \eqref{linearizedBE}. In subsection 4.3, we derive the global asymptotic formula \eqref{error-function-uniform-estimate} which describes the limit that $\epsilon$ goes to zero. Throughout this section, we set $\gamma=-3$.

\subsection{Estimate for the linear equation} We will deal with the linear equation as follows:
\ben \label{lBE-general} \partial_{t}f + v\cdot \nabla_{x} f + \mathcal{L}^{\epsilon}f= g. \een
Let us set up some notations which will be used throughout this section.

\begin{itemize}
\item  We set $f_{1} :=\mathbb{P} f$ and $f_{2} := f - \mathbb{P} f$, where $\mathbb{P}$ is the projection operator defined in \eqref{DefProj}. By \eqref{DefProj} and
\eqref{Defabc},
\ben \label{definition-f-1} f_{1}(t,x,v) = \big(a(t,x) + b(t,x) \cdot v + c(t,x)|v|^{2}\big)\mu^{1/2},\een which solves
\ben \label{macro-micro-LBE-2} \partial_{t}f_{1} + v\cdot \nabla_{x} f_{1}  = r + l + g,\een
 where $r = - \partial_{t}f_{2}$ and $ l = - v\cdot \nabla_{x} f_{2} - \mathcal{L}^{\epsilon}f_{2}$.

\item  $\{e_{j}\}_{1\leq j \leq 13}$ is defined  explicitly as
\beno e_{1} = \mu^{1/2}, e_{2} = v_{1}\mu^{1/2}, e_{3} = v_{2}\mu^{1/2},e_{4} = v_{3}\mu^{1/2}, \\ e_{5} = v_{1}^{2}\mu^{1/2}, e_{6} = v_{2}^{2}\mu^{1/2},e_{7} = v_{3}^{2}\mu^{1/2}, e_{8} = v_{1}v_{2}\mu^{1/2}, e_{9} = v_{2}v_{3}\mu^{1/2},e_{10} = v_{3}v_{1}\mu^{1/2}, \\ e_{11} = |v|^{2}v_{1}\mu^{1/2}, e_{12} = |v|^{2}v_{2}\mu^{1/2},e_{13} = |v|^{2}v_{3}\mu^{1/2}. \eeno

\item Let $A = (a_{ij})_{1\leq i \leq 13, 1\leq j \leq 13}$ be the matrix defined by $a_{ij} = \langle e_{i}, e_{j} \rangle $ and $y$ be the 13-dimensional vector with components $\partial_{t} a, \{\partial_{t}b_{i}+ \partial_{i} a \}_{1\leq i \leq 3}, \{\partial_{t}c+ \partial_{i} b_{i} \}_{1\leq i \leq 3},  \{\partial_{i}b_{j}+ \partial_{j} b_{i} \}_{1\leq i < j  \leq 3}, \{\partial_{i}c \}_{1\leq i \leq 3}$. Set $z = (z_i)_{i=1}^{13}=(\langle r+l+g, e_i\rangle)_{i=1}^{13}$. By  \eqref{definition-f-1} and taking inner product between \eqref{macro-micro-LBE-2} and $\{e_{j}\}_{1\leq j \leq 13}$ in the space $L^{2}(\R^{3})$ for variable $v$,   one has $Ay = z$, which implies
 $ y = A^{-1}z $.
We denote
\beno \tilde{r}= (r^{(0)}, \{r^{(1)}_{i}\}_{1\leq i \leq 3}, \{r^{(2)}_{i}\}_{1\leq i \leq 3}, \{r^{(2)}_{ij}\}_{1\leq i < j \leq 3}, \{r^{(3)}_{i}\}_{1\leq i \leq 3})^{T} = A^{-1} ( \langle r, e_i\rangle)_{i=1}^{13},
\\ \tilde{l}= (l^{(0)}, \{l^{(1)}_{i}\}_{1\leq i \leq 3}, \{l^{(2)}_{i}\}_{1\leq i \leq 3}, \{l^{(2)}_{ij}\}_{1\leq i < j \leq 3}, \{l^{(3)}_{i}\}_{1\leq i \leq 3})^{T} = A^{-1} ( \langle l, e_i\rangle)_{i=1}^{13},
\\ \tilde{g}=(g^{(0)}, \{g^{(1)}_{i}\}_{1\leq i \leq 3}, \{g^{(2)}_{i}\}_{1\leq i \leq 3}, \{g^{(2)}_{ij}\}_{1\leq i < j \leq 3}, \{g^{(3)}_{i}\}_{1\leq i \leq 3})^{T} = A^{-1} ( \langle g, e_i\rangle)_{i=1}^{13},
\\ \tilde{f}= (\tilde{f}^{(0)}, \{\tilde{f}^{(1)}_{i}\}_{1\leq i \leq 3}, \{\tilde{f}^{(2)}_{i}\}_{1\leq i \leq 3}, \{\tilde{f}^{(2)}_{ij}\}_{1\leq i < j \leq 3}, \{\tilde{f}^{(3)}_{i}\}_{1\leq i \leq 3})^{T} = A^{-1} (\langle f_{2}, e_{i}\rangle)_{i=1}^{13}. \eeno
Then one has
$ \tilde{r} =  - \partial_{t}\tilde{f} $, which yields
\ben
\label{linear-equation-abc-3}
y  = -\partial_{t}\tilde{f} + \tilde{l} + \tilde{g}.\een

 \item  We define the
 temporal  energy functional  $\mathcal{I}^{N}(f)$ as
\ben \label{interactive-INf} \mathcal{I}^{N}(f) := \sum_{|\alpha|\leq N-1}\sum_{i=1}^{3}( \mathcal{I}^{a}_{\alpha,i}(f)+\mathcal{I}^{b}_{\alpha,i}(f)+\mathcal{I}^{c}_{\alpha,i}(f)+\mathcal{I}^{ab}_{\alpha,i}(f)), \een
where
\beno
 \mathcal{I}^{a}_{\alpha,i}(f):= \langle \partial^{\alpha} \tilde{f}^{(1)}_{i}, \partial_{i}\partial^{\alpha} a\rangle_{x}, \mathcal{I}^{c}_{\alpha,i}(f) := \langle \partial^{\alpha} \tilde{f}^{(3)}_{i}, \partial_{i}\partial^{\alpha} c\rangle_{x}, \mathcal{I}^{ab}_{\alpha,i}(f) := \langle \partial_{i}\partial^{\alpha} a, \partial^{\alpha} b_{i}\rangle_{x}
  \\ \mathcal{I}^{b}_{\alpha,i}(f) := -\sum_{j \neq i}\langle \partial^{\alpha} \tilde{f}^{(2)}_{j}, \partial_{i}\partial^{\alpha} b_{i}\rangle_{x} + \sum_{j \neq i}\langle \partial^{\alpha} \tilde{f}^{(2)}_{ji}, \partial_{j}\partial^{\alpha} b_{i}\rangle_{x}  + 2 \langle \partial^{\alpha} \tilde{f}^{(2)}_{i}, \partial_{i}\partial^{\alpha} b_{i}\rangle_{x}.
\eeno
Here $\langle a, b\rangle_{x} := \int_{\mathbb{T}^{3}} a(x)b(x) dx$ is the inner product in $L^{2}(\mathbb{T}^{3})$ for variable $x$.
\end{itemize}

Note that there is some universal constant $C_{1}$ such that
\ben \label{temporal-bounded-by-energy}
|\mathcal{I}^{N}(f)| \leq C_{1} \|f\|^{2}_{H^{N}_{x}L^{2}}.\een
The temporal energy functional $\mathcal{I}^{N}(f)$ is used to capture the dissipation of $(a,b,c)$. Based on \eqref{linear-equation-abc-3}, one can study the evolution of the macroscopic quantities $(a,b,c)$ in terms of the microscopic part $f_{2}$.
\begin{lem}\label{estimate-for-highorder-abc}
There exists are two universal constants $C, c_{0} > 0$ such that for any $N \geq 2$,
\ben \label{solution-property-part2} \frac{d}{dt}\mathcal{I}^{N}(f) + c_{0}|(a,b,c)|^{2}_{H^{N}_{x}} \leq C(\|f_{2}\|^{2}_{H^{N}_{x}L^2_{\epsilon,\gamma/2}} + \sum_{|\alpha| \leq N-1}\sum_{j=1}^{13}\int_{\mathbb{T}^{3}}|\langle  \pa^{\alpha}g, e_j\rangle|^{2} dx).\een
\end{lem}
\begin{proof} Referring to \cite{duan2008cauchy,guo2003classical}, we have
\beno\frac{d}{dt}\mathcal{I}^{N}(f) + \frac{1}{2}|\nabla_{x}(a,b,c)|^{2}_{H^{N-1}_{x}} \leq C(\|f_{2}\|^{2}_{H^{N}_{x}L^2_{\epsilon,\gamma/2}} + \sum_{|\alpha| \leq N-1}\sum_{j=1}^{13}\int_{\mathbb{T}^{3}}|\langle  \pa^{\alpha}g, e_j\rangle|^{2} dx).\eeno
Thanks to \eqref{Nuspace}, we can apply Poincare inequality $|\mathcal{MA}|_{H^N_x} \lesssim |\na_x\mathcal{MA}|_{H^{N-1}_x}$ to end the proof.
\end{proof}

The above set-up is standard for the near Maxwellian framework. We refer readers to \cite{duan2008cauchy,guo2003classical} for more details. Before giving the estimates for \eqref{lBE-general}, we prepare some technical lemmas to deal with the inner products that will appear in energy estimates.

\begin{lem} \label{coercivity-L-alpha-beta}
Let $|\alpha|+|\beta| \leq N$, $q\ge 0$, then
\beno
(\mathcal{L}^{\epsilon, \gamma,\eta} W_{q}\pa^{\alpha}_{\beta}f, W_{q}\pa^{\alpha}_{\beta}f) \geq
 (7/8)\lambda_0\| W_{q}\pa^{\alpha}_{\beta} f_2\|^2_{L^{2}_{x}L^2_{\epsilon,\gamma/2}}- C_{q,N}(\|\pa^{\alpha} f_2\|^2_{L^{2}_{x}L^{2}_{\gamma/2}}+|\partial^{\alpha}(a,b,c)|^{2}_{L^{2}_{x}}).
\eeno
\end{lem}
\begin{proof} By Theorem \ref{micro-weight-dissipation}, we have
\beno (\mathcal{L}^{\epsilon, \gamma,\eta} W_{q}\pa^{\alpha}_{\beta}f, W_{q}\pa^{\alpha}_{\beta}f)\ge \lambda_0\|(\mathbb{I}-\mathbb{P})W_{q}\pa^{\alpha}_{\beta}f\|_{L^{2}_{x}L^2_{\epsilon,\gamma/2}}^2. \eeno
Thanks to the Macro-Micro decomposition $f=f_{1}+f_{2}$, we deduce that
\beno (\mathcal{L}^{\epsilon, \gamma,\eta} W_{q}\pa^{\alpha}_{\beta}f, W_{q}\pa^{\alpha}_{\beta}f)\ge \lambda_0\|(\mathbb{I}-\mathbb{P})W_{q}\pa^{\alpha}_{\beta}(f_1+f_2)\|_{L^{2}_{x}L^2_{\epsilon,\gamma/2}}^{2}\\
 \ge (7/8)\lambda_0 \| W_{q}\pa^{\alpha}_{\beta} f_2\|_{L^{2}_{x}L^2_{\epsilon,\gamma/2}}^2- C_{q,N}(\| \pa^{\alpha} f_2\|_{L^{2}_{\gamma/2}}^2+|\partial^{\alpha}(a,b,c)|^{2}_{L^{2}_{x}}), \eeno
where we use \eqref{basic-inequality-1} to take out $W_{q}\pa^{\alpha}_{\beta} f_2$ as the leading term and integration by parts formula to deal with the operator $\partial_{\beta}$. In addition, all polynomial weights can be controlled by the factor $\mu^{1/2}$ in $f_{1}$.
\end{proof}

\begin{lem} \label{commutator-L-alpha-beta}
Let $|\alpha|+|\beta| \leq N, \beta_{0} + \beta_{1} + \beta_{2} = \beta$, $q \geq 2$,  then for any $0<\delta \leq 1$, we have
\beno
|([W_{q}, \mathcal{L}^{\epsilon,\beta_{0},\beta_{1}}]\pa^{\alpha}_{\beta_{2}}f,W_{q}\pa^{\alpha}_{\beta}f)| &\leq&
\delta (\|\pa^{\alpha}_{\beta}f_{2}\|_{L^{2}_{x}L^{2}_{\epsilon,q+\gamma/2}}^{2}  + \|\pa^{\alpha}_{\beta_{2}}f_{2}\|^{2}_{L^{2}_{x}L^{2}_{\epsilon,\gamma/2}})
\\&&+ C_{\delta, q}\|\pa^{\alpha}_{\beta_{2}}f_{2}\|^{2}_{L^{2}_{x}L^{2}_{q+\gamma/2}}
 + C_{\delta,q,N}|\pa^{\alpha}(a,b,c)|^{2}_{L^{2}_{x}}.
\eeno
\end{lem}
\begin{proof}
By Corollary \ref{commutator-L-gamma-minus3},
we get for any $0<\eta<1$,
\beno |([W_{q}, \mathcal{L}^{\epsilon,\beta_{0},\beta_{1}}]\pa^{\alpha}_{\beta_{2}}f,W_{q}\pa^{\alpha}_{\beta}f)|  &\leq&
\eta^{-6}C_{q}\|\pa^{\alpha}_{\beta_{2}}f\|_{L^{2}_{x}L^{2}_{q+\gamma/2}}\|\pa^{\alpha}_{\beta}f\|_{L^{2}_{x}L^{2}_{\epsilon,q+\gamma/2}}
\\&&+ \eta^{1/2}C_{q}\|\pa^{\alpha}_{\beta_{2}}f\|_{L^{2}_{x}L^{2}_{\epsilon,\gamma/2}}\|\pa^{\alpha}_{\beta}f \|_{L^{2}_{x}L^{2}_{\epsilon,\gamma/2}}
\\ &\leq&
\eta^{1/2}C_{q}\|\pa^{\alpha}_{\beta}f\|_{L^{2}_{x}L^{2}_{\epsilon,q+\gamma/2}}^{2}+ \eta^{-12-1/2}C_{q}\|\pa^{\alpha}_{\beta_{2}}f\|_{L^{2}_{x}L^{2}_{q+\gamma/2}}^{2}
\\&&+ \eta^{1/2}C_{q}\|\pa^{\alpha}_{\beta_{2}}f\|^{2}_{L^{2}_{x}L^{2}_{\epsilon,\gamma/2}}
+ \eta^{1/2}C_{q}\|\pa^{\alpha}_{\beta}f \|^{2}_{L^{2}_{x}L^{2}_{\epsilon,\gamma/2}}
\\ &\leq& \delta \|\pa^{\alpha}_{\beta}f\|_{L^{2}_{x}L^{2}_{\epsilon,q+\gamma/2}}^{2} + \delta \|\pa^{\alpha}_{\beta_{2}}f\|^{2}_{L^{2}_{x}L^{2}_{\epsilon,\gamma/2}}
+ \delta^{-25}C_{q}\|\pa^{\alpha}_{\beta_{2}}f\|_{L^{2}_{x}L^{2}_{q+\gamma/2}}^{2},
\eeno
where we set $\delta = \eta^{1/2}C_{q}$, and the constant $C_{q}$ may change across different lines. From which together with the decomposition $f=f_1+f_2=\mathbb{P}f+f_2$,  we get the lemma.
\end{proof}

\begin{lem} \label{commutator-L-with-beta}
Let $|\alpha|+|\beta| \leq N, |\beta| \geq 1$, $q \geq 2$ or $q=0$,  then
\beno
|(W_{q}[\partial_{\beta}, \mathcal{L}^{\epsilon}]\pa^{\alpha}f ,W_{q}\pa^{\alpha}_{\beta}f)| &\lesssim&
(\delta + C_{N}\epsilon)\|\pa^{\alpha}_{\beta}f_{2}\|_{L^{2}_{x}L^{2}_{\epsilon,q+\gamma/2}}^{2}
 + C_{\delta,q,N}\|\partial^{\alpha}(a,b,c)\|_{L^{2}_{x}}^{2}
\\&&+   C_{\delta,q,N} \sum_{\beta_{2}<\beta} (\|\pa^{\alpha}_{\beta_{2}}f_{2}\|_{L^{2}_{x}L^{2}_{\epsilon,q+\gamma/2}}^{2}
+ \|\partial^{\alpha}_{\beta_{2}}f_{2}\|_{L^{2}_{x}H^{1}_{\gamma/2}}^{2}).
\eeno
\end{lem}
\begin{proof}
Recalling
$ \mathcal{L}^{\epsilon}g = -\Gamma^{\epsilon}(\mu^{1/2},g) - \Gamma^{\epsilon}(g, \mu^{1/2}),$ \eqref{Gamma-beta}, \eqref{alpha-beta-on-Gamma}, and \eqref{beta-version-L},
we have
\ben  \pa_{\beta}\mathcal{L}^{\epsilon}g &=& \mathcal{L}^{\epsilon}\pa_{\beta}g
-\sum_{\beta_{0}+\beta_{1}+\beta_{2}= \beta, \beta_{2} < \beta} C^{\beta_{0},\beta_{1},\beta_{2}}_{\beta}
 [\Gamma^{\epsilon}(\pa_{\beta_{1}}\mu^{1/2}, \pa_{\beta_{2}}g;\beta_{0}) + \Gamma^{\epsilon}(\pa_{\beta_{1}}g, \pa_{\beta_{2}}\mu^{1/2};\beta_{0})]
\nonumber \\ \label{alpha-beta-Lep}
&=& \mathcal{L}^{\epsilon}\pa_{\beta}g
-\sum_{\beta_{0}+\beta_{1}+\beta_{2}= \beta, \beta_{2} < \beta} C^{\beta_{0},\beta_{1},\beta_{2}}_{\beta}
 \mathcal{L}^{\epsilon,\beta_{0},\beta_{1}} \partial_{\beta_{2}}g. \een
In this way, we have $[\partial_{\beta}, \mathcal{L}^{\epsilon}] = \sum_{\beta_{2}<\beta} C^{\beta_{0},\beta_{1},\beta_{2}}_{\beta}\mathcal{L}^{\epsilon,\beta_{0},\beta_{1}} \partial_{\beta_{2}}$,  and thus
\beno W_{q}[\partial_{\beta}, \mathcal{L}^{\epsilon}]\pa^{\alpha}f &=& W_{q}\sum_{\beta_{2}<\beta} C^{\beta_{0},\beta_{1},\beta_{2}}_{\beta}\mathcal{L}^{\epsilon,\beta_{0},\beta_{1}} \partial^{\alpha}_{\beta_{2}}f
\\&=&\sum_{\beta_{2}<\beta} C^{\beta_{0},\beta_{1},\beta_{2}}_{\beta} \mathcal{L}^{\epsilon,\beta_{0},\beta_{1}} W_{q}\partial^{\alpha}_{\beta_{2}}f +
  \sum_{\beta_{2}<\beta}C^{\beta_{0},\beta_{1},\beta_{2}}_{\beta}
  [W_{q},\mathcal{L}^{\epsilon,\beta_{0},\beta_{1}}]\partial^{\alpha}_{\beta_{2}}f. \eeno
By upper bound estimate in Lemma \ref{l-full-estimate}, we get
\beno |(\mathcal{L}^{\epsilon,\beta_{0},\beta_{1}} W_{q}\partial^{\alpha}_{\beta_{2}}f ,W_{q}\pa^{\alpha}_{\beta}f)|
&\lesssim&
\delta^{-6}\|\partial^{\alpha}_{\beta_{2}}f\|_{L^{2}_{x}L^{2}_{\epsilon,q+\gamma/2}}\|\pa^{\alpha}_{\beta}f\|_{L^{2}_{x}L^{2}_{\epsilon,q+\gamma/2}}
\\&&+ (\delta^{1/2}+\epsilon^{1/2})\|\mu^{1/16}\partial^{\alpha}_{\beta_{2}}f\|_{L^{2}_{x}H^{1}_{q}}\|\pa^{\alpha}_{\beta}f\|_{L^{2}_{x}L^{2}_{\epsilon,q+\gamma/2}}
\\&\lesssim&
(\delta+\epsilon) \|\pa^{\alpha}_{\beta}f_{2}\|_{L^{2}_{x}L^{2}_{\epsilon,q+\gamma/2}}^{2}
 + C_{\delta,q,N}\|\partial^{\alpha}(a,b,c)\|_{L^{2}_{x}}^{2}
\\&&+   C_{\delta,q,N}\|\pa^{\alpha}_{\beta_{2}}f_{2}\|_{L^{2}_{x}L^{2}_{\epsilon,q+\gamma/2}}^{2}
+ C_{q}\|\partial^{\alpha}_{\beta_{2}}f_{2}\|_{L^{2}_{x}H^{1}_{\gamma/2}}^{2}.\eeno
where we use $f = f_{1} + f_{2}$ and the definition of $a,b,c$.

If $q\geq 2$, by Lemma \ref{commutator-L-alpha-beta}, we have
\beno
|([W_{q}, \mathcal{L}^{\epsilon,\beta_{0},\beta_{1}}]\pa^{\alpha}_{\beta_{2}}f,W_{q}\pa^{\alpha}_{\beta}f)| &\leq&
\delta (\|\pa^{\alpha}_{\beta}f_{2}\|_{L^{2}_{x}L^{2}_{\epsilon,q+\gamma/2}}^{2}  + \|\pa^{\alpha}_{\beta_{2}}f_{2}\|^{2}_{L^{2}_{x}L^{2}_{\epsilon,\gamma/2}})
\\&&+ C_{\delta, q}\|\pa^{\alpha}_{\beta_{2}}f_{2}\|^{2}_{L^{2}_{x}L^{2}_{q+\gamma/2}}
 + C_{\delta,q,N}|\pa^{\alpha}(a,b,c)|^{2}_{L^{2}_{x}}.
\eeno
If $q=0$, the commutator $[W_{q},\mathcal{L}^{\epsilon,\beta_{0},\beta_{1}}]=0$.
Taking sum over $\beta_{2}<\beta$,
we get the result.
\end{proof}

For non-negative integers $n,m$, we recall:
\beno \|f\|^{2}_{H^{n}_{x}\dot{H}^{m}_{l}} = \sum_{|\alpha|\leq n,|\beta| = m}\|\partial^{\alpha}_{\beta}f\|^{2}_{L^{2}_{x}L^{2}_{l}},
\|f\|^{2}_{H^{n}_{x}\dot{H}^{m}_{\epsilon,l}} = \sum_{|\alpha|\leq n,|\beta| = m}\|\partial^{\alpha}_{\beta}f\|^{2}_{L^{2}_{x}L^{2}_{\epsilon,l}}. \eeno
Let $N \geq 4, l\geq 3N+2$. For some universal constants $M, L, K_{j}, 0 \leq j \leq N$ (which could depend on $N, l$ and will be explicitly determined later), we define
\ben \label{def-energy-combination}
\Xi^{N,l}(f) &=& M\mathcal{I}^{N}(f)+ L\|f\|^{2}_{H^{N}_{x}L^{2}}+ \sum_{j=0}^{N} K_{j}\|f\|^{2}_{H^{N-j}_{x}\dot{H}^{j}_{l+j\gamma}},
\\ \label{def-dissipation}
\mathcal{D}_{\epsilon}^{N,l}(f)&=& c_{0}M|\mathcal{MA}|^{2}_{H^{N}_{x}} + \lambda_{0}L\|f_{2}\|^{2}_{H^{N}_{x}L^{2}_{\epsilon,\gamma/2}}
+\lambda_{0}\sum_{j=0}^{N}K_{j}\|f\|^{2}_{H^{N-j}_{x}\dot{H}^{j}_{\epsilon,l+j\gamma+\gamma/2}}, \een
where $\mathcal{MA}(t,x) = (a(t,x), b(t,x), c(t,x))$ is a vector of length 5,
which stands for the macro-part of a solution $f$.

Now we are in a position to prove
\begin{prop}\label{essential-estimate-of-micro-macro} Let $N \geq 4, l\geq 3N+2$. Suppose $f$ is a smooth solution to \eqref{lBE-general}, then there exist a constant $\epsilon_{1}$ verifying $0< \epsilon_{1} \leq \epsilon_{0}$, such that for any $0 \leq \eta \leq \eta_{0}, 0 \leq \epsilon \leq  \epsilon_{1}$,
there holds
\ben \label{essential-micro-macro-result-2} \frac{d}{dt}\Xi^{N,l}(f) +  \frac{1}{4} \mathcal{D}_{\epsilon}^{N,l}(f) &\leq& M C\sum_{|\alpha| \leq N-1}\sum_{j=1}^{13}
\int_{\mathbb{T}^{3}}|\langle  \pa^{\alpha}g, e_j\rangle|^{2} dx
+2 L \sum_{|\alpha|\leq N}(\pa^{\alpha}g- \mathcal{L}^{\epsilon, \gamma}_{\eta}\pa^{\alpha}f, \pa^{\alpha}f) \\&&+\sum_{j=0}^{N}2K_{j}\sum_{|\alpha|\leq N-j,|\beta|=j}(W_{l+j\gamma}\pa^{\alpha}_{\beta}g- \mathcal{L}^{\epsilon, \gamma}_{\eta}W_{l+j\gamma}\pa^{\alpha}_{\beta}f, W_{l+j\gamma}\pa^{\alpha}_{\beta}f).
\nonumber \een
The constant $\epsilon_{1}>0$ could depend on $N,l$. Here $\eta_{0}, \epsilon_{0}$ are the universal constants in Theorem \ref{micro-weight-dissipation}. Moreover, $M \sim L$.
\end{prop}

\begin{proof} We divide the proof into three steps to construct the energy functional $\Xi^{N,l}(f)$ in \eqref{def-energy-combination}.

{\it Step 1: Propagation of $\|f\|^{2}_{H^{N}_{x}L^{2}}$}.
Applying $\partial^{\alpha}$ to equation \eqref{lBE-general}, taking inner product with $\partial^{\alpha}f$, taking sum over $|\alpha|\leq N$, we have
\ben \label{inner-product-with-pure-x}
\frac{1}{2}\frac{d}{dt}\|f\|^{2}_{H^{N}_{x}L^{2}}  + \sum_{|\alpha| \leq N} (\mathcal{L}^{\epsilon}\pa^{\alpha}f, \pa^{\alpha}f) = \sum_{|\alpha| \leq N}(\pa^{\alpha}g, \pa^{\alpha}f). \een
Split $\mathcal{L}^{\epsilon,\gamma} = \mathcal{L}^{\epsilon, \gamma,\eta} + \mathcal{L}^{\epsilon, \gamma}_{\eta}$.
Thanks to $\partial^{\alpha}f_{2} = (\partial^{\alpha}f)_{2}$,
by Theorem \ref{micro-weight-dissipation}, for $\eta \leq \eta_{0}, \epsilon \leq \epsilon_{0}$,
we have
\beno (\mathcal{L}^{\epsilon, \gamma,\eta} \partial^{\alpha}f, \partial^{\alpha}f) \geq \lambda_{0}\|\partial^{\alpha}f_{2}\|^{2}_{L^{2}_{x}L^2_{\epsilon,\gamma/2}}.\eeno
Plugging which into \eqref{inner-product-with-pure-x}, we have
\ben \label{solution-property-part-g}\frac{1}{2}\frac{d}{dt}\|f\|^{2}_{H^{N}_{x}L^{2}} + \lambda_{0}\|f_{2}\|^{2}_{H^{N}_{x}L^2_{\epsilon,\gamma/2}} \leq \sum_{|\alpha| \leq N}
(\pa^{\alpha}g- \mathcal{L}^{\epsilon, \gamma}_{\eta}\partial^{\alpha}f, \pa^{\alpha}f).\een
Multiplying \eqref{solution-property-part-g} by a large constant $2M_{1}$ and adding it to \eqref{solution-property-part2},
we get
\ben \label{essential-micro-macro-result} &&\frac{d}{dt}(M_{1}\|f\|^{2}_{H^{N}_{x}L^{2}}+\mathcal{I}^{N}(f))+ (c_{0}|\mathcal{MA}|^{2}_{H^{N}_{x}}+M_{1}\lambda_{0}\|f_{2}\|^{2}_{H^{N}_{x}L^2_{\epsilon,\gamma/2}}) \\&\leq& 2M_{1}\sum_{|\alpha| \leq N}
(\pa^{\alpha}g- \mathcal{L}^{\epsilon, \gamma}_{\eta}\partial^{\alpha}f, \pa^{\alpha}f)+ C\sum_{|\alpha| \leq N-1}\sum_{j=1}^{13}
\int_{\mathbb{T}^{3}}|\langle  \pa^{\alpha}g, e_j\rangle|^{2} dx. \nonumber \een
Here $M_{1}$ is large enough such that $M_{1} \geq 2C_{1}$ and $M_{1} \lambda_{0} \geq C$ to insure $ M_{1}\|f\|^{2}_{H^{N}_{x}L^{2}}+\mathcal{I}^{N}(f) \sim \|f\|^{2}_{H^{N}_{x}L^{2}}$ by \eqref{temporal-bounded-by-energy} and cancel the term $C\|f_{2}\|^{2}_{H^{N}_{x}L^2_{\epsilon,\gamma/2}}$ on the right hand side of \eqref{solution-property-part2}.

{\it Step 2: Propagation of $\|f\|^{2}_{H^{N}_{x}L^{2}_{l}}$}.
Applying $W_{l}\partial^{\alpha}$ to equation \eqref{lBE-general}, taking inner product with $W_{l}\partial^{\alpha}f$, taking sum over $|\alpha|\leq N$, we have
\ben \label{inner-product-with-pure-x-weight}
\frac{1}{2}\frac{d}{dt}\|f\|^{2}_{H^{N}_{x}L^{2}_{l}}  + \sum_{|\alpha| \leq N} (W_{l}\mathcal{L}^{\epsilon}\pa^{\alpha}f,W_{l}\pa^{\alpha}f) = \sum_{|\alpha| \leq N}(W_{l}\pa^{\alpha}g,W_{l}\pa^{\alpha}f). \een
Using commutator to transfer weight and splitting $\mathcal{L}^{\epsilon,\gamma} = \mathcal{L}^{\epsilon, \gamma,\eta} + \mathcal{L}^{\epsilon, \gamma}_{\eta}$, we get
\beno W_{l}\mathcal{L}^{\epsilon}\pa^{\alpha}f = \mathcal{L}^{\epsilon}W_{l}\pa^{\alpha}f + [W_{l}, \mathcal{L}^{\epsilon}]\pa^{\alpha}f
= \mathcal{L}^{\epsilon, \gamma,\eta}W_{l}\pa^{\alpha}f +\mathcal{L}^{\epsilon, \gamma}_{\eta}W_{l}\pa^{\alpha}f+ [W_{l}, \mathcal{L}^{\epsilon}]\pa^{\alpha}f.
\eeno
By Lemma \ref{coercivity-L-alpha-beta}, we get
\beno
(\mathcal{L}^{\epsilon, \gamma,\eta} W_{l}\pa^{\alpha}f, W_{l}\pa^{\alpha}f) \geq
 (7/8)\lambda_{0} \|\pa^{\alpha}f_{2}\|_{L^{2}_{x}L^{2}_{\epsilon,l+\gamma/2}}^{2}
-  C_{l,N}( \|\partial^{\alpha}f_{2}\|^{2}_{L^{2}_{x}L^2_{\gamma/2}}
+|\partial^{\alpha}(a,b,c)|^{2}_{L^{2}_{x}}).
\eeno
Thanks to Lemma \ref{commutator-L-alpha-beta}, we have
\beno
|([W_{l}, \mathcal{L}^{\epsilon}]\pa^{\alpha}f,W_{l}\pa^{\alpha}f)| \leq
\delta \|\pa^{\alpha}f_{2}\|_{L^{2}_{x}L^{2}_{\epsilon,l+\gamma/2}}^{2}
 + C_{\delta, l}\|\pa^{\alpha}f_{2}\|^{2}_{L^{2}_{x}L^{2}_{l+\gamma/2}}
 + C_{\delta,l,N}|\pa^{\alpha}(a,b,c)|^{2}_{L^{2}_{x}}.
\eeno
Since $|h|_{L^{2}_{q}} \lesssim (\delta_{2}^{1/2} + \epsilon^{1/2})|W^{\epsilon}h|_{L^{2}_{q}} + C(\delta_{2}, q) |h|_{L^{2}}$ for any $\delta_{2}>0$, we have
\ben \label{reduce-to-gamma/2-2}
 \|\pa^{\alpha}f_{2}\|^{2}_{L^{2}_{x}L^{2}_{l+\gamma/2}} \leq (\delta_{2} + \epsilon)\|\pa^{\alpha}f_{2}\|^{2}_{L^{2}_{x}L^{2}_{\epsilon,l+\gamma/2}}
+ C(\delta_{2}, l)\|\pa^{\alpha}f_{2}\|^{2}_{L^{2}_{x}L^{2}_{\epsilon,\gamma/2}}.\een
First taking  $\delta = \lambda_{0}/8$, then taking $\delta_{2}$ such that $\delta_{2}C_{\delta,l}=\lambda_{0}/8$,
when $\epsilon C_{\delta,l} \leq \lambda_{0}/8$, we get
\ben \label{weighted-pure-x-2}
\frac{d}{dt}\|f\|^{2}_{H^{N}_{x}L^{2}_{l}}  + (\lambda_{0}/2)\|f_{2}\|^{2}_{H^{N}_{x}L^{2}_{\epsilon,l+\gamma/2}} &\leq& C_{l,N}\|f_{2}\|^{2}_{H^{N}_{x}L^{2}_{\epsilon,\gamma/2}}+C_{l,N}|\partial^{\alpha}(a,b,c)|^{2}_{L^{2}_{x}}
\nonumber\\&&+2\sum_{|\alpha| \leq N}(W_{l}\pa^{\alpha}g-\mathcal{L}^{\epsilon, \gamma}_{\eta}W_{l}\pa^{\alpha}f,W_{l}\pa^{\alpha}f).  \een

There is a constant $C_{l}$ such that $\|f_{1}\|^{2}_{H^{N}_{x}L^{2}_{\epsilon,l+\gamma/2}} \leq C_{l}|\mathcal{MA}|^{2}_{H^{N}_{x}}.$
We choose a constant $M_{2}$ large enough such that $c_{0}M_{2}/4 \geq C_{l,N}, c_{0}M_{2}/4 \geq C_{l}\lambda_{0}/2, M_{2}M_{1}\lambda_{0}/2 \geq C_{l,N}$. Multiplying \eqref{essential-micro-macro-result} by the constant $M_{2}$ and adding the resulting inequality to \eqref{weighted-pure-x-2},
we get
\ben \label{essential-micro-macro-result-3} &&\frac{d}{dt}(M_{2}\mathcal{I}^{N}(f)+M_{1}M_{2}\|f\|^{2}_{H^{N}_{x}L^{2}}+\|f\|^{2}_{H^{N}_{x}L^{2}_{l}})
\\&&+ \frac{1}{2}(M_{2}c_{0}|\mathcal{MA}|^{2}_{H^{N}_{x}}+M_{2}M_{1}\lambda_{0}\|f_{2}\|^{2}_{H^{N}_{x}L^2_{\epsilon,\gamma/2}}
+\lambda_{0}\|f_{1}\|^{2}_{H^{N}_{x}L^2_{\epsilon,l+\gamma/2}}+ \lambda_{0}\|f_{2}\|^{2}_{H^{N}_{x}L^2_{\epsilon,l+\gamma/2}})
\nonumber \\&\leq&  M_{2}C\sum_{|\alpha| \leq N-1}\sum_{j=1}^{13}
\int_{\mathbb{T}^{3}}|\langle  \pa^{\alpha}g, e_j\rangle|^{2} dx
+2M_{2}M_{1}\sum_{|\alpha| \leq N}
(\pa^{\alpha}g- \mathcal{L}^{\epsilon, \gamma}_{\eta}\partial^{\alpha}f, \pa^{\alpha}f), \nonumber
\\&&+2\sum_{|\alpha| \leq N}(W_{l}\pa^{\alpha}g-\mathcal{L}^{\epsilon, \gamma}_{\eta}W_{l}\pa^{\alpha}f,W_{l}\pa^{\alpha}f).
\nonumber \een

{\it Step 3: Propagation of  $ \sum_{j=1}^{N}K_{j}\|f\|^{2}_{H^{N-j}_{x}\dot{H}^{j}_{l+j\gamma}}$}.  We shall use mathematical induction to prove that for any $0\leq i \leq N$, there are some constants $M^{i}, L^{i}, K^{i}_{j}, 0 \leq j \leq i$, such that
\ben \label{essential-micro-macro-result-4} &&\frac{d}{dt}(M^{i}\mathcal{I}^{N}(f)+ L^{i}\|f\|^{2}_{H^{N}_{x}L^{2}}+\sum_{0 \leq j \leq i}K^{i}_{j}\|f\|^{2}_{H^{N-j}_{x}\dot{H}^{j}_{l+j\gamma}})\\&&+  2^{-1-i/N}(c_{0}M^{i}|\mathcal{MA}|^{2}_{H^{N}_{x}}
+\lambda_{0}L^{i} \|f_{2}\|^{2}_{H^{N}_{x}L^{2}_{\epsilon,\gamma/2}}
\nonumber\\&&+\lambda_{0}\sum_{j=0}^{i}K^{i}_{j}(\|f_{1}\|^{2}_{H^{N-j}_{x}\dot{H}^{j}_{\epsilon,l+j\gamma+\gamma/2}}
+\|f_{2}\|^{2}_{H^{N-j}_{x}\dot{H}^{j}_{\epsilon,l+j\gamma+\gamma/2}}) \nonumber \\&\leq& M^{i}C\sum_{|\alpha| \leq N-1}\sum_{j=1}^{13}
\int_{\mathbb{T}^{3}}|\langle  \pa^{\alpha}g, e_j\rangle|^{2} dx
+ 2L^{i} \sum_{|\alpha|\leq N}(\pa^{\alpha}g- \mathcal{L}^{\epsilon, \gamma}_{\eta}\pa^{\alpha}f, \pa^{\alpha}f) \nonumber \\&&+\sum_{j=0}^{i}2K^{i}_{j}\sum_{|\alpha|\leq N-j,|\beta|=j}(W_{l+j\gamma}\pa^{\alpha}_{\beta}g- \mathcal{L}^{\epsilon, \gamma}_{\eta}W_{l+j\gamma}\pa^{\alpha}_{\beta}f, W_{l+j\gamma}\pa^{\alpha}_{\beta}f). \nonumber \een

It is easy to check that \eqref{essential-micro-macro-result-4} is valid for $i=0$ thanks to \eqref{essential-micro-macro-result-3}. We also remark that our final goal \eqref{essential-micro-macro-result-2} is actually  \eqref{essential-micro-macro-result-4} with $i=N$ since
$\|f\|^{2}_{H^{N-j}_{x}\dot{H}^{j}_{\epsilon,l+j\gamma+\gamma/2}}
\leq 2\|f_{1}\|^{2}_{H^{N-j}_{x}\dot{H}^{j}_{\epsilon,l+j\gamma+\gamma/2}}
+2\|f_{2}\|^{2}_{H^{N-j}_{x}\dot{H}^{j}_{\epsilon,l+j\gamma+\gamma/2}}.$

Assume \eqref{essential-micro-macro-result-4} is valid for $i = k$ for some $0 \leq k \leq N-1$. We set to prove \eqref{essential-micro-macro-result-4} is also valid for $i=k+1$.

Let $\alpha$ and $\beta$ be multi-indices such that $|\alpha|\leq N-(k+1)$ and $|\beta|= k+1 \geq 1$. Let $q=l+(k+1)\gamma$. Applying $W_{q}\pa^{\alpha}_{\beta}$ to both sides of \eqref{lBE-general}, we obtain
\ben \label{weight-q-LBE-3} \partial_{t}W_{q}\pa^{\alpha}_{\beta}f + v\cdot \nabla_{x} W_{q}\pa^{\alpha}_{\beta}f +\sum_{\beta_{1}\leq \beta,|\beta_{1}|=1} W_{q}\pa^{\alpha+\beta_{1}}_{\beta-\beta_{1}}f + W_{q}\pa^{\alpha}_{\beta}\mathcal{L}^{\epsilon}f = W_{q}\pa^{\alpha}_{\beta}g. \een
Taking inner product with $W_{q}\pa^{\alpha}_{\beta} f$ over $(x,v)$, one has
\ben \label{weight-mixed-derivative} \frac{1}{2}\frac{d}{dt}\|\pa^{\alpha}_{\beta}f \|^{2}_{L^{2}_{q}}  +  \sum_{\beta_{1}\leq \beta,|\beta_{1}|=1}(W_{q}\pa^{\alpha+\beta_{1}}_{\beta-\beta_{1}}f,W_{q}\pa^{\alpha}_{\beta}f) + (W_{q}\pa^{\alpha}_{\beta}\mathcal{L}^{\epsilon}f,W_{q}\pa^{\alpha}_{\beta}f) = (W_{q}\pa^{\alpha}_{\beta}g,W_{q}\pa^{\alpha}_{\beta}f).  \een

\underline{\it Estimate of $(W_{q}\pa^{\alpha+\beta_{1}}_{\beta-\beta_{1}}f,W_{q}\pa^{\alpha}_{\beta}f)$.} By Cauchy-Schwartz inequality and using $f=f_{1}+f_{2}$, we get
\ben \label{weight-term-1}
|(W_{q}\pa^{\alpha+\beta_{1}}_{\beta-\beta_{1}}f, W_{q}\pa^{\alpha}_{\beta}f)| &\leq& \|\pa^{\alpha+\beta_{1}}_{\beta-\beta_{1}}f\|_{L^2_{x}L^{2}_{q-\gamma/2}}\|\pa^{\alpha}_{\beta}f\|_{L^2_{x}L^{2}_{q+\gamma/2}}
\\&\lesssim& \|f_{2}\|^{2}_{H^{N-k}_{x}\dot{H}^{k}_{\epsilon,q-\gamma+\gamma/2}} +\|f_{2}\|^{2}_{H^{N-k-1}_{x}\dot{H}^{k+1}_{q+\gamma/2}} + C_{l}|\mathcal{MA}|^{2}_{H^{N-k}_{x}}. \nonumber \een

\underline{\it Estimate of $(W_{q}\pa^{\alpha}_{\beta}\mathcal{L}^{\epsilon}f,W_{q}\pa^{\alpha}_{\beta}f) $.}
Using commutator to transfer weight and splitting $\mathcal{L}^{\epsilon,\gamma} = \mathcal{L}^{\epsilon, \gamma,\eta} + \mathcal{L}^{\epsilon, \gamma}_{\eta}$, we get
\ben \label{weight-separation}
W_{q}\pa^{\alpha}_{\beta}\mathcal{L}^{\epsilon}f = \mathcal{L}^{\epsilon, \gamma,\eta}W_{q}\pa^{\alpha}_{\beta}f +\mathcal{L}^{\epsilon, \gamma}_{\eta}W_{q}\pa^{\alpha}_{\beta}f + [W_{q}, \mathcal{L}^{\epsilon}]\pa^{\alpha}_{\beta}f+
W_{q}[\partial_{\beta}, \mathcal{L}^{\epsilon}]\pa^{\alpha}f.
\een
By Lemma \ref{coercivity-L-alpha-beta}, Lemma \ref{commutator-L-alpha-beta} and Lemma \ref{commutator-L-with-beta}, we have
\beno
&&(\mathcal{L}^{\epsilon, \gamma,\eta} W_{q}\pa^{\alpha}_{\beta}f+[W_{q}, \mathcal{L}^{\epsilon}]\pa^{\alpha}_{\beta}f+
W_{q}[\partial_{\beta}, \mathcal{L}^{\epsilon}]\pa^{\alpha}f, W_{q}\pa^{\alpha}_{\beta}f)
\\&\geq&
(7\lambda_0/8 - 3\delta-C_{N}\epsilon)\| W_{q}\pa^{\alpha}_{\beta} f_2\|^2_{L^{2}_{x}L^2_{\epsilon,\gamma/2}}
- C_{\delta, q}\|\pa^{\alpha}_{\beta}f_{2}\|^{2}_{L^{2}_{x}L^{2}_{q+\gamma/2}}
 - C_{\delta,q,N}\|\partial^{\alpha}(a,b,c)\|_{L^{2}_{x}}^{2}
\\&& -   C_{\delta,q,N} \sum_{\beta_{2}<\beta} ( \|\pa^{\alpha}_{\beta_{2}}f_{2}\|_{L^{2}_{x}L^{2}_{\epsilon,q+\gamma/2}}^{2}
+ \|\partial^{\alpha}_{\beta_{2}}f_{2}\|_{L^{2}_{x}H^{1}_{\gamma/2}}^{2} ).
\eeno
Taking $\delta$ such that $3\delta = \lambda_0/16$, when $\epsilon$ is small such that $C_{N}\epsilon \leq \lambda_0/16$, we get
\ben \label{weight-term-2}
(\mathcal{L}^{\epsilon, \gamma,\eta} W_{q}\pa^{\alpha}_{\beta}f+[W_{q}, \mathcal{L}^{\epsilon}]\pa^{\alpha}_{\beta}f+
W_{q}[\partial_{\beta}, \mathcal{L}^{\epsilon}]\pa^{\alpha}f, W_{q}\pa^{\alpha}_{\beta}f)
\geq
 (3\lambda_{0}/4)\|\pa^{\alpha}_{\beta}f_{2}\|_{L^{2}_{x}L^{2}_{\epsilon,q+\gamma/2}}^{2}
 \\- C_{q,N}\|f_{2}\|^{2}_{H^{N-k-1}_{x}H^{k+1}_{q+\gamma/2}} - C_{q,N}\|f_{2}\|^{2}_{H^{N-k-1}_{x}H^{k}_{\epsilon, q+\gamma/2}}
 - C_{q,N}\|\partial^{\alpha}(a,b,c)\|_{L^{2}_{x}}^{2}. \nonumber
\een

Plugging \eqref{weight-term-1}, \eqref{weight-separation} and \eqref{weight-term-2} into \eqref{weight-mixed-derivative},
taking sum over $|\alpha|\leq N-(k+1), |\beta| = k+1$, we have
\ben \label{essential-micro-macro-result-pure}&&\frac{d}{dt}\|f\|^{2}_{H^{N-k-1}_{x}\dot{H}^{k+1}_{q}}+ \frac{3}{2}\lambda_{0} \|f_{2}\|^{2}_{H^{N-k-1}_{x}\dot{H}^{k+1}_{\epsilon,q+\gamma/2}}  \\&\leq& 2\sum_{|\alpha|\leq N-k-1,|\beta|=k+1}(W_{q}\pa^{\alpha}_{\beta}g-\mathcal{L}^{\epsilon, \gamma}_{\eta}W_{q}\pa^{\alpha}_{\beta}f, W_{q}\pa^{\alpha}_{\beta}f)
+C_{N}\|f_{2}\|^{2}_{H^{N-k}_{x}\dot{H}^{k}_{\epsilon,q-\gamma+\gamma/2}}
\nonumber \\&&+C_{l,N}(\|f_{2}\|^{2}_{H^{N-k-1}_{x}H^{k+1}_{q+\gamma/2}} +\|f_{2}\|^{2}_{H^{N-k-1}_{x}H^{k}_{\epsilon,q+\gamma/2}}
 +|\mathcal{MA}|^{2}_{H^{N-k}_{x}})
. \nonumber\een
By Lemma \ref{interWsd}, for any $0<\delta'<1$, we have
\beno \|f_{2}\|^{2}_{H^{N-k-1}_{x}H^{k+1}_{q+\gamma/2}} \leq (\delta^{\prime}+\epsilon)\|f_{2}\|^{2}_{H^{N-k-1}_{x}H^{k+1}_{\epsilon,q+\gamma/2}}
+C(\delta',k)\|f_{2}\|_{H^{N-k-1}_{x}L^{2}_{q+\gamma/2}}^{2}.
 \eeno
Taking $\delta^{\prime}$ such that $\delta^{\prime}C_{l,N} = \lambda_{0}/4$, when $\epsilon$ satisfies $\epsilon C_{l,N} \leq \lambda_0/4$, recalling  $q=l+(k+1)\gamma$, we get
\ben \label{essential-micro-macro-result-pure-2}&&\frac{d}{dt}\|f\|^{2}_{H^{N-k-1}_{x}\dot{H}^{k+1}_{q}}+ \lambda_{0} \|f_{2}\|^{2}_{H^{N-k-1}_{x}\dot{H}^{k+1}_{\epsilon,q+\gamma/2}}\nonumber \\&\leq& 2\sum_{|\alpha|\leq N-k-1,|\beta|=k+1}(W_{q}\pa^{\alpha}_{\beta}g-\mathcal{L}^{\epsilon, \gamma}_{\eta}W_{q}\pa^{\alpha}_{\beta}f, W_{q}\pa^{\alpha}_{\beta}f)\nonumber
+C_{l,N}(\sum_{j=0}^{k}\|f_{2}\|^{2}_{H^{N-j}_{x}\dot{H}^{j}_{\epsilon,l+j\gamma+\gamma/2}}+|\mathcal{MA}|^{2}_{H^{N-k}_{x}})
. \een

For notational convenience, set $\mathcal{X}^{k}(f):=M^{k}\mathcal{I}^{N}(f)+L^{k}\|f\|^{2}_{H^{N}_{x}L^{2}}+\sum_{0 \leq j \leq k}K^{k}_{j}\|f\|^{2}_{H^{N-j}_{x}\dot{H}^{j}_{l+j\gamma}}$,
and
\beno \mathcal{Y}^{k}_{\epsilon}(f):=c_{0}M^{k}|\mathcal{MA}|^{2}_{H^{N}_{x}}
+\lambda_{0} L^{k}\|f_{2}\|^{2}_{H^{N}_{x}L^{2}_{\epsilon,\gamma/2}}
+\lambda_{0}\sum_{j=0}^{k}K^{k}_{j} ( \|f_{1}\|^{2}_{H^{N-j}_{x}\dot{H}^{j}_{\epsilon,l+j\gamma+\gamma/2}}+\|f_{2}\|^{2}_{H^{N-j}_{x}\dot{H}^{j}_{\epsilon,l+j\gamma+\gamma/2}}). \eeno
By our induction assumption, \eqref{essential-micro-macro-result-4} is true when $i=k$, that is,
\ben \label{essential-micro-macro-result-4-i-k} &&\frac{d}{dt}\mathcal{X}^{k}(f)+ 2^{1-k/N}\mathcal{Y}^{k}_{\epsilon}(f)  \\&\leq& M^{k}C\sum_{|\alpha| \leq N-1}\sum_{j=1}^{13}
\int_{\mathbb{T}^{3}}|\langle  \pa^{\alpha}g, e_j\rangle|^{2} dx
+ 2L^{k} \sum_{|\alpha|\leq N}(\pa^{\alpha}g- \mathcal{L}^{\epsilon, \gamma}_{\eta}\pa^{\alpha}f, \pa^{\alpha}f) \nonumber \\&&+\sum_{j=0}^{k}2K^{k}_{j}\sum_{|\alpha|\leq N-j,|\beta|=j}(W_{l+j\gamma}\pa^{\alpha}_{\beta}g- \mathcal{L}^{\epsilon, \gamma}_{\eta}W_{l+j\gamma}\pa^{\alpha}_{\beta}f, W_{l+j\gamma}\pa^{\alpha}_{\beta}f). \nonumber \een
There is a constant $C_{l,N}$ such that $\|f_{1}\|^{2}_{H^{N-k-1}_{x}H^{k+1}_{\epsilon,q+\gamma/2}} \leq C_{l,N}|\mathcal{MA}|^{2}_{H^{N}_{x}}.$
We choose a constant $M_{3}$ large enough such that
\beno
M_{3}(1-2^{-1/N})c_{0}M^{k}/2 \geq C_{l,N}, M_{3}(1-2^{-1/N})c_{0}M^{k}/2 \geq C_{l,N}\lambda_{0},
\\ M_{3}(1-2^{-1/N}) \lambda_{0} \min_{0 \leq j \leq k} \{K^{k}_{j}\} \geq C_{l,N}. \eeno
Multiplying \eqref{essential-micro-macro-result-4-i-k} by the constant $M_{3}$,
and adding the resulting inequality  to \eqref{essential-micro-macro-result-pure},
we get
\ben \label{essential-micro-macro-result-4-i-k-plus} &&\frac{d}{dt} (M_{3}\mathcal{X}^{k}(f)+ \|f\|^{2}_{H^{N-k-1}_{x}\dot{H}^{k+1}_{q}})+ M_{3}2^{-1/N}2^{1-k/N} \mathcal{Y}^{k}_{\epsilon}(f)
\\&& +\lambda_{0} \|f_{1}\|^{2}_{H^{N-k-1}_{x}\dot{H}^{k+1}_{\epsilon,q+\gamma/2}}+ \lambda_{0} \|f_{2}\|^{2}_{H^{N-k-1}_{x}\dot{H}^{k+1}_{\epsilon,q+\gamma/2}}
\nonumber \\&\leq& M_{3}M^{k}C\sum_{|\alpha| \leq N-1}\sum_{j=1}^{13}
\int_{\mathbb{T}^{3}}|\langle  \pa^{\alpha}g, e_j\rangle|^{2} dx
+ 2M_{3}L^{k} \sum_{|\alpha|\leq N}(\pa^{\alpha}g- \mathcal{L}^{\epsilon, \gamma}_{\eta}\pa^{\alpha}f, \pa^{\alpha}f) \nonumber \\&&+\sum_{j=0}^{k}2M_{3}K^{k}_{j}\sum_{|\alpha|\leq N-j,|\beta|=j}(W_{l+j\gamma}\pa^{\alpha}_{\beta}g- \mathcal{L}^{\epsilon, \gamma}_{\eta}W_{l+j\gamma}\pa^{\alpha}_{\beta}f, W_{l+j\gamma}\pa^{\alpha}_{\beta}f). \nonumber
\\&&+2\sum_{|\alpha|\leq N-k-1,|\beta|=k+1}(W_{q}\pa^{\alpha}_{\beta}g-\mathcal{L}^{\epsilon, \gamma}_{\eta}W_{q}\pa^{\alpha}_{\beta}f, W_{q}\pa^{\alpha}_{\beta}f).\nonumber
\een
So we get \eqref{essential-micro-macro-result-4} for $i=k+1$. In detail, we set $M^{k+1} = M_{3}M^{k}, L^{k+1}=M_{3} L^{k}, K^{k+1}_{j}=M_{3} K^{k}_{j}$ for $0 \leq j\leq k$ and
$K^{k+1}_{k+1}=1$.
\end{proof}

\subsection{Global well-posedenss and propagation of regularity}
The local well-posedness and the non-negativity of the solution to \eqref{coboltzmann} were well established in \cite{he2014well}. Thus to prove global well-posedness, we only need to
  provide the {\it a priori} estimates for the equation, which is Theorem \ref{a-priori-estimate-LBE}.

\subsubsection{A priori estimate of Boltzmann equation \eqref{linearizedBE}.}
In this subsection, we derive the following a priori estimate for solutions to the Cauchy problem \eqref{linearizedBE}.
\begin{thm}\label{a-priori-estimate-LBE} There exists a universal constant  $\delta_{0}>0$ such that the following statement is valid.
Let $N \geq 4, l \geq 3N+2$, there is a constant $\epsilon_{0}$ which may depend on $N, l$, such that if $\epsilon \leq \epsilon_{0}$ and $f^{\epsilon}$ is a solution of the Cauchy problem \eqref{linearizedBE}  satisfying
$\sup_{0 \leq t \leq T} \mathcal{E}^{4,14} (f^{\epsilon}(t)) \le \delta_{0}$, then for any $t\in[0,T]$,
\begin{enumerate}
\item if $N=4, l=14$, the solution $f^{\epsilon}$  verifies
\ben \label{uniform-estimate-propagation-N-4-l-15} \mathcal{E}^{4,14}(f^{\epsilon}(t)) +  \int_{0}^{t}\mathcal{D}_{\epsilon}^{4,14}(f^{\epsilon}(s))ds \leq C \mathcal{E}^{4,14}(f_{0});
\een
\item if $N=4, l>14$, the solution $f^{\epsilon}$  verifies	
\ben\label{uniform-estimate-propagation-N-4-l-big} \mathcal{E}^{4,l}(f^{\epsilon}(t)) +  \int_{0}^{t}\mathcal{D}_{\epsilon}^{4,l}(f^{\epsilon}(s))ds \leq C_{l}\mathcal{E}^{4,l}(f_{0});
\een	
\item if $N\geq 5, l \geq 3N+2$, the solution $f^{\epsilon}$  verifies	
\ben \label{uniform-estimate-propagation-N-geq-5-l-big} \mathcal{E}^{N,l}(f^{\epsilon}(t)) +  \int_{0}^{t}\mathcal{D}_{\epsilon}^{N,l}(f^{\epsilon}(s))ds \leq P_{N,l}(\mathcal{E}^{N,l}(f_{0})).
\een	
\end{enumerate}
Here $C$ is a universal constant, $C_{l}$ is a constant depending on $l$ and $P_{N,l}(\cdot)$ is a continuous and increasing function with $P_{N,l}(0)=0$.
\end{thm}

Recall from \eqref{energy-functional} the energy functional $\mathcal{E}^{N,l} =\sum_{j=0}^{N}\|f\|^{2}_{H^{N-j}_{x}\dot{H}^{j}_{l+j\gamma}}$. For some constants $C_{0}$ depending only on $N, l$, we have
\ben \label{energy-equivalence}
\mathcal{E}^{N,l}(f) \leq \Xi^{N,l}(f) \leq C_{0}(N,l) \mathcal{E}^{N,l}(f),
\een
In order to prove Theorem \ref{a-priori-estimate-LBE},
we employ Proposition \ref{essential-estimate-of-micro-macro} by taking $g = \Gamma^{\epsilon}(f^{\epsilon},f^{\epsilon})$ and get
\ben \label{g=gamma-f-f}
\frac{d}{dt}\Xi^{N,l}(f^{\epsilon}) +  \frac{1}{4} \mathcal{D}_{\epsilon}^{N,l}(f^{\epsilon}) \leq \sum_{j=0}^{N}2K_{j} \mathcal{A}^{N,j,l}_{\epsilon}(f^{\epsilon},f^{\epsilon})
+2 L \mathcal{B}^{N}_{\epsilon}(f^{\epsilon},f^{\epsilon})+M C \mathcal{C}^{N}_{\epsilon}(f^{\epsilon},f^{\epsilon}), \een
where
\ben \label{A-N-j-l}
\mathcal{A}^{N,j,l}_{\epsilon}(g,f) := \sum_{|\alpha|\leq N-j,|\beta|=j}(W_{l+j\gamma}\pa^{\alpha}_{\beta}\Gamma^{\epsilon}(g,f) - \mathcal{L}^{\epsilon,\gamma}_{\eta}W_{l+j\gamma}\pa^{\alpha}_{\beta}f, W_{l+j\gamma}\pa^{\alpha}_{\beta}f).
\\ \label{B-N-j-0}
\mathcal{B}^{N}_{\epsilon}(g,f) := \sum_{|\alpha|\leq N}(\pa^{\alpha}\Gamma^{\epsilon}(g,f)- \mathcal{L}^{\epsilon, \gamma}_{\eta} \pa^{\alpha}f, \pa^{\alpha}f),
\\ \label{C-N-epsilon-0}
\mathcal{C}^{N}_{\epsilon}(g,f) := \sum_{|\alpha| \leq N-1}\sum_{j=1}^{13}
\int_{\mathbb{T}^{3}}|\langle  \pa^{\alpha}\Gamma^{\epsilon}(g,f), e_j\rangle|^{2} dx.
\een

To move forward based on \eqref{g=gamma-f-f}, we need to estimate $\mathcal{A}^{N,j,l}_{\epsilon}(f^{\epsilon},f^{\epsilon}), \mathcal{B}^{N}_{\epsilon}(f^{\epsilon},f^{\epsilon}), \mathcal{C}^{N}_{\epsilon}(f^{\epsilon},f^{\epsilon})$. To this end, we will give estimates of functionals $\mathcal{A}^{N,j,l}_{\epsilon}$ and $\mathcal{B}^{N}_{\epsilon}$ in Lemma \ref{non-linear-term}, functional $\mathcal{C}^{N}_{\epsilon}$ in Lemma  \ref{pure-x-term-from-I}. To keep the proof of Lemma \ref{non-linear-term} in a reasonable length, we prepare a commutator estimate as Lemma \ref{derivative-commutator-lemma}.

Recalling from \eqref{def-dissipation} the dissipation functional $\mathcal{D}_{\epsilon}^{N,l}$, we have
\ben \label{lower-bound-of-D}
\mathcal{D}_{\epsilon}^{N,l}(f) \geq (c_{0}M/2)|\mathcal{MA}|^{2}_{H^{N}_{x}} + c_{1} \lambda_{0}L\|f\|^{2}_{H^{N}_{x}H^{0}_{\epsilon,\gamma/2}}
+\lambda_{0}\sum_{j=0}^{N}K_{j}\|f\|^{2}_{H^{N-j}_{x}\dot{H}^{j}_{\epsilon,l+j\gamma+\gamma/2}},
\een
for some universal constant $c_{1}=\left(M c_{0}/4L\lambda_{0} \right)\wedge 1$ since $M \sim L$ by Proposition \ref{essential-estimate-of-micro-macro}.
For simplicity, we define
\beno \|f\|^{2}_{H^{m}_{x,v}} := \sum_{|\alpha|+|\beta|\leq m}\|\pa^{\alpha}_{\beta}f\|^{2}_{L^{2}},  \,\, \mathcal{D}_{\epsilon}^{m}(f) := \sum_{|\alpha|+|\beta|\leq m} \|\pa_{\beta}^{\alpha}f\|_{H^{0}_xH^{0}_{\epsilon,\gamma/2}}^2. \eeno

\begin{lem} \label{derivative-commutator-lemma} Let $N \geq 4, l \geq 3N+2, 0 \leq j \leq N$.
Let $\alpha, \beta$ satisfy $|\alpha| \leq N-j, |\beta|=j$.  The following three statements hold true.
\begin{enumerate}
	\item If $N=4$ and $l=14$, then
\ben \label{derivative-commutator-N-4-l-14}
|(W_{l+j\gamma}[\pa^{\alpha}_{\beta},\Gamma^{\epsilon}(g,\cdot)]h, W_{l+j\gamma}\pa_{\beta}^{\alpha}f)| \lesssim \|g\|_{H^{4}_{x,v}} \left(\mathcal{D}_{\epsilon}^{4,14}(h)\right)^{1/2}\left(\mathcal{D}_{\epsilon}^{4,14}(f)\right)^{1/2}.
\een
\item If $N=4$ and $l>14$,  then for any $\delta>0$,
\ben \label{derivative-commutator-N-4-l-large}
|(W_{l+j\gamma}[\pa^{\alpha}_{\beta},\Gamma^{\epsilon}(g,\cdot)]h, W_{l+j\gamma}\pa_{\beta}^{\alpha}f)| &\lesssim& \|g\|_{H^{4}_{x,v}} \|h\|_{H^{N-j}_x \dot{H}^{j}_{\epsilon,l+j\gamma+\gamma/2}} \|\pa_{\beta}^{\alpha}f\|_{H^{0}_xH^{0}_{\epsilon,l+j\gamma+\gamma/2}} +
\\&&+ \delta \|\pa_{\beta}^{\alpha}f\|^{2}_{H^{0}_xH^{0}_{\epsilon,l+j\gamma+\gamma/2}} + \delta^{-1}C_{l} \|g\|^{2}_{H^{4}_{x,v}} \mathcal{E}^{4,l}(h).  \nonumber
\een
 \item If $N \geq 5$ and $l \geq 3N+2$, then for any $\delta>0$,
\ben \label{derivative-commutator-N-geq-5}
|(W_{l+j\gamma}[\pa^{\alpha}_{\beta},\Gamma^{\epsilon}(g,\cdot)]h, W_{l+j\gamma}\pa_{\beta}^{\alpha}f)| &\lesssim& \delta \|\pa_{\beta}^{\alpha}f\|^{2}_{H^{0}_xH^{0}_{\epsilon,l+j\gamma+\gamma/2}}
+ \delta^{-1}C_{N,l}\|g\|^{2}_{H^{N}_{x,v}}\mathcal{D}_{\epsilon}^{N-1,l}(h)
\\&&+ \delta^{-1}C_{N,l}\|g\|^{2}_{H^{4}_{x,v}}\|h\|^{2}_{H^{N}_{x,v}}. \nonumber
\een
\end{enumerate}

Let $N \geq 4$. Let $\alpha$ satisfy $|\alpha| \leq N$. The following two statements hold true.
\begin{enumerate}
	\item If $N=4$, then
\ben \label{derivative-commutator-N-4-no-weight}
|([\pa^{\alpha},\Gamma^{\epsilon}(g,\cdot)]h, \pa^{\alpha}f)| \lesssim \|g\|_{H^{4}_{x,v}}
\sqrt{\mathcal{D}_{\epsilon}^{4}(h)} \sqrt{\mathcal{D}_{\epsilon}^{4}(f)}.
\een
\item  If $N\geq 5$, then for any $\delta>0$,
\ben \label{derivative-commutator-N-geq-5-no-weight}
|([\pa^{\alpha},\Gamma^{\epsilon}(g,\cdot)]h, \pa^{\alpha}f)| \lesssim  \delta \|\pa^{\alpha}f\|^{2}_{H^{0}_xH^{0}_{\epsilon, \gamma/2}}
+ \delta^{-1}C_{N}\|g\|^{2}_{H^{N}_{x,v}}\mathcal{D}_{\epsilon}^{N-1}(h)
+ \delta^{-1}C_{N}\|g\|^{2}_{H^{4}_{x,v}}\|h\|^{2}_{H^{N}_{x,v}}.
\een
\end{enumerate}
\end{lem}

\begin{proof} Set $q=l+j\gamma$.
By the binomial expansion \eqref{alpha-beta-on-Gamma}, we have
\ben \label{derivative-commutator}
W_{q}[\pa^{\alpha}_{\beta},\Gamma^{\epsilon}(f,\cdot)]f &=&  W_{q}\pa^{\alpha}_{\beta}\Gamma^{\epsilon}(f,f) - W_{q}\Gamma^{\epsilon}(f,\pa^{\alpha}_{\beta}f) \nonumber \\&=& W_{q}\sum C(\alpha_{1},\alpha_{2},\beta_{0},\beta_{1},\beta_{2}) \Gamma^{\epsilon}(\partial^{\alpha_{1}}_{\beta_{1}}f, \partial^{\alpha_{2}}_{\beta_{2}}f;\beta_{0}),
\een
where the sum is over $\alpha_{1}+\alpha_{2}=\alpha, \beta_{1}+\beta_{2}\leq \beta, |\alpha_{2}+\beta_{2}|\leq |\alpha|+|\beta|-1$.
By \eqref{version-2-dissipation-Boltzmann} in Corollary \ref{Gamma-full-up-bound-with-weight}, we have for $b_{1} \geq 0,b_{2}\geq 1$ with $b_{1}+b_{2}=2$,
\ben \label{upper-bound-2}
|\langle W_{q}\Gamma^{\epsilon,\gamma}(g,h; \beta_{0}), W_{q}f\rangle| &\lesssim&
|g|_{H^{0}}|h|_{H^{0}_{\epsilon,q+\gamma/2}}|f|_{H^{0}_{\epsilon,q+\gamma/2}} + C_{q}|g|_{H^{0}}|h|_{H^{0}_{q+\gamma/2}}|f|_{H^{0}_{\epsilon,q+\gamma/2}}
\nonumber \\&&+
C_{q}|\mu^{1/16}g|_{H^{b_{1}}}|\mu^{1/32}h|_{H^{b_{2}}}|f|_{H^{0}_{\epsilon,q+\gamma/2}}.
\een
If we denote the Fourier transform of $f$ with respect to $x$ variable by $\widehat{f}$, then we have
\beno (\Gamma^\epsilon(g, h; \beta_{0}), f)=\sum_{k,m\in\Z^3} \langle \Gamma^\epsilon (\widehat{g}(k), \widehat{h}(m-k); \beta_{0}), \widehat{f}(m)\rangle. \eeno
From which together with \eqref{upper-bound-2}, we get
\beno &&|(W_{q}\Gamma^\epsilon(\pa^{\alpha_{1}}_{\beta_{1}} g, \pa^{\alpha_{2}}_{\beta_{2}} h; \beta_{0}), W_{q}f)|
\\&\lesssim&
\sum_{k,m\in\Z^3} |k|^{|\alpha_{1}|}|m-k|^{|\alpha_{2}|}|\widehat{\pa_{\beta_{1}}g}(k)|_{H^{0}}
|\widehat{\pa_{\beta_{2}}h}(m-k)|_{H^{0}_{\epsilon,q+\gamma/2}}|\widehat{f}(m)|_{H^0_{\epsilon,q+\gamma/2}}
\\&&+C_{q}\sum_{k,m\in\Z^3} |k|^{|\alpha_{1}|}|m-k|^{|\alpha_{2}|}|\widehat{\pa_{\beta_{1}}g}(k)|_{H^{0}}
|\widehat{\pa_{\beta_{2}}h}(m-k)|_{H^{0}_{q+\gamma/2}}|\widehat{f}(m)|_{H^0_{\epsilon,q+\gamma/2}}
\\&&+C_{q}
\sum_{k,m\in\Z^3} |k|^{|\alpha_{1}|}|m-k|^{|\alpha_{2}|}|\mu^{1/16}\widehat{\pa_{\beta_{1}}g}(k)|_{H^{b_{1}}}
|\mu^{1/32}\widehat{\pa_{\beta_{2}}h}(m-k)|_{H^{b_{2}}}|\widehat{f}(m)|_{H^0_{\epsilon,q+\gamma/2}}.
\eeno
From which we derive that  for $a_{1}, a_{2}\ge 0$ with $a_{1}+a_{2}=2$,
\ben\label{HNGamma1-continue}
&&|( W_{q}\Gamma^\epsilon(\pa_{\beta_{1}}^{\alpha_{1}} g, \pa_{\beta_{2}}^{\alpha_{2}} h), W_{q}\pa_{\beta}^{\alpha}f)|\lesssim
\|g\|_{H^{|\alpha_{1}|+a_{1}}_xH^{|\beta_{1}|}}(\|\pa_{\beta_{2}}h\|_{H^{|\alpha_{2}|+a_{2}}_xL^2_{\epsilon,q+\gamma/2}}+C_q\|\pa_{\beta_{2}}h\|_{H^{|\alpha_{2}|+a_{2}}_xL^2_{q+\gamma/2}})
\nonumber \\&& \quad\times\|\pa_{\beta}^{\alpha}f\|_{H^{0}_xH^{0}_{\epsilon,q+\gamma/2}}+C_{q}\|\mu^{1/16}g\|_{H^{|\alpha_{1}|+a_{1}}_xH^{|\beta_{1}|+b_{1}}}\|\mu^{1/32}h\|_{H^{|\alpha_{2}|+a_{2}}_xH^{|\beta_{2}|+b_{2}}}
\|\pa_{\beta}^{\alpha}f\|_{H^{0}_xH^{0}_{\epsilon,q+\gamma/2}}.
\een

In the following, we choose $a_{1}, a_{2}, b_{1}, b_{2} \in \{0,1,2\}$ with $a_{1}+a_{2}=b_{1}+b_{2}=2$ and $b_2\ge1$.
For $N \geq 4$ and multi-indices $\alpha, \beta$ with $|\alpha|+|\beta| \leq N$, we consider all the combinations of $\alpha_{1},\alpha_{2},\beta_{1},\beta_{2}$ such that $\alpha_{1}+\alpha_{2}=\alpha, \beta_{1}+\beta_{2} \leq \beta, |\alpha_{2}+\beta_{2}|\leq |\alpha|+|\beta|-1$ in Table \ref{parameter-2} for the choice of $a_{1}, a_{2}, b_{1}, b_{2}$.

\begin{table}[!htbp]
\centering
\caption{Parameter choice}\label{parameter-2}
\begin{tabular}{ccccc}
\hline
$(|\alpha_{1}|,|\beta_{1}|)$  &$(|\alpha_{2}|,|\beta_{2}|)$ & $(a_{1}, a_{2}, b_{1}, b_{2})$ & $|\alpha_{1}|+a_{1}+|\beta_{1}|+b_{1}$ &
$|\alpha_{2}|+a_{2}+|\beta_{2}|+b_{2}$   \\
\hline
(0,0)& $(|\alpha|, \leq |\beta|-1)$& (2,0,1,1) & 3 &  $\leq |\alpha|+|\beta|$ \\
(0,1)& $(|\alpha|, \leq |\beta|-1)$& (2,0,1,1) & 4 & $\leq |\alpha|+|\beta|$\\
(1,0)& $(|\alpha|-1, \leq |\beta|)$& (2,0,1,1) & 4 & $\leq |\alpha|+|\beta|$\\
(0,2)& $(|\alpha|, \leq |\beta|-2)$& (2,0,0,2) & 4 & $\leq |\alpha|+|\beta|$\\
(1,1)& $(|\alpha|-1, \leq |\beta|-1)$& (1,1,1,1) & 4 & $\leq |\alpha|+|\beta|$\\
(2,0)& $(|\alpha|-2, \leq |\beta|)$& (1,1,1,1) & 4 & $\leq |\alpha|+|\beta|$\\
(0,3)& $(|\alpha|, \leq |\beta|-3)$& (1,1,0,2) & 4 & $\leq |\alpha|+|\beta|$\\
(1,2)& $(|\alpha|-1, \leq |\beta|-2)$& (1,1,0,2) & 4 & $\leq |\alpha|+|\beta|$\\
(2,1)& $(|\alpha|-2, \leq |\beta|-1)$& (0,2,1,1) & 4 & $\leq |\alpha|+|\beta|$\\
(3,0)& $(|\alpha|-3, \leq |\beta|)$& (0,2,1,1) & 4 & $\leq |\alpha|+|\beta|$\\
$|\alpha_{1}|+|\beta_{1}|=4$& $(|\alpha|-|\alpha_{1}|, \leq |\beta|-|\beta_{1}|)$& (0,2,0,2) & 4 & $\leq |\alpha|+|\beta|$\\
$|\alpha_{1}|+|\beta_{1}|\geq5$& $(|\alpha|-|\alpha_{1}|, \leq |\beta|-|\beta_{1}|)$& (0,2,0,2) & N & $\leq |\alpha|+|\beta|-1$\\
\hline
\end{tabular}
\end{table}

To summarize, if $|\alpha_{1}|+|\beta_{1}| \leq 4$, with the choice of $(a_{1}, a_{2}, b_{1}, b_{2})$ in Table \ref{parameter-2}, we have
$|\alpha_{1}|+a_{1}+|\beta_{1}|+b_{1} \leq 4, |\alpha_{2}|+a_{2}+|\beta_{2}|+b_{2}\leq |\alpha|+|\beta|, |\beta_{2}| \leq |\beta|$.
\smallskip

\underline{\it Case 1: $N=4, l=14$.} We recall that $q=l+j\gamma \leq 14$ which implies $C_{q} \lesssim 1$. By the lower bound of $\mathcal{D}_{\epsilon}^{N,l}(f)$ in \eqref{lower-bound-of-D} and the estimate \eqref{HNGamma1-continue}, we have
\beno
|( W_{q}\Gamma^\epsilon(\pa_{\beta_{1}}^{\alpha_{1}} g, \pa_{\beta_{2}}^{\alpha_{2}} h), W_{q}\pa_{\beta}^{\alpha}f)|
\lesssim \|g\|_{H^{4}_{x,v}}\left(\mathcal{D}_{\epsilon}^{4,14}(h)\right)^{1/2}\left(\mathcal{D}_{\epsilon}^{4,14}(f)\right)^{1/2}.
\eeno
Since $N=4$, the constants $C(\alpha_{1},\alpha_{2},\beta_{0},\beta_{1},\beta_{2})$  in \eqref{derivative-commutator} are universally bounded.
Taking sum over $\alpha_{1}+\alpha_{2}=\alpha, \beta_{0}+\beta_{1}+\beta_{2}=\beta, |\alpha_{2}+\beta_{2}|\leq |\alpha|+|\beta|-1$, we get \eqref{derivative-commutator-N-4-l-14}.

\underline{\it Case 2: $N=4, l\ge14$.}
 If $|\beta_{2}| = |\beta|$, we use $\|h\|_{H^{|\alpha_{2}|+a_{2}}_x\dot{H}^{|\beta_{2}|}_{\epsilon,q+\gamma/2}} \leq \|h\|_{H^{N-j}_x \dot{H}^{j}_{\epsilon,q+\gamma/2}}$.
If $|\beta_{2}| < |\beta|$ which only happens when $j \geq 1$, by $|f|_{\epsilon, l} \lesssim |f|_{H^{1}_{l+1}}$, we have
\beno \|h\|_{H^{|\alpha_{2}|+a_{2}}_xH^{|\beta_{2}|}_{\epsilon,q+\gamma/2}} \leq \|h\|_{H^{N-j}_x H^{j-1}_{\epsilon,q+\gamma/2}} \lesssim  \|h\|_{H^{N-j}_x H^{j}_{q+\gamma/2+1}} \leq \|h\|_{H^{N-j}_x H^{j}_{q}}. \eeno
 Plugging these facts into \eqref{HNGamma1-continue}, we get
\beno
|( W_{q}\Gamma^\epsilon(\pa_{\beta_{1}}^{\alpha_{1}} g, \pa_{\beta_{2}}^{\alpha_{2}} h), W_{q}\pa_{\beta}^{\alpha}f)|&\lesssim&
 \mathrm{1}_{|\beta_{2}| = |\beta|} \|g\|_{H^{4}_{x,v}}\|h\|_{H^{N-j}_x \dot{H}^{j}_{\epsilon,q+\gamma/2}} \|\pa_{\beta}^{\alpha}f\|_{H^{0}_xH^{0}_{\epsilon,q+\gamma/2}} \\&&+
(\mathrm{1}_{j \geq 1} \mathrm{1}_{|\beta_{2}| < |\beta|} + C_{q})\|g\|_{H^{4}_{x,v}} \|h\|_{H^{N-j}_x H^{j}_{q}}\|\pa_{\beta}^{\alpha}f\|_{H^{0}_xH^{0}_{\epsilon,q+\gamma/2}}
\\&& +
C_{q}\|g\|_{H^{4}_{x,v}} \|h\|_{H^{4}_{x,v}} \|\pa_{\beta}^{\alpha}f\|_{H^{0}_xH^{0}_{\epsilon,q+\gamma/2}}
\\&\lesssim& \|g\|_{H^{4}_{x,v}} \|h\|_{H^{N-j}_x \dot{H}^{j}_{\epsilon,q+\gamma/2}} \|\pa_{\beta}^{\alpha}f\|_{H^{0}_xH^{0}_{\epsilon,q+\gamma/2}} +
\\&&+ \delta \|\pa_{\beta}^{\alpha}f\|^{2}_{H^{0}_xH^{0}_{\epsilon,q+\gamma/2}} + \delta^{-1}C_{l} \|g\|^{2}_{H^{4}_{x,v}} \mathcal{E}^{4,l}(h),
\eeno
where the facts $\|h\|_{H^{N-j}_x H^{j}_{q}}^{2} \leq \mathcal{E}^{4,l}(h)$ and  $\|h\|^{2}_{H^{4}_{x,v}} \leq \mathcal{E}^{4,l}(h)$ are used.
Since $N=4$, the constants $C(\alpha_{1},\alpha_{2},\beta_{0},\beta_{1},\beta_{2})$  in \eqref{derivative-commutator} are universally bounded.
Taking sum over according to \eqref{derivative-commutator}, we get \eqref{derivative-commutator-N-4-l-large}.

\underline{\it Case 3: $N\ge 5, l\ge 3N+2$.} When $N \geq 5$, since $|\alpha_{2}|+a_{2}+|\beta_{2}| \leq |\alpha|+|\beta| - b_{2} \leq N-1$,
we have
\ben \label{N-geq-5-g-leq-4}
|( W_{q}\Gamma^\epsilon(\pa_{\beta_{1}}^{\alpha_{1}} g, \pa_{\beta_{2}}^{\alpha_{2}} h), W_{q}\pa_{\beta}^{\alpha}f)| &\lesssim&
C_{q}\|g\|_{H^{4}_{x,v}} \left(\mathcal{D}_{\epsilon}^{N-1,l}(h)\right)^{1/2}\|\pa_{\beta}^{\alpha}f\|_{H^{0}_xH^{0}_{\epsilon,q+\gamma/2}}
\\&&+
C_{q}\|g\|_{H^{4}_{x,v}} \|h\|_{H^{N}_{x,v}} \|\pa_{\beta}^{\alpha}f\|_{H^{0}_xH^{0}_{\epsilon,q+\gamma/2}}
\nonumber \\&\lesssim& \delta \|\pa_{\beta}^{\alpha}f\|^{2}_{H^{0}_xH^{0}_{\epsilon,q+\gamma/2}}
+ \delta^{-1}C_{l}\|g\|^{2}_{H^{4}_{x,v}}\mathcal{D}_{\epsilon}^{N-1,l}(h) + \delta^{-1}C_{l}\|g\|^{2}_{H^{4}_{x,v}}\|h\|^{2}_{H^{N}_{x,v}}.
\nonumber
\een

If $|\alpha_{1}|+|\beta_{1}| \geq 5$, which occurs only when $N\geq 5$, with the choice of $(a_{1}, a_{2}, b_{1}, b_{2})$ in the last line of Table \ref{parameter-2}, we have
$|\alpha_{1}|+a_{1}+|\beta_{1}|+b_{1} \leq N, |\alpha_{2}|+a_{2}+|\beta_{2}|+b_{2}\leq |\alpha|+|\beta|-1 \leq N-1, |\beta_{2}| \leq |\beta|$
and thus
\ben \label{N-geq-5-g-geq-5}
|( W_{q}\Gamma^\epsilon(\pa_{\beta_{1}}^{\alpha_{1}} g, \pa_{\beta_{2}}^{\alpha_{2}} h), W_{q}\pa_{\beta}^{\alpha}f)|&\lesssim&
C_{q}\|g\|_{H^{N}_{x,v}} \left(\mathcal{D}_{\epsilon}^{N-1,l}(h)\right)^{1/2}\|\pa_{\beta}^{\alpha}f\|_{H^{0}_xH^{0}_{\epsilon,q+\gamma/2}}
\\&\lesssim& \delta \|\pa_{\beta}^{\alpha}f\|^{2}_{H^{0}_xH^{0}_{\epsilon,q+\gamma/2}}
+ \delta^{-1}C_{l}\|g\|^{2}_{H^{N}_{x,v}}\mathcal{D}_{\epsilon}^{N-1,l}(h) . \nonumber\een

Taking sum over according to \eqref{derivative-commutator}, by \eqref{N-geq-5-g-leq-4} for the case of $|\alpha_{1}|+|\beta_{1}| \leq 4$ and \eqref{N-geq-5-g-geq-5} for the case of $|\alpha_{1}|+|\beta_{1}| \geq 5$,
 we get \eqref{derivative-commutator-N-geq-5}. We remark that the sum will bring a constant depending on $N$ due to the constants $C(\alpha_{1},\alpha_{2},\beta_{0},\beta_{1},\beta_{2})$. However, thanks to the arbitrariness of $\delta$, only the latter two terms in
 \eqref{derivative-commutator-N-geq-5} depend on $N$.

We turn to the case when $q=0$. By Theorem \ref{Gamma-full-up-bound}, a counterpart to \eqref{HNGamma1-continue} is
\ben\label{HNGamma1-continue-no-weight-version}
|( \Gamma^\epsilon(\pa^{\alpha_{1}} g, \pa^{\alpha_{2}} h), \pa^{\alpha}f)|&\lesssim&
\|g\|_{H^{|\alpha_{1}|+a_{1}}_xH^{0}}\|h\|_{H^{|\alpha_{2}|+a_{2}}_xH^{0}_{\epsilon,\gamma/2}}
\|\pa^{\alpha}f\|_{H^{0}_xH^{0}_{\epsilon,\gamma/2}}
\nonumber
\\&&+\|\mu^{1/16}g\|_{H^{|\alpha_{1}|+a_{1}}_xH^{b_{1}}}\|\mu^{1/32}h\|_{H^{|\alpha_{2}|+a_{2}}_xH^{b_{2}}}
\|\pa^{\alpha}f\|_{H^{0}_xH^{0}_{\epsilon,\gamma/2}},
\een
 which gives
\eqref{derivative-commutator-N-4-no-weight} and \eqref{derivative-commutator-N-geq-5-no-weight} in a similar way.
\end{proof}

\begin{lem}\label{non-linear-term} Let $0<\eta, \epsilon<1$ and $g, f$ be  suitable  functions with $\mu^{1/2}+g \geq 0$. Recall the definition of $\mathcal{A}^{N,j,l}_{\epsilon}$ in \eqref{A-N-j-l}.
The following three statements hold true.
\begin{enumerate}
\item  If  $N=4, l=14$, then
\ben  \label{N-equals-4-x-v-weight-15}
\mathcal{A}^{N,j,l}_{\epsilon}(g,f) \leq C(\eta^{1/2}+\epsilon^{1/2} +\eta^{-6}\|g\|_{H^{4}_{x,v}})
\mathcal{D}_{\epsilon}^{4,14}(f).
\een
 \item  If $N=4, l>14$,  then for any $\delta>0$,
\ben \label{N-equals-4-x-v-weight-larger}
\mathcal{A}^{N,j,l}_{\epsilon}(g,f) &\leq& C(\delta + \eta^{1/2}+\epsilon^{1/2} +\eta^{-6}\|g\|_{H^{4}_{x,v}}) \|f\|^{2}_{H^{N-j}_{x}\dot{H}^{j}_{\epsilon,l+j\gamma+\gamma/2}}
\\&&+ \delta^{-1} C_{l}\|g\|^{2}_{H^{4}_{x,v}} \|f\|^{2}_{H^{N-j}_{x}\dot{H}^{j}_{l+j\gamma+\gamma/2}} +
\delta^{-1}C_{l} \|g\|^{2}_{H^{4}_{x,v}} \mathcal{E}^{4,l}(f). \nonumber \een
\item   If $N\geq 5, l\geq 3N+2$,  then for any $\delta>0$,
\ben \label{N-geq-5-x-v-weight-larger}
\mathcal{A}^{N,j,l}_{\epsilon}(g,f) &\leq& C(\delta+\eta^{1/2}+\epsilon^{1/2} +\eta^{-6}\|g\|_{H^{4}_{x,v}}) \|f\|^{2}_{H^{N-j}_{x}\dot{H}^{j}_{\epsilon,l+j\gamma+\gamma/2}}+\delta^{-1} C_{N,l}\|g\|^{2}_{H^{4}_{x,v}}
\\&& \times\|f\|^{2}_{H^{N-j}_{x}\dot{H}^{j}_{l+j\gamma+\gamma/2}}  + \delta^{-1}C_{N,l}\|g\|^{2}_{H^{N}_{x,v}}\mathcal{D}_{\epsilon}^{N-1,l}(f)
+ \delta^{-1}C_{N,l}\|g\|^{2}_{H^{4}_{x,v}}\|f\|^{2}_{H^{N}_{x,v}}.
\nonumber
\een
\end{enumerate}

Recall the definition of $\mathcal{B}^{N}_{\epsilon}$ in \eqref{B-N-j-0}.
The following two statements hold true.
\begin{enumerate}
\item If $N=4$, then
\ben \label{N-equals-4-x-v-no-weight}
\mathcal{B}^{N}_{\epsilon}(g,f) \leq C(\eta^{1/2}+\epsilon^{1/2} +\eta^{-6}\|g\|_{H^{4}_{x,v}})
\mathcal{D}_{\epsilon}^{4}(f).
\een
\item If $N\geq 5$, then for any $\delta>0$,
\ben \label{N-geq-5-x-v-no-weight}
\mathcal{B}^{N}_{\epsilon}(g,f) &\leq& C(\delta+\eta^{1/2}+\epsilon^{1/2} +\eta^{-6}\|g\|_{H^{4}_{x,v}}) \|f\|^{2}_{H^{N}_{x} H^{0}_{\epsilon,\gamma/2}}
\\&&+ C_{N}\delta^{-1}\|g\|^{2}_{H^{N}_{x,v}}\mathcal{D}_{\epsilon}^{N-1}(f)
+ C_{N}\delta^{-1}\|g\|^{2}_{H^{4}_{x,v}}\|f\|^{2}_{H^{N}_{x,v}} \nonumber
.\een
\end{enumerate}

We emphasize that $C$ is a universal constant independent of $N,l$.
\end{lem}
\begin{proof}   Let $q=l+j\gamma$.
Note that
a typical term in $\mathcal{A}^{N,j,l}_{\epsilon}(g,f)$ is $(W_{q}\pa^{\alpha}_{\beta}\Gamma^{\epsilon}(g,f) - \mathcal{L}^{\epsilon,\gamma}_{\eta}W_{q}\pa^{\alpha}_{\beta}f, W_{q}\pa^{\alpha}_{\beta}f)$ for some fixed $\alpha,\beta$ such that $|\alpha| \leq N-j, |\beta|=j$. We make the following decomposition,
\ben \label{decomposition-Gamma} W_{q}\pa^{\alpha}_{\beta}\Gamma^{\epsilon}(g,f) &=& W_{q}\Gamma^{\epsilon}(g,\pa^{\alpha}_{\beta}f)
+W_{q}[\pa^{\alpha}_{\beta},\Gamma^{\epsilon}(g,\cdot)]f
\nonumber \\&=& \Gamma^{\epsilon}(g,W_{q}\pa^{\alpha}_{\beta}f) + [W_{q},\Gamma^{\epsilon}(g,\cdot)]\pa^{\alpha}_{\beta}f
+W_{q}[\pa^{\alpha}_{\beta},\Gamma^{\epsilon}(g,\cdot)]f
\nonumber \\&=& \Gamma^{\epsilon,\gamma}_{\eta}(g,W_{q}\pa^{\alpha}_{\beta}f) + \Gamma^{\epsilon,\gamma,\eta}(g,W_{q}\pa^{\alpha}_{\beta}f) +
[W_{q},\Gamma^{\epsilon}(g,\cdot)]\pa^{\alpha}_{\beta}f
+W_{q}[\pa^{\alpha}_{\beta},\Gamma^{\epsilon}(g,\cdot)]f.
\een

\underline{\it Estimate of $\Gamma^{\epsilon,\gamma}_{\eta}(g,W_{q}\pa^{\alpha}_{\beta}f) - \mathcal{L}^{\epsilon,\gamma}_{\eta}W_{q}\pa^{\alpha}_{\beta}f$.}
Since $\mu^{1/2}+g \geq 0$, taking $\delta=1/2$ in Theorem \ref{small-part-L+gamma}, we have
\ben \label{key-term-novel}
(\Gamma^{\epsilon,\gamma}_{\eta}(g,W_{q}\pa^{\alpha}_{\beta}f) - \mathcal{L}^{\epsilon,\gamma}_{\eta}W_{q}\pa^{\alpha}_{\beta}f, W_{q}\pa^{\alpha}_{\beta}f) &\leq& C(\eta^{1/2}+\epsilon^{1/2})\int_{\mathbb{T}^{3}} (1+|\mu^{1/16}g|_{H^{2}})|W_{q}\pa^{\alpha}_{\beta}f|^{2}_{\epsilon,\gamma/2} dx
\nonumber \\&\leq&C(\eta^{1/2}+\epsilon^{1/2})\|\pa^{\alpha}_{\beta}f\|^{2}_{L^{2}_{x}L^{2}_{\epsilon,q+\gamma/2}}
\nonumber\\&&+ C(\eta^{1/2}+\epsilon^{1/2}) \|\mu^{1/16}g\|_{H^{2}_{x}H^{2}} \|\pa^{\alpha}_{\beta}f\|^{2}_{L^{2}_{x}L^{2}_{\epsilon,q+\gamma/2}}
\nonumber\\&\leq& C(\eta^{1/2}+\epsilon^{1/2} +\|g\|_{H^{4}_{x,v}})\|\pa^{\alpha}_{\beta}f\|^{2}_{L^{2}_{x}L^{2}_{\epsilon,q+\gamma/2}}.
\een

\underline{\it{Estimate of $\Gamma^{\epsilon,\gamma,\eta}(g,W_{q}\pa^{\alpha}_{\beta}f)$.}}
Taking $\delta=1/2, s_{1}=2, s_{2}=0$ in Theorem \ref{upGammagh-geq-eta}, by the embedding $H^{2}_{x} \rightarrow L^{\infty}_{x}$,
we have
\ben \label{regular-term}
|(\Gamma^{\epsilon,\gamma,\eta}(g,W_{q}\pa^{\alpha}_{\beta}f), W_{q}\pa^{\alpha}_{\beta}f)| &\leq& C\eta^{-6}\int_{\mathbb{T}^{3}}  |f|_{H^{2}}|W_{q}\pa^{\alpha}_{\beta}f|^{2}_{\epsilon,\gamma/2} dx
\nonumber \\&\leq& C \eta^{-6} \|g\|_{H^{4}_{x,v}} \|\pa^{\alpha}_{\beta}f\|^{2}_{L^{2}_{x}L^{2}_{\epsilon,q+\gamma/2}}.
\een

\underline{\it{Estimate of  $[W_{q},\Gamma^{\epsilon}(g,\cdot)]\pa^{\alpha}_{\beta}f$.}}
Taking $\delta=1/2, s_{3}=1, s_{4}=0$ in Theorem \ref{commutator-Gamma-geq-eta}, by the embedding $H^{2}_{x} \rightarrow L^{\infty}_{x}$, we have
\beno |([W_{q},\Gamma^{\epsilon}(g,\cdot)]\pa^{\alpha}_{\beta}f, W_{q}\pa^{\alpha}_{\beta}f)| &\leq& \eta^{-6}C_{l}\int
|g|_{L^{2}}|W_{q+\gamma/2}\pa^{\alpha}_{\beta}f|_{L^{2}}|\pa^{\alpha}_{\beta}f|_{\epsilon,q+\gamma/2} dx
\\&&+\eta^{1/2}C_{l}\int|g|_{H^{1}}|\pa^{\alpha}_{\beta}f|^{2}_{\epsilon,\gamma/2} dx
\\&\leq& C_{l}\|g\|_{H^{4}_{x,v}}  \|\pa^{\alpha}_{\beta}f\|_{L^{2}_{x}L^{2}_{q+\gamma/2}} \|\pa^{\alpha}_{\beta}f\|_{L^{2}_{x}L^{2}_{\epsilon,q+\gamma/2}} +
\|g\|_{H^{4}_{x,v}} \|\pa^{\alpha}_{\beta}f\|^{2}_{L^{2}_{x}L^{2}_{\epsilon,\gamma/2}},
\eeno
where we choose $\eta$ such that $\eta^{1/2}C_{l}=1$ and then $\eta^{-6}C_{l}$ is a constant depending only on $l$.
When $N=4, l=14$, the constant $C_{l}$ is a universal constant, which gives
\ben \label{weight-commutator-N-4-l-15}
|([W_{q},\Gamma^{\epsilon}(g,\cdot)]\pa^{\alpha}_{\beta}f, W_{q}\pa^{\alpha}_{\beta}f)| \lesssim
\|g\|_{H^{4}_{x,v}} \|\pa^{\alpha}_{\beta}f\|_{L^{2}_{x}L^{2}_{\epsilon,q+\gamma/2}}^{2}.
\een
When $N=4, l>14$ or $N \geq 5, l \geq 3N+2$,  we get
\ben \label{weight-commutator-N-4-l-large-or-N-geq-5}
|([W_{q},\Gamma^{\epsilon}(g,\cdot)]\pa^{\alpha}_{\beta}f, W_{q}\pa^{\alpha}_{\beta}f)|
&\lesssim& (\delta + \|g\|_{H^{4}_{x,v}})\|\pa^{\alpha}_{\beta}f\|^{2}_{L^{2}_{x}L^{2}_{\epsilon,q+\gamma/2}} + \delta^{-1} C_{l}\|g\|^{2}_{H^{4}_{x,v}} \|\pa^{\alpha}_{\beta}f\|^{2}_{L^{2}_{x}L^{2}_{q+\gamma/2}}.
\een

The last term $W_{q}[\pa^{\alpha}_{\beta},\Gamma^{\epsilon}(g,\cdot)]f$ in \eqref{decomposition-Gamma} is handled in Lemma \ref{derivative-commutator-lemma} by \eqref{derivative-commutator-N-4-l-14}, \eqref{derivative-commutator-N-4-l-large} and \eqref{derivative-commutator-N-geq-5}. From which together with
\eqref{key-term-novel}, \eqref{regular-term}, \eqref{weight-commutator-N-4-l-15} and \eqref{weight-commutator-N-4-l-large-or-N-geq-5}
, we get the desired results  \eqref{N-equals-4-x-v-weight-15}, \eqref{N-equals-4-x-v-weight-larger} and \eqref{N-geq-5-x-v-weight-larger}. The estimates  \eqref{N-equals-4-x-v-no-weight} and \eqref{N-geq-5-x-v-no-weight} of
 $\mathcal{B}^{N}_{\epsilon}(g,f)$ can be derived similarly, so we omit the details and end the proof of the lemma.
\end{proof}

\begin{lem}\label{pure-x-term-from-I} Let $ \epsilon\ge0 $ be small enough. Recall the definition of $\mathcal{C}^{N}_{\epsilon}$ in \eqref{C-N-epsilon-0}. The following two statements hold true.
\begin{enumerate}
\item If $N=4$, then
\ben \label{only-x-with-mu-type-N-equals-4}
\mathcal{C}^{N}_{\epsilon}(g,f) \leq C\|g\|^{2}_{H^{4}_{x,v}} \mathcal{D}_{\epsilon}^{4}(f).
\een
\item If $N\geq 5$, then
\ben \label{only-x-with-mu-type-N-bigger-5}
\sum_{|\alpha| \leq N-1}\sum_{j=1}^{13}
\int_{\mathbb{T}^{3}}|\langle  \pa^{\alpha}\Gamma^{\epsilon}(g,f), e_j\rangle|^{2} dx \leq C_{N}(\|g\|^{2}_{H^{N}_{x,v}}\mathcal{D}_{\epsilon}^{N-1}(f)
+ \mathcal{D}_{\epsilon}^{N-1}(g)\|f\|^{2}_{H^{N}_{x,v}}).
\een
\end{enumerate}
\end{lem}
\begin{proof}
Thanks to Theorem \ref{Gamma-full-up-bound}, for $a_{1}, a_{2}, b_{1}, b_{2} \in \{0,1,2\}$ with $a_{1}+a_{2}=b_{1}+b_{2}=2$ and $b_2\ge1$,
when $N=4$, we have
\beno
\int_{\mathbb{T}^{3}}|\langle  \Gamma^{\epsilon}(\pa^{\alpha_{1}} g, \pa^{\alpha_{2}}f), e_j\rangle|^{2} dx &\lesssim&
\|g\|^{2}_{H^{|\alpha_{1}|+a_{1}}_xH^{0}}\|f\|^{2}_{H^{|\alpha_{2}|+a_{2}}_xH^{0}_{\epsilon,\gamma/2}}
\\&&+\|\mu^{1/16}g\|^{2}_{H^{|\alpha_{1}|+a_{1}}_xH^{b_{1}}}\|\mu^{1/32}f\|^{2}_{H^{|\alpha_{2}|+a_{2}}_xH^{b_{2}}}
\lesssim
 \|g\|^{2}_{H^{4}_{x,v}} \mathcal{D}_{\epsilon}^{4}(f).
\eeno
Similarly, when $N\geq 5$, we have
\beno
\int_{\mathbb{T}^{3}}|\langle  \Gamma^{\epsilon}(\pa^{\alpha_{1}} g, \pa^{\alpha_{2}}f), e_j\rangle|^{2} dx &\lesssim&
\|g\|^{2}_{H^{|\alpha_{1}|+a_{1}}_xH^{0}}\|f\|^{2}_{H^{|\alpha_{2}|+a_{2}}_xH^{0}_{\epsilon,\gamma/2}}
 \\&&+\|\mu^{1/16}g\|^{2}_{H^{|\alpha_{1}|+a_{1}}_xH^{b_{1}}}\|\mu^{1/32}f\|^{2}_{H^{|\alpha_{2}|+a_{2}}_xH^{b_{2}}}
\\&\lesssim&  \|g\|^{2}_{H^{N}_{x,v}}\mathcal{D}_{\epsilon}^{N-1}(f)
+ \mathcal{D}_{\epsilon}^{N-1}(g)\|f\|^{2}_{H^{N}_{x,v}}.
\eeno
Here in both cases, we use Table \ref{parameter-2} for the choice of $a_{1}, a_{2}, b_{1}, b_{2}$.
\end{proof}

Now we are ready to prove Theorem \ref{a-priori-estimate-LBE}.
\begin{proof}[Proof of Theorem \ref{a-priori-estimate-LBE}]
Taking $g = \Gamma^{\epsilon}(f^{\epsilon},f^{\epsilon})$ in Proposition \ref{essential-estimate-of-micro-macro}, for $0< \eta \leq \eta_{0}, 0 \leq  \epsilon \leq \epsilon_{1}$, we have
\beno
\frac{d}{dt}\Xi^{N,l}(f^{\epsilon}) +  \frac{1}{4} \mathcal{D}_{\epsilon}^{N,l}(f^{\epsilon}) \leq \sum_{j=0}^{N}2K_{j} \mathcal{A}^{N,j,l}_{\epsilon}(f^{\epsilon},f^{\epsilon})
+2 L \mathcal{B}^{N}_{\epsilon}(f^{\epsilon},f^{\epsilon})+M C \mathcal{C}^{N}_{\epsilon}(f^{\epsilon},f^{\epsilon}), \eeno

 \underline{\it Case 1: $N=4, l=14$.}  In this case, the constants $M, L, K_{j}$ are universal.
Then by \eqref{N-equals-4-x-v-weight-15} and \eqref{N-equals-4-x-v-no-weight} in Lemma \ref{non-linear-term},
 and \eqref{only-x-with-mu-type-N-equals-4} in Lemma \ref{pure-x-term-from-I}, and the natural inequality
 $\mathcal{D}_{\epsilon}^{4}(f^{\epsilon}) \leq \mathcal{D}_{\epsilon}^{4,14}(f^{\epsilon})$,
 we have
\beno
\frac{d}{dt}\Xi^{4,14}(f^{\epsilon}) +  \frac{1}{4} \mathcal{D}_{\epsilon}^{4,14}(f^{\epsilon}) \leq
 C(\eta^{1/2}+\epsilon^{1/2} +\eta^{-6}\|f^{\epsilon}\|_{H^{4}_{x,v}} + \|f^{\epsilon}\|^{2}_{H^{4}_{x,v}})
\mathcal{D}_{\epsilon}^{4,14}(f^{\epsilon}).
\eeno
Let $\eta_{1}$ verify $C \eta_{1}^{1/2} = 1/32$. Let $\epsilon_{2}$ verify $C\epsilon_{2}^{1/2} = 1/32$. Let $\delta_{1}$ to be the largest number satisfying $C \eta_{2}^{-6} \delta_{1}^{1/2} \leq 1/32$ and $C\delta_{1} \leq 1/32$.
We choose $\eta= \min\{\eta_{0},\eta_{1}\}$. When $\epsilon \leq \min\{\epsilon_{1},\epsilon_{2}\}$, under the assumption $\sup_{0 \leq t \leq T}\mathcal{E}^{4,14}(f^{\epsilon}(t)) \leq \delta_{1}$, since $\|f^{\epsilon}\|^{2}_{H^{4}_{x,v}} \leq \mathcal{E}^{4,14}(f^{\epsilon})$, we have
\ben \label{N=4-l=14-diff}
\frac{d}{dt}\Xi^{4,14}(f^{\epsilon}) +  \frac{1}{8} \mathcal{D}_{\epsilon}^{4,14}(f^{\epsilon}) \leq 0.
\een
We emphasize that when $N=4, l=14$, the constants $C_{0}$ in \eqref{energy-equivalence} is universal. Therefore
we get \eqref{uniform-estimate-propagation-N-4-l-15} from \eqref{N=4-l=14-diff}.

 \underline{\it Case 2: $N=4, l>14$.}  In this case, the constants $M, L, K_{j}$ could depend on $l$.
Then by  \eqref{N-equals-4-x-v-weight-larger} and \eqref{N-equals-4-x-v-no-weight} in Lemma \ref{non-linear-term}
 and \eqref{only-x-with-mu-type-N-equals-4} in Lemma \ref{pure-x-term-from-I}, and the natural inequality
 $\|f^{\epsilon}\|^{2}_{H^{4}_{x,v}}  \leq \mathcal{D}_{\epsilon}^{4,14}(f^{\epsilon})$,
 we have
\beno
\frac{d}{dt}\Xi^{4,l}(f^{\epsilon}) +  \frac{1}{4} \mathcal{D}_{\epsilon}^{4,l}(f^{\epsilon}) &\leq&
M C \|f^{\epsilon}\|^{2}_{H^{4}_{x,v}} \mathcal{D}_{\epsilon}^{4}(f^{\epsilon}) + C_{l}(\eta^{1/2}+\epsilon^{1/2} +\eta^{-6}\|f^{\epsilon}\|_{H^{4}_{x,v}})
\mathcal{D}_{\epsilon}^{4}(f^{\epsilon})
\\&&+ \sum_{j=0}^{N}2 K_{j} C(\delta+\eta^{1/2}+\epsilon^{1/2} +\eta^{-6}\|f^{\epsilon}\|_{H^{4}_{x,v}}) \|f^{\epsilon}\|^{2}_{H^{N-j}_{x}\dot{H}^{j}_{\epsilon,l+j\gamma+\gamma/2}}
\\&&+ \delta^{-1}  C_{l}\mathcal{D}_{\epsilon}^{4,14}(f^{\epsilon}) \mathcal{E}^{4,l}(f^{\epsilon}).
\eeno
We take $\delta, \eta_{2}, \epsilon_{3}, \delta_{2}$ such that $2C\delta = \lambda_{0}/64$, $2C \eta_{2}^{1/2}= \lambda_{0}/64$, $2C \epsilon_{3}^{1/2}= \lambda_{0}/64,$ $2C \eta_{2}^{-6} \delta_{2}^{1/2}= \lambda_{0}/64$. We choose $\eta= \min\{\eta_{0},\eta_{1}, \eta_{2}\}$. When $\epsilon \leq \min\{\epsilon_{1},\epsilon_{2},\epsilon_{3}\}$, under the assumption $\sup_{0 \leq t \leq T}\mathcal{E}^{4,14}(f^{\epsilon}(t)) \leq \min\{\delta_{1},\delta_{2}\}$, since $\|f^{\epsilon}\|^{2}_{H^{4}_{x,v}} \leq \mathcal{E}^{4,14}(f^{\epsilon})$, we have
$2 C(\delta+\eta_{3}^{1/2}+\epsilon^{1/2} +\eta_{3}^{-6}\|f^{\epsilon}\|_{H^{4}_{x,v}}) \leq \lambda_{0}/16$. Recalling the definition of $\mathcal{D}_{\epsilon}^{N,l}$
in \eqref{lower-bound-of-D}, we get
\beno
\frac{d}{dt}\Xi^{4,l}(f^{\epsilon}) +  \frac{1}{8} \mathcal{D}_{\epsilon}^{4,l}(f^{\epsilon}) \leq
C_{l}\mathcal{D}_{\epsilon}^{4}(f^{\epsilon}) +  C_{l}\mathcal{D}_{\epsilon}^{4,14}(f^{\epsilon}) \mathcal{E}^{4,l}(f^{\epsilon}).
\eeno
By Grownwall inequality,  we arrive at
\ben \label{a-priori-estimate-n-4-l-large} \Xi^{4,l}(f^{\epsilon}(t))+ \frac{1}{8}  \int_{0}^{t} \mathcal{D}_{\epsilon}^{4,l}(f^{\epsilon}(s)) ds &\leq& (\Xi^{4,l}(f_{0}) + C_{l} \int_{0}^{t} \mathcal{D}_{\epsilon}^{4}(f^{\epsilon}(s)) ds)\exp\big(C_{l}\int_{0}^{t}\mathcal{D}_{\epsilon}^{4,14}(f^{\epsilon}(s))ds\big)
\\&\leq& (\Xi^{4,l}(f_{0}) + C_{l} C \mathcal{E}^{4,14}(f_{0})  )  \exp\left(C_{l} C \mathcal{E}^{4,14}(f_{0})  \right)
\leq C_{l}\Xi^{4,l}(f_{0}). \nonumber \een
where we use $\mathcal{D}_{\epsilon}^{4}(f^{\epsilon}) \leq \mathcal{D}_{\epsilon}^{4,14}(f^{\epsilon})$, and
$\int_{0}^{t}\mathcal{D}_{\epsilon}^{4,14}(f^{\epsilon}(s))ds \leq C \mathcal{E}^{4,14}(f_{0})$ by  the proved result \eqref{uniform-estimate-propagation-N-4-l-15}, and $\mathcal{E}^{4,14}(f_{0}) \leq \Xi^{4,l}(f_{0})$, and the assumption $\mathcal{E}^{4,14}(f_{0}) \leq \delta_{1} \lesssim 1$.
Then by \eqref{energy-equivalence}, we get \eqref{uniform-estimate-propagation-N-4-l-big}.

 \underline{\it Case 3: $N\geq 5, l \geq 3N+2$.}  In this case, the constants $M, L, K_{j}$ could depend on $N, l$.
Then by  \eqref{N-geq-5-x-v-weight-larger} and \eqref{N-geq-5-x-v-no-weight} in Lemma \ref{non-linear-term}
 and \eqref{only-x-with-mu-type-N-bigger-5} in Lemma \ref{pure-x-term-from-I}, and the inequalities $\|f^{\epsilon}\|^{2}_{H^{4}_{x,v}}  \leq \mathcal{D}_{\epsilon}^{N-1,l}(f^{\epsilon}),  \mathcal{D}_{\epsilon}^{N-1}(f^{\epsilon})  \leq \mathcal{D}_{\epsilon}^{N-1,l}(f^{\epsilon}), \|f^{\epsilon}\|^{2}_{H^{N}_{x,v}}  \leq \mathcal{E}_{\epsilon}^{N,l}(f^{\epsilon})$,
we have
\beno
\frac{d}{dt}\Xi^{N,l}(f^{\epsilon}) +  \frac{1}{4} \mathcal{D}_{\epsilon}^{N,l}(f^{\epsilon})
&\leq& 2C(\delta+\eta^{1/2}+\epsilon^{1/2} +\eta^{-6}\|f^{\epsilon}\|_{H^{4}_{x,v}})
 L \|f^{\epsilon}\|^{2}_{H^{N}_{x} H^{0}_{\epsilon,\gamma/2}}
\\&&+ 2 C(\delta+\eta^{1/2}+\epsilon^{1/2} +\eta^{-6}\|f^{\epsilon}\|_{H^{4}_{x,v}}) \sum_{j=0}^{N} K_{j}  \|f^{\epsilon}\|^{2}_{H^{N-j}_{x}\dot{H}^{j}_{\epsilon,q+\gamma/2}}
\\&& + \delta^{-1}C_{N,l} \mathcal{D}_{\epsilon}^{N-1,l}(f^{\epsilon}) \mathcal{E}^{N,l}(f^{\epsilon}).
\eeno
We take $\delta, \eta_{3}, \epsilon_{4}, \delta_{3}$ such that $2C\delta = c_{1}\lambda_{0}/64$, $2C \eta_{3}^{1/2}= c_{1}\lambda_{0}/64$, $2C \epsilon_{4}^{1/2}= c_{1}\lambda_{0}/64,$ $2C \eta_{3}^{-6} \delta_{3}^{1/2}= c_{1}\lambda_{0}/64$. We choose $\eta= \min_{0 \le i\le 3}\eta_i$. Let $\delta_0=\min_{1\le i\le 3}\delta_i$ and $\epsilon_0 =\min_{2\le i\le 4}\epsilon_i$.
When $\epsilon \leq \epsilon_{0}$, under the assumption $\sup_{0 \leq t \leq T}\mathcal{E}^{4,14}(f^{\epsilon}(t)) \leq \delta_0$, since $\|f^{\epsilon}\|^{2}_{H^{4}_{x,v}} \leq \mathcal{E}^{4,14}(f^{\epsilon})$, we have
$2 C(\delta+\eta_{3}^{1/2}+\epsilon^{1/2} +\eta_{3}^{-6}\|f^{\epsilon}\|_{H^{4}_{x,v}}) \leq c_{1}\lambda_{0}/16$. Recalling the definition of $\mathcal{D}_{\epsilon}^{N,l}$
in \eqref{lower-bound-of-D}, we conclude that for any $N\geq 5, l \geq 3N+2$, there holds
\ben \label{n-geq-5-l-big}
\frac{d}{dt}\Xi^{N,l}(f^{\epsilon}) +  \frac{1}{8} \mathcal{D}_{\epsilon}^{N,l}(f^{\epsilon}) \leq
C_{N,l} \mathcal{D}_{\epsilon}^{N-1,l}(f^{\epsilon}) \mathcal{E}^{N,l}(f^{\epsilon}).
\een

In the following we use mathematical induction to finish the proof. Suppose for some $k \geq 4$, \eqref{uniform-estimate-propagation-N-geq-5-l-big} is valid for $N=k$, that is,
\ben \label{a-priori-estimate-n=k} \mathcal{E}^{k,l}(f^{\epsilon}(t)) +  \int_{0}^{t}\mathcal{D}_{\epsilon}^{k,l}(f^{\epsilon}(s))ds \leq P_{k,l}(\mathcal{E}^{k,l}(f_{0})). \een
Then for $N=k+1 \geq 5, l \geq 3N+2$, by \eqref{n-geq-5-l-big}, we get
\ben \label{a-priori-estimate-n=k+1} \frac{d}{dt}\Xi^{k+1,l}(f^{\epsilon})+ \frac{1}{8} \mathcal{D}_{\epsilon}^{k+1,l}(f^{\epsilon}) \leq C_{k+1,l} \mathcal{D}_{\epsilon}^{k,l}(f^{\epsilon}) \mathcal{E}^{k+1,l}(f^{\epsilon}). \een
Now since $\int_{0}^{t}\mathcal{D}_{\epsilon}^{k,l}(f^{\epsilon}(s))ds \leq P_{k,l}(\mathcal{E}^{k,l}(f_{0}))$ by \eqref{a-priori-estimate-n=k} and $\mathcal{E}^{k+1,l}(f^{\epsilon}) \leq \Xi^{k+1,l}(f^{\epsilon})$, by Gronwall's inequality, we arrive at
\ben \label{a-priori-estimate-n=k+12} \Xi^{k+1,l}(f^{\epsilon}(t))+ \frac{1}{8}  \int_{0}^{t} \mathcal{D}_{\epsilon}^{k+1,l}(f^{\epsilon}(t)) dt &\leq& \Xi^{k+1,l}(f_{0})\exp\big(C_{k+1,l}\int_{0}^{t}\mathcal{D}_{\epsilon}^{k,l}(f^{\epsilon}(s))ds\big)
\\&\leq& \Xi^{k+1,l}(f_{0})\exp\left(C_{k+1,l}P_{k,l}(\mathcal{E}^{k,l}(f_{0}))\right). \nonumber \een
Then by the equivalence relation \eqref{energy-equivalence}, we have
\ben \label{a-priori-estimate-n=k+13} \mathcal{E}^{k+1,l}(f^{\epsilon}(t)) +  \int_{0}^{t}\mathcal{D}_{\epsilon}^{k+1,l}(f^{\epsilon}(s))ds
&\leq& C_{k+1,l}\mathcal{E}^{k+1,l}(f_{0})\exp\left(C_{k+1,l}P_{k,l}(\mathcal{E}^{k+1,l}(f_{0}))\right)
\\ &:=& P_{k+1,l}(\mathcal{E}^{k+1,l}(f_{0})). \nonumber
\een
That is, we get \eqref{uniform-estimate-propagation-N-geq-5-l-big} for the case $N =  k+1, l \geq 3N+2$. Starting from $P_{4,l}(x) = C_{l}x$, we can define $P_{N,l}(x) := C_{N,l}x \exp\left(C_{N,l}P_{N-1,l}(x)\right)$ in a iterating manner for $N \geq 5$.
\end{proof}

\begin{proof}[Proof of Theorem \ref{asymptotic-result}(global well-posedness and regularity propagation)] We remind readers that local well-posedness of the equation and the non-negativity of $\mu+\mu^{\f12}f$ were proved in \cite{he2014well}. Thanks to Theorem \ref{a-priori-estimate-LBE}, the standard continuity argument yields the global well-posedness result \eqref{uniform-controlled-by-initial}. The propagation of regularity result \eqref{propagation} is a direct consequence of Theorem \ref{a-priori-estimate-LBE}.
\end{proof}

\subsection{Asymptotic formula for the limit} We want to prove \eqref{error-function-uniform-estimate} in this subsection.    Let $f^{\epsilon}$ and $f$ be the solutions to  \eqref{linearizedBE} and \eqref{linearizedLE} respectively with the initial data $f_0$. Set $F^{\epsilon}_{R} :=  |\ln \epsilon| (f^{\epsilon}-f)$, then it solves
\ben \label{error-equation} \partial_{t}F^{\epsilon}_{R} + v \cdot \nabla_{x} F^{\epsilon}_{R} + \mathcal{L}F^{\epsilon}_{R} = |\ln \epsilon|[(\mathcal{L}-\mathcal{L}^{\epsilon})f^{\epsilon}+(\Gamma^{\epsilon}-\Gamma)(f^{\epsilon},f)]
+\Gamma^{\epsilon}(f^{\epsilon},F^{\epsilon}_{R})+\Gamma(F^{\epsilon}_{R},f). \een
We will apply Proposition  \ref{essential-estimate-of-micro-macro} to the above equation for $F^{\epsilon}_{R}$.
For notational brevity, we set
\ben \label{three-terms}
G_{1} = |\ln \epsilon|[(\mathcal{L}-\mathcal{L}^{\epsilon})f^{\epsilon}+(\Gamma^{\epsilon}-\Gamma)(f^{\epsilon},f)],\quad G_{2}=\Gamma^{\epsilon}(f^{\epsilon},F^{\epsilon}_{R}), \quad G_{3}= \Gamma(F^{\epsilon}_{R},f).
\een

When $N \geq 4, \eta=\epsilon=0$, by applying Proposition  \ref{essential-estimate-of-micro-macro} with $g = G_{1}+G_{2}+G_{3}$, since
$|\langle \pa^{\alpha}g, e_j\rangle|^{2} \leq 3(|\langle \pa^{\alpha}G_{1}, e_j\rangle|^{2} + |\langle \pa^{\alpha}G_{2}, e_j\rangle|^{2} +|\langle \pa^{\alpha}G_{3}, e_j\rangle|^{2})$,
we have
\ben \label{essential-micro-macro-error-function-2} \frac{d}{dt}\Xi^{N,l}(F^{\epsilon}_{R}) +  \frac{1}{4} \mathcal{D}_{0}^{N,l}(F^{\epsilon}_{R}) &\leq& 3 M C \sum_{i=1}^{3}\sum_{|\alpha| \leq N-1}\sum_{j=1}^{13}
\int_{\mathbb{T}^{3}}|\langle \pa^{\alpha}G_{i}, e_j\rangle|^{2} dx
+ \sum_{i=1}^{3} 2 L\sum_{|\alpha|\leq N}(\pa^{\alpha}G_{i}, \pa^{\alpha}F^{\epsilon}_{R}) \nonumber \\&&+ \sum_{i=1}^{3}\sum_{j=0}^{N}2K_{j}\sum_{|\alpha|\leq N-j,|\beta|=j}(W_{l+j\gamma}\pa^{\alpha}_{\beta}G_{i}, W_{l+j\gamma}\pa^{\alpha}_{\beta}F^{\epsilon}_{R})\nonumber
\\&=& 3 MC \left(\mathcal{X}^{N}(G_{1}) + \mathcal{C}^{N}_{\epsilon}(f^{\epsilon},F^{\epsilon}_{R}) + \mathcal{C}^{N}_{0}(F^{\epsilon}_{R},f) \right)
\nonumber\\&&+ 2 L \left(\mathcal{V}^{N}(G_{1}) + \mathcal{Z}^{N}_{\epsilon}(f^{\epsilon},F^{\epsilon}_{R},F^{\epsilon}_{R}) + \mathcal{Z}^{N}_{0}(F^{\epsilon}_{R},f,F^{\epsilon}_{R}) \right)
\nonumber \\&&+
 \sum_{j=0}^{N} 2K_{j} \left(\mathcal{W}^{N,j,l}(G_{1}) + \mathcal{Y}^{N,j,l}_{\epsilon}(f^{\epsilon},F^{\epsilon}_{R},F^{\epsilon}_{R}) + \mathcal{Y}^{N,j,l}_{0}(F^{\epsilon}_{R},f,F^{\epsilon}_{R}) \right)
 \een
where for $\epsilon \geq 0$ and general functions $g,h,f$, we define
\ben \label{mathcal-X-pure-x}
\mathcal{X}^{N}(h) := \sum_{|\alpha| \leq N-1}\sum_{j=1}^{13} \int_{\mathbb{T}^{3}}|\langle \pa^{\alpha}h, e_j\rangle|^{2} dx
\\ \label{mathcal-V-no-weight-up}
\mathcal{V}^{N}(h) := \sum_{|\alpha|\leq N}(\pa^{\alpha}h, \pa^{\alpha}F^{\epsilon}_{R})
\\\label{mathcal-Z-no-weight-up} \mathcal{Z}^{N}_{\epsilon}(g,h,f):= \sum_{|\alpha|\leq N}
(\pa^{\alpha}\Gamma^{\epsilon}(g,h), \pa^{\alpha}f).
\\\label{mathcal-W-no-weight-up}
\mathcal{W}^{N,j,l}(h) := \sum_{|\alpha|\leq N-j,|\beta|=j}(W_{l+j\gamma}\pa^{\alpha}_{\beta}h, W_{l+j\gamma}\pa^{\alpha}_{\beta}F^{\epsilon}_{R})
\\\label{mathcal-Y-no-weight-up}
\mathcal{Y}^{N,j,l}_{\epsilon}(g,h,f):= \sum_{|\alpha|\leq N-j,|\beta|=j}
(W_{l+j\gamma}\pa^{\alpha}_{\beta}\Gamma^{\epsilon}(g,h), W_{l+j\gamma}\pa^{\alpha}_{\beta}f).
\een

In order to further analyze \eqref{essential-micro-macro-error-function-2}, we need to estimate the nine terms on the right hand side.
Note that the functional $ \mathcal{C}^{N}_{\epsilon}$ is already handled in Lemma \ref{pure-x-term-from-I}. We will deal with the functionals
$\mathcal{Z}^{N}_{\epsilon}$ and $\mathcal{Y}^{N,j,l}_{\epsilon}$ in Lemma \ref{non-linear-estimate-for-error}, functional $\mathcal{X}^{N}$ in Lemma \ref{difference-terms}, functionals $\mathcal{V}^{N}$ and $\mathcal{W}^{N,j,l}$ in Lemma \ref{difference-terms-2}.

\begin{lem}\label{non-linear-estimate-for-error}
Let $\epsilon \geq 0$. Let $g, h, f$ be suitable functions.
Recall the definition of $\mathcal{Y}^{N,j,l}_{\epsilon}$ in \eqref{mathcal-Y-no-weight-up}. The following three statements hold true.
\begin{enumerate}
\item If $N=4, l=14$, we have
\ben  \label{g-h-f-N-equals-4-x-v-weight-15}
\mathcal{Y}^{N,j,l}_{\epsilon}(g,h,f) \leq  C\|g\|_{H^{4}_{x,v}} \left(\mathcal{D}_{0}^{4,14}(h)\right)^{1/2}\left(\mathcal{D}_{\epsilon}^{4,14}(f)\right)^{1/2}
.
\een
\item If $N=4, l>14$, we have for any $\delta>0$,
\ben \label{g-h-f-N-equals-4-x-v-weight-larger}
\mathcal{Y}^{N,j,l}_{\epsilon}(g,h,f) &\leq& C\|g\|_{H^{4}_{x,v}} \|h\|_{H^{N-j}_{x}\dot{H}^{j}_{0,l+j\gamma+\gamma/2}}\|f\|_{H^{N-j}_{x}\dot{H}^{j}_{\epsilon,l+j\gamma+\gamma/2}}
+\delta\|f\|^{2}_{H^{N-j}_{x}\dot{H}^{j}_{\epsilon,l+j\gamma+\gamma/2}}
\\&&    +  \delta^{-1} C_{l} \|g\|^{2}_{H^{4}_{x,v}}\mathcal{E}^{4,l}(h). \nonumber \een
\item If $N\geq 5, l\geq 3N+2$, we have for any $\delta>0$,
\ben \label{g-h-f-N-geq-5-x-v-weight-larger}
\mathcal{Y}^{N,j,l}_{\epsilon}(g,h,f) &\leq& C\|g\|_{H^{4}_{x,v}} \|h\|_{H^{N-j}_{x}\dot{H}^{j}_{0,l+j\gamma+\gamma/2}}\|f\|_{H^{N-j}_{x}\dot{H}^{j}_{\epsilon,l+j\gamma+\gamma/2}}
+\delta\|f\|^{2}_{H^{N-j}_{x}\dot{H}^{j}_{\epsilon,l+j\gamma+\gamma/2}}
\\&&  +  \delta^{-1} C_{N,l} \|g\|^{2}_{H^{4}_{x,v}}\|h\|^{2}_{H^{N-j}_{x}\dot{H}^{j}_{l+j\gamma+\gamma/2}}
\nonumber \\&&  + \delta^{-1}C_{N,l}\|g\|^{2}_{H^{N}_{x,v}}\mathcal{D}_{\epsilon}^{N-1,l}(h) + \delta^{-1}C_{N,l}\|g\|^{2}_{H^{4}_{x,v}}\|h\|^{2}_{H^{N}_{x,v}}.
\nonumber
\een \end{enumerate}
Recall the definition of $\mathcal{Z}^{N}_{\epsilon}$ in \eqref{mathcal-Z-no-weight-up}.
The following two statements hold true.
\begin{enumerate}
\item If $N=4$, then
\ben \label{g-h-f-N-equals-4-x-v-no-weight}
|\mathcal{Z}^{N}_{\epsilon}(g,h,f)| \leq
C\|g\|_{H^{4}_{x,v}} (\mathcal{D}_{0}^{4}(h))^{1/2} (\mathcal{D}_{\epsilon}^{4}(f))^{1/2}
.
\een

\item If $N\geq5$, then for any $\delta>0$,
\ben \label{g-h-f-N-geq-5-x-v-no-weight}
|\mathcal{Z}^{N}_{\epsilon}(g,h,f)| &\leq&
C\|g\|_{H^{4}_{x,v}} \|h\|_{H^{N}_{x}H^{0}_{0,\gamma/2}}\|f\|_{H^{N}_{x}H^{0}_{\epsilon,\gamma/2}}
+\delta\|f\|^{2}_{H^{N}_{x}H^{0}_{\epsilon,\gamma/2}}
\\&&+ \delta^{-1}C_{N}\|g\|^{2}_{H^{N}_{x,v}}\mathcal{D}_{\epsilon}^{N-1}(h)
+ \delta^{-1}C_{N}\|g\|^{2}_{H^{4}_{x,v}}\|h\|^{2}_{H^{N}_{x,v}}. \nonumber
\een
\end{enumerate}

We emphasize that $C$ is a universal constant independent of $N,l$.
\end{lem}
\begin{proof}
A typical term in $\mathcal{Y}^{N,j,l}_{\epsilon}(g,h,f)$ is $(W_{l+j\gamma}\pa^{\alpha}_{\beta}\Gamma^{\epsilon}(g,h), W_{l+j\gamma}\pa^{\alpha}_{\beta}f)$ for some fixed $\alpha,\beta$ such that $|\alpha|\leq N-j,|\beta|=j$. For simplicity, let $q=l+j\gamma$.
We use
\beno W_{q}\pa^{\alpha}_{\beta}\Gamma^{\epsilon}(g,h) = W_{q}\Gamma^{\epsilon}(g,\pa^{\alpha}_{\beta}h)
+W_{q}[\pa^{\alpha}_{\beta},\Gamma^{\epsilon}(g,\cdot)]h.
\eeno
Since  $(W_{q}[\pa^{\alpha}_{\beta},\Gamma^{\epsilon}(g,\cdot)]h, W_{q}\pa_{\beta}^{\alpha}f)$ is handled in Lemma \ref{derivative-commutator-lemma}, we only need to focus on the first term.

By \eqref{version-2-dissipation-landau} in Corollary \ref{Gamma-full-up-bound-with-weight}, we have
\beno |\langle W_{q}\Gamma^\epsilon (g, \pa^{\alpha}_{\beta}h), W_{q}\pa_{\beta}^{\alpha}f\rangle| \lesssim
|g|_{H^{2}}|\pa^{\alpha}_{\beta}h|_{0,q+\gamma/2}|\pa_{\beta}^{\alpha}f|_{\epsilon,q+\gamma/2} +
C_{l}|g|_{H^{2}}|\pa_{\beta}^{\alpha}h|_{L^{2}_{q+\gamma/2}}|\pa_{\beta}^{\alpha}f|_{\epsilon,q+\gamma/2}.\eeno
Then by the imbedding $H^{2}_{x} \rightarrow L^{\infty}_{x}$, we get
\beno
|(W_{q}\Gamma^{\epsilon}(g,\pa^{\alpha}_{\beta}h), W_{q}\pa_{\beta}^{\alpha}f)| &\lesssim& \|g\|_{H^{4}_{x,v}}\|\pa_{\beta}^{\alpha}h\|_{H^{0}_xH^{0}_{0,q+\gamma/2}}\|\pa_{\beta}^{\alpha}f\|_{H^{0}_xH^{0}_{\epsilon,q+\gamma/2}}
\\&&+
C_{l}\|g\|_{H^{4}_{x,v}}\|\pa_{\beta}^{\alpha}h\|_{H^{0}_xH^{0}_{q+\gamma/2}}\|\pa_{\beta}^{\alpha}f\|_{H^{0}_xH^{0}_{\epsilon,q+\gamma/2}}
.
\eeno

  When $N=4, l=14$, since $C_{l}$ is a universal constant, we have
\beno
|(W_{q}\Gamma^{\epsilon}(g,\pa^{\alpha}_{\beta}h), W_{q}\pa_{\beta}^{\alpha}f)| \lesssim
\|g\|_{H^{4}_{x,v}} \left(\mathcal{D}_{0}^{4,14}(h)\right)^{1/2}\left(\mathcal{D}_{\epsilon}^{4,14}(f)\right)^{1/2}
.
\eeno
From which together with \eqref{derivative-commutator-N-4-l-14} in Lemma \ref{derivative-commutator-lemma}, we get
\eqref{g-h-f-N-equals-4-x-v-weight-15}.
When $N=4, l>14$ or $N\geq 5, l\geq3N+2$, one has
\beno
|(W_{q}\Gamma^{\epsilon}(g,\pa^{\alpha}_{\beta}h), W_{q}\pa_{\beta}^{\alpha}f)| &\lesssim& \|g\|_{H^{4}_{x,v}}\|\pa_{\beta}^{\alpha}h\|_{H^{0}_xH^{0}_{0,q+\gamma/2}}\|\pa_{\beta}^{\alpha}f\|_{H^{0}_xH^{0}_{\epsilon,q+\gamma/2}}+
\delta  \|\pa_{\beta}^{\alpha}f\|^{2}_{H^{0}_xH^{0}_{\epsilon,q+\gamma/2}}
\\&&+ \delta^{-1}C_{l}\|g\|^{2}_{H^{4}_{x,v}}\|\pa_{\beta}^{\alpha}h\|^{2}_{H^{0}_xH^{0}_{q+\gamma/2}}
.
\eeno
From which together with \eqref{derivative-commutator-N-4-l-large} and \eqref{derivative-commutator-N-geq-5} in Lemma \ref{derivative-commutator-lemma}, we get  \eqref{g-h-f-N-equals-4-x-v-weight-larger}
and  \eqref{g-h-f-N-geq-5-x-v-weight-larger}.

For
 $\mathcal{Z}^{N}_{\epsilon}(g,h,f)$, it is not difficult to copy the above argument to get the desired result.
\end{proof}

We recall an estimate on the operator $\Gamma-\Gamma^{\epsilon}$, which can be derived similarly as in \cite{zhou2020refined}.
\begin{lem}\label{estimate-operator-difference} There holds
\beno|\langle W_{q}(\Gamma-\Gamma^{\epsilon})(g,h), f \rangle| \lesssim C_{q}|\ln \epsilon|^{-1}|\mu^{1/32}g|_{H^{3}}|h|_{H^{3}_{q+15/2}}|f|_{L^{2}_{-3/2}}.\eeno
\end{lem}
As an application of Lemma \ref{estimate-operator-difference}, we have
\begin{lem}
\label{difference-terms} Let $N \geq 4$. Recall the function $G_{1}$ in \eqref{three-terms} and the functional $\mathcal{X}^{N}$ in \eqref{mathcal-X-pure-x}. The following estimate is valid.
\beno \mathcal{X}^{N}(G_{1}) \leq C_{N}\mathcal{D}_{\epsilon}^{N+3,3N+12}(f^{\epsilon})+
C_{N}\|f^{\epsilon}\|^{2}_{H^{N+3}_{x,v}}\mathcal{D}_{0}^{N+3,3N+12}(f).\eeno
\end{lem}
\begin{proof}
By Lemma \ref{estimate-operator-difference}, we have
\beno |\langle  \pa^{\alpha}G_{1}, e_j\rangle| \lesssim |\pa^{\alpha}f^{\epsilon}|_{H^{3}_{9+\gamma/2}} + C_{N}\sum_{\alpha_{1}+\alpha_{2}=\alpha}|\pa^{\alpha_{1}}f^{\epsilon}|_{H^{3}_{\gamma/2}}|\pa^{\alpha_{2}}f|_{H^{3}_{9+\gamma/2}}.\eeno
Since $N \geq 4$, by the embedding $H^{2}_{x} \rightarrow L^{\infty}_{x}$, we get
\beno \mathcal{X}^{N}(G_{1}) &\leq& C_{N}\|f^{\epsilon}\|^{2}_{H^{N}_{x}H^{3}_{9+\gamma/2}}+
C_{N}\|f^{\epsilon}\|^{2}_{H^{N}_{x}H^{3}_{\gamma/2}}\|f\|^{2}_{H^{N}_{x}H^{3}_{9+\gamma/2}}
\\&\leq& C_{N}\mathcal{D}_{\epsilon}^{N+3,3N+12}(f^{\epsilon})+
C_{N}\|f^{\epsilon}\|^{2}_{H^{N+3}_{x,v}}\mathcal{D}_{0}^{N+3,3N+12}(f),\eeno
thanks to $\mathcal{D}_{\epsilon}^{N+3,3N+12}(f^{\epsilon}) \geq \|f^{\epsilon}\|^{2}_{H^{N+3-j}_{x}\dot{H}^{j}_{\epsilon, 3N+3+\gamma/2}} \geq \|f^{\epsilon}\|^{2}_{H^{N+3-j}_{x} \dot{H}^{j}_{3N+3+\gamma/2}}$ for any $0\leq j \leq N+3$.
\end{proof}

As another application of Lemma \ref{estimate-operator-difference}, we have
\begin{lem}\label{difference-terms-2} Let $N \geq 4, l\geq 3N+2$. Recall the function $G_{1}$ in \eqref{three-terms} and the functional $\mathcal{V}^{N}$ in  \eqref{mathcal-V-no-weight-up} and  $\mathcal{W}^{N,j,l}$ in  \eqref{mathcal-W-no-weight-up}.
For any $\delta>0$, there holds
\beno
\mathcal{V}^{N}(G_{1})+\mathcal{W}^{N,j,l}(G_{1}) \leq \delta \mathcal{D}_{0}^{N,l}(F^{\epsilon}_{R}) + \delta^{-1}C_{N,l}(\mathcal{D}_{\epsilon}^{N+3,l+18}(f^{\epsilon})+
\|f^{\epsilon}\|^{2}_{H^{N+3}_{x,v}}\mathcal{D}_{0}^{N+3,l+18}(f)).\eeno
\end{lem}
\begin{proof} It suffices to only consider $\mathcal{W}^{N,j,l}(G_{1})$.
Set $q=l+j\gamma$.
By Lemma \ref{estimate-operator-difference}, we have
\beno |\langle  W_{q}\pa^{\alpha}_{\beta}G_{1}, W_{q}\pa^{\alpha}_{\beta}F^{\epsilon}_{R}\rangle| &\lesssim& \sum_{\beta_{1}\leq\beta}C_{N,q}|\pa^{\alpha}_{\beta_{1}}f^{\epsilon}|_{H^{3}_{q+9+\gamma/2}}
|\pa^{\alpha}_{\beta}F^{\epsilon}_{R}|_{L^{2}_{q+\gamma/2}}
\\&&+ \sum_{\alpha_{1}+\alpha_{2}=\alpha,\beta_{1}+\beta_{2}\leq\beta}
C_{N,q}|\pa^{\alpha_{1}}_{\beta_{1}}f^{\epsilon}|_{H^{3}}|\pa^{\alpha_{2}}_{\beta_{2}}f|_{H^{3}_{q+9+\gamma/2}}|
\pa^{\alpha}_{\beta}F^{\epsilon}_{R}|_{L^{2}_{q+\gamma/2}}.\eeno
Since $N \geq 4$, by embedding $H^{2}_{x} \rightarrow L^{\infty}_{x}$, we get
\beno \mathcal{W}^{N,j,l}(G_{1}) &=& \sum_{|\alpha|+|\beta|\leq N}
|(W_{l+|\beta|\gamma}\pa^{\alpha}_{\beta}G, W_{l+|\beta|\gamma}\pa^{\alpha}_{\beta}F^{\epsilon}_{R})|
\\&\lesssim& C_{N,l}\sqrt{\mathcal{D}_{\epsilon}^{N+3,l+18}(f^{\epsilon})}
\sqrt{\mathcal{D}_{0}^{N,l}(F^{\epsilon}_{R})}
+C_{N,l} \|f^{\epsilon}\|_{H^{N+3}_{x,v}}
\sqrt{\mathcal{D}_{0}^{N+3,l+18}(f)}
\sqrt{\mathcal{D}_{0}^{N,l}(F^{\epsilon}_{R})},
\eeno
since $\mathcal{D}_{\epsilon}^{N+3,l+18}(f^{\epsilon}) \geq \|f^{\epsilon}\|^{2}_{H^{N+3-j}_{x} \dot{H}^{j}_{\epsilon, l+18-3j+\gamma/2}} \geq \|f^{\epsilon}\|^{2}_{H^{N+3-j}_{x} \dot{H}^{j}_{l+18-3j+\gamma/2}}$ for any $0\leq j \leq N+3$. Then by the basic inequality $2 ab \leq \delta a^{2} + \delta^{-1} b^{2}$, we get the result.
\end{proof}

We are ready to prove  \eqref{error-function-uniform-estimate}.
\begin{proof}[Proof of Theorem \ref{asymptotic-result}(Asymptotic formula)]
We give a detailed proof to the case $N=4, l=14$. For the other two cases, we only illustrate the main differences.

\underline{\it{Case 1: $N=4, l=14$.}} In this case the constants  $M, L, K_{j}$ in \eqref{essential-micro-macro-error-function-2} are universal. By \eqref{only-x-with-mu-type-N-equals-4} in Lemma \ref{pure-x-term-from-I} for $\mathcal{C}^{N}_{\epsilon}(f^{\epsilon},F^{\epsilon}_{R})$ and $\mathcal{C}^{N}_{0}(F^{\epsilon}_{R},f)$, \eqref{g-h-f-N-equals-4-x-v-weight-15} in Lemma \ref{non-linear-estimate-for-error} for $\mathcal{Y}^{N,j,l}_{\epsilon}(f^{\epsilon},F^{\epsilon}_{R},F^{\epsilon}_{R})$ and $\mathcal{Y}^{N,j,l}_{0}(F^{\epsilon}_{R},f,F^{\epsilon}_{R})$,
\eqref{g-h-f-N-equals-4-x-v-no-weight} in Lemma \ref{non-linear-estimate-for-error} for
$\mathcal{Z}^{N}_{\epsilon}(f^{\epsilon},F^{\epsilon}_{R},F^{\epsilon}_{R})$ and $\mathcal{Z}^{N}_{0}(F^{\epsilon}_{R},f,F^{\epsilon}_{R})$,
Lemma \ref{difference-terms} for $\mathcal{X}^{N}(G_{1})$, Lemma \ref{difference-terms-2} for $\mathcal{V}^{N}(G_{1})$ and $\mathcal{W}^{N,j,l}(G_{1})$, we get
\beno\frac{d}{dt}\Xi^{4,14}(F^{\epsilon}_{R}) +  \frac{1}{4} \mathcal{D}_{0}^{4,14}(F^{\epsilon}_{R}) &\leq&
C(\|f^{\epsilon}\|^{2}_{H^{4}_{x,v}}+\|f^{\epsilon}\|_{H^{4}_{x,v}}+\delta) \mathcal{D}_{0}^{4,14}(F^{\epsilon}_{R})
+ \delta^{-1}C\|F^{\epsilon}_{R}\|^{2}_{H^{4}_{x,v}} \mathcal{D}_{0}^{4,14}(f)
\\&&+ \delta^{-1}C (\mathcal{D}_{\epsilon}^{7,32}(f^{\epsilon})+
\|f^{\epsilon}\|^{2}_{H^{7}_{x,v}}\mathcal{D}_{0}^{7,32}(f)).
 \eeno

For the moment, let $\delta_{0}$ be the universal constant such that Theorem \ref{a-priori-estimate-LBE},  global well-posedness \eqref{uniform-controlled-by-initial} and propagation of regularity \eqref{propagation} in Theorem \ref{asymptotic-result} are valid.

Let $\delta$ verify $C\delta = 1/16$. Let $\delta_{4}$ be the largest number verifying
$C(\delta_{4}+\delta_{4}^{1/2}) \leq 1/16,$ and $\delta_{4} \leq \delta_{0}$.
Let $\delta_{5}$ be the largest number verifying $C\delta_{5} \leq \delta_{4}$ and $\delta_{5} \leq \delta_{0}$. Then by \eqref{uniform-controlled-by-initial}, if $\mathcal{E}^{4,14}(f_{0})\le \delta_{5}$, we have $\sup_{t\ge 0} \mathcal{E}^{4,14}(f^{\epsilon}({t}))\leq C\delta_{5} \leq \delta_{4}$, which gives $\sup_{t\geq0}\|f^{\epsilon}(t)\|^{2}_{H^{4}_{x,v}} \leq \delta_{4}$ since $\|f^{\epsilon}\|^{2}_{H^{4}_{x,v}} \leq \mathcal{E}^{4,14}(f^{\epsilon})$. Therefore $C(\|f^{\epsilon}\|^{2}_{H^{4}_{x,v}}+\|f^{\epsilon}\|_{H^{4}_{x,v}}+\delta) \leq 1/8$ and thus

\beno\frac{d}{dt}\Xi^{4,14}(F^{\epsilon}_{R}) +  \frac{1}{8} \mathcal{D}_{0}^{4,14}(F^{\epsilon}_{R}) \leq
 C\|F^{\epsilon}_{R}\|^{2}_{H^{4}_{x,v}} \mathcal{D}_{0}^{4,14}(f)+ C(\mathcal{D}_{\epsilon}^{7,32}(f^{\epsilon})+
\|f^{\epsilon}\|^{2}_{H^{7}_{x,v}}\mathcal{D}_{0}^{7,32}(f)).
 \eeno

 By the propagation result \eqref{uniform-estimate-propagation-N-geq-5-l-big} in Theorem \ref{a-priori-estimate-LBE}, we have for $\epsilon \geq 0$ small enough,
\ben \label{propagation-7-33} \mathcal{E}^{7,32}(f^{\epsilon}(t)) +  \int_{0}^{t}\mathcal{D}_{\epsilon}^{7,32}(f^{\epsilon}(s))ds \leq P_{7,32}(\mathcal{E}^{7,32}(f_{0})).
\een	
Recall the natural relation $\|F^{\epsilon}_{R}\|^{2}_{H^{4}_{x,v}} \leq \Xi^{4,14}(F^{\epsilon}_{R}), \|f^{\epsilon}\|^{2}_{H^{7}_{x,v}} \leq \mathcal{E}^{7,32}(f^{\epsilon})$.
By Gronwall's inequality and the initial condition $F^{\epsilon}_{R}(0)=0$, we have
\beno &&\Xi^{4,14}(F^{\epsilon}_{R}(t)) +  \frac{1}{8} \int_{0}^{t} \mathcal{D}_{0}^{4,14}(F^{\epsilon}_{R}) d \tau
\\&\leq& \exp(C \int_{0}^{\infty}\mathcal{D}_{0}^{4,14}(f) dt) (C\int_{0}^{\infty}\mathcal{D}_{\epsilon}^{7,32}(f^{\epsilon}) dt+
C\sup_{t \geq 0} \|f^{\epsilon}(t)\|^{2}_{H^{7}_{x,v}} \int_{0}^{\infty} \mathcal{D}_{0}^{7,32}(f) dt)
\\&\leq& C \exp(CP_{7,32}(\mathcal{E}^{7,32}(f_{0})))  P_{7,32}(\mathcal{E}^{7,32}(f_{0})) \left(1+P_{7,32}(\mathcal{E}^{7,32}(f_{0}))\right) .
 \eeno
By the equivalence relation \eqref{energy-equivalence} and recalling $F^{\epsilon}_{R}(t) = |\ln \epsilon| (f^{\epsilon}-f)$, we get
\beno
&&\mathcal{E}^{4,14}(f^{\epsilon}(t)-f(t)) + \int_{0}^{t} \mathcal{D}_{0}^{4,14}(f^{\epsilon}(t)-f(t)) d \tau \\&\leq& |\ln \epsilon|^{-2} C \exp(C P_{7,32}(\mathcal{E}^{7,32}(f_{0}))  P_{7,32}(\mathcal{E}^{7,32}(f_{0})) \left(1+P_{7,32}(\mathcal{E}^{7,32}(f_{0}))\right)
:= |\ln \epsilon|^{-2} U_{4,14}(\mathcal{E}^{7,32}(f_{0})).
\eeno

\underline{\it{Case 2: $N=4, l>14$.}}  We use \eqref{g-h-f-N-equals-4-x-v-weight-larger}  in Lemma \ref{non-linear-estimate-for-error} to deal with $\mathcal{Y}^{N,j,l}_{\epsilon}(f^{\epsilon},F^{\epsilon}_{R},F^{\epsilon}_{R})$ and $\mathcal{Y}^{N,j,l}_{0}(F^{\epsilon}_{R},f,F^{\epsilon}_{R})$. The other terms can be handled in the same way as in {\it Case 1.}  We skip the details here.

\underline{\it{Case 3: $N \geq 5, l \geq 3N+2$.}}
We use \eqref{only-x-with-mu-type-N-bigger-5} in Lemma \ref{pure-x-term-from-I} to handle $\mathcal{C}^{N}_{\epsilon}(f^{\epsilon},F^{\epsilon}_{R})$ and $\mathcal{C}^{N}_{0}(F^{\epsilon}_{R},f)$, \eqref{g-h-f-N-geq-5-x-v-weight-larger} in Lemma \ref{non-linear-estimate-for-error} to handle $\mathcal{Y}^{N,j,l}_{\epsilon}(f^{\epsilon},F^{\epsilon}_{R},F^{\epsilon}_{R})$ and $\mathcal{Y}^{N,j,l}_{0}(F^{\epsilon}_{R},f,F^{\epsilon}_{R})$, \eqref{g-h-f-N-geq-5-x-v-no-weight} in Lemma \ref{non-linear-estimate-for-error} to handle
$\mathcal{Z}^{N}_{\epsilon}(f^{\epsilon},F^{\epsilon}_{R},F^{\epsilon}_{R})$ and $\mathcal{Z}^{N}_{0}(F^{\epsilon}_{R},f,F^{\epsilon}_{R})$. As in the {\it Proof of Theorem \ref{a-priori-estimate-LBE}}, we can apply mathematical induction on $N$ and a sequence of functions $U_{N,l}$ can be defined in an iterating manner. We skip the details here. However, the smallness assumption on $\mathcal{E}^{4,14}(f_{0})$ (bounded by a universal constant) is not affected in the process.
\end{proof}

\setcounter{equation}{0}

\section{Appendix}\label{appendix}
In this appendix, we prove several results for the sake of completeness.
We first recall the definition of the symbol class $S^{m}_{1,0}$.
\begin{defi}\label{psuopde} A smooth function $a(v,\xi)$ is said to be a symbol of type $S^{m}_{1,0}$ if for any multi-indices $\alpha$ and $\beta$,
\beno |(\pa^\alpha_\xi\pa^\beta_v a)(v,\xi)|\le C_{\alpha,\beta} \langle \xi\rangle^{m-|\alpha|}, \eeno
where $C_{\alpha,\beta}$ is a constant depending only on   $\alpha$ and $\beta$.
\end{defi}

The following is a result on the commutator between multipliers in frequency and phase spaces.
\begin{lem}[Lemma 5.3 in \cite{he2018sharp}]\label{operatorcommutator1}
Let $l, s, r \in \R, M(\xi) \in S^{r}_{1,0}$ and $\Phi(v) \in S^{l}_{1,0}$. There holds
\beno
|[M(D), \Phi]f|_{H^{s}} \lesssim C|f|_{H^{r+s-1}_{l-1}}.
\eeno
\end{lem}

As a direct application, we get the following result.
\begin{lem}\label{piece-to-whole} Recall the localizers $\varphi_{k}$ in \eqref{dyadic-decomposition}. There holds
$
\sum_{k \geq -1}^{\infty}|W^{\epsilon}(D)\varphi_{k}f|^{2}_{L^{2}}
\lesssim |W^{\epsilon}(D)f|^{2}_{L^{2}}.
$
\end{lem}
\begin{proof}
Since $W^{\epsilon}\in S^{1}_{1,0}, 2^{k} \varphi_{k} \in S^{1}_{1,0}$,
then by Lemma \ref{operatorcommutator1},
we have
\beno
\sum_{k \geq -1}^{\infty}|W^{\epsilon}(D)\varphi_{k}f|^{2}_{L^{2}} =\sum_{k \geq -1}^{\infty}2^{-2k}|W^{\epsilon}(D)2^{k}\varphi_{k}f|^{2}_{L^{2}}
\lesssim\sum_{k \geq -1}^{\infty}2^{-2k}(|2^{k}\varphi_{k}W^{\epsilon}(D)f|^{2}_{L^{2}}+|f|^{2}_{H^{0}}).
\eeno
The lemma then follows accordingly.
\end{proof}


\begin{prop}[Theorem 3.1 in \cite{he2014well}] \label{symbol}Suppose $ A^\epsilon(\xi):= \int b^\epsilon(\f{\xi}{|\xi|}\cdot \sigma)\min\{ |\xi|^2\sin^2(\theta/2),1\} d\sigma$. Then we have
$A^\epsilon(\xi)\sim |\xi|^2\mathrm{1}_{|\xi|\le2}+\mathrm{1}_{|\xi|\ge2}(W^\epsilon(\xi))^2$.
\end{prop}

\begin{prop} \label{fourier-transform-cross-term} There holds
	\beno
	\int_{\R^3\times\mathbb{S}^2} b(\f{u}{|u|}\cdot \sigma) h(u)(f(u^+)-f(\f{|u|}{|u^+|}u^+)) d\sigma du
	=\int_{\R^3\times\mathbb{S}^2} b(\f{\xi}{|\xi|}\cdot \sigma)  (\hat{h}(\xi^+)-\hat{h}(\f{|\xi|}{|\xi^+|}\xi^+))\bar{\hat{f}}(\xi) d\sigma d\xi.
	\eeno
\end{prop}
\begin{proof} By Plancherel equality, we have
	\beno &&\int_{\R^3\times\SS^2} b(\f{u}{|u|}\cdot \sigma) h(u) f(\f{|u|}{|u^+|}u^+) d\sigma du\\
	&=&\int_{\R^3} h(u) \bigg(\int_{\SS^2}  b(\f{u}{|u|}\cdot \sigma)  f(\f{|u|}{|u^+|}u^+)d\sigma\bigg) du
	=\int_{\R^3} \hat{h}(\xi) \bar{\hat{F}}(\xi)d\xi,\eeno
where we set $F(u):=\int_{\SS^2}  b(\f{u}{|u|}\cdot \sigma)  f(\f{|u|}{|u^+|}u^+)d\sigma$.
Let us derive the Fourier transform $\hat{F}$ of $F$. By definition, we have
\beno
\hat{F}(\xi)=\int_{\R^3} e^{-iu\cdot \xi} F(u)du
=\f1{(2\pi)^{3/2}}\int_{\R^3} \int_{\SS^2} \int_{\R^3} e^{-iu\cdot \xi}e^{i\f{|u|}{|u^+|}u^+\cdot\eta} b(\f{u}{|u|}\cdot \sigma)  \hat{f}(\eta)  d\sigma d\eta du.
\eeno
Since $\f{|u|}{|u^+|}u^+\cdot\eta= \f12 \big((\f{u}{|u|}\cdot \sigma +1)/2\big)^{-\f12}(u
\cdot \eta+|u\|\eta| \f{\eta}{|\eta|}\cdot\sigma),$
then by the fact $\int_{\SS^2} b(\kappa\cdot \sigma) d(\tau\cdot \sigma)d\sigma=\int_{\SS^2} b(\tau\cdot \sigma) d(\kappa\cdot \sigma)d\sigma$, one has
\beno
\hat{F}(\xi)
&=&\f1{(2\pi)^{3/2}}\int_{\R^3} \int_{\SS^2} \int_{\R^3} e^{-iu\cdot \xi}e^{i\f{|\eta|}{|\eta^+|}\eta^+\cdot u} b(\f{u}{|u|}\cdot \sigma)  \hat{f}(\eta)  d\sigma d\eta du\\
&=&\f1{(2\pi)^{3/2}}\int_{\R^3} \int_{\SS^2}    b(\f{\eta}{|\eta|}\cdot \sigma)  \hat{f}(\eta)\delta [\xi=\f{|\eta|}{|\eta^+|}\eta^+]  d\sigma d\eta,
\eeno
which gives
\beno \int_{\R^3\times\SS^2} b(\f{u}{|u|}\cdot \sigma) h(u) f(\f{|u|}{|u^+|}u^+) d\sigma du
=\int_{\R^3\times\SS^2} b(\f{\xi}{|\xi|}\cdot \sigma)   \hat{h}(\f{|\xi|}{|\xi^+|}\xi^+)\bar{\hat{f}}(\xi) d\sigma d\xi.\eeno
A similar argument works for the remainder term and then we get the proposition.
\end{proof}

\begin{lem}[Lemma 5.8 in \cite{he2018sharp}]\label{comWep} Let $\mathcal{F}$ be Fourier transform, then  $\mathcal{F}W^\epsilon((-\triangle_{\mathbb{S}^2})^{1/2})=W^\epsilon((-\triangle_{\mathbb{S}^2})^{1/2})\mathcal{F}$.
\end{lem}

The following lemma relies on some localization techniques in phase space and then in frequency space.

 \begin{lem}\label{a-technical-lemma} There holds
\beno \mathcal{A} := |\ln \epsilon|^{-1}\int_{\R^{3}}\int_{\epsilon}^{\pi/4} \theta^{-3}|f(v) - f(v/\cos\theta)|^{2} dv d\theta \lesssim |W^{\epsilon}(D)f|^{2}_{L^{2}} + |W^{\epsilon}f|^{2}_{L^{2}}.\eeno
\end{lem}
\begin{proof}
First applying dyadic decomposition in phase space and using the fact $1/\sqrt{2} \leq \cos\theta \leq 1$ when $0 \leq \theta \leq \pi/4$, we have
\beno
\mathcal{A} &=& |\ln \epsilon|^{-1}\int_{\R^{3}}\int_{\epsilon}^{\pi/4} \theta^{-3}| \sum_{k=-1}^{\infty}(\varphi_{k}f)(v)- \sum_{k=-1}^{\infty}(\varphi_{k}f)(v/\cos\theta)|^{2} dv d\theta
\\&\lesssim& |\ln \epsilon|^{-1}\sum_{k=-1}^{\infty}\int_{\R^{3}}\int_{\epsilon}^{\pi/4} \theta^{-3}| (\varphi_{k}f)(v)- (\varphi_{k}f)(v/\cos\theta)|^{2} dv d\theta := |\ln \epsilon|^{-1}\sum_{k=-1}^{\infty}\mathcal{A}_{k}.
\eeno
Using $\int_{\epsilon}^{\pi/4} \theta^{-3}d\theta \lesssim \epsilon^{-2}$, it is easy to check $|\ln \epsilon|^{-1}\sum_{2^{k} \geq 1/\epsilon} \mathcal{A}_{k} \lesssim |W^{\epsilon}f|^{2}_{L^{2}}$.
For the case $2^{k} \leq 1/\epsilon$, by By Plancherel's theorem and dyadic decomposition in frequency space, we have
\beno
\mathcal{A}_{k}&=&
\int_{\R^{3}}\int_{\epsilon}^{\pi/4} \theta^{-3}|\widehat{\varphi_{k}f}(\xi)- \cos^{3}\theta\widehat{\varphi_{k}f}(\xi\cos\theta)|^{2} d\xi d\theta
\\&\lesssim& \int_{\R^{3}}\int_{\epsilon}^{\pi/4} \theta^{-3}|\widehat{\varphi_{k}f}(\xi)- \widehat{\varphi_{k}f}(\xi\cos\theta)|^{2} d\xi d\theta + |\ln \epsilon| |\varphi_{k}f|^{2}_{L^{2}}
\\&=& \int_{\R^{3}}\int_{\epsilon}^{\pi/4} \theta^{-3}|\sum_{l=-1}^{\infty}(\varphi_{l}\widehat{\varphi_{k}f})(\xi)- \sum_{l=-1}^{\infty}(\varphi_{l}\widehat{\varphi_{k}f})(\xi\cos\theta)|^{2} d\xi d\theta + |\ln \epsilon\|\varphi_{k}f|^{2}_{L^{2}}
\\&\lesssim& \sum_{l=-1}^{\infty} \int_{\R^{3}}\int_{\epsilon}^{\pi/4} \theta^{-3}|(\varphi_{l}\widehat{\varphi_{k}f})(\xi)- (\varphi_{l}\widehat{\varphi_{k}f})(\xi\cos\theta)|^{2} d\xi d\theta + |\ln \epsilon\|\varphi_{k}f|^{2}_{L^{2}}
\\&:=& \sum_{l=-1}^{\infty}\mathcal{A}_{k,l}  + |\ln \epsilon\|\varphi_{k}f|^{2}_{L^{2}},
\eeno
where we use $\int_{\epsilon}^{\pi/4}\theta^{-3} (1-\cos^{3}\theta) d\theta \lesssim |\ln \epsilon|$. Note that $|\ln \epsilon|^{-1}\sum_{2^{l} \geq 1/\epsilon} \mathcal{A}_{k,l} \lesssim |W^{\epsilon}(D)\varphi_{k}f|^{2}_{L^{2}}$, thus \beno |\ln \epsilon|^{-1}\mathcal{A}_{k} \lesssim |\ln \epsilon|^{-1}\sum_{2^{l} \leq 1/\epsilon} \mathcal{A}_{k,l} + |W^{\epsilon}(D)\varphi_{k}f|^{2}_{L^{2}}.\eeno
By Lemma \ref{piece-to-whole}, we have
\beno
\mathcal{A}  \lesssim |\ln \epsilon|^{-1} \sum_{2^{k}, 2^{l} \leq 1/\epsilon}\mathcal{A}_{k,l}+ |W^{\epsilon}(D)f|^{2}_{L^{2}} + |W^{\epsilon}f|^{2}_{L^{2}}.
\eeno
For each pair $k,l$ such that $2^{k}, 2^{l}\leq 1/\epsilon$, we have
\beno
\mathcal{A}_{k,l}  &=&  \int_{\R^{3}}\int_{\epsilon}^{2^{-k/2-l/2}} \theta^{-3}|(\varphi_{l}\widehat{\varphi_{k}f})(\xi)- (\varphi_{l}\widehat{\varphi_{k}f})(\xi\cos\theta)|^{2} d\xi d\theta
\\&&+ \int_{\R^{3}}\int_{2^{-k/2-l/2}}^{\pi/4} \theta^{-3}|(\varphi_{l}\widehat{\varphi_{k}f})(\xi)- (\varphi_{l}\widehat{\varphi_{k}f})(\xi\cos\theta)|^{2} d\xi d\theta
\\&\lesssim& \int_{\R^{3}}\int_{\epsilon}^{2^{-k/2-l/2}} \theta^{-3}|(\varphi_{l}\widehat{\varphi_{k}f})(\xi)- (\varphi_{l}\widehat{\varphi_{k}f})(\xi\cos\theta)|^{2} d\xi d\theta
+ 2^{l+k}|\varphi_{l}\widehat{\varphi_{k}f}|^{2}_{L^{2}}
:=\mathcal{B}_{k,l}+ 2^{l+k}|\varphi_{l}\widehat{\varphi_{k}f}|^{2}_{L^{2}}.
\eeno
By Taylor expansion,
$
(\varphi_{l}\widehat{\varphi_{k}f})(\xi)- (\varphi_{l}\widehat{\varphi_{k}f})(\xi\cos\theta) = (1-\cos\theta)\int_{0}^{1}
(\nabla \varphi_{l}\widehat{\varphi_{k}f})(\xi(\kappa))\cdot \xi d\kappa,
$
where $\xi(\kappa) = (1-\kappa)\xi\cos\theta + \kappa \xi$. Thus  we obtain
\beno
\mathcal{B}_{k,l}
\lesssim \int_{\R^{3}}\int_{\epsilon}^{2^{-k/2-l/2}}\int_{0}^{1} \theta|\xi|^{2}|(\nabla \varphi_{l}\widehat{\varphi_{k}f})(\xi(\kappa)) |^{2} d\xi d\theta d\kappa.
\eeno
In the change of variable $\xi \rightarrow \eta = \xi(\kappa)$, the Jacobean is $|d \eta / d \xi| = ((1-\kappa)\cos\theta + \kappa)^{3} \sim 1$ since $1/\sqrt{2} \leq \cos\theta \leq 1$. By the change,
we have
\beno
\mathcal{B}_{k,l}
&=& \int_{\R^{3}}\int_{\epsilon}^{2^{-k/2-l/2}}\int_{0}^{1} \theta\frac{|\eta|^{2}}{((1-\kappa)\cos\theta + \kappa)^{5}}|\nabla (\varphi_{l}\widehat{\varphi_{k}f})(\eta)|^{2} d\eta d\theta d\kappa
\\&\lesssim&\int_{\R^{3}}\int_{\epsilon}^{2^{-k/2-l/2}} \theta^{}|\eta|^{2}|(\nabla \varphi_{l}\widehat{\varphi_{k}f}) (\eta)|^{2} d\eta d\theta
\\&\lesssim& 2^{-(l+k)} \int_{\R^{3}} |\eta|^{2}|\nabla \varphi_{l}\widehat{\varphi_{k}f} (\eta)|^{2} d\eta
\lesssim 2^{l+k}2^{-2k} \int_{\R^{3}} |\nabla \varphi_{l}\widehat{\varphi_{k}f} (\eta)|^{2} d\eta.
\eeno
Since
$
(\nabla \varphi_{l}\widehat{\varphi_{k}f})(\eta) = (\nabla \varphi_{l})(\eta) \widehat{\varphi_{k}f}(\eta) + (\varphi_{l} \nabla \widehat{\varphi_{k}f})(\eta)
= 2^{-l}(\nabla \varphi) (\eta/2^{l}) \widehat{\varphi_{k}f}(\eta) - i (\varphi_{l} \widehat{v\varphi_{k}f}) (\eta),
$
we have
\beno
|(\nabla \varphi_{l}\widehat{\varphi_{k}f})(\eta)|^{2} \lesssim 2^{-2l}|\nabla \varphi|^{2}_{L^{\infty}} |\widehat{\varphi_{k}f} (\eta)|^{2}
+  |\varphi_{l}\widehat{v \varphi_{k}f} (\eta)|^{2},
\eeno
which gives
$
\mathcal{B}_{k,l}
\lesssim 2^{-(l+k)} |\varphi_{k}f|^{2}_{L^{2}} + 2^{l-k}|\varphi_{l} \widehat{v\varphi_{k}f}|^{2}_{L^{2}}.
$
We finally arrive at
\beno
\mathcal{A}_{k,l} \lesssim  2^{-(l+k)} |\varphi_{k}f|^{2}_{L^{2}} + 2^{l-k}|\varphi_{l} \widehat{v\varphi_{k}f}|^{2}_{L^{2}} + 2^{l+k}|\varphi_{l}\widehat{\varphi_{k}f}|^{2}_{L^{2}}:=\mathcal{A}_{k,l,1}+\mathcal{A}_{k,l,2}+\mathcal{A}_{k,l,3}.
\eeno
The first term is estimated by $\sum_{2^{k}, 2^{l}\leq 1/\epsilon} \mathcal{A}_{k,l,1} \lesssim |f|^{2}_{L^{2}}.$
In the following, we repeatedly use the fact $W^{\epsilon}(\xi) \gtrsim |\ln \epsilon|^{-1/2} \langle \xi \rangle$ for $|\xi| \lesssim \epsilon^{-1}$ in \eqref{low-frequency-lb-cf}. For the second term $\mathcal{A}_{k,l,2}$, we have
\beno
|\ln \epsilon|^{-1}\sum_{2^{k}, 2^{l}\leq 1/\epsilon} \mathcal{A}_{k,l,2}
&=& |\ln \epsilon|^{-1}\sum_{j=1}^{3}\sum_{2^{k}, 2^{l}\leq 1/\epsilon} 2^{2l}2^{-2k}|\varphi_{l} \widehat{v_{j}\varphi_{k}f}|^{2}_{L^{2}}
\\&\lesssim& \sum_{j=1}^{3}\sum_{2^{k} \leq 1/\epsilon} 2^{-2k}|W^{\epsilon}\widehat{v_{j}\varphi_{k}f}|^{2}_{L^{2}}
= \sum_{j=1}^{3}\sum_{2^{k} \leq 1/\epsilon} 2^{-2k}|W^{\epsilon}(D) v_{j} \varphi_{k}f|^{2}_{L^{2}}
\\&\lesssim& |W^{\epsilon}(D)f|^{2}_{L^{2}} + |f|^{2}_{L^{2}}.
\eeno
In the last inequality, we apply Lemma \ref{operatorcommutator1}  to get  that $|W^{\epsilon}(D) v_{j} \varphi_{k}f|^{2}_{L^{2}} \lesssim |v_{j} \varphi_{k}W^{\epsilon}(D)f|^{2}_{L^{2}} +  |f|^{2}_{H^{0}}$
thanks to $W^{\epsilon}\in S^{1}_{1,0}, v_{j}\varphi_{k} \in S^{1}_{1,0}$.
As for the sum of the last term $\mathcal{A}_{k,l,3}$, we have
\beno
|\ln \epsilon|^{-1}\sum_{2^{k}, 2^{l}\leq 1/\epsilon} \mathcal{A}_{k,l,3}
&\lesssim& |\ln \epsilon|^{-1}\sum_{2^{k}, 2^{l}\leq 1/\epsilon} 2^{2l}|\varphi_{l}\widehat{\varphi_{k}f}|^{2}_{L^{2}}  +
|\ln \epsilon|^{-1}\sum_{2^{k}, 2^{l}\leq 1/\epsilon} 2^{2k}|\varphi_{l}\widehat{\varphi_{k}f}|^{2}_{L^{2}}
\\&\lesssim&  \sum_{2^{k} \leq 1/\epsilon} |W^{\epsilon}(D)\varphi_{k}f|^{2}_{L^{2}} + |\ln \epsilon|^{-1}\sum_{2^{k} \leq 1/\epsilon} 2^{2k}|\widehat{\varphi_{k}f}|^{2}_{L^{2}}
\lesssim |W^{\epsilon}(D)f|^{2}_{L^{2}} + |W^{\epsilon}f|^{2}_{L^{2}}.
\eeno	
Patching together the above estimates, we finish the proof.
\end{proof}

We now prepare a decomposition on the unit sphere. Let $\omega \in C^\infty_c(\R_{+})$ be a non-negative and non-increasing function. Assume that $\omega(x)=1$ on $[0, 2/3]$ and has support $[0, 3/4]$. Moreover, $\omega$ is strictly decreasing on $[2/3, 3/4]$.

In the following, let $\chi$ be a smooth non-decreasing function on $\R$ verifying
 \beno  \chi(x)=\left\{\begin{aligned} & 1, \quad\mbox{if}\quad x\ge0;\\
&  0,\quad\mbox{if}\quad x<-\f1{10}.\end{aligned}\right.\eeno
Suppose  $u=(u_1, u_2, u_3) \in \R^3$. Then it is easy to check for $u\neq 0$,
\ben \label{positive-result} \sum_{i=1}^3 \omega(\sum_{j\neq i} \f{u_j^2}{|u|^2})\ge 1.\een
Indeed, since $ \frac{u^{2}_{1}+u^{2}_{2}}{|u|^{2}} + \frac{u^{2}_{2}+u^{2}_{3}}{|u|^{2}} + \frac{u^{2}_{3}+u^{2}_{1}}{|u|^{2}} = 2,$ at least one of them is no larger than $2/3.$ Recalling $\omega(x)= 1$ when $x \in [0, 2/3],$ so we get \eqref{positive-result}.

Then for $u\neq 0, m=1,2,3,$ we define
\beno \vartheta_{m+}(u):=\f{\omega(\sum\limits_{j\neq m} \f{u_j^2}{|u|^2})}{\sum\limits_{i=1}^3 \omega(\sum\limits_{j\neq i} \f{u_j^2}{|u|^2})}\chi(\f{u_m}{|u|})\quad
\mbox{and}\quad \vartheta_{m-}(u):=\f{\omega(\sum\limits_{j\neq m} \f{u_j^2}{|u|^2})}{\sum\limits_{i=1}^3 \omega(\sum\limits_{j\neq i} \f{u_j^2}{|u|^2})}\chi(-\f{u_m}{|u|}). \eeno
It is obvious $0 \leq \vartheta_{ m \pm } \leq 1$. Note that the functions $\vartheta_{m \pm}$ are homogeneous of degree 0, so it suffices consider them on the unit sphere $\SS^2$. Taking $\vartheta_{3+}$ for example, let us summarize its property in the following lemma.
\begin{lem}\label{decomposition-on-sphere} For $u \in \mathbb{S}^{2}$, one has
\ben
\label{values}
 \vartheta_{3+}(u) = \left\{\begin{aligned} & 1, \quad\mbox{if}\quad u_{3} \ge \sqrt{3}/2;\\
&  0,\quad\mbox{if}\quad u_{3} \leq 1/2.\end{aligned}\right.
\\ \label{0-region}
\vartheta_{3+}(u) = 0 \Leftrightarrow u_{3} \leq 1/2.
\\ \label{decomposition} \sum_{m=1}^3 \big[\vartheta_{m+}(u)+\vartheta_{m-}(u)\big]=1.
\\ \label{equivalence-cut-6} 1/6 \leq \sum_{m=1}^3 \big[\vartheta^{2}_{m+}(u)+\vartheta^{2}_{m-}(u)\big] \leq 1.
\een
\end{lem}
The proof of Lemma \ref{decomposition-on-sphere} is elementary, so we skip it.
We are ready to prove gain of anisotropic regularity.
\begin{prop}\label{anisotropic-key-pro} Set $K^{\epsilon}(r)= |\ln \epsilon|^{-1} r^{-4}\mathrm{1}_{2 \geq r \geq \epsilon}$.
For any smooth function $f$ defined on $\mathbb{S}^2$, we have
\begin{eqnarray}\label{similarlemma5.5}
\int_{\mathbb{S}^2\times\mathbb{S}^2}|f(\sigma)-f(\tau)|^{2} K^{\epsilon}(|\sigma-\tau|)  d\sigma d\tau + |f|^{2}_{L^{2}(\mathbb{S}^{2})}
 \sim |W^{\epsilon}((-\Delta_{\mathbb{S}^{2}})^{1/2})f|^{2}_{L^{2}(\mathbb{S}^{2})} + |f|^{2}_{L^{2}(\mathbb{S}^{2})}.
\end{eqnarray}
As a consequence, for any smooth function $f$ defined on $\mathbb{R}^3$, we have
\begin{eqnarray}\label{similarlemma5.6}
\int_{\mathbb{R}_{+}\times\mathbb{S}^2\times\mathbb{S}^2}|f(r\sigma)-f(r\tau)|^{2} K^{\epsilon}(|\sigma-\tau|) r^{2}d\sigma d\tau dr + |f|^{2}_{L^{2}}
 \sim |W^{\epsilon}((-\Delta_{\mathbb{S}^{2}})^{1/2})f|^{2}_{L^{2}} + |f|^{2}_{L^{2}}.
\end{eqnarray}
\end{prop}

\begin{proof}
We prove it in the sprit of \cite{he2018sharp}.  By Lemma 5.4 of \cite{he2018sharp}, we have
\beno
\int_{\mathbb{S}^2\times\mathbb{S}^2} |f(\sigma)-f(\tau)|^{2} K^{\epsilon}(|\sigma-\tau|) d\sigma d\tau
=\sum_{l=0}^{\infty}\sum_{m=-l}^{l}(f^{m}_{l})^{2} \int_{\mathbb{S}^2\times\mathbb{S}^2}|Y^{m}_{l}(\sigma)-Y^{m}_{l}(\tau)|^{2}K^{\epsilon}(|\sigma-\tau|) d\sigma d\tau,
\eeno where $f^m_l=\int_{\mathbb{S}^2} f(\sigma)Y^m_l(\sigma)d\sigma$.
For simplicity, set \beno \mathcal{A}^{\epsilon}_{l} :=  \int_{\mathbb{S}^2\times\mathbb{S}^2}|Y^{m}_{l}(\sigma)-Y^{m}_{l}(\tau)|^{2}K^{\epsilon}(|\sigma-\tau|) d\sigma d\tau.\eeno

We set to analyze $\mathcal{A}^{\epsilon}_{l}$. By \eqref{equivalence-cut-6} in Lemma \ref{decomposition-on-sphere}, we have
\ben\label{equivalence-to-1-6-2} \sum_{i=1}^3 \sum_{j=+,-}
\mathcal{A}^{\epsilon}_{l,i,j} \leq \mathcal{A}^{\epsilon}_{l}
\leq 6\sum_{i=1}^3 \sum_{j=+,-} \mathcal{A}^{\epsilon}_{l,i,j}, \een
where
$
\mathcal{A}^{\epsilon}_{l,i,j}:=\int_{\SS^2 \times \SS^2} |Y^{m}_{l}(\sigma)-Y^{m}_{l}(\tau)|^2K^{\epsilon}(|\sigma-\tau|)\vartheta_{ij}^2(\sigma) d\sigma d\tau.
$
Then due to the symmetric structure, we only need to focus on the estimate of
\beno \mathcal{A}^{\epsilon}_{l,3,+} = \int_{\SS^2 \times \SS^2} |Y^{m}_{l}(\sigma)-Y^{m}_{l}(\tau)|^2K^{\epsilon}(|\sigma-\tau|)\vartheta_{3+}^2(\sigma) d\sigma d\tau. \eeno

We give a detailed proof to the $\gtrsim$ direction in \eqref{similarlemma5.5} while omit the proof of the $\lesssim$ direction in \eqref{similarlemma5.5}. The proof is divided into three steps.

{\it Step 1:  $\epsilon^{2} l(l+1) \leq \eta^{2}$.} Here $\eta>0$ is a small constant which will be determined later.
By \eqref{0-region}, we have
\beno  \mathcal{A}^{\epsilon}_{l,3,+}&=&
\int_{\SS^2 \times \SS^2} |Y^{m}_{l}(\sigma)-Y^{m}_{l}(\tau)|^2K^{\epsilon}(|\sigma-\tau|)\vartheta_{3+}^2(\sigma) \mathrm{1}_{\sigma_{3} \geq 1/2}d\sigma d\tau
\\&=&
\int_{\SS^2 \times \SS^2} |(\vartheta_{3+}Y^{m}_{l})(\sigma)-(\vartheta_{3+}Y^{m}_{l})(\tau) + \left(\vartheta_{3+}(\tau) -\vartheta_{3+}(\sigma)\right)Y^{m}_{l}(\tau)|^2 K^{\epsilon}(|\sigma-\tau|)
\mathrm{1}_{\sigma_{3} \geq 1/2}d\sigma d\tau
\\&\geq&
\frac{1}{2}\int_{\SS^2 \times \SS^2} |(\vartheta_{3+}Y^{m}_{l})(\sigma)-(\vartheta_{3+}Y^{m}_{l})(\tau) |^2 K^{\epsilon}(|\sigma-\tau|)\mathrm{1}_{\sigma_{3} \geq 1/2} \mathrm{1}_{\tau_{3} \geq \sqrt{5}/5}d\sigma d\tau
\\&&-
\int_{\SS^2 \times \SS^2} |\vartheta_{3+}(\tau) -\vartheta_{3+}(\sigma)|^2|Y^{m}_{l}(\tau)|^2 K^{\epsilon}(|\sigma-\tau|) \mathrm{1}_{\sigma_{3} \geq 1/2} \mathrm{1}_{\tau_{3} \geq \sqrt{5}/5}d\sigma d\tau
:= \frac{1}{2}\mathcal{I}_{1}-\mathcal{I}_{2}.
 \eeno
By mean value theorem, one has for some $\kappa \in [0,1]$,
\beno \vartheta_{3+}(\tau) -\vartheta_{3+}(\sigma) = (\nabla \vartheta_{3+})(\kappa \tau + (1-\kappa)\sigma)\cdot(\tau-\sigma).
\eeno
Since $\sigma_{3} \geq 1/2, \tau_{3} \geq \sqrt{5}/5$, then $\sqrt{5}/5 \leq |\kappa \tau + (1-\kappa)\sigma| \leq 1$. So we have
\beno |\vartheta_{3+}(\tau) -\vartheta_{3+}(\sigma)| \leq  |\nabla \vartheta_{3+}|_{L^{\infty}(B(\sqrt{5}/5,1))}|\sigma-\tau| := C|\sigma-\tau|.
\eeno
where $B(\sqrt{5}/5,1) : = \{x \in \R^{3}, \sqrt{5}/5 \leq |x| \leq 1\}$. Here and in the following, $C$ is a universal constant and it may change across different lines.
Note that $\vartheta_{3+}(u)$ is smooth on $|u| >0$ and then its derivative is bounded on the compact set  $B(\sqrt{5}/5,1)$.
Then we get
\beno \mathcal{I}_{2} &\leq& C\int_{\SS^2 \times \SS^2} |Y^{m}_{l}(\tau)|^2|\sigma-\tau|^{2} K^{\epsilon}(|\sigma-\tau|)d\sigma d\tau
\\&\leq& C \int_{ \SS^2} |Y^{m}_{l}(\tau)|^2 \left( \int_{\SS^2} |\sigma-\tau|^{2} K^{\epsilon}(|\sigma-\tau|) d\sigma \right) d\tau \leq C,
 \eeno
where we used the following computation,  $|\sigma-\tau| = 2 \sin(\theta/2)$, $d\sigma = \sin \theta d\theta d\varphi = 4 \sin(\theta/2) d \sin(\theta/2) d\varphi$,
\beno \int_{\SS^2} |\sigma-\tau|^{2} K^{\epsilon}(|\sigma-\tau|) d\sigma
&=&  16 \int_{0}^{\pi}\int_{0}^{2\pi}  \sin^{3}(\theta/2) K^{\epsilon}(2 \sin(\theta/2)) d \sin(\theta/2) d\varphi
\\&=& 2 \pi \int_{0}^{2} K^{\epsilon}(r) r^{3} d r \leq  C. \eeno

Set $F_{3+}(x):=(\vartheta_{3+}Y^{m}_{l})(x_1,x_2, \sqrt{1-x_1^2-x_2^2})$ for $x=(x_1,x_2) \in \{ x \in \mathbb{R}^{2}: |x| < 1\}$. We make the following change of variables:
\beno \sigma = (\sigma_1, \sigma_2, \sigma_3) \in \{ \sigma \in \mathbb{S}^{2}: \sigma_{3} \geq 1/2 \}
 \rightarrow x = (\sigma_1, \sigma_2)\in \{ x \in \mathbb{R}^{2}: |x| \leq \sqrt{3/4} \},
  \\
\tau = (\tau_1, \tau_2, \tau_3) \in \{ \tau \in \mathbb{S}^{2}: y_{3} \geq \sqrt{5}/5 \}  \rightarrow y = (\tau_1, \tau_2) \in \{ y \in \mathbb{R}^{2}: |y| \leq \sqrt{4/5} \}.
\eeno
In the changes, the Jacobean is $\frac{d \sigma}{d x} = \f{1}{\sqrt{1-|x|^{2}}}$ and $\frac{d\tau}{d y} = \f{1}{\sqrt{1-|y|^{2}}}$.
Let $K^{\epsilon, b}(r) := K^{\epsilon}(r) \mathrm{1}_{r \leq b}$ for $b \geq \epsilon$  and $m(\eta, l) := 4 \sqrt{5}\eta\left(l(l+1)\right)^{-1/2}$.
Then we have
\beno  \mathcal{I}_{1} &\geq& \int_{\SS^2 \times \SS^2} |(\vartheta_{3+}Y^{m}_{l})(\sigma)-(\vartheta_{3+}Y^{m}_{l})(\tau)|^2
 K^{\epsilon, m(\eta,l)}(|\sigma-\tau|) \mathrm{1}_{\sigma_{3} \geq 1/2} \mathrm{1}_{\tau_{3} \geq \sqrt{1/5}}d\sigma d\tau\\
&=& \int_{|x|\leq \sqrt{3/4},|y|\le \sqrt{4/5}} |F_{3+}(x)-F_{3+}(y)|^2 K^{\epsilon, m(\eta,l)}(d(x,y))
\f1{\sqrt{1-|x|^2}}\f1{\sqrt{1-|y|^2}}dxdy,
 \eeno
where $d(x,y):=\big(|x-y|^2+|\sqrt{1-|x|^2}-\sqrt{1-|y|^2}|^2\big)^{1/2} = |\sigma-\tau|$. Since $|x|\leq \sqrt{3/4},|y|\le \sqrt{4/5}$, it is elementary to derive
$|x-y| \leq  d(x,y) \leq \sqrt{5}|x-y|. $
So $\epsilon \leq |x-y| \leq 4  \eta \left(l(l+1)\right)^{-1/2}$ gives $\epsilon \leq |\sigma-\tau| \leq  4 \sqrt{5}\eta\left(l(l+1)\right)^{-1/2}$, and thus
\ben \label{intermediate-lower-bound-2}
\mathcal{I}_{1}
\geq \frac{1}{5^{2}} \int_{|x| \leq \sqrt{3/4},|y| \leq \sqrt{4/5}} |F_{3+}(x)-F_{3+}(y)|^2 K^{\epsilon, h(\eta,l)}(|x-y|)
dx dy,
 \een
where $h(\eta,l) = 4  \eta \left(l(l+1)\right)^{-1/2}$.
By Taylor expansion,
and the basic inequality $(a-b)^{2} \geq a^{2}/2-b^{2},$ similar to \eqref{main-term-and-lower-order}, we have
\beno
|F_{3+}(y)-F_{3+}(x)|^{2} \geq  \frac{1}{2}|\nabla F_{3+}(x) \cdot (y-x)|^{2} - \frac{1}{4}|y-x|^{4} \int_{0}^{1} (1-\kappa)^{2} |(\nabla^{2} F_{3+}) \left(x+\kappa(y-x)\right)|^{2} d \kappa.
\eeno
Plugging which into \eqref{intermediate-lower-bound-2}, we get
\beno  \mathcal{I}_{1}
&\geq& \frac{1}{50}  \int_{|x| \leq \sqrt{3/4},|y|\le \sqrt{4/5}} |\nabla F_{3+}(x) \cdot (y-x)|^{2} K^{\epsilon, h(\eta,l)}(|x-y|)
dxdy
\\&&- \frac{1}{100} \int_{0}^{1}\int_{|x| \leq \sqrt{3/4},|y|\le \sqrt{4/5}}  (1-\kappa)^{2} |(\nabla^{2} F_{3+}) \left(x+\kappa(y-x)\right)|^{2}  |x-y|^{4}K^{\epsilon, h(\eta,l)}(|x-y|)
d \kappa dxdy
\\&:=& \frac{1}{50}\mathcal{I}_{1,1} - \frac{1}{100}\mathcal{I}_{1,2}
. \eeno
For $x$ with $|x| \leq \sqrt{3/4}$, if $|\nabla F_{3+}(x)| \neq 0$, set the unit vector $e = \nabla F_{3+}(x)/|\nabla F_{3+}(x)|$. Then we have
\ben \label{mathcal-I-11}
\mathcal{I}_{1,1} =
\int_{|x| \leq \sqrt{3/4}}  |\nabla F_{3+}(x)|^{2}  \mathrm{1}_{|\nabla F_{3+}(x)| \neq 0}
 \big(\int_{|y|\le \sqrt{4/5}}  |e \cdot (y-x)|^{2} K^{\epsilon, h(\eta,l)}(|x-y|) dy \big) dx
. \een
Note that if $\eta$ is small enough, then for $x$ with $|x|\leq  \sqrt{3/4}$, the condition $|x-y| \leq h(\eta,l) \leq 4\eta$ gives $|y| \leq \sqrt{4/5}$. Then we have
\beno
\int_{|y|\le \sqrt{4/5}}  |e \cdot (y-x)|^{2} K^{\epsilon, h(\eta,l)}(|x-y|) dy &=& |\ln \epsilon|^{-1} \int_{\epsilon \leq |u| \leq 4 \eta \left(l(l+1)\right)^{-1/2}}  \f{|e \cdot u|^{2}}{|u|^{4}} d u
\\&=& |\ln \epsilon|^{-1} \int_{0}^{2\pi}\int_{\epsilon}^{4 \eta \left(l(l+1)\right)^{-1/2}} \frac{\cos^{2}\theta}{r} d r d\theta
\\&=&\pi|\ln \epsilon|^{-1} \big(\ln(4\eta)- \frac{1}{2}\ln(l(l+1)) + |\ln \epsilon|\big)
. \eeno
Plugging which into \eqref{mathcal-I-11}, we get
\beno \mathcal{I}_{1,1} = \pi |\ln \epsilon|^{-1} \big(\ln(4\eta)- \frac{1}{2}\ln(l(l+1)) + |\ln \epsilon|\big) \|F_{3+}\|_{\dot{H}^1(B_{\sqrt{3}/2})}^2.
\eeno

Let us deal with $\mathcal{I}_{1,2}.$ For fixed $\kappa, y$, in the change of variable $u=x+\kappa(y-x)$, one has $du = (1-\kappa)^{2} dx$. Since $u \in B(\kappa y, (1-\kappa)\sqrt{3/4}) \subset B(0, \sqrt{4/5})$, we get
\beno
\mathcal{I}_{1,2} &=& \int_{0}^{1}\int_{|x| \leq \sqrt{3/4},|y|\le \sqrt{4/5}}  (1-\kappa)^{2} |(\nabla^{2} F_{3+}) \left(x+\kappa(y-x)\right)|^{2} |x-y|^{4}K^{\epsilon, h(\eta,l)}(|x-y|)
d \kappa dxdy
\\&\leq& \int_{0}^{1}\int_{|y|\le \sqrt{4/5}}  \bigg(\int_{|u|\le \sqrt{4/5}} |(\nabla^{2} F_{3+}) (u)|^{2} (1-\kappa)^{-4}|u-y|^{4}K^{\epsilon, h(\eta,l)}((1-\kappa)^{-1}|u-y|) du\bigg)
d \kappa dy
\\&=& \int_{0}^{1} \int_{|u|\le \sqrt{4/5}} |(\nabla^{2} F_{3+}) \left(u\right)|^{2}
\bigg(\int_{|y|\le \sqrt{4/5}} (1-\kappa)^{-4}|u-y|^{4}K^{\epsilon, h(\eta,l)}((1-\kappa)^{-1}|u-y|) dy \bigg)
d \kappa du
\\&\leq& \pi 16 \eta^{2} \left(l(l+1)\right)^{-1}
|\ln \epsilon|^{-1}\int_{0}^{1} \int_{|u|\le \sqrt{4/5}} |(\nabla^{2} F_{3+}) \left(u\right)|^{2}
(1-\kappa)^{2}
d \kappa du
\\&\leq& 8\pi \eta^{2} \left(l(l+1)\right)^{-1} |\ln \epsilon|^{-1} \|F_{3+}\|_{\dot{H}^2(B_{2/\sqrt{5}})}^2,
\eeno
where we use
\beno
&&\int_{|y|\le \sqrt{4/5}} (1-\kappa)^{-4}|u-y|^{4}K^{\epsilon, h(\eta,l)}((1-\kappa)^{-1}|u-y|) dy
\\&=& |\ln \epsilon|^{-1} \int_{|y|\le \sqrt{4/5}} \mathrm{1}_{(1-\kappa) \epsilon \leq |u-y| \leq (1-\kappa)4 \eta \left(l(l+1)\right)^{-1/2}} dy
\leq
|\ln \epsilon|^{-1} \pi 16 \eta^{2} \left(l(l+1)\right)^{-1} (1-\kappa)^{2}.
\eeno

Patching together the above estimates for $\mathcal{I}_{1,1}, \mathcal{I}_{1,2}, \mathcal{I}_{2}$, we get
\beno
\mathcal{A}^{\epsilon}_{l,3,+} &\geq& |\ln \epsilon|^{-1}
C_{1}\big(\ln(4\eta)- \frac{1}{2}\ln(l(l+1)) + |\ln \epsilon|\big) \|F_{3+}\|_{\dot{H}^1(B_{2/\sqrt{5}})}^2
\\&&- |\ln \epsilon|^{-1}  C_{2} \eta^{2} \left(l(l+1)\right)^{-1} \|F_{3+}\|_{\dot{H}^2(B_{2/\sqrt{5}})}^2
- C_{3},
\eeno
for some universal constants $C_{1} , C_{2}, C_{3}$. By the equivalence \eqref{equivalence-to-1-6-2}, the relations
\eqref{order-1-lower-bound} and \eqref{order-2-upper-bound} in Lemma \ref{relation-order-1-2},
we have
\beno  \mathcal{A}^{\epsilon}_{l}
\geq |\ln \epsilon|^{-1}\big( C_{1}\big(\ln(4\eta)- \frac{1}{2}\ln(l(l+1)) + |\ln \epsilon|\big) - C_{2} \eta^{2} \big) l(l+1) -C_{3}.
 \eeno
Since $(l(l+1))^{1/2} \leq \eta \epsilon^{-1}$, we have $\ln(4\eta)- \frac{1}{2}\ln(l(l+1)) + |\ln \epsilon| \geq  \ln 4$. By taking $\eta$ small such that $C_{2} \eta^{2}  \leq \frac{1}{2}C_{1} \ln 4$, we have
\beno  \mathcal{A}^{\epsilon}_{l}
\geq  \frac{1}{2} C_{1}|\ln \epsilon|^{-1}\big(\ln(4\eta)- \frac{1}{2}\ln(l(l+1)) + |\ln \epsilon|\big) l(l+1) -C_{3}.
 \eeno
Thanks to \eqref{reduce-to-the-exact-form}, we get
\ben \label{case-1-small}
\mathcal{A}^{\epsilon}_{l} \geq   \frac{1}{2} \frac{\ln 4}{1-\ln
\eta} C_{1}|\ln \epsilon|^{-1} \big(1- \frac{1}{2}\ln(l(l+1)) + |\ln \epsilon|\big) l(l+1) -C_{3}.
 \een

{\it Step 2:  $\epsilon^{2} l(l+1) \geq R^{2}$.}  Let $\zeta$ be a smooth function on $\mathbb{R}$ with compact support verifying that $0\le \zeta\le1$, $\zeta(x)=1$ if $4\geq x \geq 2$ and $\zeta(x)=0$ if $x \leq  1$ and $x \geq 5$.
We have
\ben \label{main-and-lower-term}
\mathcal{A}^{\epsilon}_{l} &\ge& \int_{\mathbb{S}^2\times\mathbb{S}^2}\left(|Y^{m}_{l}(\sigma)|^{2}+|Y^{m}_{l}(\tau)|^{2}-2Y^{m}_{l}(\sigma)Y^{m}_{l}(\tau)\right)
K^{\epsilon}(|\sigma-\tau|) \zeta(\epsilon^{-1}|\sigma-\tau|) d\sigma d\tau
\\&\geq& \mu_{0}|\ln\epsilon|^{-1}\epsilon^{-2} - 2\int_{\mathbb{S}^2\times\mathbb{S}^2}Y^{m}_{l}(\sigma)Y^{m}_{l}(\tau) K^{\epsilon}(|\sigma-\tau|) \zeta(\epsilon^{-1}|\sigma-\tau|) d\sigma d\tau
:= \mu_{0}|\ln\epsilon|^{-1}\epsilon^{-2} - 2\mathcal{B}^{\epsilon}_{l}, \nonumber
\een
where we use
$ \int_{\mathbb{S}^{2}} K^{\epsilon}(|\sigma-\tau|) \zeta(\epsilon^{-1}|\sigma-\tau|) d\sigma \geq |\ln\epsilon|^{-1} \int_{\mathbb{S}^{2}} |\sigma-\tau|^{-4} \mathrm{1}_{2\epsilon \leq |\sigma-\tau| \leq 4\epsilon}
 d\sigma = \mu_{0} |\ln\epsilon|^{-1} \epsilon^{-2}$ for some constant $\mu_{0}$ and $|Y^{m}_{l}|_{L^{2}(\mathbb{S}^{2})}=1$.

Since $(-\Delta_{\mathbb{S}^{2}})Y^{m}_{l} = l(l+1)Y^{m}_{l}$, by the decomposition  \eqref{decomposition},
we have
\beno
\mathcal{B}^{\epsilon}_{l} &=& [l(l+1)]^{-1}\int_{\mathbb{S}^2\times\mathbb{S}^2}(-\Delta_{\mathbb{S}^{2}})Y^{m}_{l}(\sigma)Y^{m}_{l}(\tau)
K^{\epsilon}(|\sigma-\tau|) \zeta(\epsilon^{-1}|\sigma-\tau|) d\sigma d\tau
\\ &=& |\ln\epsilon|^{-1}\sum_{i=1}^3\sum_{j=+,-} [l(l+1)]^{-1}\int_{\mathbb{S}^2\times\mathbb{S}^2}(-\Delta_{\mathbb{S}^{2}})(\vartheta_{ij}Y^{m}_{l})
(\sigma)Y^{m}_{l}(\tau)|\sigma-\tau|^{-4}\zeta(\epsilon^{-1}|\sigma-\tau|) d\sigma d\tau
\\&:=& |\ln\epsilon|^{-1}[l(l+1)]^{-1} \sum_{i=1}^3\sum_{j=+,-}  \mathcal{B}^{\epsilon}_{l,i,j}.
\eeno

It suffices to consider $\mathcal{B}^{\epsilon}_{l,3,+}$. For simplicity, set $F(y):= Y^{m}_{l}(y_1,y_2, \sqrt{1-y_1^2-y_2^2})$. Recall $F_{3+}(x)=(\vartheta_{3+}Y^{m}_{l})(x_1,x_2, \sqrt{1-x_1^2-x_2^2})$. We make the following of change variables:
\ben \label{change-sigma-to-x}
\sigma = (\sigma_1, \sigma_2, \sigma_3) \in \{ \sigma \in \mathbb{S}^{2}, \sigma_{3} \geq \sqrt{1/5} \} \rightarrow x = (\sigma_1, \sigma_2)\in \{ x \in \mathbb{R}^{2}, |x| \leq \sqrt{4/5} \}, \\
\tau = (\tau_1, \tau_2, \tau_3) \in \{ \tau \in \mathbb{S}^{2}, \tau_3 \geq \sqrt{1/5} \}  \rightarrow y = (\tau_1, \tau_2)\in \{ y \in \mathbb{R}^{2}, |y| \leq \sqrt{4/5} \}.  \nonumber \een
In the above changes, the Jacobean is $\frac{d \sigma}{d x} = \f{1}{\sqrt{1-|x|^{2}}}$ and $\frac{d\tau}{d y} = \f{1}{\sqrt{1-|y|^{2}}}$.
Then we have
\beno
\mathcal{B}^{\epsilon}_{l,3,+} &=& \int_{\mathbb{S}^2\times\mathbb{S}^2} (-\Delta_{\mathbb{S}^{2}})(\vartheta_{3+}Y^{m}_{l})
(\sigma)Y^{m}_{l}(\tau)|\sigma-\tau|^{-4}\zeta(\epsilon^{-1}|\sigma-\tau|) d\sigma d\tau
\\&=&\int_{\mathbb{S}^2\times\mathbb{S}^2}(-\sum\limits_{1\le i<j\le 3} (\Omega_{ij}^2 \vartheta_{3+}Y^{m}_{l})(\sigma_1, \sigma_2, \sigma_3))Y^{m}_{l}(\tau)|\sigma-\tau|^{-4}\zeta(\epsilon^{-1}|\sigma-\tau|) d\sigma d\tau
\\&:=& \sum\limits_{1\le i<j\le 3}I_{ij} =I_{13} + I_{23} +I_{12}.
\eeno
Here $\Omega_{ij}=\sigma_i \pa_j- \sigma_j \pa_i$.
It is easy to   see that for $i=1,2$, \beno \pa_i F_{3+}(x)(x_1,x_2)=\f1{\sqrt{1-|x|^2}}\big(-\Omega_{i3}(\vartheta_{3+}Y^{m}_{l})\big)(x,\sqrt{1-|x|^2}), \eeno
which yields
\ben \label{differential-sphere-to-plane}
\big(\Omega_{i3}^2(\vartheta_{3+}Y^{m}_{l})\big)(x,\sqrt{1-|x|^2})=\big((\sqrt{1-|x|^2}\pa_i)^2F_{3+}\big)(x).\een
From which we have
\beno
I_{13} &=& -\int_{|x|,|y|\leq \sqrt{4/5}} \big((\sqrt{1-|x|^2}\pa_1)^2F_{3+}\big)(x)F(y) d^{-4}(x,y)\zeta(\epsilon^{-1}d(x,y))
\\&&\times
\f1{\sqrt{1-|x|^2}}\f1{\sqrt{1-|y|^2}}dxdy
\\&=& -\int_{|x|,|y|\leq \sqrt{4/5}} \big((\pa_1\sqrt{1-|x|^2}\pa_1)F_{3+}\big)(x)F(y)d^{-4}(x,y)\zeta(\epsilon^{-1}d(x,y))
\f1{\sqrt{1-|y|^2}}dxdy.
\eeno
For $y$ with $|y| \leq \sqrt{4/5}$, let us consider the following integral $I(y)$. Using integrating by parts formula, we have
\beno
I(y)&:=&-\int_{|x|\leq \sqrt{4/5}} \big((\pa_1\sqrt{1-|x|^2}\pa_1)F_{3+}\big)(x) d^{-4}(x,y)\zeta(\epsilon^{-1}d(x,y))
dx
\\&=&\int_{|x|\leq \sqrt{4/5}} \big((\sqrt{1-|x|^2}\pa_1)F_{3+}\big)(x)\pa_1 \left(d^{-4}(x,y)\zeta(\epsilon^{-1}d(x,y))\right)
dx
\\&=&-\int_{|x|\leq \sqrt{4/5}} F_{3+}(x) \pa_1\left(\sqrt{1-|x|^2}\pa_1 \left(d^{-4}(x,y)\zeta(\epsilon^{-1}d(x,y))\right)\right)
dx
\\&:=&-\int_{|x|\leq \sqrt{4/5}} F_{3+}(x) K(x,y)
dx,
\eeno
where $K(x,y):=\pa_1\left(\sqrt{1-|x|^2}\pa_1 \left(d^{-4}(x,y)\zeta(\epsilon^{-1}d(x,y))\right)\right)$. Note that we have used $\pa_1F_{3+}(x)=F_{3+}(x)=0$  on the boundary $|x|=\sqrt{4/5}$. Then we have
\ben \label{integrating-by-parts}
|I_{13}| &=& |\int_{|x|,|y|\leq \sqrt{4/5}} F_{3+}(x) F(y) K(x,y)\f1{\sqrt{1-|y|^2}}dxdy|
\\&\leq& \sqrt{5} \big(\int_{|x|,|y|\leq \sqrt{4/5}} |F_{3+}(x)|^{2}  |K(x,y)| dx dy \big)^{1/2}
\big(\int_{|x|,|y|\leq \sqrt{4/5}} |F(y)|^{2}  |K(x,y)| dx dy \big)^{1/2}. \nonumber
\een
Direct calculation gives
\beno
K(x,y) &=& \big(\pa_1\sqrt{1-|x|^2}\big)\left(\pa_1\zeta(\epsilon^{-1}d(x,y))\right)d^{-4} + \sqrt{1-|x|^2}\left(\pa^{2}_1\zeta(\epsilon^{-1}d(x,y))\right)d^{-4}
\\&&+
\big(\pa_1\sqrt{1-|x|^2}\big)\zeta(\epsilon^{-1}d(x,y))\pa_1d^{-4}+
\sqrt{1-|x|^2}\left(\pa_1 \zeta(\epsilon^{-1}d(x,y)) \right)\pa_1 d^{-4}
\\&&+ \sqrt{1-|x|^2}\zeta(\epsilon^{-1}d(x,y))\pa^{2}_1 d^{-4}.
\eeno
Since $|x-y| \leq d(x,y) \leq \sqrt{5}|x-y| \leq 3|x-y|$ and $\zeta$ has support $[1,5]$, we have
\beno \sqrt{1-|x|^2} \leq 1, |\pa_1\sqrt{1-|x|^2}| \leq 2, \\
|\left(\pa_1\zeta(\epsilon^{-1}d(x,y))\right)d^{-4}| \leq C\epsilon^{-5}\mathrm{1}_{\epsilon/3 \leq |x-y| \leq 5 \epsilon},
|\left(\pa^{2}_1\zeta(\epsilon^{-1}d(x,y))\right)d^{-4}| \leq C\epsilon^{-6}\mathrm{1}_{\epsilon/3 \leq |x-y| \leq 5 \epsilon},
\\
|\zeta(\epsilon^{-1}d(x,y))\pa_1d^{-4}| \leq C\epsilon^{-5}\mathrm{1}_{\epsilon/3 \leq |x-y| \leq 5 \epsilon},
|\left(\pa_1 \zeta(\epsilon^{-1}d(x,y)) \right)\pa_1 d^{-4}| \leq C\epsilon^{-5}\mathrm{1}_{\epsilon/3 \leq |x-y| \leq 5 \epsilon},
\\
|\zeta(\epsilon^{-1}d(x,y))\pa^{2}_1 d^{-4}| \leq C\epsilon^{-6}\mathrm{1}_{\epsilon/3 \leq |x-y| \leq 5 \epsilon}.
\eeno
Patching together the above estimates, we get
$
|K(x,y)| \leq C\epsilon^{-6}\mathrm{1}_{\epsilon/3 \leq |x-y| \leq 5 \epsilon},
$
which gives
\beno
\int_{|x|\leq \sqrt{4/5}}  |K(x,y)| dx \leq C\epsilon^{-4}, \int_{|y|\leq \sqrt{4/5}}  |K(x,y)| dy \leq C\epsilon^{-4}.
\eeno
Plugging which into \eqref{integrating-by-parts}, we have
\beno
|I_{13}| \leq C \epsilon^{-4} |F_{3+}|_{L^{2}(B(\sqrt{4/5}))} |F|_{L^{2}(B(\sqrt{4/5}))} C \epsilon^{-4}
\eeno
where we use $|F_{3+}|_{L^{2}(B(\sqrt{4/5}))} \leq |Y^{m}_{l}|_{L^{2}(\mathbb{S}^{2})}=1$ and
$|F|_{L^{2}(B(\sqrt{4/5}))}\leq |Y^{m}_{l}|_{L^{2}(\mathbb{S}^{2})}=1$.
One can use similar techniques to deal with $I_{23}, I_{12}$ and they also have upper bound $C \epsilon^{-4}$.
Finally, we arrive at
\beno \mathcal{B}^{\epsilon}_{l} \leq C |\ln\epsilon|^{-1} [l(l+1)]^{-1} \epsilon^{-4}  \leq C R^{-2} |\ln\epsilon|^{-1} \epsilon^{-2}.  \eeno
From which together with \eqref{main-and-lower-term}, we have
\ben \label{case-2-large}
\mathcal{A}^{\epsilon}_{l} \geq |\ln\epsilon|^{-1}\epsilon^{-2}(\mu_{0}-CR^{-2}).
\een

{\it Step 3:  $\epsilon^{2}l(l+1)\geq \eta^{2}$.} Here $\eta$ is fixed in {\it Step 1}. Since $\epsilon^{2}l(l+1)\geq \eta^{2}$, then $(M\epsilon)^{2}l(l+1)\geq M^{2}\eta^{2}$. Choosing $M$ large enough such that $C(M\eta)^{-2} \leq \mu_{0}/2$.
Applying the estimate \eqref{case-2-large} with $\epsilon:= M\epsilon, R:= M \eta$, we obtain that
\beno
\mathcal{A}^{M\epsilon}_{l} \geq |\ln (M\epsilon)|^{-1}\epsilon^{-2}M^{-2}(\mu_{0}-C(M\eta)^{-2}) \geq |\ln (M\epsilon)|^{-1}\epsilon^{-2}M^{-2}\mu_{0}/2.
\eeno
From which we get
\ben \label{case-3-middle}
\mathcal{A}^{\epsilon}_{l} \geq |\ln \epsilon|^{-1}|\ln (M\epsilon)|\mathcal{A}^{M\epsilon}_{l}  \geq |\ln \epsilon|^{-1}\epsilon^{-2}M^{-2}\mu_{0}/2.
\een

Patching together \eqref{case-1-small} and
\eqref{case-3-middle}, recalling the definition of $W^{\epsilon}$ in \eqref{charicter function}, we have
\beno
\mathcal{A}^{\epsilon}_{l} \geq C_{1} W^{\epsilon}\big((l(l+1))^{1/2}\big) - C_{2}.
\eeno
for some universal constants $C_{1}$ and $C_{2}$, which ends the proof.
\end{proof}

\begin{rmk} \label{2-epsilon-statement}
If we the function $K^{\epsilon}(r)$
in Proposition \ref{anisotropic-key-pro} is changed to $K^{\epsilon}(r) = |\ln \epsilon|^{-1} r^{-4}\mathrm{1}_{2 \geq r \geq 2\epsilon}$, the results in Proposition \ref{anisotropic-key-pro} are still valid.
\end{rmk}

\begin{lem}\label{relation-order-1-2}
Let $F_{3+}(x):=(\vartheta_{3+}Y^{m}_{l})(x_1,x_2, \sqrt{1-x_1^2-x_2^2}), F_{3-}(x):=(\vartheta_{3+}Y^{m}_{l})(x_1,x_2, -\sqrt{1-x_1^2-x_2^2})$ for $x=(x_1,x_2)$. Similarly, $F_{i+}, F_{i-}$ for $i=1,2$ can be defined. Then we have
\ben \label{order-1-lower-bound} \sum_{i=1}^{3}\sum_{j=+,-}\|F_{ij}\|_{\dot{H}^1(B_{2/\sqrt{5}})}^2 \sim (l(l+1)),
\\  \label{order-2-upper-bound}\sum_{i=1}^{3}\sum_{j=+,-}\|F_{ij}\|_{\dot{H}^2(B_{2/\sqrt{5}})}^2 \lesssim (l(l+1))^{2}. \een
\end{lem}
\begin{proof}
Observe  that $ |(-\triangle_{\SS^2})^{1/2}  (\vartheta_{3+}Y^{m}_{l})|_{L^2(\SS^2)}^2=\int_{ \SS^2} \big((-\triangle_{\SS^2})(\vartheta_{3+}Y^{m}_{l})\big)(\sigma) (\vartheta_{3+}Y^{m}_{l})(\sigma)d\sigma.$
Thanks to the fact $(-\triangle_{\SS^2}f)(\sigma)=-\sum\limits_{1\le i<j\le 3} (\Omega_{ij}^2 f)(\sigma_1, \sigma_2, \sigma_3)$ with $\sigma=(\sigma_1, \sigma_2, \sigma_3)$ and $\Omega_{ij}=\sigma_i \pa_j- \sigma_j \pa_i$, by the change of variable in  \eqref{change-sigma-to-x}, we obtain that
\beno &&|(-\triangle_{\SS^2})^{1/2}  (\vartheta_{3+}Y^{m}_{l})|_{L^2(\SS^2)}^2
\\&=& - \int_{|x|\le \sqrt{\f45}} \sum\limits_{1\le i<j\le 3} \big(\Omega_{ij}^2(\vartheta_{3+}Y^{m}_{l})\big)(x,\sqrt{1-|x|^2})(\vartheta_{3+}Y^{m}_{l})(x,\sqrt{1-|x|^2})\f1{\sqrt{1-|x|^2}}dx.\eeno
Recalling \eqref{differential-sphere-to-plane}, we have
\ben  &&|(-\triangle_{\SS^2})^{1/2}  (\vartheta_{3+}Y^{m}_{l})|_{L^2(\SS^2)}^2
\nonumber\\
&=&-\int_{|x|\le \sqrt{\f45}}  \big(\pa_1(\sqrt{1-|x|^2}\pa_1)F_{3+}(x)\big)F_{3+}(x)dx -\int_{|x|\le \sqrt{\f45}}  \big(\pa_2(\sqrt{1-|x|^2}\pa_2) F_{3+}(x)\big)F_{3+}(x) dx
\nonumber\\
&&-\int_{|x|\le \sqrt{\f45}}  (\Omega_{12})^2F_{3+}(x)F_{3+}(x)\f1{\sqrt{1-|x|^2}}dx
\nonumber\\ \label{transfer-from-sphere-to-plane}
&=&\int_{|x|\le \sqrt{\f45}} \sqrt{1-|x|^2}\left(|\pa_1F_{3+}(x)|^{2} + |\pa_2F_{3+}(x)|^{2}\right) dx
+ \int_{|x|\le \sqrt{\f45}}  \f{|\Omega_{12}F_{3+}(x)|^{2}}{\sqrt{1-|x|^2}}dx
.\een
It is easy to check the following two results,
\beno \int_{|x|\le \sqrt{\f45}} \sqrt{1-|x|^2}\left(|\pa_1F_{3+}(x)|^{2} + |\pa_2F_{3+}(x)|^{2}\right) dx \leq \|F_{3+}(x)\|_{\dot{H}^1(B_{2/\sqrt{5}})}^2,
 \\
\int_{|x|\le \sqrt{\f45}}  \f{|\Omega_{12}F_{3+}(x)|^{2}}{\sqrt{1-|x|^2}}dx \leq \sqrt{5} \int_{|x|\le \sqrt{\f45}}  |(x_{1}\pa_{2}-x_{2}\pa_{1})F_{3+}(x)|^{2}dx \leq \frac{4\sqrt{5}}{5} \|F_{3+}(x)\|_{\dot{H}^1(B_{2/\sqrt{5}})}^2. \eeno
Plugging which into \eqref{transfer-from-sphere-to-plane}, we get
\beno  |(-\triangle_{\SS^2})^{1/2}  (\vartheta_{3+}Y^{m}_{l})|_{L^2(\SS^2)}^2
\leq (1+\frac{4\sqrt{5}}{5})\|F_{3+}(x)\|_{\dot{H}^1(B_{2/\sqrt{5}})}^2  \leq 3 \|F_{3+}(x)\|_{\dot{H}^1(B_{2/\sqrt{5}})}^2
.\eeno
On the other direction, we have
\beno  |(-\triangle_{\SS^2})^{1/2}  (\vartheta_{3+}Y^{m}_{l})|_{L^2(\SS^2)}^2
&\geq&\int_{|x|\le \sqrt{\f45}} \sqrt{1-|x|^2}\left(|\pa_1F_{3+}(x)|^{2} + |\pa_2F_{3+}(x)|^{2}\right) dx
\\ &\geq& \sqrt{1/5}\int_{|x|\le \sqrt{\f45}} \left(|\pa_1F_{3+}(x)|^{2} + |\pa_2F_{3+}(x)|^{2}\right) dx
= \sqrt{1/5}\|F_{3+}\|_{\dot{H}^1(B_{2/\sqrt{5}})}^2
.\eeno
In summary, we get
$  \|F_{3+}(x)\|_{\dot{H}^1(B_{2/\sqrt{5}})}^2 \sim |(-\triangle_{\SS^2})^{1/2}  (\vartheta_{3+}Y^{m}_{l})|_{L^2(\SS^2)}^2
.$
Then by \eqref{decomposition} and \eqref{equivalence-cut-6}, we have
\beno
\sum_{i=1}^{3}\sum_{j=+,-}\|F_{ij}\|_{\dot{H}^1(B_{2/\sqrt{5}})}^2 \sim \sum_{i=1}^{3}\sum_{j=+,-}
|(-\triangle_{\SS^2})^{1/2}  (\vartheta_{ij}Y^{m}_{l})|_{L^2(\SS^2)}^2 \sim |(-\triangle_{\SS^2})^{1/2}  Y^{m}_{l}|_{L^2(\SS^2)}^2.
\eeno
We get \eqref{order-1-lower-bound} by the fact $|(-\triangle_{\SS^2})^{1/2}  Y^{m}_{l}|_{L^2(\SS^2)}^2 = l(l+1)$.
The second result \eqref{order-2-upper-bound} can be derived similarly by working on
$ |(-\triangle_{\SS^2}) (\vartheta_{3+}Y^{m}_{l})|_{L^2(\SS^2)}^2=\int_{ \SS^2} |\big(-\triangle_{\SS^2} (\vartheta_{3+}Y^{m}_{l})\big)(\sigma)|^{2} d\sigma.$
We skip the details.
\end{proof}

 {\bf Acknowledgments.} Ling-Bing He is supported by NSF of China under the grant 11771236. Yu-Long Zhou is supported by the Fundamental Research Funds for the Central Universities, under the grant 19lgpy242.

\bibliographystyle{plain}
\bibliography{CP4}

\end{document}